\documentclass[onecolumn]{article}
\usepackage{float}
\usepackage{capt-of}
\usepackage{amssymb}
\usepackage{amsmath}
\usepackage{geometry}
\usepackage{lscape}
\usepackage{rotating}
\usepackage[onehalfspacing]{setspace}
\usepackage{hyperref}
\usepackage{cite}
\usepackage{appendix}
\usepackage{graphicx}
\usepackage{caption}

\makeatletter
\@ifundefined{dint}{}{\let\dint\relax}
\makeatother
\providecommand{\dint}{\displaystyle\int}

\providecommand{\func}[1]{\operatorname{#1}}

\newtheorem{theorem}{Theorem}

\newtheorem{assumption}[theorem]{Assumption}

\newtheorem{claim}[theorem]{Claim}
\newtheorem{conclusion}[theorem]{Conclusion}

\newtheorem{conjecture}[theorem]{Conjecture}
\newtheorem{corollary}[theorem]{Corollary}

\newtheorem{definition}[theorem]{Definition}
\newtheorem{example}[theorem]{Example}
\let\olddefinition\definition
\let\endolddefinition\enddefinition
\renewenvironment{definition}{\olddefinition\normalfont}{\endolddefinition}
\let\oldassumption\assumption
\let\endoldassumption\endassumption
\renewenvironment{assumption}{\oldassumption\normalfont}{\endoldassumption}
\let\oldexample\example
\let\endoldexample\endexample
\renewenvironment{example}{\oldexample\normalfont}{\endoldexample}

\newtheorem{lemma}[theorem]{Lemma}

\newtheorem{proposition}[theorem]{Proposition}
\newenvironment{remark}{  \refstepcounter{theorem}  \par\medskip\noindent\textbf{Remark~\thetheorem.}\rmfamily\ }{\par\medskip}

\newenvironment{proof}[1][Proof]{\noindent\textbf{#1.} }{\ \rule{0.5em}{0.5em}}

\makeatletter
\@ifundefined{dint}{  \newcommand{\dint}{\displaystyle\int}}{  \renewcommand{\dint}{\displaystyle\int}}
\makeatother
\renewcommand{\dint}{\displaystyle\int}

\begin{document}

\author{{\large Vilimir Yordanov}$\thanks{{\footnotesize Technical
University Vienna, Financial and Actuarial Mathematics and Vienna University
of Economics and Business, Vienna Graduate School of Finance, e-mail:
vilimir.yordanov@tuwien.ac.at and vilimir.yordanov@vgsf.ac.at. Non-academic
e-mail: villyjord@gmail.com.}}$ $\thanks{{\footnotesize Acknowledgment: The
theorem in Section 3 was initially formulated as a conjecture and was
brought to the attention of the author as an open problem by Prof. Zvetan
Ignatov, Faculty of Economics and Business Administration, Sofia University.
For the exact chronology, see the declaration at the end of this paper.
Prof. Zvetan Ignatov passed away on January 14, 2024. V.Y. dedicates the
paper to his memory. The paper is distributed in accordance with the ICMJE
authorship guidelines. The approach to solving the problem, its
implementation, the writing of the paper, and any errors or omissions are
solely the responsibility of the author. He is grateful to Stefan Gerhold
for helpful comments and to Jordan Stoyanov for valuable general suggestions
on the notation.}}$}
\date{{\normalsize June, 2026}}
\title{{\Large On Iterated Lorenz Curves with Applications: The Multivariate
Case }}
\maketitle

\begin{abstract}
It is well known that a Lorenz curve, derived from the distribution function
of a random variable, can itself be viewed as a probability distribution
function of a new random variable \cite{[2]}. In a previous work of ours 
\cite{[3]}, we proved the surprising result that a sequence of consecutive
iterations of this map leads to a non-corner case convergence, independent
of the initial random variable. Namely, the limiting distribution follows a
power-law distribution. In this paper, we generalize our result to the
multivariate setting. We do so using Arnold's type definition \cite{[2]} of
a Lorenz curve, which offers the greatest parsimony among its counterparts.
The situation becomes more complex in higher dimensions as the map affects
not only the marginals but also their dependence structure. Nevertheless, we
prove the equally surprising result that under reasonable restrictions, the
marginals again converge uniformly to a power-law distribution, with an
exponent equal to the golden section. Furthermore, they become independent
in the limit. To emphasize the multifaceted nature of the problem and
broaden the scope of potential applications, our approach utilizes a variety
of mathematical tools, extending beyond very specialized methods.

\ \ \ \ \ \ \ \ \ \ \ \ \ \ \ \ \ \ \ \ \ \ \ \ \ \ \ \ \ \ \ \ \ \ \ \ \ \
\ 

\end{abstract}

\section{\protect\Large Introduction}

\bigskip The classical \textit{Lorenz curve} finds numerous applications in
applied statistics \cite{[2]}, stochastic orders \cite{[1]}, income
inequality \cite{[10]}, risk analysis \cite{[5]}, portfolio theory \cite%
{[43]}, etc. While this is well established in the univariate case, the
situation is different in the multivariate setting, where applications are
rare and lack a uniform framework. This is not due to a scarcity of
opportunities or a diminished need---after all, random vectors are prevalent
in many areas---but rather to the technical challenges inherent in higher
dimensions. The obstacles arise from the outset, as there is no universally
accepted definition of a multivariate \textit{Lorenz curve}. Three are known
in the literature. Historically, the first was introduced by Taguchi, as
described in \cite{[7]} and \cite{[8]}. This was followed by Arnold's
definition in \cite{[4]}. Later, the most elaborate formulation emerged in 
\cite{[9]}, where the concept of the \textit{Lorenz zonoid} was developed. A
detailed discussion of these three versions can be found in \cite{[2]}, with
a more recent technical elaboration in \cite{[6]}.

The concept of the \textit{Lorenz zonoid} is advantageous in that it avoids
difficulties such as defining and inverting a multivariate function, as well
as the need for \textit{order statistics} and their associated machinery.
Nevertheless, working with it remains technically involved and, to a large
extent, non-parsimonious---an important drawback for practical applications.
Interestingly, the simplest form of the multivariate \textit{Lorenz curve},
based on Arnold's definition, remains largely undeveloped, with no major
works dedicated to it \cite{[6]}. We believe that significant insights into
the structure and potential applications of this formulation have yet to be
fully explored.

We demonstrate this in the context of our previous work \cite{[3]}, with the
present paper serving as a generalization of the results established there.
As elaborated in \cite{[14]} and \cite{[4]}, the distribution function
(d.f.) of an arbitrary univariate random variable with a finite first moment
can be reconstructed from its mean and the \textit{Lorenz curve} it
generates. However, the \textit{Lorenz curve} itself can be viewed as a d.f.
of a new random variable. By focusing on this new variable and its
stochastic properties, we can characterize the parent distribution in a
potentially more convenient way \cite{[4]}, since the derived random
variable is, by construction, guaranteed to have finite moments.
Additionally, due to general \textit{Lorenz curve} properties, its d.f. is
convex and defined on the domain $[0,1]$, allowing for finer numerical
calculations and facilitating the application of powerful convex analysis
tools. Rather than following the classical approach of estimating the
moments of this derived random variable, we can take an alternative
perspective: deriving its \textit{Lorenz curve} and treating it as the d.f.
of yet another new random variable. This iterative process naturally
unfolds. In \cite{[3]}, we proved that the limiting d.f. corresponds to a
random variable following a power-law distribution with an exponent equal to
the golden section. This was a surprising result, as it holds regardless of
the parent distribution. The remarkable presence of the golden section in
this context---manifesting within the \textit{Lorenz curve} framework---was
an unexpected and graceful discovery. Given this, it is natural to seek a
similar iterative tool for characterizing multivariate distributions. We
prove that under reasonable restrictions, an analogous result holds in the
multivariate case when Arnold's definition of the \textit{Lorenz curve} is
used in the iteration process.

The paper presents key mathematical results, reserving detailed applications
and further theoretical analysis for future work. This work is organized as
follows: \textit{Section 2} defines the bivariate \textit{Lorenz curve} and
the iterative map under consideration. \textit{Section 3} establishes the
main theorem, the proof of which is divided into several parts for clarity
and relies on the appendices. Following this, \textit{Section 4} analyzes
special cases that arise from initial distributions coinciding with the 
\textit{Fr\'{e}chet-Hoeffding} \textit{bounds}. \textit{Section 5 }%
generalizes the results to a multivariate setting\textit{. Section 6}
provides applications in quantitative finance with a special focus on
dependence modeling, risk, and portfolio theory. Machine learning
implications are also discussed. Finally, the appendices provide the
extensive technical details that support the proofs.

\section{\protect\Large Two-dimensional Lorenz curve iteration }

We present our results using a main theorem with supporting lemmas, claims,
and corollaries. This section begins with several key definitions and
observations before formulating the main theorem.

\begin{definition}
\label{main-def} Let $X=(X_{1},X_{2})$ be a non-negative random variable
with a distribution function $F_{12}(x_{1},x_{2})$ such that $%
0<E(X_{i})<+\infty $ for $i=1,2$ and $0<E(X_{1}X_{2})<+\infty $. Denoting by 
$F_{1}$ and $F_{2}$ the marginal distribution functions of $X_{1}$ and $%
X_{2} $, respectively, we define the bivariate \textit{Lorenz curve} as 
\begin{equation}
L_{F_{12}}(x_{1},x_{2})=\frac{\dint\nolimits_{-\infty
}^{s_{1}}\dint\nolimits_{-\infty }^{s_{2}}u_{1}u_{2}dF_{12}(u_{1},u_{2})}{%
\mu _{12}^{F}},0\leq x_{1}\leq 1,0\leq x_{2}\leq 1,  \label{1}
\end{equation}

where $s_{1}=F_{1}^{-1}(x_{1})$ and $s_{2}=F_{2}^{-1}(x_{2})$ are the
univariate quantiles, given by the generalized inverse $F^{-1}(u)=\inf
\left\{ y:F(y)\geq u\right\} $ and $\mu _{12}^{F}=\int_{-\infty }^{+\infty
}\int_{-\infty }^{+\infty }u_{1}u_{2}dF_{12}(u_{1},u_{2})<+\infty $ is the
mean $E(X_{1}X_{2})$. We will denote by $L$ the \textit{Lorenz curve} above
and by $\mathcal{L}$ the operator $\mathcal{L}(F_{12}(x_{1},x_{2}))(.):[0,1]%
\longrightarrow \lbrack 0,1]$ which maps $F_{12}(.)$ to $L_{F_{12}}(.)$.
\end{definition}

Using the above notation, we can also define the marginal \textit{Lorenz
curves} by $L_{F_{1}}(x_{1})=L_{F_{12}}(x_{1},1)$ and $%
L_{F_{2}}(x_{2})=L_{F_{12}}(1,x_{2})$. Clearly, if $X_{1}$ and $X_{2}$ are
independent, then $F_{12}(x_{1},x_{2})=F_{1}(x_{1})F_{2}(x_{2})$ and $%
L_{F_{12}}(x_{1},x_{2})=L_{F_{1}}(x_{1})L_{F_{2}}(x_{2})$ hold. We will use
the notations $L_{F_{12}}(x_{1},x_{2})$ and $L^{F_{12}}(x_{1},x_{2})$
interchangeably, emphasizing on the distribution function $F$ that generates 
$L$ and the same logic holds for $\mu _{12}^{F}$. Often for convenience we
will also just write $F(x_{1},x_{2})$ instead of $F_{12}(x_{1},x_{2})$ and $%
\mu ^{F}$ instead of $\mu _{12}^{F}.$

It is interesting to observe that for the above-defined \textit{Lorenz curve}%
, the following technical result holds (equations (11) and (12) in \cite{[6]}%
) in case the density exists.

\begin{claim}
\label{aux-cl} If \bigskip $X=(X_{1},X_{2})$ possesses a density, then the
following alternative representation of (\ref{1}) holds%
\begin{equation}
L_{F_{12}}(x_{1},x_{2})=\dint\nolimits_{0}^{x_{1}}\dint%
\nolimits_{0}^{x_{2}}A(u_{1},u_{2})du_{1}du_{2},  \label{2}
\end{equation}%
with 
\begin{equation}
A(u_{1},u_{2})=\frac{1}{E(X_{1}X_{2})}\frac{%
F_{1}^{-1}(u_{1})F_{2}^{-1}(u_{2})f_{12}(F_{1}^{-1}(u_{1}),F_{2}^{-1}(u_{2}))%
}{f_{1}(F_{1}^{-1}(u_{1}))f_{2}(F_{2}^{-1}(u_{2}))},  \label{3}
\end{equation}

where we denote the density of\ $X=(X_{1},X_{2})$ by $f_{12}(.,.)=f(.,.)$
and the corresponding marginal densities by $f_{1}(.)$ and $f_{2}(.)$.
\end{claim}

\begin{proof}
The proof follows directly from a change of variables and is not provided in 
\cite{[6]}. We will compose it here for completeness, because similar
mathematical operations will be used below. Take $y_{1}=F_{1}^{-1}(u_{1})$
and $y_{2}=F_{2}^{-1}(u_{2})$ and plug $A(u_{1},u_{2})$ inside the integral
to get consecutively%
\begin{eqnarray}
&&\dint\nolimits_{0}^{x_{1}}\dint\nolimits_{0}^{x_{2}}\frac{1}{E(X_{1}X_{2})}%
\frac{%
F_{1}^{-1}(u_{1})F_{2}^{-1}(u_{2})f_{12}(F_{1}^{-1}(u_{1}),F_{2}^{-1}(u_{2}))%
}{f_{1}(F_{1}^{-1}(u_{1}))f_{2}(F_{2}^{-1}(u_{2}))}du_{1}du_{2}  \notag \\
&=&\frac{1}{E(X_{1}X_{2})}\dint\nolimits_{0}^{F_{1}^{-1}(x_{1})}\dint%
\nolimits_{0}^{F_{2}^{-1}(x_{2})}\frac{y_{1}y_{2}f_{12}(y_{1},y_{2})}{%
f_{1}(y_{1})f_{2}(y_{2})}dF_{1}(y_{1})dF_{2}(y_{2})  \notag \\
&=&\frac{1}{E(X_{1}X_{2})}\dint\nolimits_{0}^{F_{1}^{-1}(x_{1})}\dint%
\nolimits_{0}^{F_{2}^{-1}(x_{2})}y_{1}y_{2}f_{12}(y_{1},y_{2})dy_{1}dy_{2} 
\notag \\
&=&\frac{\dint\nolimits_{0}^{F_{1}^{-1}(x_{1})}\dint%
\nolimits_{0}^{F_{2}^{-1}(x_{2})}y_{1}y_{2}dF_{12}(y_{1},y_{2})}{\mu _{12}}=%
\frac{\dint\nolimits_{0}^{s_{1}}\dint%
\nolimits_{0}^{s_{2}}y_{1}y_{2}dF_{12}(y_{1},y_{2})}{\mu _{12}},  \label{4}
\end{eqnarray}

which is exactly $L_{F_{12}}(x_{1},x_{2})$.
\end{proof}

\section{Main result}

We use the above Arnold's definition of a \textit{Lorenz curve} and pose the
following theorem in the spirit of the univariate case from \cite{[3]}:

\begin{theorem}
\label{main-th} Given the above definition of a bivariate \textit{Lorenz
curve} and its associated operator, if for $n=0,1,..$we consider the
sequence $L_{n+1}^{F}(x_{1},x_{2}),$ defined by%
\begin{equation}
L_{n+1}^{F}(x_{1},x_{2})=\frac{\int_{0}^{L_{1}^{n,-1}(x_{1})}%
\int_{0}^{L_{2}^{n,-1}(x_{2})}u_{1}u_{2}dL_{n}^{F}(u_{1},u_{2})}{%
\int_{0}^{1}\int_{0}^{1}u_{1}u_{2}dL_{n}^{F}(u_{1},u_{2})},  \label{5.1}
\end{equation}

where%
\begin{equation}
L_{1}^{n}(x)=L_{n}^{F}(x,1),\qquad L_{2}^{n}(x)=L_{n}^{F}(1,x),  \label{5.11}
\end{equation}

or equivalently by%
\begin{eqnarray}
L_{0}^{F}(x_{1},x_{2}) &=&L_{F}^{0}(x_{1},x_{2})=\mathcal{L(}F(x_{1},x_{2}))
\notag \\
L_{1}^{F}(x_{1},x_{2})
&=&L_{F}^{1}(x_{1},x_{2})=L^{L_{0}^{F}}(x_{1},x_{2})=L^{F}(x_{1},x_{2})=L_{F}(x_{1},x_{2})=%
\mathcal{L(}L_{0}(x_{1},x_{2}))  \notag \\
&&...  \notag \\
L_{n+1}^{F}(x_{1},x_{2}) &=&L_{F}^{n+1}(x_{1},x_{2})=L^{L_{n}^{F}}(x)=%
\mathcal{L}(L_{n}(x_{1},x_{2})),  \label{5}
\end{eqnarray}

where in $L$ the subscript $n$ denotes the iteration step, while the
superscript indicates starting distribution $F$, then the functions $%
L_{n}^{F}(x_{1},x_{2})$ are themselves distribution functions on $[0,1]^{2}$%
. When extended to $R^{2}$, this extension is understood in the usual
distribution function sense, namely 
\begin{equation*}
L_{n}^{F}(x_{1},x_{2})=L_{n}^{F}(\pi (x_{1}),\pi (x_{2})),\qquad \pi
(x)=\min \{1,\max \{0,x\}\}.
\end{equation*}

Assume additionally the existence of a strictly positive density for $%
X=(X_{1},X_{2})$ on $(0,1)^{2}$ which is either totally positive of order $2$
$(TP_{2})$ or strictly reverse regular of order $2$ $(SRR_{2})$ in the sense
of \cite{[36]}. Then $L_{n}^{F}(x_{1},x_{2})$ converges uniformly to the
distribution function $G(x_{1},x_{2})$, where 
\begin{equation}
G(x_{1},x_{2})=\pi (x_{1})^{\frac{1+\sqrt{5}}{2}}\pi (x_{2})^{\frac{1+\sqrt{5%
}}{2}},\qquad \pi (x)=\min \{1,\max \{0,x\}\}.  \label{6}
\end{equation}
\end{theorem}

\begin{proof}
We emphasize the following preliminary points, which not only introduce the
necessary notation but also aid in making the exposition of the proof
clearer.

First, we need to verify whether $L_{n}^{F}(x_{1},x_{2})$ can be viewed as a
distribution function of a new bivariate random variable for every $n$.
Clearly, it is enough to consider only the case $n=0$, the others follow by
induction. We use the representation of the bivariate \textit{Lorenz curve}
as a change of measure. Let $P_{F}$ be the law of $X=(X_{1},X_{2})$, and
define the probability measure $Q_{F}$ by 
\begin{equation}
dQ_{F}(y_{1},y_{2})=\frac{y_{1}y_{2}}{E(X_{1}X_{2})}\,dP_{F}(y_{1},y_{2}).
\label{7.11}
\end{equation}%
Equivalently, for every Borel set $A\subseteq \lbrack 0,\infty )^{2}$, 
\begin{equation}
Q_{F}(A)=\frac{E\left[ X_{1}X_{2}1_{\{X\in A\}}\right] }{E(X_{1}X_{2})}.
\label{7.21}
\end{equation}%
This is indeed a probability measure, since $X_{1}X_{2}\geq 0$ and $%
Q_{F}([0,\infty )^{2})=1$.

With this notation, the bivariate \textit{Lorenz curve} can be written as 
\begin{equation}
L_{0}^{F}(x_{1},x_{2})=Q_{F}([0,F_{1}^{-1}(x_{1})]\times \lbrack
0,F_{2}^{-1}(x_{2})]),\qquad 0\leq x_{1},x_{2}\leq 1.  \label{7.31}
\end{equation}%
Thus $L_{0}^{F}$ is the distribution function, in the probability
coordinates $(x_{1},x_{2})$, of the measure $Q_{F}$. Consequently, the usual
distribution function properties follow immediately: boundary conditions,
coordinatewise monotonicity, right-continuity, and the rectangle inequality.

For completeness, we spell out the rectangle inequality. Take $u_{1}\leq
u_{2}$ and $v_{1}\leq v_{2}$, and set 
\begin{equation*}
x_{1,i}=F_{1}^{-1}(u_{i}),\qquad x_{2,j}=F_{2}^{-1}(v_{j}),\qquad i,j=1,2.
\end{equation*}%
Then 
\begin{eqnarray}
\Delta _{L}
&=&L_{0}^{F}(u_{2},v_{2})-L_{0}^{F}(u_{2},v_{1})-L_{0}^{F}(u_{1},v_{2})+L_{0}^{F}(u_{1},v_{1})
\notag \\
&=&Q_{F}((x_{1,1},x_{1,2}]\times (x_{2,1},x_{2,2}])\geq 0.  \label{7.41}
\end{eqnarray}%
If $F$ has density $f_{X_{1},X_{2}}$, this becomes 
\begin{equation}
\Delta _{L}=\frac{\int_{x_{1,1}}^{x_{1,2}}%
\int_{x_{2,1}}^{x_{2,2}}y_{1}y_{2}f_{X_{1},X_{2}}(y_{1},y_{2})\,dy_{2}dy_{1}%
}{E(X_{1}X_{2})}\geq 0.  \label{7.51}
\end{equation}%
Hence the rectangle inequality is satisfied. Therefore $L_{0}^{F}$,
analogously to the univariate case, can be viewed as a distribution function
of a bivariate random variable. Applying the same argument inductively to $%
L_{n}^{F}$ gives the same conclusion for the entire sequence $%
\{L_{n}^{F}\}_{n\geq 0}$.

Second, it should be noted that while the iterations in (\ref{5.1}) for the
bivariate case resemble their univariate counterpart from \cite{[3]} in
structure, they differ in key details. Once the bivariate distribution
functions $L_{n}^{F}(x_{1},x_{2})$ are defined, we must first derive their
marginal distributions and only then determine the iterative process for
them, rather than assuming it a priori. This distinction is made explicit in
the calculations that follow. As observed from \textit{Definition }\ref%
{main-def}, the two iterations coincide---resulting in identical marginal
distributions---only when the initial random variables are independent,
i.e., when $F_{12}(x_{1},x_{2})=F_{1}(x_{1})F_{2}(x_{2}).$

Third, we introduce more precise notation for the marginal distribution
functions and their densities throughout the iterations. The marginal
distribution functions are given by%
\begin{eqnarray}
F_{1}(x) &=&F_{12}(x,+\infty )\text{, }F_{2}(x)=F_{12}(+\infty ,x)  \notag \\
L_{1}^{0}(x) &=&L_{1}^{F}(x)=L_{0}(x,1)\text{, }%
L_{2}^{0}(x)=L_{2}^{F}(x)=L_{0}(1,x)  \notag \\
&&...  \notag \\
L_{1}^{n}(x) &=&L_{1}^{L^{n-1}}(x)=L_{n}(x,1)\text{, }%
L_{2}^{n}(x)=L_{2}^{L^{n-1}}(x)=L_{n}(1,x),  \label{6a}
\end{eqnarray}

and the corresponding marginal densities by%
\begin{eqnarray}
\text{ }f_{1}(x) &=&\int\nolimits_{0}^{+\infty }f_{12}(x,u)du,\text{ }%
f_{2}(x)=\int\nolimits_{0}^{+\infty }f_{12}(u,x)du  \label{6b} \\
l_{1}^{0}(x) &=&\frac{F_{1}^{-1}(x)\int\nolimits_{0}^{+\infty
}uf_{12}(F_{1}^{-1}(x),u)du}{%
f_{1}(F_{1}^{-1}(x))E(X_{1}^{F_{1}}X_{2}^{F_{2}})}\text{, }l_{2}^{0}(x)=%
\frac{F_{2}^{-1}(x)\int\nolimits_{0}^{+\infty }uf_{12}(u,F_{2}^{-1}(x))du}{%
f_{2}(F_{2}^{-1}(x))E(X_{1}^{F_{1}}X_{2}^{F_{2}})}  \notag \\
&&...  \notag \\
l_{1}^{n}(x) &=&\frac{L_{1}^{n-1,-1}(x)\int%
\nolimits_{0}^{1}ul_{12}^{n-1}(L_{1}^{n-1,-1}(x),u)du}{%
l_{1}^{n-1}(L_{1}^{n-1,-1}(x))E(X_{1}^{L_{1}^{n-1}}X_{2}^{L_{1}^{n-1}})}%
\text{, }l_{2}^{n}(x)=\frac{L_{2}^{n-1,-1}(x)\int%
\nolimits_{0}^{1}ul_{12}^{n-1}(u,L_{2}^{n-1,-1}(x))du}{%
l_{2}^{n-1}(L_{2}^{n-1,-1}(x))E(X_{1}^{L_{1}^{n-1}}X_{2}^{L_{1}^{n-1}})}, 
\notag
\end{eqnarray}

where for $n\geq 0,$ $l_{12}^{n}(x_{1},x_{2})$ is the density of $%
L_{F_{12}}^{n}(x_{1},x_{2}),$ with the subscript occasionally omitted for
convenience. The functions $L_{1}^{n}(x_{1})$ and $L_{2}^{n}(x_{2})$
represent the corresponding marginal distribution functions of $%
L^{n}(x_{1},x_{2})$, while $l_{1}^{n}(x_{1})$ and $l_{2}^{n}(x_{2})$ are
their respective marginal densities. Furthermore, for $i=1,2$, $%
X_{i}^{F_{i}} $ indicates that $X_{i}$ follows the distribution function $%
F_{i}$, while $X_{i}^{L_{i}^{n}}$ denotes that $X_{i}$ follows the
distribution function $L_{i}^{n}$. Additionally, $L_{i}^{n,-1}(x)$
represents the generalized inverse of $L_{i}^{n}(x)$.

Fourth, it can be verified that the marginal \textit{Lorenz curves} satisfy
all the standard properties of their univariate counterparts (e.g., as
listed in \cite{[2]} and \cite{[10]}), with one exception: convexity is not
guaranteed. This means the term \textit{Lorenz curve} is somewhat premature
when applied to the marginals. The term \textquotedblleft \textit{marginal
Lorenz curve}\textquotedblright\ remains appropriate for these functions.
This is not because they inherently satisfy the definition of a \textit{%
Lorenz curve}, but rather because they are derived from a bivariate \textit{%
Lorenz curve}, which is itself a valid distribution function. For the sake
of convenience, however, we will often omit the \textquotedblleft
marginal\textquotedblright\ prefix in the subsequent text.

Fifth, the existence of density is a crucial assumption, as will become
clear in the proof below. In certain cases when the distribution lacks a
density, the theorem does not hold. A separate discussion will be devoted to
such situations. The assumption of a strictly positive density implies that
the marginals are continuous and strictly increasing.

Sixth, the assumption that the joint density is either \textit{totally
positive of order 2} ($TP_{2}$) or \textit{reverse regular of order 2} ($%
RR_{2}$), is also crucial for the proof and acts as a suitable foundation to
build upon. Later we will discuss generalizations. For the two definitions%
\footnote{{\footnotesize If the }$TP_{2}${\footnotesize \ concept is
standard and is extensively discussed in the monograph \cite{[36]} as well
as on textbook level in \cite{[38]} and \cite{[24]}, among others, its }$%
RR_{2}${\footnotesize \ }or $RR_{2}$ {\footnotesize counterpart is more
elusive due to the no universally accepted negative dependence analogs of
it. It is discussed both in \cite[page 12]{[36]} and \cite{[32]} with many
variations later on getting popularity. Recently both concepts gained
attention also in the financial mathematics literature, see \cite{[37]}. }},
we will follow the classical text of \cite{[36]}. Although these assumptions
may initially appear restrictive and abstract since both imply high level of
dependence, they are, in fact, widely applicable. Many common
copulas---including \textit{Gaussian}, \textit{Student-t}, \textit{Joe}, 
\textit{Gumbel}, \textit{AMH}, \textit{Cuadras-Aug\'{e}}, and \textit{Frank}%
---as well as the boundary comonotonic, countermonotonic and independence
copulas---satisfy either the $TP_{2}$ or $RR_{2}$ conditions globally or at
least in segments, depending on their parametrization. The latter leads to
outright positive or negative dependence. Since we will prove in the theorem
eventual independence, it is clear that it would be good to apply the
iterative map to distributions subject to strong dependence\footnote{%
{\footnotesize See \textit{Appendix F} for a classification of sample
copulas by }$RR_{2}${\footnotesize \ and }$TP_{2}${\footnotesize .
Information on joint log-concavity is also provided because it implies the
weaker condition of coordinate log-concavity. This weaker condition, which
holds for most standard copulas, is a candidate for a potential regularity
condition in the }$RR_{2}${\footnotesize \ case and will be discussed later.}%
}. Finally, it happens that technically we need to work not just with 
\textit{reverse regular of order 2} ($RR_{2}$) property for the density but
with it in a strict sense, i.e., $SRR_{2}$, which from practical point of
view is not a big restriction since most of the $RR_{2}$ copulas are also $%
SRR_{2}$.

Now we proceed to the proof. Using \textit{Claim }\ref{aux-cl}, the defined
iterative process for the bivariate \textit{Lorenz curves}, and the assumed
notation for the distribution functions and their densities, we obtain the
following%
\begin{equation}
L_{F_{12}}^{0}(x_{1},x_{2})=\dint\nolimits_{0}^{x_{1}}\dint%
\nolimits_{0}^{x_{2}}l_{12}^{0}(u_{1},u_{2})du_{1}du_{2},  \label{7}
\end{equation}

where 
\begin{equation}
l_{12}^{0}(x_{1},x_{2})=\frac{1}{E(X_{1}^{F_{1}}X_{2}^{F_{2}})}\frac{%
F_{1}^{-1}(x_{1})F_{2}^{-1}(x_{2})f_{12}(F_{1}^{-1}(x_{1}),F_{2}^{-1}(x_{2}))%
}{f_{1}(F_{1}^{-1}(x_{1}))f_{2}(F_{2}^{-1}(x_{2}))}.  \label{8}
\end{equation}

Analogously, for any $i\geq 0$, we also have%
\begin{equation}
l_{12}^{i+1}(x_{1},x_{2})=\frac{1}{E(X_{1}^{L_{1}^{i}}X_{2}^{L_{2}^{i}})}%
\frac{%
L_{1}^{i,-1}(x_{1})L_{2}^{i,-1}(x_{2})l_{12}^{i}(L_{1}^{i,-1}(x_{1}),L_{2}^{i,-1}(x_{2}))%
}{l_{1}^{i}(L_{1}^{i,-1}(x_{1}))l_{2}^{i}(L_{2}^{i,-1}(x_{2}))}.  \label{9}
\end{equation}

To better isolate the dependence structure induced by the bivariate
probability distributions, we introduce the corresponding copulas. Following
the established notation, let $c^{i}(.,.)$ denote the copula density of the
joint probability density $l_{12}^{i}(.,.)$ for $i\geq 0$, and let $%
c^{F}(.,.)$ be the copula density of the joint probability density $%
f_{12}(.,.)$. By definition of the copula density, we obtain the identity%
\begin{equation}
l_{12}^{i}(x_{1},x_{2})=c^{i}(L_{1}^{i}(x_{1}),L_{2}^{i}(x_{2}))l_{1}^{i}(x_{1})l_{2}^{i}(x_{2}).
\label{10}
\end{equation}

Substituting (\ref{10}) into (\ref{9}) for $i\geq 0$, we derive%
\begin{eqnarray}
l_{12}^{i+1}(x_{1},x_{2}) &=&\frac{1}{E(X_{1}^{L_{1}^{i}}X_{2}^{L_{2}^{i}})}%
\frac{%
L_{1}^{i,-1}(x_{1})L_{2}^{i,-1}(x_{2})c^{i}(L_{1}^{i}(L_{1}^{i,-1}(x_{1})),L_{2}^{i}(L_{2}^{i,-1}(x_{2})))l_{1}^{i}(L_{1}^{i,-1}(x_{1}))l_{2}^{i}(L_{2}^{i,-1}(x_{2}))%
}{l_{1}^{i}(L_{1}^{i,-1}(x_{1}))l_{2}^{i}(L_{2}^{i,-1}(x_{2}))}  \notag \\
&=&\frac{1}{E(X_{1}^{L_{1}^{i}}X_{2}^{L_{2}^{i}})}%
L_{1}^{i,-1}(x_{1})L_{2}^{i,-1}(x_{2})c^{i}(x_{1},x_{2}).  \label{11}
\end{eqnarray}

Rewriting the left-hand side of (\ref{11}) in terms of the copula density,
we obtain%
\begin{equation}
c^{i+1}(L_{1}^{i+1}(x_{1}),L_{2}^{i+1}(x_{2}))l_{1}^{i+1}(x_{1})l_{2}^{i+1}(x_{2})=%
\frac{1}{E(X_{1}^{L_{1}^{i}}X_{2}^{L_{2}^{i}})}%
L_{1}^{i,-1}(x_{1})L_{2}^{i,-1}(x_{2})c^{i}(x_{1},x_{2}).  \label{12}
\end{equation}

This simplifies to%
\begin{equation}
c^{i+1}(x_{1},x_{2})=\frac{1}{E(X_{1}^{L_{1}^{i}}X_{2}^{L_{2}^{i}})}\frac{%
L_{1}^{i,-1}(L_{1}^{i+1,-1}(x_{1}))L_{2}^{i,-1}(L_{2}^{i+1,-1}(x_{2}))}{%
l_{1}^{i+1}(L_{1}^{i+1,-1}(x_{1}))l_{2}^{i+1}(L_{2}^{i+1,-1}(x_{2}))}%
c^{i}(L_{1}^{i+1,-1}(x_{1}),L_{2}^{i+1,-1}(x_{2})).  \label{13}
\end{equation}

Proceeding further, we obtain%
\begin{equation}
c^{i+2}(x_{1},x_{2})=\frac{1}{E(X_{1}^{L_{1}^{i+1}}X_{2}^{L_{2}^{i+1}})}%
\frac{%
L_{1}^{i+1,-1}(L_{1}^{i+2,-1}(x_{1}))L_{2}^{i+1,-1}(L_{2}^{i+2,-1}(x_{2}))}{%
l_{1}^{i+2}(L_{1}^{i+2,-1}(x_{1}))l_{2}^{i+2}(L_{2}^{i+2,-1}(x_{2}))}%
c^{i+1}(L_{1}^{i+2,-1}(x_{1}),L_{2}^{i+2,-1}(x_{2})).  \label{14}
\end{equation}

Substituting (\ref{13}) into (\ref{14}) gives%
\begin{eqnarray}
&&c^{i+2}(x_{1},x_{2})=\frac{1}{%
E(X_{1}^{L_{1}^{i+1}}X_{2}^{L_{2}^{i+1}})E(X_{1}^{L_{1}^{i}}X_{2}^{L_{2}^{i}})%
}  \notag \\
&&\times \frac{%
L_{1}^{i+1,-1}(L_{1}^{i+2,-1}(x_{1}))L_{2}^{i+1,-1}(L_{2}^{i+2,-1}(x_{2}))L_{1}^{i,-1}(L_{1}^{i+1,-1}(L_{1}^{i+2,-1}(x_{1})))L_{2}^{i,-1}(L_{2}^{i+1,-1}(L_{2}^{i+2,-1}(x_{2})))%
}{%
l_{1}^{i+2}(L_{1}^{i+2,-1}(x_{1}))l_{2}^{i+2}(L_{2}^{i+2,-1}(x_{2}))l_{1}^{i+1}(L_{1}^{i+1,-1}(L_{1}^{i+2,-1}(x_{1})))l_{2}^{i+1}(L_{2}^{i+1,-1}(L_{2}^{i+1,-1}(x_{2})))%
}  \notag \\
&&\times
c^{i}(L_{1}^{i+1,-1}(L_{1}^{i+2,-1}(x_{1})),L_{2}^{i+1,-1}(L_{2}^{i+2,-1}(x_{1}))).
\label{15}
\end{eqnarray}

By trivial induction, we obtain the general form%
\begin{equation}
c^{n}(x_{1},x_{2})=I^{n}D^{n}(x_{1},x_{2})\frac{%
P_{1}^{n}(x_{1})P_{2}^{n}(x_{2})}{Q_{1}^{n}(x_{1})Q_{2}^{n}(x_{2})},
\label{16}
\end{equation}

where $I^{n}$ and $D^{n}(x_{1},x_{2})$ denote the terms%
\begin{eqnarray}
I^{n} &=&\frac{1}{%
E(X_{1}^{F_{1}}X_{2}^{F_{2}})E(X_{1}^{L_{1}^{0}}X_{2}^{L_{2}^{0}})(\Pi
_{i=1}^{n-1}E(X_{1}^{L_{1}^{i}}X_{2}^{L_{2}^{i}}))}  \notag \\
D^{n}(x_{1},x_{2}) &=&c^{F}(L_{1}^{0,-1}\circ ...\circ \text{ }%
L_{1}^{n,-1}(x_{1}),L_{2}^{0,-1}\circ ...\circ \text{ }L_{2}^{n,-1}(x_{2})),
\label{16.1}
\end{eqnarray}

as well as $P_{1}^{n}(x_{1}),$ $Q_{1}^{n}(x_{1}),P_{2}^{n}(x_{2})$ and $%
Q_{2}^{n}(x_{2})$ the terms%
\begin{eqnarray}
P_{1}^{n}(x_{1}) &=&(F_{1}^{-1}\circ ...\circ \text{ }%
L_{1}^{n,-1}(x_{1}))(L_{1}^{0,-1}\circ ...\circ \text{ }L_{1}^{n,-1}(x_{1}))%
\Pi _{i=1}^{n-1}[L_{1}^{i,-1}\circ ...\circ \text{ }L_{1}^{n,-1}(x_{1})] 
\notag \\
Q_{1}^{n}(x_{1}) &=&l_{1}^{0}(L_{1}^{0,-1}\circ ...\circ \text{ }%
L_{1}^{n,-1}(x_{1}))\Pi _{i=1}^{n-1}[l_{1}^{i}(L_{1}^{i,-1}\circ ...\circ 
\text{ }L_{1}^{n,-1}(x_{1}))l_{1}^{n}(L_{1}^{n,-1}(x_{1}))]  \notag \\
P_{2}^{n}(x_{2}) &=&(F_{2}^{-1}\circ ...\circ \text{ }%
L_{2}^{n,-1}(x_{2}))(L_{2}^{0,-1}\circ ...\circ \text{ }L_{2}^{n,-1}(x_{2}))%
\Pi _{i=1}^{n-1}[L_{2}^{i,-1}\circ ...\circ \text{ }L_{2}^{n,-1}(x_{2})] 
\notag \\
Q_{2}^{n}(x_{2}) &=&l_{2}^{0}(L_{2}^{0,-1}\circ ...\circ \text{ }%
L_{2}^{n,-1}(x_{2}))\Pi _{i=1}^{n-1}[l_{2}^{i}(L_{2}^{i,-1}\circ ...\circ 
\text{ }L_{2}^{n,-1}(x_{2}))l_{2}^{n}(L_{2}^{n,-1}(x_{2}))].  \label{16.2}
\end{eqnarray}

Let's return now to (\ref{16}). We can write the equation as 
\begin{eqnarray}
&&c^{n}(x_{1},x_{2})=\frac{P_{1}^{n}(x_{1})P_{2}^{n}(x_{2})D^{n}(x_{1},x_{2})%
}{Q_{1}^{n}(x_{1})Q_{2}^{n}(x_{1})}\frac{I^{n-1}}{%
E(X_{1}^{L_{1}^{n-1}}X_{2}^{L_{2}^{n-1}})}  \notag \\
&=&I^{n-1}\frac{P_{1}^{n}(x_{1})P_{2}^{n}(x_{2})}{%
Q_{1}^{n}(x_{1})Q_{2}^{n}(x_{2})}\frac{D^{n}(x_{1},x_{2})}{%
\dint\nolimits_{0}^{1}\dint%
\nolimits_{0}^{1}L_{1}^{n-1,-1}(u_{1})L_{2}^{n-1,-1}(u_{2})c^{n-1}(u_{1},u_{2})du_{1}du_{2}%
}  \notag \\
&=&I^{n-1}\frac{P_{1}^{n}(x_{1})P_{2}^{n}(x_{2})}{%
Q_{1}^{n}(x_{1})Q_{2}^{n}(x_{2})}\frac{D^{n}(x_{1},x_{2})}{%
\dint\nolimits_{0}^{1}\dint%
\nolimits_{0}^{1}L_{1}^{n-1,-1}(u_{1})L_{2}^{n-1,-1}(u_{2})I^{n-1}\frac{%
P_{1}^{n-1}(u_{1})P_{2}^{n-1}(u_{2})}{Q_{1}^{n-1}(u_{1})Q_{2}^{n-1}(u_{2})}%
D^{n-1}(u_{1},u_{2})du_{1}du_{2}}  \notag \\
&=&\frac{P_{1}^{n}(x_{1})P_{2}^{n}(x_{2})}{Q_{1}^{n}(x_{1})Q_{2}^{n}(x_{2})}%
\frac{D^{n}(x_{1},x_{2})}{\dint\nolimits_{0}^{1}\dint%
\nolimits_{0}^{1}L_{1}^{n-1,-1}(u_{1})L_{2}^{n-1,-1}(u_{2})\frac{%
P_{1}^{n-1}(u_{1})P_{2}^{n-1}(u_{2})}{Q_{1}^{n-1}(u_{1})Q_{2}^{n-1}(u_{2})}%
D^{n-1}(u_{1},u_{2})du_{1}du_{2}}.  \label{69}
\end{eqnarray}

In (\ref{16.1}), we have defined $D^{n}(x_{1},x_{2})$ by%
\begin{equation}
D^{n}(x_{1},x_{2})=c^{F}(L_{1}^{0,-1}\circ ...\circ \text{ }%
L_{1}^{n,-1}(x_{1}),L_{2}^{0,-1}\circ ...\circ \text{ }L_{2}^{n,-1}(x_{2})).
\label{72}
\end{equation}

Applying the results from \textit{Appendices D.1-D.3 }and particularly 
\textit{Theorem }\ref{phi-conv}, we get that the arguments of the function $%
c^{F}(.,.)$ are uniformly subsequentially convergent to constants for $%
x_{1}\in \lbrack \delta ,\eta ]$ and $x_{2}\in \lbrack \delta ,\eta ],$
where $(\delta ,\eta )$ $\subset $ $(0,1)$.

More precisely, for any sequence of indices $\{n_{k}^{(\alpha )}\}_{k\geq 0}$
with $n_{k}^{(\alpha )}\rightarrow \infty $ there exists a (not relabeled)
subsequence, still denoted $\{n_{k}^{(\alpha )}\}$, and constants $%
c_{1}^{(\alpha )}$ and $c_{2}^{(\alpha )}$ such that the uniform
subsequential convergence 
\begin{equation}
D^{n_{k}^{(\alpha )}}(x_{1},x_{2})\ \longrightarrow \ c^{F}(c_{1}^{(\alpha
)},c_{2}^{(\alpha )}),\qquad (x_{1},x_{2})\in \lbrack \delta ,\eta ]^{2},
\label{79}
\end{equation}%
holds. Here $\alpha $ denotes a generic subsequence of $n$, and $%
(c_{1}^{(\alpha )},c_{2}^{(\alpha )})$ belongs to the cluster set of
subsequential limits 
\begin{equation}
\mathcal{C}=\{(c_{1},c_{2}):\ \exists \,n_{k}\rightarrow +\infty \text{ such
that the arguments }\rightarrow (c_{1},c_{2})\text{ uniformly on }[\delta
,\eta ]^{2}\}.  \label{79.1}
\end{equation}

Let us write the convergence (\ref{79}) in $\varepsilon $--form. Fix $%
\varepsilon >0$. Then there exists a natural number $K(\varepsilon )$ such
that for all $k\geq K(\varepsilon )$ and all $(x_{1},x_{2})\in \lbrack
\delta ,\eta ]^{2}$, 
\begin{equation}
\left\vert D^{n_{k}^{(\alpha )}}(x_{1},x_{2})-c^{F}(c_{1}^{(\alpha
)},c_{2}^{(\alpha )})\right\vert <\varepsilon .  \label{79.2}
\end{equation}

Take now (\ref{69}). What we will prove next is that in the limit we have an
independence copula, in the sense that every cluster point of $%
c^{n}(x_{1},x_{2})$ equals $1$ on $[\delta ,\eta ]^{2}$. For $k$ being
\textquotedblleft large"\ we apply the above definition of uniform
subsequential convergence. Namely, take a fixed $\varepsilon $ and consider $%
k\geq K(\varepsilon )$. Since $P_{i}^{n}(x_{i})$ and $Q_{i}^{n}(x_{i})$ are
bounded, for the denominator of (\ref{69}) (with $n=n_{k}^{(\alpha )}$) we
get consecutively%
\begin{eqnarray}
&&\dint\nolimits_{0}^{1}\dint%
\nolimits_{0}^{1}L_{1}^{n-1,-1}(u_{1})L_{2}^{n-1,-1}(u_{2})\frac{%
P_{1}^{n-1}(u_{1})P_{2}^{n-1}(u_{2})}{Q_{1}^{n-1}(u_{1})Q_{2}^{n-1}(u_{2})}%
D^{n-1}(u_{1},u_{2})du_{1}du_{2}  \notag \\
&\leq &\left( c^{F}(c_{1},c_{2})+\varepsilon \right)
\dint\nolimits_{0}^{1}\dint%
\nolimits_{0}^{1}L_{1}^{n-1,-1}(u_{1})L_{2}^{n-1,-1}(u_{2})\frac{%
P_{1}^{n-1}(u_{1})P_{2}^{n-1}(u_{2})}{Q_{1}^{n-1}(u_{1})Q_{2}^{n-1}(u_{2})}%
du_{1}du_{2}  \notag \\
&=&c^{F}(c_{1},c_{2})\dint\nolimits_{0}^{1}\dint%
\nolimits_{0}^{1}L_{1}^{n-1,-1}(u_{1})L_{2}^{n-1,-1}(u_{2})\frac{%
P_{1}^{n-1}(u_{1})P_{2}^{n-1}(u_{2})}{Q_{1}^{n-1}(u_{1})Q_{2}^{n-1}(u_{2})}%
du_{1}du_{2}+\varepsilon _{1}  \notag \\
&=&c^{F}(c_{1},c_{2})\left( \dint\nolimits_{0}^{1}L_{1}^{n-1,-1}(u_{1})\frac{%
P_{1}^{n-1}(u_{1})}{Q_{1}^{n-1}(u_{1})}du_{1}\right) \left(
\dint\nolimits_{0}^{1}L_{2}^{n-1,-1}(u_{2})\frac{P_{2}^{n-1}(u_{2})}{%
Q_{2}^{n-1}(u_{2})}du_{2}\right) +\varepsilon _{1},  \label{81}
\end{eqnarray}

where $\varepsilon _{1}$ is also \textquotedblleft small" and for it we have%
\footnote{{\footnotesize Obviously, the boundedness of }$P_{i}^{n}(x_{i})$%
{\footnotesize \ and }$Q_{i}^{n}(x_{i})${\footnotesize \ matter so far both
in (\ref{81}) and (\ref{82}).}}%
\begin{equation}
\varepsilon _{1}=\varepsilon
\dint\nolimits_{0}^{1}\dint%
\nolimits_{0}^{1}L_{1}^{n-1,-1}(u_{1})L_{2}^{n-1,-1}(u_{2})\frac{%
P_{1}^{n-1}(u_{1})P_{2}^{n-1}(u_{2})}{Q_{1}^{n-1}(u_{1})Q_{2}^{n-1}(u_{2})}%
du_{1}du_{2}.  \label{82}
\end{equation}

Analogously, we have also%
\begin{eqnarray}
&&\dint\nolimits_{0}^{1}\dint%
\nolimits_{0}^{1}L_{1}^{n-1,-1}(u_{1})L_{2}^{n-1,-1}(u_{2})\frac{%
P_{1}^{n-1}(u_{1})P_{2}^{n-1}(u_{2})}{Q_{1}^{n-1}(u_{1})Q_{2}^{n-1}(u_{2})}%
D^{n-1}(u_{1},u_{2})du_{1}du_{2}  \notag \\
&\geq &c^{F}(c_{1},c_{2})\left( \dint\nolimits_{0}^{1}L_{1}^{n-1,-1}(u_{1})%
\frac{P_{1}^{n-1}(u_{1})}{Q_{1}^{n-1}(u_{1})}du_{1}\right) \left(
\dint\nolimits_{0}^{1}L_{2}^{n-1,-1}(u_{2})\frac{P_{2}^{n-1}(u_{2})}{%
Q_{2}^{n-1}(u_{2})}du_{2}\right) -\varepsilon _{1}.  \label{83}
\end{eqnarray}

For the numerator we get%
\begin{equation}
\frac{P_{1}^{n}(x_{1})P_{2}^{n}(x_{2})}{Q_{1}^{n}(x_{1})Q_{2}^{n}(x_{2})}%
c^{F}(c_{1},c_{2})-\varepsilon _{2}\leq \frac{%
P_{1}^{n}(x_{1})P_{2}^{n}(x_{2})}{Q_{1}^{n}(x_{1})Q_{2}^{n}(x_{2})}%
D^{n}(x_{1},x_{2})\leq \frac{P_{1}^{n}(x_{1})P_{2}^{n}(x_{2})}{%
Q_{1}^{n}(x_{1})Q_{2}^{n}(x_{2})}c^{F}(c_{1},c_{2})+\varepsilon _{2},
\label{84}
\end{equation}

where $\varepsilon _{2}$ is again \textquotedblleft small" and for it we have%
\begin{equation}
\varepsilon _{2}=\varepsilon \frac{P_{1}^{n}(x_{1})P_{2}^{n}(x_{2})}{%
Q_{1}^{n}(x_{1})Q_{2}^{n}(x_{2})}.  \label{85}
\end{equation}

So from (\ref{69}) and taking the evaluations of its numerator and
denominator from (\ref{81}), (\ref{83}), and (\ref{84}), we get%
\begin{eqnarray}
c^{n}(x_{1},x_{2}) &\leq &\frac{\frac{P_{1}^{n}(x_{1})P_{2}^{n}(x_{2})}{%
Q_{1}^{n}(x_{1})Q_{2}^{n}(x_{2})}c^{F}(c_{1},c_{2})+\varepsilon _{2}}{%
c^{F}(c_{1},c_{2})\left( \dint\nolimits_{0}^{1}L_{1}^{n-1,-1}(u_{1})\frac{%
P_{1}^{n-1}(u_{1})}{Q_{1}^{n-1}(u_{1})}du_{1}\right) \left(
\dint\nolimits_{0}^{1}L_{2}^{n-1,-1}(u_{2})\frac{P_{2}^{n-1}(u_{2})}{%
Q_{2}^{n-1}(u_{2})}du_{2}\right) -\varepsilon _{1}}  \notag \\
&\leq &\frac{\frac{P_{1}^{n}(x_{1})P_{2}^{n}(x_{2})}{%
Q_{1}^{n}(x_{1})Q_{2}^{n}(x_{2})}c^{F}(c_{1},c_{2})}{c^{F}(c_{1},c_{2})%
\left( \dint\nolimits_{0}^{1}L_{1}^{n-1,-1}(u_{1})\frac{P_{1}^{n-1}(u_{1})}{%
Q_{1}^{n-1}(u_{1})}du_{1}\right) \left(
\dint\nolimits_{0}^{1}L_{2}^{n-1,-1}(u_{2})\frac{P_{2}^{n-1}(u_{2})}{%
Q_{2}^{n-1}(u_{2})}du_{2}\right) }+\varepsilon _{3}  \notag \\
&=&\frac{\frac{P_{1}^{n}(x_{1})P_{2}^{n}(x_{2})}{%
Q_{1}^{n}(x_{1})Q_{2}^{n}(x_{2})}}{\left(
\dint\nolimits_{0}^{1}L_{1}^{n-1,-1}(u_{1})\frac{P_{1}^{n-1}(u_{1})}{%
Q_{1}^{n-1}(u_{1})}du_{1}\right) \left(
\dint\nolimits_{0}^{1}L_{2}^{n-1,-1}(u_{2})\frac{P_{2}^{n-1}(u_{2})}{%
Q_{2}^{n-1}(u_{2})}du_{2}\right) }+\varepsilon _{3},  \label{86}
\end{eqnarray}

where $\varepsilon _{3}$ is also small by a classical linear fraction
representation. Analogously, we have also%
\begin{equation}
\frac{\frac{P_{1}^{n}(x_{1})P_{2}^{n}(x_{2})}{%
Q_{1}^{n}(x_{1})Q_{2}^{n}(x_{2})}}{\left(
\dint\nolimits_{0}^{1}L_{1}^{n-1,-1}(u_{1})\frac{P_{1}^{n-1}(u_{1})}{%
Q_{1}^{n-1}(u_{1})}du_{1}\right) \left(
\dint\nolimits_{0}^{1}L_{2}^{n-1,-1}(u_{2})\frac{P_{2}^{n-1}(u_{2})}{%
Q_{2}^{n-1}(u_{2})}du_{2}\right) }-\varepsilon _{3}\leq c^{n}(x_{1},x_{2}).
\label{87}
\end{equation}

From (\ref{69}), (\ref{86}), and (\ref{87}) we get effectively that for $%
k\geq K(\varepsilon )$ (equivalently, for $n=n_{k}^{(\alpha )}$ sufficiently
large along the chosen subsequence)%
\begin{equation}
\left\vert c^{n}(x_{1},x_{2})-h_{1}^{n}(x_{1})h_{2}^{n}(x_{2})\right\vert
\leq \varepsilon _{3},  \label{88}
\end{equation}

where%
\begin{eqnarray}
h_{1}^{n}(x) &=&\frac{P_{1}^{n}(x)}{Q_{1}^{n}(x)}\frac{1}{\left(
\dint\nolimits_{0}^{1}L_{1}^{n-1,-1}(u_{1})\frac{P_{1}^{n-1}(u_{1})}{%
Q_{1}^{n-1}(u_{1})}du_{1}\right) }  \label{89} \\
h_{2}^{n}(x) &=&\frac{P_{2}^{n}(x)}{Q_{2}^{n}(x)}\frac{1}{\left(
\dint\nolimits_{0}^{1}L_{2}^{n-1,-1}(u_{2})\frac{P_{2}^{n-1}(u_{2})}{%
Q_{2}^{n-1}(u_{2})}du_{2}\right) }.  \label{90}
\end{eqnarray}

We can re-write (\ref{88})$~$as%
\begin{eqnarray}
\left\vert
c^{n}(L_{1}^{n}(x_{1}),L_{2}^{n}(x_{2}))-h_{1}^{n}(L_{1}^{n}(x_{1}))h_{2}^{n}(L_{2}^{n}(x_{2}))\right\vert &\leq &\varepsilon _{3}
\label{101} \\
\left\vert \frac{l_{12}^{n}(x_{1},x_{2})}{l_{1}^{n}(x_{1})l_{2}^{n}(x_{2})}%
-h_{1}^{n}(L_{1}^{n}(x_{1}))h_{2}^{n}(L_{2}^{n}(x_{2}))\right\vert &\leq
&\varepsilon _{3}.  \label{102}
\end{eqnarray}

Thus we get a product separability of the density $l_{12}^{n}(x_{1},x_{2})$
in the limit along the chosen subsequence, which implies that every cluster
point of the copula density $c^{n}(.,.)$ factorizes on $[\delta ,\eta ]^{2}$
to the product of its marginals. Since any copula density has uniform
marginals, this factorization forces the cluster point to be identically $1$
on $[\delta ,\eta ]^{2}$, i.e.,\ the corresponding limiting copula is the
independence copula\footnote{{\footnotesize Note that }$\frac{P_{1}^{n}(x)}{%
Q_{1}^{n}(x)}${\footnotesize \ and }$\frac{P_{2}^{n}(x)}{Q_{2}^{n}(x)},$%
{\footnotesize \ as well as }$h_{1}^{n}(.)${\footnotesize \ and }$%
h_{2}^{n}(.),${\footnotesize \ converging uniformly to 1 comes only as a
byproduct of the separability in the limit. }}.

Once we have the independence, we can resort to the remark from the
beginning that if we start with an independent copula, we get effectively
the iterations from \cite{[3]}. There the following theorem was proven:

\textbf{\textit{Theorem }}\textit{(Univariate case):}\textbf{\ }\textit{Let }%
$X$\textit{\ be an arbitrary non-negative random variable with a
distribution function }$F$\textit{\ and a positive finite mean }$\mu _{F}$%
\textit{. It gives rise to a Lorenz curve, }$L_{F}(x)$%
\begin{equation}
L_{F}(x)=\frac{\dint\nolimits_{0}^{x}F^{-1}(u)du}{\mu _{F}}=\frac{%
\dint\nolimits_{0}^{x}F^{-1}(u)du}{\dint\nolimits_{0}^{1}F^{-1}(u)du},
\label{104}
\end{equation}

\textit{where }$F^{-1}(u)=\inf \left\{ y:F(y)\geq u\right\} $\textit{\ for }$%
0\leq u\leq 1$\textit{\ is the generalized inverse of }$F(u)$\textit{\ and }$%
\mu _{F}=$\textit{\ }$\dint\nolimits_{0}^{1}F^{-1}(u)du<\infty $\textit{\ is
the mean of }$X\,$\textit{. We will denote by }$L$\textit{\ both the Lorenz
curve above (using the notation }$L_{F}(x)$\textit{\ and }$L^{F}(x)$\textit{%
\ interchangeably, emphasizing on the distribution function }$F$\textit{\
that generates }$L$\textit{\ and the same logic holds for }$\mu _{F}$\textit{%
) and the linear operator }$L(F(x))(.):[0,1]\rightarrow \lbrack 0,1]$\textit{%
\ which maps }$F(.)$\textit{\ to }$L_{F}(.)$\textit{.}

\textit{Since }$L_{F}(x)$\textit{\ by itself could be viewed as a
distribution function of a random variable (with the possible extension }$%
L_{F}(x)=0$\textit{\ for }$x<0$\textit{\ and }$L_{F}(x)=1$\textit{\ for }$%
x>1 $\textit{), if for }$i=0,1,...$ \textit{we consider the sequence of
distribution functions }$H_{i}^{F}(x)$%
\begin{eqnarray}
H_{0}^{F}(x) &=&F(x)  \label{104.1} \\
H_{1}^{F}(x) &=&L^{H_{0}^{F}}(x)=L^{F}(x)=L_{F}(x)=L_{0}(x)  \notag \\
H_{2}^{F}(x) &=&L^{H_{1}^{F}}(x)=L(L^{F}(x))=L_{1}(x)  \notag \\
H_{3}^{F}(x) &=&L^{H_{2}^{F}}(x)=L(L_{1}(x))=L_{2}(x)  \notag \\
&&...  \notag \\
H_{n}^{F}(x) &=&L^{H_{n-1}^{F}}(x)=L(L_{n-2}(x))=L_{n-1}(x)  \notag \\
H_{n+1}^{F}(x) &=&L^{H_{n}^{F}}(x)=L(L_{n-1}(x))=L_{n}(x),  \notag
\end{eqnarray}

\textit{where in }$H$\textit{\ and }$L$\textit{\ by the subscript }$n$%
\textit{\ we indicate the iteration and by the superscript the starting
distribution }$F$\textit{, we have that }$H_{n}^{F}(x)$\textit{\ converges
uniformly to the distribution function }$G(x)$\textit{\ }%
\begin{equation}
G(x)=\left\{ 
\begin{array}{c}
x^{\frac{1+\sqrt{5}}{2}},0\leq x\leq 1 \\ 
0,x<0 \\ 
1,x>1.%
\end{array}%
\right.  \label{104.2}
\end{equation}

So the limit marginals are the ones proved in the theorem - power-laws with
golden section exponent. This completes the proof.
\end{proof}

\section{Analysis of extremal dependence cases}

In the previous section, the density assumption was essential to the proof.
Specifically, it was used to establish that both the product copula and
marginals of the form presented in (\ref{104.2}) are attracting fixed points
of the iteration map (\ref{5.1}). This raises the natural question of
whether other such fixed points exist under weaker assumptions, a
possibility we investigate in the present section.

In contrast to our prior work \cite{[3]}, which relied on straight
polynomial bounds, the present analysis must address also the more intricate
problem of an evolving dependence structure. As detailed in \textit{Appendix
D}, our approach leverages the $TP_{2}$ and $RR_{2}$ properties of the
initial density, which induce specific crossing patterns in the \textit{%
marginal} \textit{Lorenz curves} that helped to produce a solution. In
retrospect, the integrand in (\ref{5.1}) can be seen as a special normalized
kernel that possesses both $TP_{2}$ and $RR_{2}$ properties. This kernel
served to reduce the dependence of the integrator, iteratively driving any
initial distribution with strong positive ($TP_{2}$) or negative ($RR_{2}$)
dependence toward complete independence.

This leads to a natural question, which we address in this section: What is
the outcome when the iterative process begins with distributions at the
limits of maximal dependence---namely, the \textit{Fr\'{e}chet-Hoeffding
bounds}---particularly when a density does not exist?

Our approach to answering this question is structured as follows. The
initial distribution $F(x_{1},x_{2})$ is bounded by the \textit{Fr\'{e}%
chet-Hoeffding bounds}, which are determined by its marginals, $F_{1}(x_{1})$
and $F_{2}(x_{2})$. Therefore, the first step in our analysis is to
investigate how these bounds evolve under the iteration map. Specifically,
we aim to characterize the evolution of the bounds for $L_{n}(x_{1},x_{2})$.
This characterization will, in turn, inform our second objective: the
detection of additional fixed points. The \textit{Fr\'{e}chet--Hoeffding
bounds} in the two-dimensional case constitute distribution functions which
are non-smooth and hence lack densities. Although this suggests a relaxation
of the conditions of \textit{Theorem \ref{main-th}}, we will retain the
strictly positive density assumption for the initial marginals $F_{1}(x_{1})$
and $F_{2}(x_{2})$, both here and in the multivariate case. This assumption
is crucial for the results to hold; we will relax it only later.

The key to the next analysis is a well-known inequality for the expectation
of the product of two random variables. In our notation, for any random
vector $(X_{1},$ $X_{2})$ with joint distribution function $%
F_{12}(x_{1},x_{2})$ and marginals $F_{i}(x_{i}),i=1,2$, as well as a random
variable $U\sim U(0,1),$ the following inequality holds (see \cite{[5]})%
\begin{equation}
EF_{1}^{-1}(U)F_{2}^{-1}(1-U)\leq EX_{1}X_{2}\leq
EF_{1}^{-1}(U)F_{2}^{-1}(U).  \label{17}
\end{equation}

In terms of integrals, it reads: 
\begin{equation}
\dint\nolimits_{-\infty }^{+\infty }F_{1}^{-1}(u)F_{2}^{-1}(1-u)du\leq
\dint\nolimits_{-\infty }^{+\infty }\dint\nolimits_{-\infty }^{+\infty
}u_{1}u_{2}dF_{12}(u_{1},u_{2})\leq \dint\nolimits_{-\infty }^{+\infty
}F_{1}^{-1}(u)F_{2}^{-1}(u)du.  \label{18}
\end{equation}

This structure resembles to a great extent exactly Arnold's definition of a 
\textit{Lorenz curve} we employ, with the difference that the integration in
(\ref{18}) is improper (i.e., up to $+\infty $), while in (\ref{1}), it is
proper (i.e., up to a finite variable). Additionally, the \textit{upper} and
the \textit{lower-bounds} above are reached for comonotonic and
countermonotonic random vectors $(X_{1}^{F_{1}},X_{2}^{F_{2}}),$
respectively. This gives further insight into what to expect for the \textit{%
Lorenz curve} iterations.

Let's turn to them. We proceed as follows: take the well-known \textit{Fr%
\'{e}chet-Hoeffding bounds} of $F(x_{1},x_{2})$. Following \cite{[5]} and 
\cite{[11]}, and using the standard notation from there for the two
inequality sides, we have%
\begin{equation}
\overset{F_{-}}{\overbrace{Max(F_{1}(x_{1})+F_{2}(x_{2})-1,0)}}\leq
F(x_{1},x_{2})\leq \overset{F_{+}}{\overbrace{Min(F_{1}(x_{1}),F_{2}(x_{2}))}%
},  \label{19}
\end{equation}

or in terms of copulas%
\begin{equation}
\overset{W^{F}(x_{1},x_{2})}{\overbrace{Max(x_{1}+x_{2}-1,0)}}\leq
C^{F}(x_{1},x_{2})\leq \overset{M^{F}(x_{1},x_{2})}{\overbrace{%
Min(x_{1},x_{2})}.}  \label{20}
\end{equation}

It is well known (e.g., again \cite{[5]} and \cite{[11]}) that in the
bivariate case, both bounds are sharp and form distribution functions ($%
F_{-} $ and $F_{+}$ ) as in (\ref{19}), or copulas ($W^{F}(x_{1},x_{2})$ -
countermonotonic and $M^{F}(x_{1},x_{2})$ - comonotonic) as in (\ref{20}).

\subsection{Fr\'{e}chet-Hoeffding upper-bound}

Let's see what happens with the \textit{Lorenz curve} $%
L_{F_{+}}(x_{1},x_{2}) $ generated by the \textit{upper-bound}. By utilizing
the properties of the \textit{Dirac delta function} $\delta (.)$ in the
context of integration theory, we address the occurrence of a non-smooth
integrator in the corresponding \textit{Stieltjes integrals}. They are also
considered in \textit{Claim }\ref{fh-claim1} of \textit{Appendix A}, which
provides details on the calculations below. We get\footnote{{\footnotesize %
To our knowledge, while the bounds (\ref{17}) are well known and can be
derived in various ways (see \cite[Remark~3.25]{[5]} for more details), the
analysis of the \textit{Lorenz curve} case considered in this
paper---specifically, obtaining the proper integrals in (\ref{18})---has not
been explicitly addressed in the literature.}
\par
{\footnotesize Although our approach is technically involved at certain
points, we have deliberately used general methods from the theory of
distributions to improve problem comprehension and broaden the scope of
potential applications, rather than relying on specialized probabilistic
techniques.}
\par
{\footnotesize However, it is worth noting that an elegant heuristic
probabilistic derivation exists. We have the representation }%
\begin{equation*}
L_{F_{+}}(x_{1},x_{2})=\frac{\dint\nolimits_{0}^{x_{1}}\dint%
\nolimits_{0}^{x_{2}}F_{1}^{-1}(u_{1})F_{2}^{-1}(u_{2})dM^{F}(u_{1},u_{2})}{%
\mu _{12}^{F_{+}}}=\frac{E[1_{\{U_{1}<x_{1}\}}1_{\{U_{2}<x_{2}%
\}}F_{1}^{-1}(U_{1})F_{2}^{-1}(U_{2})]}{E[F_{1}^{-1}(U_{1})F_{2}^{-1}(U_{2})]%
},
\end{equation*}%
\par
{\footnotesize where }$(U_{1},U_{2})${\footnotesize \ is a bivariate
distribution with uniform marginals. The presence of the comonotonic copula }%
$M^{F}(.,.)${\footnotesize \ in the differential implies perfect correlation
between }$U_{1}${\footnotesize \ and }$U_{2}${\footnotesize , i.e., }$%
P(U_{1}=U_{2})=1.${\footnotesize \ As a result, the expectation }%
\begin{equation*}
E[1_{\{U_{1}<x_{1}\}}1_{\{U_{2}<x_{2}\}}F_{1}^{-1}(U_{1})F_{2}^{-1}(U_{2})]%
{\footnotesize \ }
\end{equation*}%
\par
{\footnotesize reduces to }%
\begin{equation*}
E[1_{\{U<x_{1}\wedge x_{2}\}}F_{1}^{-1}(U)F_{2}^{-1}(U)],{\footnotesize \ }
\end{equation*}%
\par
{\footnotesize where }$U${\footnotesize \ is uniformly distributed. This
yields the same result as in (\ref{21}).}
\par
{\footnotesize The derivation can also be extended by incorporating deeper
aspects of the theory of distributions using test functions (see \cite{[12]}
and \cite[Chapter 9]{[19]} for an introduction, \cite[Chapter 1]{[21]} for a
modern treatment, and \cite{[15]} for detailed theoretical and applied
discussions). This approach notably allows non-smooth copulas---such as
extremal copulas attaining \textit{Fr\'{e}chet-Hoeffding} \textit{bounds} or
those induced by atoms---to be interpreted as PDE solutions, which offers
some modeling advantages. In a related context, similar techniques are
utilized in \cite{[16]} for analyzing copulas within \textit{Sobolev spaces}%
. However, such methods exceed the scope of this paper and are unnecessary
for our specific problem, except briefly in \textit{Claim }\ref{fh-claim2}
of \textit{Appendix A}.}}%
\begin{eqnarray}
&&L_{F_{+}}(x_{1},x_{2})  \notag \\
&=&\frac{\dint\nolimits_{-\infty
}^{F_{1}^{-1}(x_{1})}\dint\nolimits_{-\infty
}^{F_{2}^{-1}(x_{2})}u_{1}u_{2}dF_{+}(u_{1},u_{2})}{\mu _{12}^{F_{+}}}=\frac{%
\dint\nolimits_{0}^{x_{1}}\dint%
\nolimits_{0}^{x_{2}}F_{1}^{-1}(u_{1})F_{2}^{-1}(u_{2})dM^{F}(u_{1},u_{2})}{%
\mu _{12}^{F_{+}}}  \notag \\
&=&\frac{\dint\nolimits_{0}^{x_{1}}\dint%
\nolimits_{0}^{x_{2}}F_{1}^{-1}(u_{1})F_{2}^{-1}(u_{2})dMin(u_{1},u_{2})}{%
\mu _{12}^{F_{+}}}=\frac{\dint\nolimits_{0}^{x_{1}}\dint%
\nolimits_{0}^{x_{2}}F_{1}^{-1}(u_{1})F_{2}^{-1}(u_{2})\delta
(u_{2}-u_{1})du_{1}du_{2}}{\mu _{12}^{F_{+}}}  \notag \\
&=&\frac{\dint\nolimits_{0}^{x_{1}\wedge x_{2}}F_{1}^{-1}(u)F_{2}^{-1}(u)du}{%
\dint\nolimits_{0}^{1}F_{1}^{-1}(u)F_{2}^{-1}(u)du}.  \label{21}
\end{eqnarray}

Now, we can observe that the particular form of the last expression in (\ref%
{21}) allows us to derive the following%
\begin{eqnarray}
1)\text{ }x_{1} &\leq &x_{2}:  \label{25} \\
L_{F_{+}}(x_{1},x_{2}) &=&\frac{\dint%
\nolimits_{0}^{x_{1}}F_{1}^{-1}(u)F_{2}^{-1}(u)du}{\dint%
\nolimits_{0}^{1}F_{1}^{-1}(u)F_{2}^{-1}(u)du}=\frac{\dint%
\nolimits_{0}^{x_{1}\wedge 1}F_{1}^{-1}(u)F_{2}^{-1}(u)du}{%
\dint\nolimits_{0}^{1}F_{1}^{-1}(u)F_{2}^{-1}(u)du}%
=L_{F_{+}}(x_{1},1)=L_{1}^{F_{+}}(x_{1})  \notag
\end{eqnarray}%
\begin{eqnarray}
2)\text{ }x_{2} &<&x_{1}:  \label{25ab} \\
L_{F_{+}}(x_{1},x_{2}) &=&\frac{\dint%
\nolimits_{0}^{x_{2}}F_{1}^{-1}(u)F_{2}^{-1}(u)du}{\dint%
\nolimits_{0}^{1}F_{1}^{-1}(u)F_{2}^{-1}(u)du}=\frac{\dint%
\nolimits_{0}^{x_{2}\wedge 1}F_{1}^{-1}(u)F_{2}^{-1}(u)du}{%
\dint\nolimits_{0}^{1}F_{1}^{-1}(u)F_{2}^{-1}(u)du}%
=L_{F_{+}}(1,x_{2})=L_{2}^{F_{+}}(x_{2}).  \notag
\end{eqnarray}

Three important implications follow. First, by the definition of the \textit{%
Fr\'{e}chet-Hoeffding bounds} for the \textit{Lorenz curve} $%
L_{F}(x_{1},x_{2})$ viewed as a distribution function, we have%
\begin{equation}
L_{F}(x_{1},x_{2})\leq L_{+}^{F}(x_{1},x_{2})=Min\left(
L_{1}^{F}(x_{1}),L_{2}^{F}(x_{2})\right) .  \label{25a}
\end{equation}

Here, we continue using the notation introduced earlier to indicate the
extremal distribution $L_{+}^{F}(x_{1},x_{2})$ with a plus subscript. But
from (\ref{25}), it becomes further valid that for any $0\leq x_{1}\leq 1$
and $0\leq x_{2}\leq 1$ 
\begin{equation}
L_{F_{+}}(x_{1},x_{2})=Min\left(
L_{1}^{F_{+}}(x_{1}),L_{2}^{F_{+}}(x_{2})\right) .  \label{26}
\end{equation}

This means that $L_{F}(x_{1},x_{2})$ reaches its \textit{Fr\'{e}%
chet-Hoeffding upper-bound} if $F(x_{1},x_{2})$ does the same, i.e., $%
L_{F}(x_{1},x_{2})=L_{+}^{F}(x_{1},x_{2})$ in (\ref{25a}) when $%
F(x_{1},x_{2})=F_{+}(x_{1},x_{2})$. As shown in \textit{Claim }%
{\footnotesize \ref{fh-claim2} }of \textit{Appendix A}, the reverse also
holds: $L_{F}(x_{1},x_{2})=L_{+}^{F}(x_{1},x_{2})$ implies $%
F(x_{1},x_{2})=F_{+}(x_{1},x_{2})$. Thus, we may conclude that $%
L_{+}^{F}(x_{1},x_{2})=L_{F_{+}}(x_{1},x_{2})\footnote{{\footnotesize It is
important to emphasize that }$L_{+}^{F}(x_{1},x_{2})${\footnotesize \
simultaneously denotes two distinct concepts:}
\par
{\footnotesize 1) }$L${\footnotesize , viewed as a distribution function,
attains its \textit{Fr\'{e}chet-Hoeffding upper-bound}; thus, it equals the
minimum of its marginals. As previously mentioned, this interpretation is
signified by the superscript '+' notation.}
\par
{\footnotesize 2) }$L${\footnotesize , interpreted as a\textit{\ Lorenz curve%
} according to the Arnold's definition, possesses a specific integral form
and is associated with a parent distribution }$F${\footnotesize . This
interpretation is denoted by both }$L${\footnotesize \ itself and the
superscript }$F${\footnotesize .}
\par
{\footnotesize This dual interpretation imposes certain restrictions on }$%
L_{+}^{F}(x_{1},x_{2}).${\footnotesize \ Specifically, they lead directly to
the implication derived from the \textquotedblleft if and only if" argument
just elaborated. To be more eloquent, just by putting the '+' superscript on
the distribution function }$L^{F}(x_{1},x_{2})${\footnotesize , we force it
to be equal to the minimum of its marginals, but the integral structure of
the \textit{Lorenz curve} }$L^{F}(x_{1},x_{2})${\footnotesize \ dependent on 
}$F${\footnotesize \ will allow that to happen if and only if }$F=F_{+}$%
{\footnotesize . This may initially seem to be at odds with the paradigm
"distributions with given marginals"\ when working with \textit{Fr\'{e}%
chet-Hoeffding bounds}, but a careful reading will give that there is no
contradiction in applying them to the particular situation. }
\par
{\footnotesize \bigskip }
\par
{\footnotesize {}}}$ and $L_{i+}^{F}(x_{i})=L_{i}^{F_{+}}(x_{i}),$ where for 
$i=1,2,$ $L_{i+}^{F}(x_{i})$ denotes the marginals of $%
L_{+}^{F}(x_{1},x_{2}) $.

Second, the exact expressions for the values participating in the \textit{%
upper-bound} (i.e., the marginals $L_{i+}^{F}(x_{i})$) are%
\begin{eqnarray}
L_{1+}^{F}(x_{1}) &=&L_{1}^{F_{+}}(x_{1})=\frac{\dint%
\nolimits_{0}^{x_{1}}F_{1}^{-1}(u)F_{2}^{-1}(u)du}{\dint%
\nolimits_{0}^{1}F_{1}^{-1}(u)F_{2}^{-1}(u)du}  \label{27} \\
L_{2+}^{F}(x_{2}) &=&L_{2}^{F_{+}}(x_{2})=\frac{\dint%
\nolimits_{0}^{x_{2}}F_{1}^{-1}(u)F_{2}^{-1}(u)du}{\dint%
\nolimits_{0}^{1}F_{1}^{-1}(u)F_{2}^{-1}(u)du}.  \notag
\end{eqnarray}

Third, given a priori marginals $F_{1}(x_{1})$ and $F_{2}(x_{2}),$ it is
always possible to construct a two-dimensional distribution function based
on them in such a way that the \textit{Fr\'{e}chet-Hoeffding upper-bound} is
reached. This is done by taking $%
F(x_{1},x_{2})=F_{+}=Min(F_{1}(x_{1}),F_{2}(x_{2}))$. For an arbitrary $%
F(x_{1},x_{2})\neq F_{+}$ with the same marginals, we have $%
F(x_{1},x_{2})\leq $ $F_{+}=Min(F_{1}(x_{1}),F_{2}(x_{2}))$. Now, if we take 
$x_{1}=1$ $(x_{2}=1)$ in the latter, we get the trivial inequality $%
F_{2}(x_{2})=F(1,x_{2})\leq Min(1,F_{2}(x_{2}))\leq F_{2}(x_{2})$, which
does not offer any new information. We mention this reasoning because the
situation is not the same for a distribution function $L_{F}(x_{1},x_{2})$
defined by (\ref{1}). As discussed before and visible in (\ref{1}), (\ref%
{5.1}), and (\ref{6b}), here the marginals $L_{F}(x_{1})$ and $L_{F}(x_{2})$
are derived quantities from $L(x_{1},x_{2})$ and not a priori postulated
ones, in contrast to the freedom we had in the previous case when
constructing $F(x_{1},x_{2})$. So, whether the \textit{Fr\'{e}chet-Hoeffding
upper-bound} $L_{+}^{F}(x_{1},x_{2})$ is reached is actually determined by
the parent distribution $F$, and not by constructing $L_{F}(x_{1},x_{2})$ as 
$Min\left( L_{1}^{F}(x_{1}),L_{2}^{F}(x_{2})\right) $. We would be able to
do so if and only if $F$ is taken to be $F_{+}$, as we have already
elaborated in the first implication and footnote 6. However, both the
discussion above and equations (\ref{25a}),(\ref{26}), and (\ref{27}) are
incomplete because an important observation is still missing. The natural
question of how $Min\left( L_{1}^{F}(x_{1}),L_{2}^{F}(x_{2})\right) $ and $%
L_{F_{+}}(x_{1},x_{2})=Min\left(
L_{1}^{F_{+}}(x_{1}),L_{2}^{F_{+}}(x_{2})\right) $ are ordered remains open.
Filling this gap is important, as it will enable a comparison of the
marginals $L_{i}^{F}(x_{1})$ and $L_{i}^{F+}(x_{1})$, for $i=1,2,$ along the
iterative procedure. We will address this in more detail later\footnote{%
{\footnotesize To be more precise, an alternative reasoning behind the three
implications is that we work within the \textit{Fr\'{e}chet class} }$%
M(P_{1},P_{2})${\footnotesize \ (see \cite{[5]} for detailed definitions),
which is determined by the probability measures }$P_{1}${\footnotesize {}
and }$P_{2}${\footnotesize {} of the marginal distributions }$F_{1}$%
{\footnotesize {} and }$F_{2}${\footnotesize {}. We begin with fixed
marginals and then modify their copula dependence to observe how it evolves
through the iterations. To this end, we rely on the well-established
framework of "distributions with given marginals". However, when deviations
from this setting occur (driven by different parent distributions }$F)$%
{\footnotesize , we must carefully analyze the situation and apply
appropriate theorems.}}.

Now we can move forward. A direct inductive argument shows that for the 
\textit{Fr\'{e}chet-Hoeffding upper-bound} $L_{+}^{i+1}(x_{1},x_{2})$ of $%
L^{i+1}(x_{1},x_{2})$, in addition to satisfying the inequalities 
\begin{equation}
L^{i+1}(x_{1},x_{2})\leq L_{+}^{i+1}(x_{1},x_{2})=Min\left(
L_{1}^{i+1}(x_{1}),L_{2}^{i+1}(x_{2})\right) ,  \label{28}
\end{equation}

it is also valid%
\begin{equation}
L_{+}^{i+1}(x_{1},x_{2})=L^{L_{+}^{i}}(x_{1},x_{2})=Min\left(
L_{1}^{L_{+}^{i}}(x_{1}),L_{2}^{L_{+}^{i}}(x_{2})\right) .  \label{28a}
\end{equation}

Thus, we obtain $%
L_{+}^{i+1}(x_{1},x_{2})=L_{+}^{L^{i}}(x_{1},x_{2})=L^{L_{+}^{i}}(x_{1},x_{2}) 
$ as well as $%
L_{1+}^{i+1}(x_{1})=L_{1+}^{L^{i}}(x_{1})=L_{1}^{L_{+}^{i}}(x_{1})$ and $%
L_{2+}^{i+1}(x_{2})=L_{2+}^{L^{i}}(x_{2})=L_{2}^{L_{+}^{i}}(x_{2})$. Since
it is the initial distribution $F$ that determines whether the iterations
will take place at the \textit{upper-bound}, we can also move further
backwards and write that (\ref{28a}) implies%
\begin{eqnarray}
L_{F_{+}}^{i+1}(x_{1},x_{2}) &=&Min\left(
L_{F_{+}}^{i+1}(x_{1},1),L_{F_{+}}^{i+1}(1,x_{2})\right) =Min\left(
L_{1}^{L_{F_{+}}^{i}}(x_{1}),L_{2}^{L_{F_{+}}^{i}}(x_{2})\right)  \label{28b}
\\
&=&Min\left( L_{1+}^{i+1}(x_{1}),L_{2+}^{i+1}(x_{2})\right) .  \notag
\end{eqnarray}

Once we have completed the induction and made the initial distribution
explicit in the notation, the difference between $%
L_{1}^{L_{F_{+}}^{i}}(x_{1})$ ($L_{2}^{L_{F_{+}}^{i}}(x_{2})$) from (\ref%
{28b}) and $L_{1}^{L_{+}^{i}}(x_{1})$ ($L_{2}^{L_{+}^{i}}(x_{2})$) from (\ref%
{28a}) needs clarification. Both quantities are fully in line with our
notational logic. The previously used $L_{1}^{L_{+}^{i}}(x_{1})$ in (\ref%
{28a}) refers to the marginal of $L^{L_{+}^{i}}(x_{1},x_{2}),$ which is
straightforward. On the other hand, $L_{1}^{L_{F_{+}}^{i}}(x_{1})$ refers to
the marginal of $L_{F_{+}}^{i+1}(x_{1},x_{2})$, where we can indicate the
starting distribution by $F_{+}$, but in the former case, this is not
directly possible. Clearly, the induction shows that they both share the
same starting distribution $F_{+}$ and are thus equal. To avoid this
ambiguity, we introduced the notation $L_{1+}^{F}(x_{1})$ and $%
L_{2+}^{F}(x_{2})$ for the marginals at the bounds with a clear meaning of
the plus sign, as used in (\ref{27}) and then used it in (\ref{28b}). The
carryover of the initial distribution $F_{+}$ is automatic and does not need
special marking.

Using the notation discussed above, we also obtain the following expressions
for the marginals%
\begin{eqnarray}
L_{1+}^{i+1}(x_{1}) &=&\frac{\dint%
\nolimits_{0}^{x_{1}}L_{1+}^{i,-1}(u)L_{2+}^{i,-1}(u)du}{\dint%
\nolimits_{0}^{1}L_{1+}^{i,-1}(u)L_{2+}^{i,-1}(u)du}  \notag \\
L_{2+}^{i+1}(x_{2}) &=&\frac{\dint%
\nolimits_{0}^{x_{2}}L_{1+}^{i,-1}(u)L_{2+}^{i,-1}(u)du}{\dint%
\nolimits_{0}^{1}L_{1+}^{i,-1}(u)L_{2+}^{i,-1}(u)du}.  \label{29}
\end{eqnarray}

For $x_{1}=x_{2},$ we get $L_{1+}^{i+1}(x)=L_{2+}^{i+1}(x)$. So, for $i>0$
and given initial marginals $F_{1}(x)$ and $F_{2}(x),$ the system (\ref{29})
simplifies to the separated form%
\begin{eqnarray}
L_{1+}^{i+1}(x) &=&\frac{\dint\nolimits_{0}^{x}\left(
L_{1+}^{i,-1}(u)\right) ^{2}du}{\dint\nolimits_{0}^{1}\left(
L_{1+}^{i,-1}(u)\right) ^{2}du},\text{ with }L_{1+}^{0}(x)=\frac{%
\dint\nolimits_{0}^{x}\left( F_{1}^{-1}(u)\right) ^{2}du}{%
\dint\nolimits_{0}^{1}\left( F_{1}^{-1}(u)\right) ^{2}du}  \notag \\
L_{2+}^{i+1}(x) &=&\frac{\dint\nolimits_{0}^{x}\left(
L_{2+}^{i,-1}(u)\right) ^{2}du}{\dint\nolimits_{0}^{1}\left(
L_{2+}^{i,-1}(u)\right) ^{2}du},\text{ with }L_{2+}^{0}(x)=\frac{%
\dint\nolimits_{0}^{x}\left( F_{2}^{-1}(u)\right) ^{2}du}{%
\dint\nolimits_{0}^{1}\left( F_{2}^{-1}(u)\right) ^{2}du}.  \label{43}
\end{eqnarray}

As established in \textit{Appendix C.2}, the functions $L_{i+}^{n}(x)$ for $%
i=1,2$ converge uniformly to $G_{+}(x)=x^{2}$ on$\,$\ $x\in \lbrack 0,1]$.
The analysis also shows that each iterate is bounded by a polynomial
majorization.

Furthermore, the composite inverse functions $\Phi
_{n}^{i+}(x)=L_{i+}^{0,-1}(...(L_{i+}^{n-1,-1}(L_{i+}^{n,-1}(x_{i})))$,
which are key components of equation (\ref{72}), also converge, in this case
to a limit of $1$. This latter result, however, is not useful in the present
context. The proof technique associated with equation (\ref{72}) is
inapplicable because it requires $L_{n}(x_{1},x_{2})$ to have a density, a
condition not met by the \textit{Fr\'{e}chet-Hoeffding} \textit{bounds}.
These extremal cases constitute exceptions to \textit{Theorem} \ref{main-th}%
, for which we prove separate results using a different set of analytical
techniques.

\subsection{Fr\'{e}chet-Hoeffding lower-bound}

Let's turn attention now to the \textit{Fr\'{e}chet-Hoeffding lower-bound}.
By (\ref{20}), we have%
\begin{eqnarray}
&&L_{F_{-}}(x_{1},x_{2})  \notag \\
&=&\frac{\dint\nolimits_{-\infty
}^{F_{1}^{-1}(x_{1})}\dint\nolimits_{-\infty
}^{F_{2}^{-1}(x_{2})}u_{1}u_{2}dF_{-}(u_{1},u_{2})}{\mu _{12}^{F_{-}}}=\frac{%
\dint\nolimits_{0}^{x_{1}}\dint%
\nolimits_{0}^{x_{2}}F_{1}^{-1}(u_{1})F_{2}^{-1}(u_{2})dW^{F}(u_{1},u_{2})}{%
\mu _{12}^{F_{-}}}  \notag \\
&=&\frac{\dint\nolimits_{0}^{x_{1}}\dint%
\nolimits_{0}^{x_{2}}F_{1}^{-1}(u_{1})F_{2}^{-1}(u_{2})dMax(u_{1}+u_{2}-1,0)%
}{\mu _{12}^{F_{-}}}=\frac{\dint%
\nolimits_{1-x_{2}}^{x_{1}}F_{1}^{-1}(u)F_{2}^{-1}(1-u)du}{%
\dint\nolimits_{0}^{1}F_{1}^{-1}(u)F_{2}^{-1}(1-u)du},  \label{30}
\end{eqnarray}

where we provide all the details in \textit{Claim }{\footnotesize \ref%
{fh-claim3} }of \textit{Appendix A}\footnote{{\footnotesize Again, similarly
to (\ref{21}), we can observe that a more elegant heuristic probabilistic
derivation exists. We have the representation: }%
\begin{equation*}
L_{F_{-}}(x_{1},x_{2})=\frac{\dint\nolimits_{0}^{x_{1}}\dint%
\nolimits_{0}^{x_{2}}F_{1}^{-1}(u_{1})F_{2}^{-1}(u_{2})dW^{F}(u_{1},u_{2})}{%
\mu _{12}^{F_{-}}}=\frac{E[1_{\{U_{1}<x_{1}\}}1_{\{U_{2}<x_{2}%
\}}F_{1}^{-1}(U_{1})F_{2}^{-1}(U_{2})]}{E[F_{1}^{-1}(U_{1})F_{2}^{-1}(U_{2})]%
},
\end{equation*}%
\par
{\footnotesize where }$(U_{1},U_{2})${\footnotesize \ is a bivariate
distribution with uniform marginals. The presence of the countermonotonic
copula }$W^{F}(u_{1},u_{2})${\footnotesize \ in the differential implies
that }$U_{1}${\footnotesize \ and }$U_{2}${\footnotesize \ are related by }$%
P(U_{1}+U_{2}=1)=1.${\footnotesize \ As a result, the expectation }%
\begin{equation*}
E[1_{\{U_{1}<x_{1}\}}1_{\{U_{2}<x_{2}\}}F_{1}^{-1}(U_{1})F_{2}^{-1}(U_{2})]%
{\footnotesize \ }
\end{equation*}%
\par
{\footnotesize reduces to }%
\begin{equation*}
E[1_{\{U<x_{1}\}}1_{\{U\geq 1-x_{2}\}}F_{1}^{-1}(U)F_{2}^{-1}(1-U)],%
{\footnotesize \ }
\end{equation*}%
\par
{\footnotesize where }$U${\footnotesize \ is uniformly distributed. This
yields the same result as in (\ref{30}).}}.

From (\ref{30}), we get for the counterpart of (\ref{21}) 
\begin{eqnarray}
Max(L_{1}^{F_{-}}(x_{1})+L_{2}^{F_{-}}(x_{2})-1,0)
&=&Max(L_{F_{-}}(x_{1},1)+L_{F_{-}}(1,x_{2})-1,0)  \notag \\
&=&\frac{\dint\nolimits_{0}^{x_{1}}F_{1}^{-1}(u)F_{2}^{-1}(1-u)du}{%
\dint\nolimits_{0}^{1}F_{1}^{-1}(u)F_{2}^{-1}(1-u)du}+\frac{%
\dint\nolimits_{1-x_{2}}^{1}F_{1}^{-1}(u)F_{2}^{-1}(1-u)du}{%
\dint\nolimits_{0}^{1}F_{1}^{-1}(u)F_{2}^{-1}(1-u)du}-1  \notag \\
&=&\frac{\dint\nolimits_{1-x_{2}}^{x_{1}}F_{1}^{-1}(u)F_{2}^{-1}(1-u)du}{%
\dint\nolimits_{0}^{1}F_{1}^{-1}(u)F_{2}^{-1}(1-u)du}=L_{F_{-}}(x_{1},x_{2}).
\label{36}
\end{eqnarray}

Again, this leads to three important implications. First, by the definition
of the \textit{Fr\'{e}chet-Hoeffding bounds} for the \textit{Lorenz curve} $%
L_{F}(x_{1},x_{2})$ viewed as a distribution function, we have%
\begin{equation}
L_{F}(x_{1},x_{2})\geq
L_{-}^{F}(x_{1},x_{2})=Max(L_{1}^{F}(x_{1})+L_{2}^{F}(x_{2})-1,0),
\label{36a}
\end{equation}

where we continue to use the notation logic to indicate the extremal
distribution $L_{-}^{F}(x_{1},x_{2})$ with a minus subscript. But by (\ref%
{36}), it follows that%
\begin{equation*}
L_{F_{-}}(x_{1},x_{2})=Max(L_{1}^{F_{-}}(x_{1})+L_{2}^{F_{-}}(x_{2})-1,0).
\end{equation*}

This means that $L_{F}(x_{1},x_{2})$ reaches its \textit{Fr\'{e}%
chet-Hoeffding lower-bound} when $F(x_{1},x_{2})$ does the same, i.e., $%
L_{F}(x_{1},x_{2})=L_{-}^{F}(x_{1},x_{2})$ when $F(x_{1},x_{2})=F_{-}$. As
shown in \textit{Claim }{\footnotesize \ref{fh-claim4} } of \textit{Appendix
A}, the reverse also holds: $L_{F}(x_{1},x_{2})=L_{-}^{F}(x_{1},x_{2})$
implies $F(x_{1},x_{2})=F_{-}(x_{1},x_{2})$. Thus, we may conclude that $%
L_{-}^{F}(x_{1},x_{2})=L_{F_{-}}(x_{1},x_{2})$ and $%
L_{i-}^{F}(x_{i})=L_{i}^{F_{-}}(x_{i}),$ where for $i=1,2,$ $%
L_{i-}^{F}(x_{i})$ denotes the marginals of $L_{-}^{F}(x_{1},x_{2})$.

Second, the exact expressions for the values participating in the \textit{%
lower-bound} (i.e., the marginals of $L_{i-}^{F}(x_{i})$) are%
\begin{eqnarray}
L_{1-}^{F}(x_{1}) &=&L_{1}^{F_{-}}(x_{1})=\frac{\dint%
\nolimits_{0}^{x_{1}}F_{1}^{-1}(u)F_{2}^{-1}(1-u)du}{\dint%
\nolimits_{0}^{1}F_{1}^{-1}(u)F_{2}^{-1}(1-u)du}  \notag \\
L_{2-}^{F}(x_{2}) &=&L_{2}^{F_{-}}(x_{2})=\frac{\dint%
\nolimits_{1-x_{2}}^{1}F_{1}^{-1}(u)F_{2}^{-1}(1-u)du}{\dint%
\nolimits_{0}^{1}F_{1}^{-1}(u)F_{2}^{-1}(1-u)du}.  \label{37}
\end{eqnarray}

Third, the natural question of how $%
Max(L_{1}^{F}(x_{1})+L_{2}^{F}(x_{2})-1,0)$ and $%
L_{F_{-}}(x_{1},x_{2})=Max(L_{1}^{F_{-}}(x_{1})+L_{2}^{F_{-}}(x_{2})-1,0)$
are ordered stays, and this gap should be filled. We will focus on this
later on when we compare the marginals $L_{i}^{F}(x_{1})$ and $%
L_{i}^{F-}(x_{1})$, $i=1,2,$ along the iterative procedure.

Again, an inductive argument gives that for the \textit{Fr\'{e}%
chet-Hoeffding lower-bound} $L_{-}^{i+1}(x_{1},x_{2})$ of $%
L^{i+1}(x_{1},x_{2})$, apart from satisfying the inequalities 
\begin{equation}
L^{i+1}(x_{1},x_{2})\geq
L_{-}^{i+1}(x_{1},x_{2})=Max(L_{1}^{i+1}(x_{1})+L_{2}^{i+1}(x_{2})-1,0),
\label{39}
\end{equation}

it is also valid:%
\begin{equation}
L_{-}^{i+1}(x_{1},x_{2})=L^{L_{-}^{i}}(x_{1},x_{2})=Max\left(
L_{1}^{L_{-}^{i}}(x_{1})+L_{2}^{L_{-}^{i}}(x_{2})-1,0\right) .  \label{40}
\end{equation}

So we get $%
L_{-}^{i+1}(x_{1},x_{2})=L_{-}^{L^{i}}(x_{1},x_{2})=L^{L_{-}^{i}}(x_{1},x_{2}) 
$ and $L_{1-}^{i+1}(x_{1})=L_{1-}^{L^{i}}(x_{1})=L_{1}^{L_{-}^{i}}(x_{1})$
and $L_{2-}^{i+1}(x_{2})=L_{2-}^{L^{i}}(x_{2})=L_{2}^{L_{-}^{i}}(x_{2})$.
Since it is the initial distribution $F$ which determines whether the
iterations will take place at the \textit{Fr\'{e}chet-Hoeffding lower-bound}%
, we can also write that equation (\ref{40}) implies%
\begin{eqnarray}
L_{F_{-}}^{i+1}(x_{1},x_{2}) &=&Max\left(
L_{F_{-}}^{i+1}(x_{1},1)+L_{F_{-}}^{i+1}(1,x_{2})-1,0\right)  \notag \\
&=&Max\left(
L_{1}^{L_{F_{-}}^{i}}(x_{1})+L_{2}^{L_{F_{-}}^{i}}(x_{2})-1,0\right)  \notag
\\
&=&Max\left( L_{1-}^{i+1}(x_{1})+L_{2-}^{i+1}(x_{2})-1,0\right) ,  \label{41}
\end{eqnarray}

where the same comments about the notation as for equation (\ref{28b}) hold.
We also get for the marginals%
\begin{eqnarray}
L_{1-}^{i+1}(x_{1}) &=&\frac{\dint%
\nolimits_{0}^{x_{1}}L_{1-}^{i,-1}(u)L_{2-}^{i,-1}(1-u)du}{%
\dint\nolimits_{0}^{1}L_{1-}^{i,-1}(u)L_{2-}^{i,-1}(1-u)du}  \notag \\
L_{2-}^{i+1}(x_{2}) &=&\frac{\dint%
\nolimits_{1-x_{2}}^{1}L_{1-}^{i,-1}(u)L_{2-}^{i,-1}(1-u)du}{%
\dint\nolimits_{0}^{1}L_{1-}^{i,-1}(u)L_{2-}^{i,-1}(1-u)du}.  \label{42}
\end{eqnarray}

Take now $x_{2}=1-x_{1}$ in the second equation of (\ref{42}). We get%
\begin{equation}
L_{2-}^{i+1}(1-x_{1})=\frac{\dint%
\nolimits_{x_{1}}^{1}L_{1-}^{i,-1}(u)L_{2-}^{i,-1}(1-u)du}{%
\dint\nolimits_{0}^{1}L_{1-}^{i,-1}(u)L_{2-}^{i,-1}(1-u)du}.  \label{42.1}
\end{equation}

Adding the first equation of (\ref{42}) to equation (\ref{42.1}), we get
that for the direct and the inverse functions, it holds%
\begin{eqnarray}
L_{1-}^{i+1}(x_{1})+L_{2-}^{i+1}(1-x_{1}) &=&1  \notag \\
L_{1-}^{i+1,-1}(1-x_{1})+L_{2-}^{i+1,-1}(x_{1}) &=&1  \label{42.2}
\end{eqnarray}

or equivalently%
\begin{eqnarray}
L_{1-}^{i+1}(1-x_{1})+L_{2-}^{i+1}(x_{1}) &=&1  \notag \\
L_{1-}^{i+1,-1}(x_{1})+L_{2-}^{i+1,-1}(1-x_{1}) &=&1.  \label{42.3}
\end{eqnarray}

Substitute now $L_{2-}^{i+1,-1}(1-x_{1})$ from (\ref{42.3}) in the first
equation of (\ref{42}). Similarly, substitute $L_{1-}^{i+1,-1}(x_{1})$ from
equation (\ref{42.2}) in the second equation of (\ref{42}) and additionally
change variables. So, given initial marginals $F_{1}(x)$ and $F_{2}(x)$, we
get a simplified separated form of the system (\ref{42}) 
\begin{eqnarray}
L_{1-}^{i+1}(x_{1}) &=&\frac{\dint\nolimits_{0}^{x_{1}}L_{1-}^{i,-1}(u)%
\left( 1-L_{1-}^{i,-1}(u)\right) du}{\dint\nolimits_{0}^{1}L_{1-}^{i,-1}(u)%
\left( 1-L_{1-}^{i,-1}(u)\right) du},\text{with }L_{1-}^{0}(x_{1})=\frac{%
\dint\nolimits_{0}^{x_{1}}F_{1}^{-1}(u)\left( 1-F_{1}^{-1}(u)\right) du}{%
\dint\nolimits_{0}^{1}F_{1}^{-1}(u)\left( 1-F_{1}^{-1}(u)\right) du}  \notag
\\
L_{2-}^{i+1}(x_{2}) &=&\frac{\dint\nolimits_{0}^{x_{2}}L_{2-}^{i,-1}(u)%
\left( 1-L_{2-}^{i,-1}(u)\right) du}{\dint\nolimits_{0}^{1}L_{2-}^{i,-1}(u)%
\left( 1-L_{2-}^{i,-1}(u)\right) du},\text{with }L_{2-}^{0}(x_{2})=\frac{%
\dint\nolimits_{0}^{x_{2}}F_{2}^{-1}(u)\left( 1-F_{2}^{-1}(u)\right) du}{%
\dint\nolimits_{0}^{1}F_{2}^{-1}(u)\left( 1-F_{2}^{-1}(u)\right) du}.
\label{42.4}
\end{eqnarray}

In \textit{Appendix C.1}, we establish both the convergence of $%
L_{i-}^{n}(x) $ and several principal results related to it. Furthermore,
again the composite inverse functions $\Phi
_{n}^{i-}(x)=L_{i-}^{0,-1}(...(L_{i-}^{n-1,-1}(L_{i-}^{n,-1}(x_{i})))$,
which are key components of equation (\ref{72}), also converge, but in this
case to a limit of $0.5$. Although this latter result is inapplicable to the
density-dependent proof of \textit{Theorem} \ref{main-th}, the detailed
analysis of $L_{i-}^{n}(x)$ from \textit{Appendix C.1} is nevertheless
essential. This analysis provides the necessary foundation to investigate $%
\Phi _{n}^{i}(x)$ in the general case and, ultimately, to prove the theorem
itself as done in \textit{Appendix D} and even generalize it. \textit{\ }

\subsection{Discussion}

The main findings from \textit{Sections 4.1} and \textit{4.2} can be
distilled into the following key points:

\begin{theorem}
An initial distribution $F$ with a comonotonic copula and square-law
marginals is a fixed point of the iteration map (\ref{5.1}).
\end{theorem}

\begin{theorem}
An initial distribution $F$ with a countermonotonic copula and suitable
marginals satisfying the limit condition (\ref{42.4}) is a fixed point of
the iteration map (\ref{5.1}).
\end{theorem}

Notably, in contrast to the $TP_{2}$ and $RR_{2}$ cases, the iteration map (%
\ref{5.1}) does not drive initial distributions with extremal copulas toward
independence. Instead, it preserves the maximal positive or negative
dependence structure. Our analysis in \textit{Appendix C}, supported by
empirical observations, strongly suggests that only the extremal copulas
resist this convergence to independence. This implies that the $TP_{2}$ or $%
RR_{2}$ conditions in our main \textit{Theorem} \ref{main-th} may not be
strictly necessary. However, it remains unclear whether the density
requirement can be relaxed. This leads us to formulate the following
conjecture:

\begin{conjecture}
\label{conj} \textit{Theorem} \ref{main-th} is valid for any bivariate
distribution that admits density.
\end{conjecture}

We conclude the discussion with several plots showing the evolution of the
iterated \textit{Lorenz curves} marginal d.f.s - $L_{i}^{n}(x_{i})$,
dependence measures - \textit{Kendall's tau} and \textit{Spearman's rho},
and the compounds of the inverse marginal d.f.s - $\Phi _{n}^{i}(x)$. They
are some of the key quantities participating in the proofs and give a good
general picture of how the iteration map behaves. We present both the $%
RR_{2} $ and the $TP_{2}$ cases. All the discussed patterns are clearly
visible - convergence, crossings, dependence, etc. \textit{Figures 1-5}
represent the $RR_{2}$ case. We take: $F_{1}\thicksim Lognormal(0.5$, $0.2)$%
, $F_{2}$ $\thicksim $ $Beta(2,2)$, and the copula for $F$ Gaussian with
correlation parameter $\rho =-0.8$.

% === Corrected compact figure layout (Figures 1–5) ===

\begin{center}
\begin{minipage}[t]{0.48\linewidth}\centering
    \includegraphics[width=\linewidth]{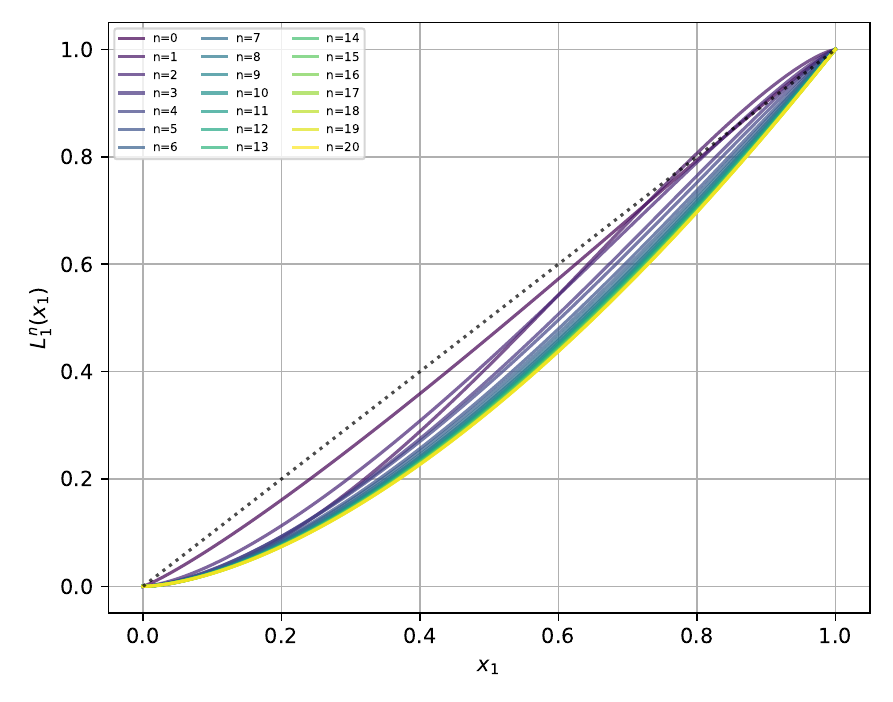}
\captionof{figure}[Figure 1:: $L_{1}^{n}(x_{1})$ evolution]{Figure 1:\\ $L_{1}^{n}(x_{1})$ evolution}
    \label{fig:figure1}
  \end{minipage}\hfill 
\begin{minipage}[t]{0.48\linewidth}\centering
    \includegraphics[width=\linewidth]{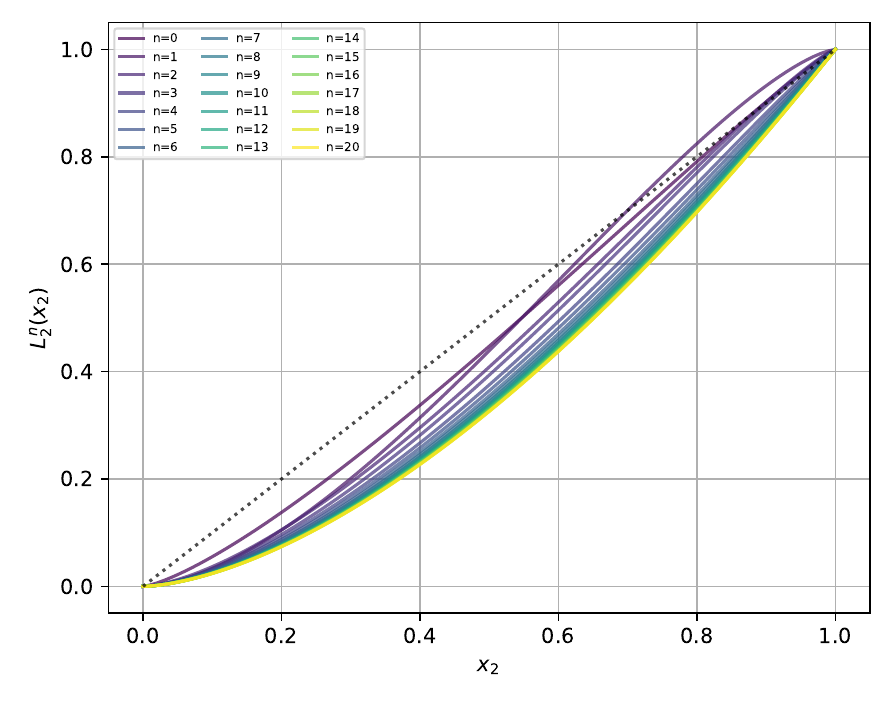}
\captionof{figure}[Figure 2:: $L_{2}^{n}(x_{2})$ evolution]{Figure 2:\\ $L_{2}^{n}(x_{2})$ evolution}
    \label{fig:figure2}
  \end{minipage}

\begin{minipage}[t]{0.48\linewidth}\centering
    \includegraphics[width=\linewidth]{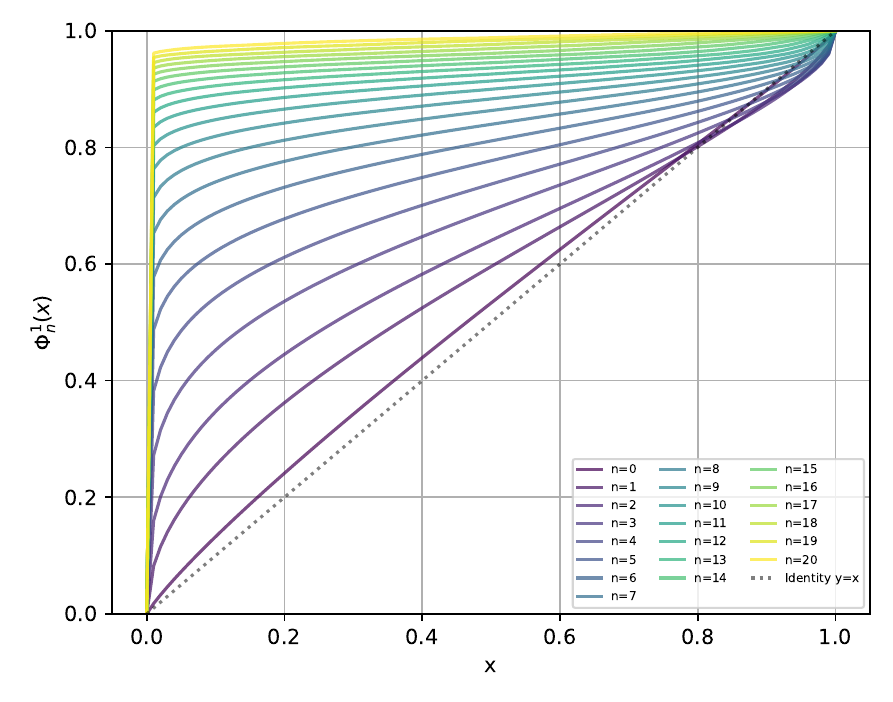}
\captionof{figure}[Figure 3:: Compounds of the inverse marginal d.f.s – $\Phi_{n}^{1}(x)$]{Figure 3:\\ Compounds of the inverse marginal d.f.s – $\Phi_{n}^{1}(x)$}
    \label{fig:figure3}
  \end{minipage}\hfill 
\begin{minipage}[t]{0.48\linewidth}\centering
    \includegraphics[width=\linewidth]{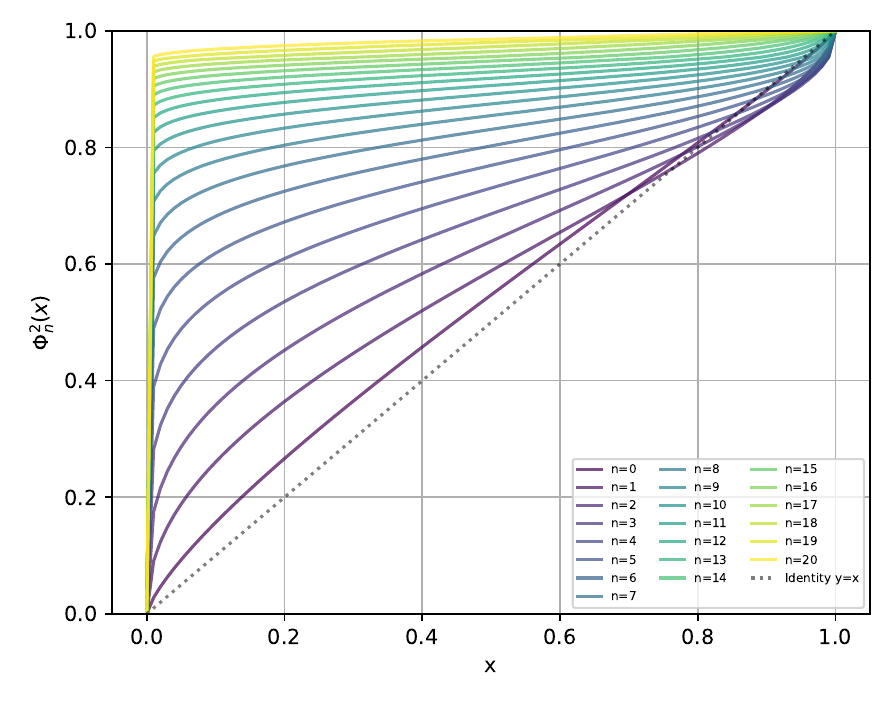}
\captionof{figure}[Figure 4:: Compounds of the inverse marginal d.f.s – $\Phi_{n}^{2}(x)$]{Figure 4:\\ Compounds of the inverse marginal d.f.s – $\Phi_{n}^{2}(x)$}
    \label{fig:figure4}
  \end{minipage}

\includegraphics[width=0.48\linewidth]{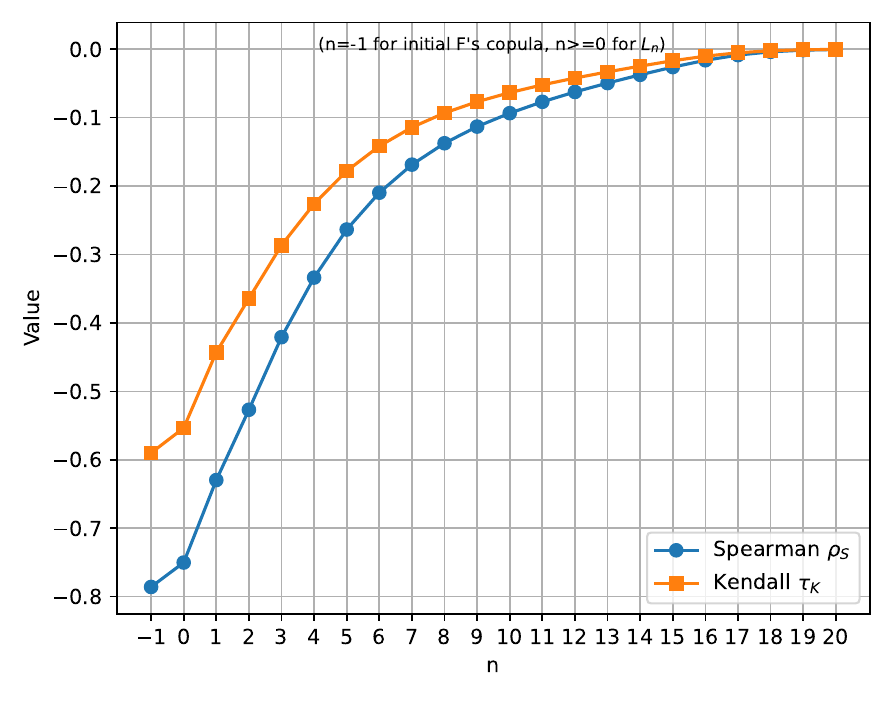} 
\captionof{figure}[Figure 5:: Dependence measures]{Figure 5:\\ Dependence
measures} \label{fig:figure5}
\end{center}

\medskip \noindent Figures 6-11 represent the $TP_{2}$ case\textbf{.} $%
F_{1}\sim $~Sinewave with frequency~$3$ and amplitude~$0.5$, $F_{2}\sim
\Gamma (2,1)$, and the copula for $F$ is Clayton with parameter $\theta =2$.
\medskip % === Corrected compact layout for Figures 6–11 ===

\begin{center}
\begin{minipage}[t]{0.48\linewidth}\centering
    \includegraphics[width=\linewidth]{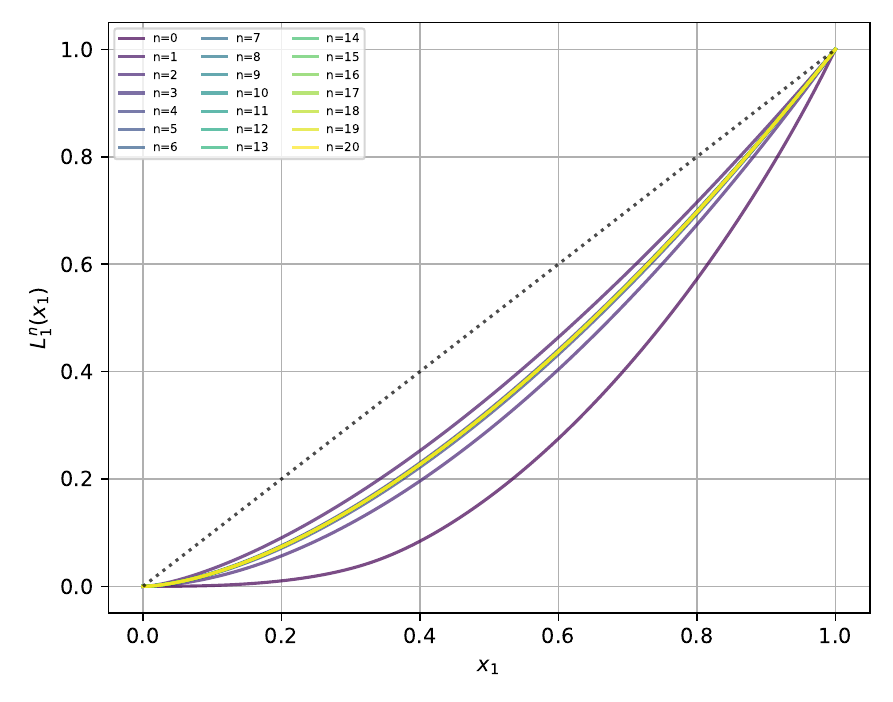}
\captionof{figure}[Figure 6:: $L_{1}^{n}(x_{1})$ evolution ]{Figure 6:\\ $L_{1}^{n}(x_{1})$ evolution}
    \label{fig:figure6}
  \end{minipage}\hfill 
\begin{minipage}[t]{0.48\linewidth}\centering
    \includegraphics[width=\linewidth]{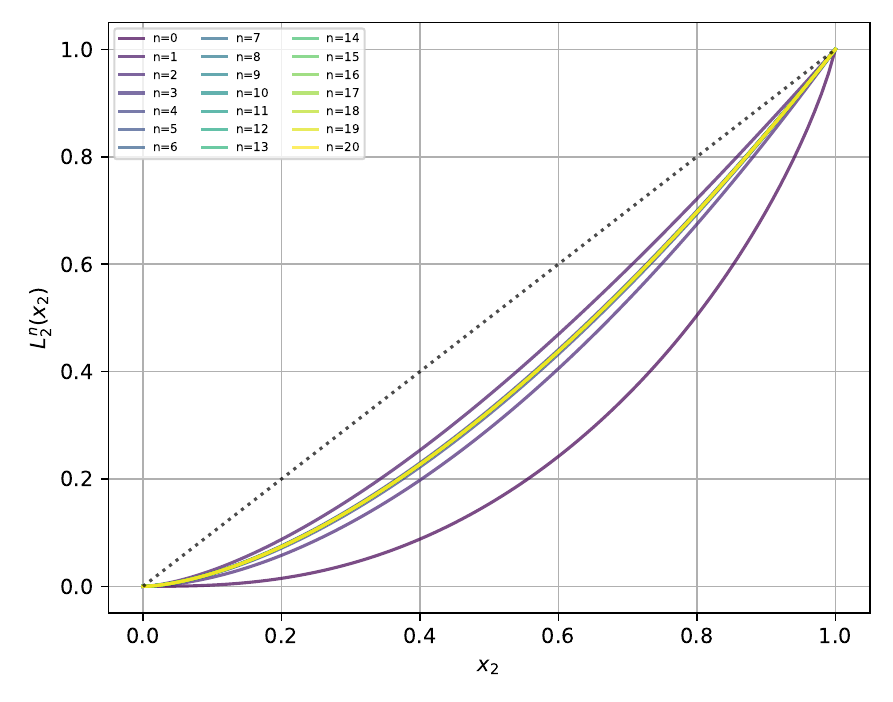}
\captionof{figure}[Figure 7:: $L_{2}^{n}(x_{2})$ evolution]{Figure 7:\\ $L_{2}^{n}(x_{2})$ evolution}
    \label{fig:figure7}
  \end{minipage}

\begin{minipage}[t]{0.48\linewidth}\centering
    \includegraphics[width=\linewidth]{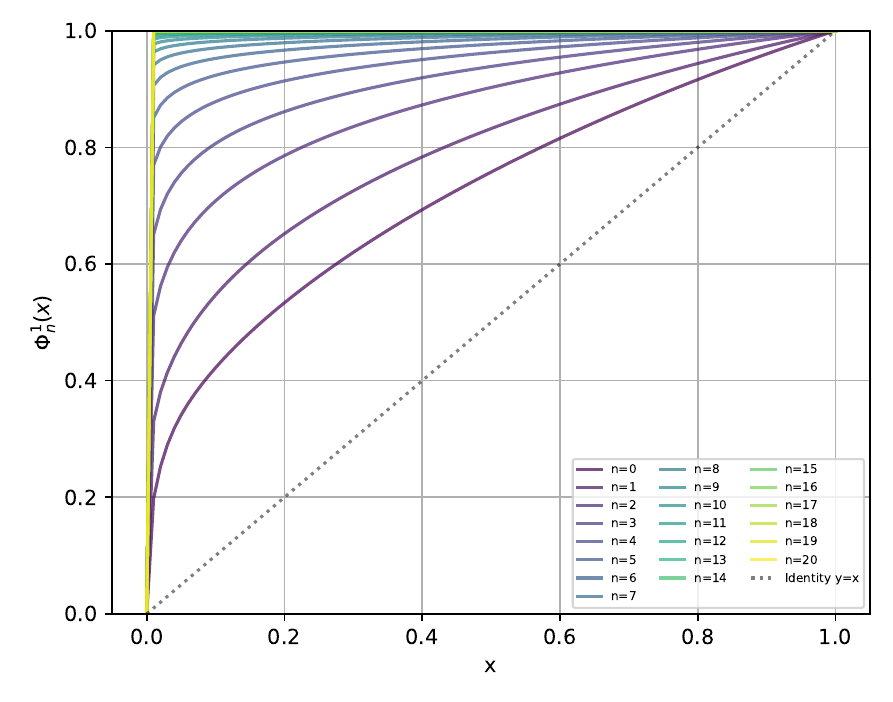}
\captionof{figure}[Figure 8:: Compounds of the inverse marginal d.f.s - $\Phi_{n}^{1}(x)$ ]{Figure 8:\\ Compounds of the inverse marginal d.f.s - $\Phi_{n}^{1}(x)$}
    \label{fig:figure8}
  \end{minipage}\hfill 
\begin{minipage}[t]{0.48\linewidth}\centering
    \includegraphics[width=\linewidth]{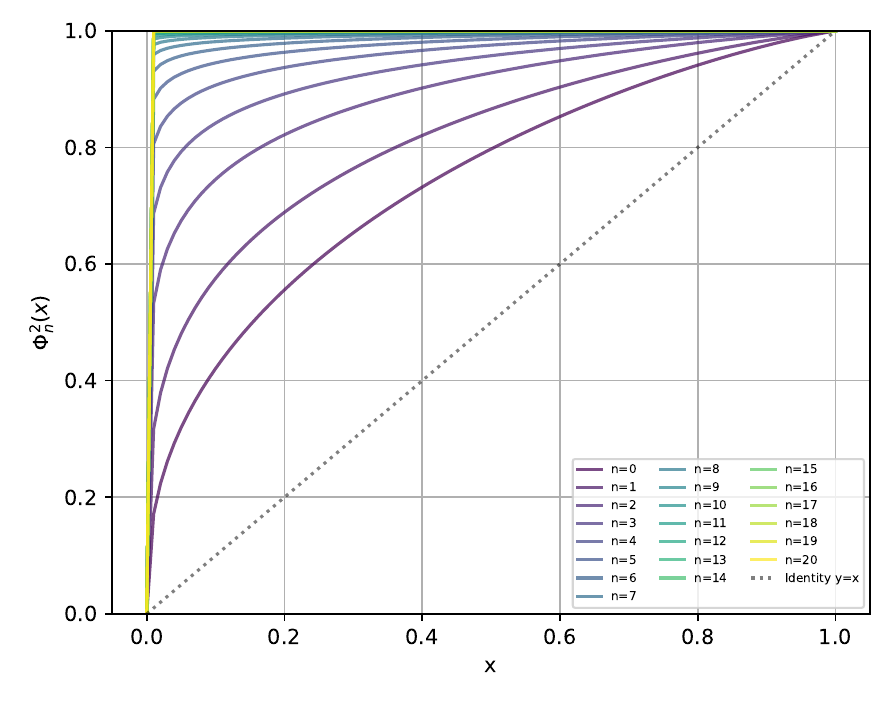}
\captionof{figure}[Figure 9:: Compounds of the inverse marginal d.f.s - $\Phi_{n}^{2}(x)$ ]{Figure 9:\\ Compounds of the inverse marginal d.f.s - $\Phi_{n}^{2}(x)$ }
    \label{fig:figure9}
  \end{minipage}

\begin{minipage}[t]{0.48\linewidth}\centering
    \includegraphics[width=\linewidth]{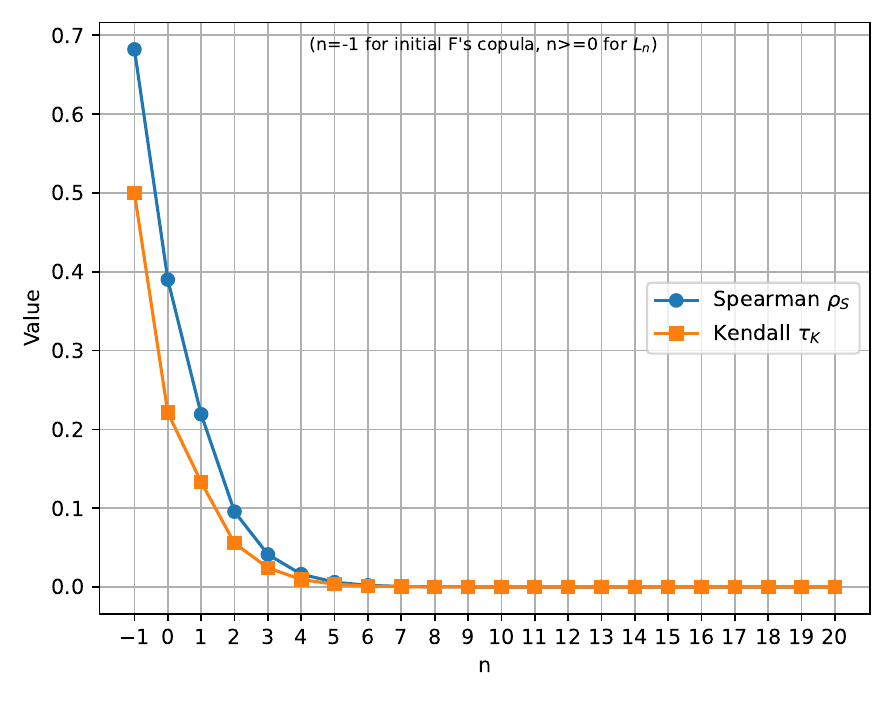}
\captionof{figure}[Figure 10:: Dependence measures ]{Figure 10:\\ Dependence measures}
    \label{fig:figure10}
  \end{minipage}\hfill 
\begin{minipage}[t]{0.48\linewidth}\centering
    \includegraphics[width=\linewidth]{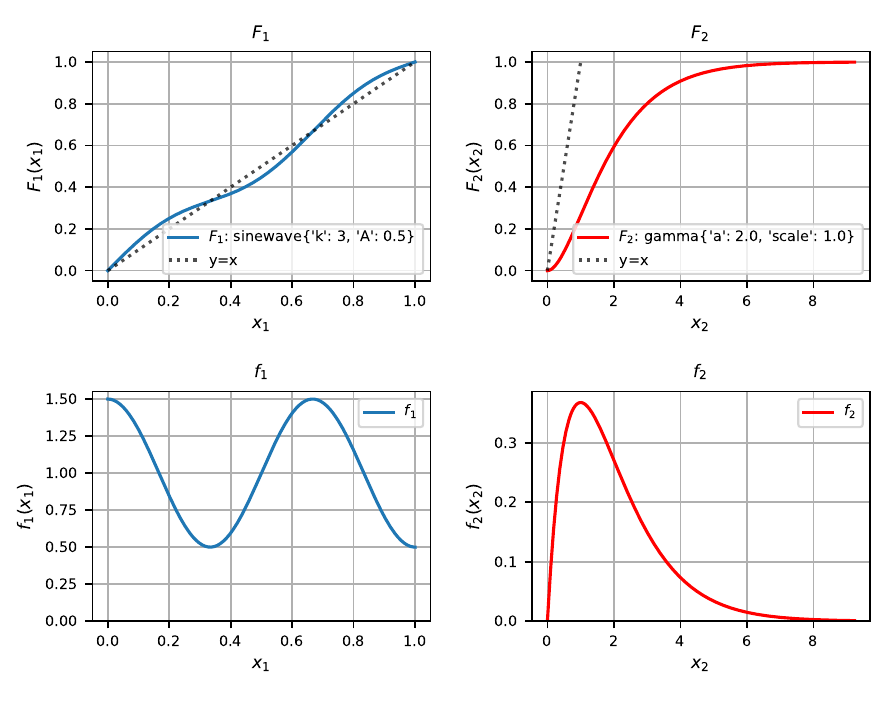}
\captionof{figure}[Figure 11:: Initial distribution d.f.s ($F_{1}$, $F_{2}$) and p.d.f.s ($f_{1}$, $f_{2}$)]{Figure 11:\\ Initial distribution d.f.s ($F_{1}$, $F_{2}$) and p.d.f.s ($f_{1}$, $f_{2}$)}
    \label{fig:figure11}
  \end{minipage}
\end{center}

All plots confirm our theoretical findings. We can observe both quick
convergence and quick decorrelation. Of special interest are the crossing
patterns. For the $TP_{2}$ case, the subdiagonality is clear. For the $%
RR_{2} $ case, we may observe eventual subdiagonality achieved in several
iterations.

\section{Multivariate setting}

In this section we extend the two-dimensional analysis developed earlier in
the main body and the appendices to arbitrary dimension $d\geq 2$. We show
that the entire structure of the theory persists:

\begin{itemize}
\item under $d$-dimensional analogues of $TP_{2}$ and $RR_{2}$, the unique
attractor is the $d$-dimensional independence copula;

\item under \textit{extreme positive dependence}, the unique fixed point is
the $d$-dimensional \textit{Fr\'{e}chet--Hoeffding upper-bound} $%
M_{d}(x)=\min (x_{1},\dots ,x_{d})$, with marginals $L_{i}^{\ast }(x)=x^{%
\frac{1+\sqrt{1+4d}}{2}}$;

\item under suitable \textit{extreme negative dependence} concepts (see \cite%
{[48]} for an overview), fixed points are feasible, which generalizes the 
\textit{Fr\'{e}chet--Hoeffding lower-bound} $W$ from $d=2$.
\end{itemize}

We also show that the $d$-dimensional analogue of the $\Phi$%
--maps---appropriately defined---converges uniformly to a constant and
identifies the fixed point in all three regimes.

Throughout this section the notation and functional--analytic structure
follow those of \textit{Sections~2--4} of the paper with suitable updates. 
\textit{Appendices~A-D} may be consulted for the two--dimensional results
that are used as prototypes in what follows.

\subsection{Setup and notation}

\begin{definition}
\noindent \label{d-centering} Given a continuous distribution $F$ function
on $[0,1]^{d}$ with density $f_{F}$ and continuous marginal d.f.'s $F_{i}$
and densities $f_{i}$, $i=1,\dots ,d$, we set $T_{i}^{F}=F_{i}^{-1}$ and
write $X^{F}=(X_{1}^{F},\dots ,X_{d}^{F})$ for a random vector with law $F$.
Denote also the mixed moment by%
\begin{equation}
I_{F}\;=\;E(\prod_{i=1}^{d}X_{i}^{F})\;=\;\int_{[0,1]^{d}}(%
\prod_{i=1}^{d}u_{i})\,dF(u_{1},\dots ,u_{d}).  \label{940}
\end{equation}

Given $F$ as above with $I_{F}>0$, the \textit{Lorenz curve} $L_{F}$ is
defined by 
\begin{equation}
L_{F}(x_{1},\dots ,x_{d})\;=\;\frac{1}{I_{F}}\int_{[0,T_{1}^{F}(x_{1})]%
\times \cdots \times \lbrack
0,T_{d}^{F}(x_{d})]}(\prod_{i=1}^{d}u_{i})\,dF(u_{1},\dots ,u_{d}),\qquad
(x_{1},\dots ,x_{d})\in \lbrack 0,1]^{d}.  \label{941}
\end{equation}
\end{definition}

Thus $L_{F}$ is the distribution of the random vector $X^{F}$ reweighted by
the product kernel $\prod_{i}X_{i}^{F}$. When $d=2$ this reduces to the
operator studied in \textit{Sections~2--4}.

\begin{definition}
\noindent \label{d-centering-n} Given a distribution $L_{n}$ on $[0,1]^{d}$
with density $l_{n}$ and strictly increasing continuous marginals $L_{i}^{n}$%
, $i=1,\dots ,d$, we set $T_{i}^{n}=(L_{i}^{n})^{-1}$ and 
\begin{equation}
X^{L_{n}}=(X_{1}^{L_{n}},\dots ,X_{d}^{L_{n}})\sim L_{n},\qquad
I_{n}=E[\prod_{i=1}^{d}X_{i}^{L_{n}}].  \label{941a}
\end{equation}

The $d$--dimensional analogue of the map (\ref{5.1}) is given by 
\begin{equation}
L_{n+1}(x_{1},\dots ,x_{d})=\frac{1}{I_{n}}\int_{[0,T_{1}^{n}(x_{1})]\times
\cdots \times \lbrack
0,T_{d}^{n}(x_{d})]}(\prod_{i=1}^{d}u_{i})\,dL_{n}(u_{1},\dots ,u_{d}),
\label{941b}
\end{equation}%
for $(x_{1},\dots ,x_{d})\in \lbrack 0,1]^{d}$, with initial condition $%
L_{0} $ obtained from $F$ by \textit{Definition~\ref{d-centering}}.
\end{definition}

We can view $\{L_{n}\}_{n\geq 0}$ is the orbit of $F$ under repeated
centering by the product integrand. Under density existence hypothesis,
differentiating under the integral sign yields the density recursion 
\begin{equation}
\begin{split}
l_{n+1}(x_{1},\dots ,x_{d})& =\frac{1}{I_{n}}%
(\prod_{i=1}^{d}T_{i}^{n}(x_{i}))\,l_{n}(T_{1}^{n}(x_{1}),\dots
,T_{d}^{n}(x_{d}))(\prod_{i=1}^{d}(T_{i}^{n})^{\prime }(x_{i})), \\
& (x_{1},\dots ,x_{d})\in (0,1)^{d}
\end{split}
\label{941c}
\end{equation}

We now define the dependence concepts under which the main theorem of this
section is stated. In much detail they are discussed in \textit{Appendix D.4.%
}

\begin{definition}
\label{MTP2-SMRR2} A strictly positive density $l:[0,1]^{d}\rightarrow
(0,\infty )$ is multivariate totally positive of order $2$ ($MTP_{2}$) if 
\begin{equation}
l(x\wedge y)\,l(x\vee y)\ \geq \ l(x)\,l(y),\qquad x,y\in \lbrack 0,1]^{d},
\label{2000a}
\end{equation}%
where $\wedge ,\vee $ denote coordinatewise minimum and maximum. It is 
\textit{multivariate reverse regular of order }$2$ ($MRR_{2}$) if the
reverse inequality holds: 
\begin{equation}
l(x\wedge y)\,l(x\vee y)\ \leq \ l(x)\,l(y),\qquad x,y\in \lbrack 0,1]^{d}.
\label{2000b}
\end{equation}%
We say that $l$ is strictly $MRR_{2}$ ($SMRR_{2}$) if for every compact set $%
K\subset (0,1)^{d}$ there exists a set of full \textit{Lebesgue measure} in $%
K\times K$ on which (\ref{2000b}) holds with strict inequality whenever $x$
and $y$ are incomparable (i.e.\ neither $x\leq y$ nor $y\leq x$
coordinatewise). Equivalently (and more concretely), on each $K\subset
(0,1)^{d}$ on which $l$ is bounded away from $0$, (\ref{2000b}) is strict
for all incomparable $x,y\in K$.
\end{definition}

\begin{remark}
\label{rem:SMRR2-analogue} In $d=2$, strict $MRR_{2}$ coincides with $%
SRR_{2} $ as used in the bivariate theorem: strictness is only needed on
compact subsets away from the boundary, precisely where the crossing and
return--shift arguments are carried out in \textit{Appendix~D.2}.
\end{remark}

As before, we extend d.f.s from $[0,1]^{d}$ to $R^{d}$ by clipping. Let $\pi
:R^{d}\rightarrow \lbrack 0,1]^{d}$ be the coordinatewise clipping map and
set 
\begin{equation}
\mathcal{R}_{0}=\{x\in R^{d}:\exists i,\ x_{i}\leq 0\},\qquad \mathcal{R}%
_{1}=\{x\in R^{d}:\forall i,\ x_{i}\geq 1\}.  \label{2000}
\end{equation}%
For any d.f. $H$ on $[0,1]^{d}$, define its extension by 
\begin{equation}
H(x)=%
\begin{cases}
0, & x\in \mathcal{R}_{0}, \\ 
1, & x\in \mathcal{R}_{1}, \\ 
H(\pi (x)), & \text{otherwise}.%
\end{cases}
\label{2001}
\end{equation}

\subsection{Main theorem: long and short formulations}

The following theorem is the multivariate counterpart to \textit{Theorem~\ref%
{main-th}.}

\begin{theorem}
\label{Dd-main-long} Let $d\geq 2$ and let $F$ be a distribution function on 
$[0,1]^{d}$ with strictly positive density $f$ on $(0,1)^{d}$. Define $%
\{L_{n}\}_{n\geq 0}$ by $L_{0}=L^{F}$ and (\ref{941b}), and extend each $%
L_{n}$ to $R^{d}$ by (\ref{2001}). Assume that $f$ is either $MTP_{2}$ or $%
SMRR_{2}$. Then 
\begin{equation*}
\sup_{x\in R^{d}}|L_{n}(x)-G(x)|\text{ }\overset{n\rightarrow +\infty }{%
\longrightarrow }\text{ }0,
\end{equation*}%
where $G(x)=\prod_{i=1}^{d}x_{i}^{\frac{1+\sqrt{5}}{2}}$ on $[0,1]^{d}$ and
is extended by (\ref{2001}).
\end{theorem}

\begin{proof}
The core proof strategy for \textit{Theorem~\ref{main-th}}, which is based
on the iteration (\ref{12}), extends straightforwardly to the multivariate
case. Similarly, the cancellation mechanism for the initial copula density $%
c^{F}$ in (\ref{86}) and (\ref{87}) also applies without modification.

The crucial remaining step is therefore to prove the subsequential
convergence of the compositions $\Phi
_{n}^{i}(x)=L_{i}^{0,-1}(...(L_{i}^{n-1,-1}(L_{i}^{n,-1}(x)))$, $i=1,2$ to
constants. A detailed proof is provided in \textit{Appendix D.4}, and it
builds upon the conceptual framework established in \textit{Appendices
D.1-D.3}.
\end{proof}

We can also have a shorter formulation based on the \textit{Lorenz map}.

\begin{theorem}
\label{Dd-main-short} Let $\mathcal{L}$ denote the operator on distribution
functions induced by (\ref{941b}). Under the assumptions of Theorem~\ref%
{Dd-main-long}, we have $\mathcal{L}^{n}(F)\rightarrow G$ uniformly on $%
R^{d} $.
\end{theorem}

\begin{remark}
\label{Dd-two-forms} \textit{Theorem~\ref{Dd-main-short}} is a notational
reformulation of \textit{Theorem~\ref{Dd-main-long}}. We keep the iterated
notation $L_{n}$ since it matches the bivariate presentation and the $\Phi
_{n}$--construction in \textit{Appendices~D.1--D.2}.
\end{remark}

\subsection{Extreme dependence}

In this addendum we treat the two \textquotedblleft extreme
dependence\textquotedblright\ regimes in the same computational spirit as
the bivariate sections~\textit{Fr\'{e}chet--Hoeffding upper-bound} and 
\textit{lower-bound}. The guiding principle is that, under these extremal
dependence structures, the $d$--dimensional \textit{Lorenz operator} reduces
to a one--dimensional recursion for each marginal, exactly as in the
bivariate case. We therefore (i) write down the $d$--dimensional analogues
of (\ref{43}) and (\ref{42.4}), and (ii) investigate whether the
corresponding extreme dependence concepts are invariant (hence
\textquotedblleft fixed\textquotedblright ) at the dependence level.

Throughout, we work with the \textit{Lorenz operator} from \textit{%
Definition~\ref{d-centering}} and the iteration $L^{0}=L^{F}$, $%
L^{n+1}=L_{L^{n}}$. Furthermore, consistent with the extremal bivariate
case, we maintain the standing assumption of strictly positive marginal
densities. Nevertheless, we explicitly examine scenarios where this
assumption is lifted; a more detailed treatment of the latter is provided in 
\textit{Appendix E} as an extension.

\subsubsection{Extreme positive dependence: Fr\'{e}chet--Hoeffding
upper-bound $M_{d}$}

\paragraph{The dependence concept and its invariance}

Let 
\begin{equation}
M_{d}(x_{1},\dots ,x_{d})=\min (x_{1},\dots ,x_{d}),\qquad (x_{1},\dots
,x_{d})\in \lbrack 0,1]^{d},  \label{2002}
\end{equation}%
be the \textit{Fr\'{e}chet--Hoeffding upper-bound}. A random vector $%
U=(U_{1},\dots ,U_{d})$ with uniform marginals has copula $M_{d}$ if and
only if it is comonotone, i.e.,\ there exists a single uniform $V\sim 
\mathrm{U}(0,1)$ such that $U_{i}=V$ a.s.\ for all $i$ (equivalently, the
law is supported on the diagonal). We have the following result.

\begin{lemma}
\label{Md-invariance} \medskip \noindent Assume that $F$ has comonotone
dependence in the sense that there exists a random variable $V\in \lbrack
0,1]$ and non--decreasing maps $g_{i}:[0,1]\rightarrow \lbrack 0,1]$ such
that $X_{i}=g_{i}(V)$ a.s.\ and $X\sim F$. Then $L_{F}$ is comonotone as
well (hence has copula $M_{d}$).
\end{lemma}

\begin{proof}
\medskip The defining integral for $L_{F}$ is a product tilt by $%
\prod_{i=1}^{d}X_{i}$ restricted to rectangles, followed by marginal
reparametrization via the quantiles $T_{i}^{F}$; see \textit{Definition~\ref%
{d-centering}}. If $X$ is supported on the comonotone curve $%
\{(g_{1}(v),\dots ,g_{d}(v)):v\in \lbrack 0,1]\}$, then multiplying by $%
\prod_{i}X_{i}$ cannot create mass off that support, and the coordinatewise
quantile maps preserve the same one--parameter representation. This is
exactly the heuristic argument used in the bivariate \textit{Fr\'{e}%
chet--Hoeffding upper-bound}, now with $d$ coordinates. Alternatively, we
can replicate also the formal algebraic proof from there without problems to 
$d$ dimensions, but we will skip this for space considerations.
\end{proof}

\medskip Thus, in the \textit{Fr\'{e}chet--Hoeffding upper-bound} regime the
entire iteration stays within the comonotone class and the only nontrivial
evolution is in the marginals.

\paragraph{The $d$--dimensional marginal recursion}

Fix $i\in \{1,\dots ,d\}$ and write $L_{i,+}^{n}$ for the $i$-th marginal of
the $n$-th iterate in the comonotone regime (the \textquotedblleft $+$%
\textquotedblright\ reminds of the \textit{upper-bound} case). Exactly as in
the bivariate derivation of (\ref{43}), comonotonicity implies that, when we
write the \textit{Lorenz} \textit{update} in quantile form, the product
kernel $\prod_{k=1}^{d}X_{k}$ collapses to a one-dimensional kernel
proportional to $(\cdot )^{d}$ along the common driving variable. This
yields the direct $d$--dimensional generalization of (\ref{43}):

\begin{equation}
L_{i,+}^{n+1}(x)=\frac{\int_{0}^{x}(L_{i,+}^{n,-1}(u))^{d}\,du}{%
\int_{0}^{1}(L_{i,+}^{n,-1}(u))^{d}\,du},\qquad L_{i,+}^{0}(x)=\frac{%
\int_{0}^{x}(F_{i}^{-1}(u))^{d}\,du}{\int_{0}^{1}(F_{i}^{-1}(u))^{d}\,du}%
,\qquad x\in \lbrack 0,1].  \label{43d}
\end{equation}

\begin{remark}
\label{43d-eq} For $d=2$, (\ref{43d}) reduces exactly to (\ref{43}). The
derivation is verbatim the same as in the main text: comonotonicity reduces
the $d$--dimensional product kernel $\prod_{k}X_{k}$ to a one-dimensional
power, the only change being that the power becomes $d$ instead of $2$.
\end{remark}

\paragraph{Fixed point within the comonotone class}

\textit{Lemma~\ref{Md-invariance}} shows that the dependence structure
(copula $M_{d}$) is invariant. A genuine fixed point $L^{\ast }$ in this
regime is therefore characterized by solving the one-dimensional fixed-point
problem induced by (\ref{43d}) for each marginal.

\begin{claim}
\medskip $\noindent $\label{Md-fixed} Assume $L^{\ast }$ is comonotone and
satisfies $L_{L^{\ast }}=L^{\ast }$. Then its marginals $L_{i,+}^{\ast }$
satisfy the fixed-point equation 
\begin{equation}
L_{i,+}^{\ast }(x)=\frac{\int_{0}^{x}((L_{i,+}^{\ast })^{-1}(u))^{d}\,du}{%
\int_{0}^{1}((L_{i,+}^{\ast })^{-1}(u))^{d}\,du},\qquad x\in \lbrack 0,1].
\label{2003}
\end{equation}%
In particular, the single solution is the power-law solution $L_{i,+}^{\ast
}(x)=x^{a}$, where the exponent satisfies the relation 
\begin{equation}
a=1+\frac{d}{a}\qquad \Longleftrightarrow \qquad a^{2}-a-d=0\text{ }%
\Longleftrightarrow \text{ }a=\frac{1+\sqrt{1+4d}}{2}  \label{2004}
\end{equation}
\end{claim}

\begin{proof}
The first statement is immediate from (\ref{43d}) by setting $%
L_{i,+}^{n+1}=L_{i,+}^{n}=L_{i,+}^{\ast }$. For the power-law computation,
we can use similar polynomial majorization methods as in \textit{Appendix C.2%
} and \cite{[3]}.
\end{proof}

\begin{remark}
Notably, for large $d$, $L_{i,+}^{\ast }(x)$ exhibits the behavior of a
univariate \textit{Lorenz curve} under extreme inequality, effectively
becoming an indicator function of the form $1_{x\in \lbrack 1-\varepsilon
,1]}$.
\end{remark}

\begin{remark}
Since each function $L_{i,+}^{n}(x)$ in the sequence is subdiagonal, it
follows that the corresponding sequence of $\Phi $--maps converges to 1.
\end{remark}

\subsubsection{Extreme negative dependence}

\paragraph{Scope and aims}

In dimension $d=2$, the \textit{Fr\'{e}chet--Hoeffding lower-bound} is a
genuine copula and yields the canonical notion of extreme negative
dependence, namely countermonotonicity. This allows the \textit{Lorenz
integrals} to be reduced to one-dimensional integrals along the line $%
u_{2}=1-u_{1}$, and leads to closed marginal recursions.

For $d\geq 3$, the pointwise \textit{Fr\'{e}chet--Hoeffding lower-bound} is
not, in general, a copula. Hence there is no unique multivariate analogue of
countermonotonicity obtained by simply replacing the bivariate \textit{%
lower-bound} by a $d$-dimensional \textit{lower-bound} copula. We must
instead choose a substitute negative-dependence concept. In this section we
study three such concepts (see \cite{[48]}): \textit{paired
countermonotonicity}, \textit{joint mixability}, and $\Sigma $\textit{%
-countermonotonicity}.

Throughout this section we remain in the regular regime of the main text:
the relevant marginal distribution functions are assumed to be continuous
and nondegenerate. This atom-free assumption is essential whenever we
replace events of the form 
\begin{equation*}
\{T_{i}(u)\leq T_{i}(x)\}
\end{equation*}%
by intervals such as $[0,x]$. If atoms are present, this replacement must be
corrected by the generalized-inverse formulas of \textit{Appendix~E}. In
particular, the discrete and endpoint atomic cases are not treated in this
section.

The objectives of this section are therefore:

\begin{enumerate}
\item to identify negative-dependence concepts which are meaningful in all
dimensions;

\item to determine which of them are preserved by the \textit{Lorenz map} in
the regular continuous setting;

\item to record the explicit one-step \textit{Lorenz formulas} available
under each concept;

\item to explain where closure fails in dimension $d\geq 3$ without
additional dependence information.
\end{enumerate}

\paragraph{Three workable concepts and why we focus on them}

Because the $d$-dimensional \textit{Fr\'{e}chet--Hoeffding lower-bound} is
not a copula for $d\geq 3$, multivariate extreme negative dependence is not
represented by a single universal copula. We focus on the following three
concepts.

\begin{enumerate}
\item \emph{Paired countermonotonicity} (PCM). This imposes ordinary
bivariate countermonotonicity inside fixed disjoint pairs of coordinates,
and independence across the pairs. It is a direct and computable
tensorization of the bivariate lower-bound theory.

\item \emph{Joint mixability} (JM). This is the constant-sum condition 
\begin{equation*}
\sum_{i=1}^{d}X_{i}=k\quad \text{a.s.}
\end{equation*}%
It is a standard multivariate negative-dependence concept, but in dimensions 
$d\geq 3$ it defines a class of \textit{couplings} rather than a unique
dependence structure.

\item \emph{$\Sigma $-countermonotonicity} ($\Sigma $-CTM). This requires
complementary partial sums to be countermonotone. It is weaker than
specifying a complete coupling, and therefore does not by itself determine
the \textit{Lorenz update} from the marginals alone.
\end{enumerate}

We do not use the full condition that every pair $(X_{i},X_{j})$ be
countermonotone, because in dimensions $d\geq 3$ this is typically
infeasible except in degenerate or very small-support cases. PCM is the
feasible paired version of that idea. JM and $\Sigma $-CTM are important
structural concepts, but, as shown below, they do not generally give closed
all-$n$ \textit{Lorenz iterations} without additional invariant structure.

\paragraph{Concept I: Paired countermonotonicity (PCM)}

\subparagraph{Definition and interpretation}

Let $d=2m$. Partition the coordinates into disjoint pairs 
\begin{equation}
(1,2),(3,4),\ldots ,(2m-1,2m).  \label{2050}
\end{equation}%
Let 
\begin{equation*}
W(s,t)=\max \{s+t-1,0\}
\end{equation*}%
be the bivariate \textit{Fr\'{e}chet--Hoeffding lower-bound} copula.

\begin{definition}
\medskip \noindent \label{def:PCM} A random vector $X=(X_{1},\ldots ,X_{d})$
is \textit{paired countermonotone}, or PCM, with respect to the pairing (\ref%
{2050}), if:

\begin{enumerate}
\item for each $r=1,\ldots ,m$, the pair $(X_{2r-1},X_{2r})$ is
countermonotone in the bivariate sense;

\item the blocks 
\begin{equation*}
(X_{1},X_{2}),\ (X_{3},X_{4}),\ldots ,(X_{2m-1},X_{2m})
\end{equation*}%
are independent.
\end{enumerate}
\end{definition}

\begin{remark}
The \textit{pairwise countermonotonicity} (PW-CTM) from \cite{[48]} requires
every pair $(X_{i},X_{j})$ to be countermonotone. PCM requires this only for
a fixed pairing, and hence is feasible and stable under many constructions.
In particular, PCM should be viewed as a structured pairwise-extreme
negative dependence model rather than a full PW-CTM condition.
\end{remark}

\subparagraph{Lorenz map under PCM}

Assume $X\sim F$ is PCM, and assume the one-dimensional marginals are
continuous. Let $F_{i}$ be the marginal distribution functions and 
\begin{equation*}
T_{i}^{F}=F_{i}^{-1}.
\end{equation*}%
For each block $r$, there exists $Z_{r}\sim U(0,1)$ such that 
\begin{equation}
(X_{2r-1},X_{2r})=(T_{2r-1}^{F}(Z_{r}),T_{2r}^{F}(1-Z_{r})),  \label{2051}
\end{equation}%
and the random variables $Z_{1},\ldots ,Z_{m}$ are independent.

For each block define 
\begin{equation}
\mu _{r}=\int_{0}^{1}T_{2r-1}^{F}(u)T_{2r}^{F}(1-u)\,du.  \label{2052}
\end{equation}%
Assume $\mu _{r}>0$. Define the bivariate block \textit{Lorenz factor} 
\begin{equation}
L_{r}(x,y)=\frac{\int_{1-y}^{x}T_{2r-1}^{F}(u)T_{2r}^{F}(1-u)\,du}{%
\int_{0}^{1}T_{2r-1}^{F}(u)T_{2r}^{F}(1-u)\,du},  \label{2053}
\end{equation}%
with the convention that the numerator is zero if $x<1-y$.

\begin{theorem}
\medskip \noindent \label{PCM-product} Let $d=2m$, and let $F$ be PCM with
continuous marginals and positive block normalizers $\mu _{r}$. Then the
Lorenz image $L_{F}$ factorizes blockwise: 
\begin{equation}
L_{F}(x_{1},\ldots ,x_{d})=\prod_{r=1}^{m}L_{r}(x_{2r-1},x_{2r}).
\label{2054}
\end{equation}%
In particular, the $i$-th marginal of $L_{F}$ depends only on the
countermonotone block containing $i$. \medskip
\end{theorem}

\begin{proof}
\noindent The \textit{Lorenz numerator} can be written as 
\begin{equation}
E\left[ \left( \prod_{i=1}^{d}X_{i}\right) \prod_{i=1}^{d}1_{\{X_{i}\leq
T_{i}^{F}(x_{i})\}}\right] .  \label{2055}
\end{equation}%
Using the PCM representation (\ref{2051}), both the product $\prod_{i}X_{i}$
and the indicator product factor over the independent blocks. The
expectation therefore factorizes into a product of bivariate countermonotone
integrals. Continuity of the block marginals gives 
\begin{equation*}
\{T_{2r-1}^{F}(Z_{r})\leq T_{2r-1}^{F}(x)\}=\{Z_{r}\leq x\}
\end{equation*}%
and 
\begin{equation*}
\{T_{2r}^{F}(1-Z_{r})\leq T_{2r}^{F}(y)\}=\{1-Z_{r}\leq y\}
\end{equation*}%
up to null sets. Hence the $r$-th block contributes exactly (\ref{2053}).
Multiplying over $r$ gives (\ref{2054}). \hfill\ \medskip
\end{proof}

\begin{remark}
\medskip \noindent The continuity assumption in \textit{Theorem~\ref%
{PCM-product}} is not cosmetic. If a block marginal has an atom, then the
event 
\begin{equation*}
\{T_{i}^{F}(u)\leq T_{i}^{F}(x)\}
\end{equation*}%
need not equal $[0,x]$. The correct atomic replacement uses the plateau
endpoint $F_{i}(T_{i}^{F}(x))$, as explained in \textit{Appendix~E}. \medskip
\end{remark}

\subparagraph{Invariance of PCM}

\begin{lemma}
\noindent \label{PCM-invariance} Let $d=2m$, and let $F$ be PCM with respect
to the fixed pairing (\ref{2050}). Assume the relevant marginals are
continuous and the Lorenz normalizer is positive. Then $L_{F}$ is again PCM
with respect to the same pairing. Consequently, if $L^{0}=L^{F}$ and $%
L^{n+1}=L_{L^{n}}$, then $L^{n}$ is PCM for all $n\geq 0$, as long as the
iterates remain in the regular atom-free regime. \medskip
\end{lemma}

\begin{proof}
Write $X=(B_{1},\ldots ,B_{m})$, where 
\begin{equation*}
B_{r}=(X_{2r-1},X_{2r}).
\end{equation*}%
By PCM, the blocks are independent and each block is bivariate
countermonotone.

The \textit{Lorenz step} consists of two operations. First one tilts the law
by the weight 
\begin{equation*}
\prod_{i=1}^{d}X_{i}=\prod_{r=1}^{m}X_{2r-1}X_{2r}.
\end{equation*}%
This weight factorizes over the blocks. Hence the tilted law still has
independent blocks. Inside each block, tilting by $X_{2r-1}X_{2r}$ does not
move mass off the countermonotone support; it only changes the marginal
distributions along that support. Therefore each tilted block remains
countermonotone with respect to its new marginals.

Second, the \textit{Lorenz map} standardizes the coordinates by their
marginal distribution functions. In the present continuous setting, these
coordinatewise monotone transformations preserve countermonotonicity inside
each block and preserve independence across blocks. Therefore $L_{F}$ is
again PCM. Iteration gives the result.
\end{proof}

\subparagraph{Marginal iteration under PCM}

Let $L^{0}=L^{F}$, $L^{n+1}=L_{L^{n}}$, and suppose the trajectory remains
PCM in the regular continuous sense. For a block $(i,j)=(2r-1,2r)$, write 
\begin{equation*}
T_{i}^{n}=(L_{i}^{n})^{-1},\qquad T_{j}^{n}=(L_{j}^{n})^{-1}.
\end{equation*}%
Then the block denominator is 
\begin{equation*}
\mu _{ij,n}=\int_{0}^{1}T_{i}^{n}(u)T_{j}^{n}(1-u)\,du.
\end{equation*}%
The marginal updates are 
\begin{equation}
L_{i}^{n+1}(x)=\frac{\int_{0}^{x}T_{i}^{n}(u)T_{j}^{n}(1-u)\,du}{%
\int_{0}^{1}T_{i}^{n}(u)T_{j}^{n}(1-u)\,du},  \label{2056a}
\end{equation}%
and 
\begin{equation}
L_{j}^{n+1}(x)=\frac{\int_{1-x}^{1}T_{i}^{n}(u)T_{j}^{n}(1-u)\,du}{%
\int_{0}^{1}T_{i}^{n}(u)T_{j}^{n}(1-u)\,du}.  \label{2056b}
\end{equation}%
Under the usual complementary symmetry inside the block, these reduce to the
one-dimensional lower-bound recursion studied in the bivariate part of the
paper.

\begin{remark}
\medskip \noindent PCM is the closest multivariate analogue of the bivariate 
\textit{Fr\'{e}chet--Hoeffding lower-bound} case that remains fully explicit
in the regular continuous setting. It is a tensorization of the bivariate
countermonotone dynamics over independent blocks. The atomic version is
different and is treated in \textit{Appendix~E}. \medskip
\end{remark}

\paragraph{Concept II: Joint mixability (JM)}

\subparagraph{Definition and interpretation}

\begin{definition}
\textbf{\label{JM}} A random vector $X=(X_{1},\ldots ,X_{d})$ is \textit{%
jointly mixable} if there exists $k\in R$ such that 
\begin{equation}
\sum_{i=1}^{d}X_{i}=k\quad \text{a.s.}  \label{2057}
\end{equation}%
\medskip
\end{definition}

If (\ref{2057}) holds, then for every subset $A\subset \{1,\ldots ,d\}$, 
\begin{equation}
\sum_{j\notin A}X_{j}=k-\sum_{i\in A}X_{i}.  \label{2058}
\end{equation}%
Thus complementary partial sums move in exact opposition.

Assume $X\sim F$ is supported on $[0,1]^{d}$ and satisfies (\ref{2057}) for
some $k\in (0,d)$. Let $\mu $ be the law of $(X_{1},\ldots ,X_{d-1})$, and
define 
\begin{equation*}
h(u_{1},\ldots ,u_{d-1})=k-\sum_{j=1}^{d-1}u_{j}.
\end{equation*}%
Then $\mu $ is supported on 
\begin{equation*}
H=\{u\in \lbrack 0,1]^{d-1}:0\leq h(u)\leq 1\}.
\end{equation*}

\begin{proposition}
\medskip \noindent \textbf{\label{JM-dm1}} The Lorenz normalizer is 
\begin{equation}
I_{F}=E\left[ \prod_{i=1}^{d}X_{i}\right] =\int_{H}\left(
\prod_{j=1}^{d-1}u_{j}\right) h(u)\,d\mu (u).  \label{2059}
\end{equation}%
For $x=(x_{1},\ldots ,x_{d})\in \lbrack 0,1]^{d}$, 
\begin{equation}
L_{F}(x)=\frac{1}{I_{F}}\int_{R_{F}(x)}\left( \prod_{j=1}^{d-1}u_{j}\right)
h(u)\,d\mu (u),  \label{2060}
\end{equation}%
where 
\begin{equation*}
R_{F}(x)=\left\{ u\in H:0\leq u_{j}\leq T_{j}^{F}(x_{j}),\ 1\leq j\leq
d-1,\quad 0\leq h(u)\leq T_{d}^{F}(x_{d})\right\} .
\end{equation*}%
In particular, 
\begin{equation}
L_{d}^{F}(x)=\frac{1}{I_{F}}\int_{\{u\in H:\,0\leq h(u)\leq
T_{d}^{F}(x)\}}\left( \prod_{j=1}^{d-1}u_{j}\right) h(u)\,d\mu (u),
\label{2061}
\end{equation}%
and, for $i=1,\ldots ,d-1$, 
\begin{equation}
L_{i}^{F}(x)=\frac{1}{I_{F}}\int_{\{u\in H:\,0\leq u_{i}\leq
T_{i}^{F}(x)\}}\left( \prod_{j=1}^{d-1}u_{j}\right) h(u)\,d\mu (u).
\label{2062}
\end{equation}%
\medskip
\end{proposition}

\begin{remark}
\medskip For $d=2$, the joint-mixability manifold is one-dimensional, and
the formula above reduces to a single integral. For $d\geq 3$, the manifold
has dimension $d-1\geq 2$, and the right-hand sides depend on the full
coupling $\mu $, not only on the one-dimensional marginals. Therefore JM
does not give a closed marginal recursion in terms of $(F_{1},\ldots ,F_{d})$
alone. \medskip
\end{remark}

\subparagraph{JM is not generally invariant under the Lorenz map}

The bivariate \textit{lower-bound} case closes because the countermonotone
coupling is essentially unique. In dimension $d\geq 3$, \textit{joint
mixability} specifies a class of \textit{couplings}. The \textit{Lorenz map}
depends on the full joint law through the product kernel and the truncation
sets, and therefore the one-dimensional marginals do not determine the 
\textit{Lorenz update}.

Moreover, JM is not automatically preserved by the full \textit{Lorenz map}.
The \textit{Lorenz iteration} may be decomposed into two parts:

\begin{enumerate}
\item product tilting by $\prod_iX_i$;

\item coordinatewise marginal standardization.
\end{enumerate}

The product tilt preserves the support constraint 
\begin{equation*}
\sum_{i}X_{i}=k,
\end{equation*}%
because it does not move mass off the support. However, the subsequent
coordinatewise standardization applies generally different nonlinear
monotone maps to the coordinates. Such transformations do not preserve the
identity 
\begin{equation*}
\sum_{i}X_{i}=k
\end{equation*}%
in general. Consequently, JM gives correct one-step \textit{Lorenz formulas}%
, but it does not in general define an invariant all-$n$ \textit{Lorenz
dynamics}.

\subparagraph{One-factor rearrangement representation}

A useful way to describe particular \textit{JM couplings} is by
rearrangements of a common uniform variable. Let $U\sim U(0,1)$. Choose
measure-preserving maps $\pi _{i}:[0,1]\rightarrow \lbrack 0,1]$ and set 
\begin{equation}
X_{i}=T_{i}^{F}(\pi _{i}(U))=q_{i}(U).  \label{2063}
\end{equation}%
Then $X_{i}\sim F_{i}$. The JM constraint becomes the pointwise identity 
\begin{equation}
\sum_{i=1}^{d}q_{i}(u)=\sum_{i=1}^{d}T_{i}^{F}(\pi _{i}(u))\equiv k\quad 
\text{for a.e. }u.  \label{2064}
\end{equation}

Define 
\begin{equation}
K(u)=\prod_{j=1}^{d}q_{j}(u)=\prod_{j=1}^{d}T_{j}^{F}(\pi _{j}(u)),\qquad
I_{F}=\int_{0}^{1}K(u)\,du.  \label{2065}
\end{equation}%
Then the exact one-step marginal formula is 
\begin{equation}
L_{i}^{F}(x)=\frac{\int_{0}^{1}K(u)1_{\{q_{i}(u)\leq T_{i}^{F}(x)\}}\,du}{%
\int_{0}^{1}K(u)\,du}.  \label{2067}
\end{equation}%
If the marginal $F_{i}$ is continuous, then 
\begin{equation*}
\{q_{i}(u)\leq T_{i}^{F}(x)\}=\{\pi _{i}(u)\leq x\}\quad \text{up to null
sets}.
\end{equation*}%
Thus, in the regular continuous case, 
\begin{equation}
L_{i}^{F}(x)=\frac{\int_{\{u:\,\pi _{i}(u)\leq x\}}K(u)\,du}{%
\int_{0}^{1}K(u)\,du}.  \label{2066}
\end{equation}

\begin{remark}
\medskip Equation (\ref{2066}) is a one-step formula. It should not be read
as an all-$n$ iteration unless one separately proves that the \textit{Lorenz
map} preserves the chosen rearrangement structure. Discrete
complete-mixability and other atomic invariant subclasses are treated in 
\textit{Appendix~E}. \medskip
\end{remark}

\paragraph{Concept III: $\Sigma $-countermonotonicity ($\Sigma $-CTM)}

\subparagraph{Definition and interpretation}

For a nonempty proper subset $A\subset \{1,\ldots ,d\}$, define the partial
sums 
\begin{equation}
S_{A}=\sum_{i\in A}X_{i},\qquad S_{A^{c}}=\sum_{j\notin A}X_{j}.
\label{2072}
\end{equation}

\begin{definition}
\medskip \noindent \label{SigmaCTM} A random vector $X$ is $\Sigma $\textit{%
-countermonotone} if, for every nonempty proper subset $A$, the bivariate
pair 
\begin{equation*}
(S_{A},S_{A^{c}})
\end{equation*}%
is countermonotone. \medskip

\textit{Joint mixability} implies $\Sigma $\textit{-countermonotonicity},
since if 
\begin{equation*}
\sum_{i}X_{i}=k\quad \text{a.s.},
\end{equation*}%
then 
\begin{equation*}
S_{A^{c}}=k-S_{A}.
\end{equation*}
\end{definition}

\subparagraph{No closure from marginals alone}

Unlike PCM, $\Sigma $\textit{-countermonotonicity} does not impose a fixed
low-dimensional product structure on $X$. It constrains complementary sums,
but it does not specify a unique coupling for the coordinates themselves.

\begin{proposition}
\medskip \noindent \label{Sigma-noclosure} For $d\geq 3$, the condition that 
$F$ is $\Sigma $-countermonotone does not determine the Lorenz marginals $%
L_{i}^{F}$ as functionals of the marginal vector $(F_{1},\ldots ,F_{d})$
alone. In particular, for fixed marginals, different $\Sigma $%
-countermonotone couplings may yield different Lorenz marginals. \medskip
\end{proposition}

\begin{remark}
\medskip \noindent The \textit{Lorenz map} depends on the full joint law
through the multiplicative kernel 
\begin{equation*}
\prod_{i}X_{i}
\end{equation*}%
and the coordinatewise truncation sets. Since $\Sigma $-CTM defines a class
of \textit{couplings} rather than a unique copula, additional structure is
required to obtain explicit closed recursions. \medskip
\end{remark}

\subparagraph{One-factor formulas}

A useful restricted regime is the one-factor case. Suppose 
\begin{equation*}
X_{i}=q_{i}(U),\qquad i=1,\ldots ,d,\qquad U\sim U(0,1).
\end{equation*}%
Define 
\begin{equation*}
K(u)=\prod_{i=1}^{d}q_{i}(u),\qquad I=\int_{0}^{1}K(u)\,du.
\end{equation*}%
Then, independently of whether $X$ is $\Sigma $\textit{-countermonotone},
the one-step \textit{Lorenz marginal} is 
\begin{equation}
L_{i}^{F}(x)=\frac{\int_{0}^{1}K(u)1_{\{q_{i}(u)\leq T_{i}^{F}(x)\}}\,du}{%
\int_{0}^{1}K(u)\,du}.  \label{2073}
\end{equation}

Now impose the $\Sigma $-CTM structure at the level of the functions $q_{i}$%
: for every nonempty proper subset $A\subset \{1,\ldots ,d\}$, assume there
is a decreasing function $\varphi _{A}$ such that 
\begin{equation}
\sum_{j\notin A}q_{j}(u)=\varphi _{A}\left( \sum_{i\in A}q_{i}(u)\right)
\quad \text{for a.e. }u.  \label{2074}
\end{equation}%
Then the law is $\Sigma $\textit{-countermonotone} in this one-factor sense,
and (\ref{2073}) gives its one-step \textit{Lorenz marginals}.

If, moreover, $q_{i}=T_{i}^{F}\circ \pi _{i}$ for measure-preserving maps $%
\pi _{i}$, then in the continuous marginal setting 
\begin{equation*}
\{q_{i}(u)\leq T_{i}^{F}(x)\}=\left\{ \pi _{i}(u)\leq x\right\} \quad \text{%
up to null sets},
\end{equation*}%
and therefore 
\begin{equation}
L_{i}^{F}(x)=\frac{\int_{\{u:\,\pi _{i}(u)\leq
x\}}\prod_{j=1}^{d}T_{j}^{F}(\pi _{j}(u))\,du}{\int_{0}^{1}%
\prod_{j=1}^{d}T_{j}^{F}(\pi _{j}(u))\,du}.  \label{2076}
\end{equation}

\medskip \noindent \textbf{Remark.} Formula (\ref{2076}) is again a one-step
formula unless the one-factor $\Sigma $-CTM structure is known to be
invariant under the full \textit{Lorenz map}. In general it is not. Product
tilting preserves support, but the subsequent coordinatewise standardization
may destroy the functional relations (\ref{2074}). Thus $\Sigma $-CTM, like
JM, is a meaningful one-step negative-dependence input, but it does not by
itself close the \textit{Lorenz dynamics}. Atomic and finite-support
invariant subclasses are treated separately in \textit{Appendix~E}. \medskip

\section{Applications}

This section sketches several application directions in statistics and in
mathematical finance. Throughout, $F$ denotes a $d$-variate distribution on $%
[0,1]^{d}$ with strictly positive density satisfying the dependence
regularity assumed in \textit{Theorems~\ref{Dd-main-long}--\ref%
{Dd-main-short}}, and $L_{n}^{F}$ denotes the $n$-th iterated \textit{Lorenz
transform} of $F$. The main results imply the existence of a universal
attractor 
\begin{equation}
G(x)\;=\;\prod_{i=1}^{d}x_{i}^{\varphi },\qquad x\in \lbrack 0,1]^{d},\qquad
\varphi =\frac{1+\sqrt{5}}{2},  \label{2077}
\end{equation}%
such that $L_{n}^{F}\rightarrow G$ uniformly. We interpret this convergence
as an explicit \textit{balancing transform}: after sufficiently many
iterations, heterogeneous distributions are mapped to a common, explicitly
known and comparatively \textquotedblleft regular\textquotedblright\
benchmark. This motivates robust summary statistics, diagnostics, and stress
transforms.

\subsection{Statistical applications}

\subsubsection{Multivariate dispersion / concentration indices via distance
to the universal limit}

The convergence $L_{n}^{F}\rightarrow G$ provides a natural reference
distribution against which we may quantify \textit{joint dispersion} or 
\textit{concentration}. Let $w:[0,1]^{d}\rightarrow \lbrack 0,\infty )$ be
an integrable weight emphasizing regions of interest (e.g.\ corners/tails).
Define for $n\geq 1$ the weighted $L^{1}$- and $L^{\infty }$-type indices 
\begin{align}
\mathcal{C}_{w,1}(F;n)& =\int_{[0,1]^{d}}w(x)\,\left\vert
L_{n}^{F}(x)-G(x)\right\vert dx,  \label{2078} \\
\mathcal{C}_{w,\infty }(F;n)& =\sup_{x\in \lbrack 0,1]^{d}}w(x)\,\left\vert
L_{n}^{F}(x)-G(x)\right\vert .  \label{2079}
\end{align}%
When $w\equiv 1$ these quantify global deviation; choices such as $%
w(x)=\prod_{i=1}^{d}x_{i}^{-\alpha }(1-x_{i})^{-\beta }$ (with $\alpha
,\beta \geq 0$ chosen so that $w$ is integrable) emphasize lower/upper tails
and corners, producing tail-sensitive multivariate concentration summaries.
The iteration depth $n$ plays the role of a scale parameter: small $n$
captures near-original features, while larger $n$ probes more structural
aspects of concentration, since repeated \textit{Lorenz transforms}
\textquotedblleft regularize\textquotedblright\ $F$ toward the balanced
benchmark $G$. In empirical work, $F$ may be replaced by an estimator $%
\widehat{F}$ (e.g.\ empirical, smoothed, or copula-based), yielding $%
\mathcal{C}_{w,\cdot }(\widehat{F};n)$ as a stable scalar functional for
comparing populations, time periods, or experimental conditions.

\subsubsection{Dimension reduction via Lorenz-iteration coordinates}

For high-dimensional data, the mapping $F\mapsto (L_{1}^{F},\dots
,L_{N}^{F}) $ offers a structured nonlinear feature extraction mechanism.
Fix a grid $\Xi =\{\xi ^{(1)},\dots ,\xi ^{(m)}\}\subset \lbrack 0,1]^{d}$
(or a basis $\{b_{j}\}$ on $[0,1]^{d}$) and define the \textit{iteration
embedding} 
\begin{equation}
\Phi _{N,\Xi }(F)=(L_{1}^{F}(\xi ^{(1)}),\dots ,L_{1}^{F}(\xi
^{(m)}),\,\dots ,\,L_{N}^{F}(\xi ^{(1)}),\dots ,L_{N}^{F}(\xi ^{(m)}))\in
R^{Nm}.  \label{2080}
\end{equation}%
Applied to rolling-window estimates $\widehat{F}_{t}$ of a multivariate time
series, or to a collection of subpopulations $\widehat{F}^{(k)}$, the
vectors $\Phi _{N,\Xi }(\widehat{F}_{t})$ can be analyzed via \textit{PCA},
factor models, clustering, or manifold learning. The key point is that the 
\textit{Lorenz operator} acts as a \textit{balancing transform} that
suppresses idiosyncratic irregularities while retaining the joint
concentration geometry: thus the dominant directions in $\Phi _{N,\Xi }$ can
provide a low-dimensional description of dependence/concentration regimes.
Coordinate-wise sensitivity can also be used for variable importance:
perturbing or omitting one component and tracking the change in $\mathcal{C}%
_{w,\cdot }$ or in $\Phi _{N,\Xi }$ identifies which margins drive
multivariate concentration.

\subsubsection{A decorrelation / dependence-stripping transform and
privacy-oriented synthetic data}

The universal attractor $G$ has product form, hence corresponds to an
independence-shaped benchmark on $[0,1]^{d}$. The iteration $F\mapsto
L_{n}^{F}$ may therefore be viewed as a principled \textit{%
dependence-stripping} transform within the regularity class: as $n$
increases, the transformed distributions become closer (in the \textit{Lorenz%
} sense) to an explicitly known, \textquotedblleft
balanced\textquotedblright\ target. This motivates the following
preprocessing pipeline. Given a sample $X_{1},\dots ,X_{N}$ from $F$,
estimate $L_{n}^{\widehat{F}}$ for moderate $n$ and use the transformed
object to construct decorrelated summaries or synthetic samples with
attenuated dependence. Operationally, we may select the smallest $n$ such
that $\mathcal{C}_{w,\infty }(\widehat{F};n)$ falls below a tolerance. The
resulting output retains controlled marginal structure (through the \textit{%
Lorenz}-type transformation) while systematically weakening joint
dependence, which is attractive both for (i) stabilizing downstream learning
algorithms on heavy-tailed dependent features, and (ii) generating
privacy-oriented synthetic data where sensitive cross-feature dependencies
are intentionally damped.

\subsubsection{Monitoring time-varying multivariate concentration in
econometrics}

For a multivariate time series (returns, losses, macro indicators), let $%
\widehat{F}_{t}$ be the empirical distribution on a rolling window ending at
time $t$. The scalar process 
\begin{equation}
t\;\longmapsto \;\mathcal{C}_{w,1}(\widehat{F}_{t};n)\quad \text{or}\quad 
\mathcal{C}_{w,\infty }(\widehat{F}_{t};n)  \label{2081}
\end{equation}%
can be used as a regime indicator capturing changes in joint concentration
and dependence structure, with $n$ controlling the aggressiveness of the
balancing transform. Compared to direct tail-dependence or copula-parameter
tracking, these indicators are functionals of the \textit{Lorenz} \textit{%
iterates} and can be numerically stable even in moderate samples, because
the iterated \textit{Lorenz map} tends to regularize the empirical object.

\subsubsection{Nonparametric testing and model diagnostics}

Theorem-driven convergence provides a basis for diagnostics of
dependence/regularity assumptions. Under the model class, $L_{n}^{F}$ should
approach $G$ in sup norm. Given data and an estimator $\widehat{F}$,
consider test statistics of the form 
\begin{equation}
T_{n}=d(L_{n}^{\widehat{F}},G),  \label{2082}
\end{equation}%
where $d$ is $\Vert \cdot \Vert _{\infty }$ or a weighted integrated
distance as in~(\ref{2078})--(\ref{2079}). We may calibrate $T_{n}$ via
bootstrap or subsampling. Large persistent deviations for moderate $n$
indicate misspecification of the dependence/regularity class (e.g.\
violation of the required total positivity / reverse-regularity). Similarly,
in copula modeling, comparing empirical $L_{n}^{\widehat{F}}$ to the implied 
$L_{n}$ under a fitted copula yields a functional \textit{goodness-of-fit }%
diagnostic that avoids direct estimation of high-dimensional density
constraints.

\subsection{Applications in portfolio theory and mathematical finance}

\subsubsection{Lorenz-based risk functionals as stable multi-criteria
objectives}

Let $X=(X_{1},\dots ,X_{d})$ denote a vector of losses (or negative returns)
associated with $d$ assets, desks, or risk factors, and let $F_{X}$ be the
distribution of a suitable transformation to $[0,1]^{d}$ (e.g.\ via marginal
probability integral transforms on each component). The iterated \textit{%
Lorenz} surfaces $L_{n}^{F_{X}}$ suggest a family of risk/concentration
functionals that complement tail-based measures and can be more stable in
optimization. A generic class is 
\begin{equation}
\rho _{n,w}(X)=\int_{[0,1]^{d}}w(u)\,\psi \!(L_{n}^{F_{X}}(u))\,du,
\label{2083}
\end{equation}%
where $w$ emphasizes relevant regions (e.g.\ tail corners) and $\psi $ is
increasing and convex, encoding risk aversion. Alternatively, we may use
deviation-from-benchmark penalties 
\begin{equation}
\rho _{n,w}^{\mathrm{dev}}(X)=\int_{[0,1]^{d}}w(u)\,\left\vert
L_{n}^{F_{X}}(u)-G(u)\right\vert \,du\qquad \text{or}\qquad \sup_{u\in
\lbrack 0,1]^{d}}w(u)\,\left\vert L_{n}^{F_{X}}(u)-G(u)\right\vert .
\label{2084}
\end{equation}%
These functionals treat the universal limit $G$ as a canonical
\textquotedblleft balanced\textquotedblright\ reference. The iteration depth 
$n$ acts as a robustness knob: for small $n$ the functional remains close to
the empirical dependence structure; for larger $n$ it emphasizes persistent
concentration features while dampening small-sample irregularities. This can
yield portfolio weights with lower turnover when combined with classical
objectives, e.g.\ \textit{mean--variance} or \textit{mean--CVaR}, by adding $%
\lambda \,\rho _{n,w}^{\mathrm{dev}}(X)$ as a regularization term.

\subsubsection{Strategy signatures and fragility: Lorenz iteration profiles}

For a trading strategy $S$ with P\&L (or drawdown) series, define on each
time window an estimated distribution $\widehat{F}_{S}$ (after
standardization to $[0,1]^{d}$ when multivariate) and compute the\textit{\
iteration profile} 
\begin{equation}
\Gamma _{N}(S)=(d(L_{1}^{\widehat{F}_{S}},G),\,d(L_{2}^{\widehat{F}%
_{S}},G),\,\dots ,\,d(L_{N}^{\widehat{F}_{S}},G)),  \label{2085}
\end{equation}%
with $d$ a weighted integrated distance or $\Vert \cdot \Vert _{\infty }$.
Even when strategies have similar first-order performance metrics, $\Gamma
_{N}(S)$ can separate them by their dependence/concentration dynamics. Fast
decay of $d(L_{n}^{\widehat{F}_{S}},G)$ suggests that dependence patterns
are easily \textquotedblleft washed out\textquotedblright\ under the
balancing transform (a potential proxy for robustness), whereas slow decay
indicates structural co-movement, crowding, hidden factor exposure, or tail
co-dependence. These profiles can be used for clustering strategies,
monitoring changes in strategy behavior, and defining stress scenarios (see
below).

\subsubsection{Model risk and stress testing via Lorenz-generated
perturbation paths}

A central advantage of the present framework is that it does not require
restricting attention to a fixed-marginal \textit{Fr\'{e}chet class}: the 
\textit{Lorenz} \textit{iteration} defines an intrinsic path 
\begin{equation*}
F\;\longmapsto \;L_{1}^{F}\;\longmapsto \;L_{2}^{F}\;\longmapsto \;\cdots
\end{equation*}%
within the regularity class, converging to the explicit benchmark $G$. This
supplies a structured family of distributional perturbations for model risk
assessment. Given a pricing functional or risk measure $\mathfrak{R}(F)$
(e.g.\ \textit{VaR/ES} of portfolio loss, risk contributions, or a convex
capital requirement), we may examine the envelope 
\begin{equation}
\sup_{1\leq n\leq N}\,\mathfrak{R}(L_{n}^{F})\qquad \text{or}\qquad \{%
\mathfrak{R}(L_{n}^{F}):1\leq n\leq N\},  \label{2086}
\end{equation}%
interpreting it as a \textit{Lorenz}-induced stress test that gradually
\textquotedblleft balances\textquotedblright\ the distribution. Because the
limit is known explicitly, $G$ can serve as a universal stress reference. In
the spirit of modern model risk treatments (where uncertainty is represented
by plausible transformations rather than parametric perturbations), the
family $\{L_{n}^{F}\}$ provides a mathematically grounded and
computationally tractable stress path.

\subsubsection{Extreme dependence, fixed points, and optimization
perspectives}

Classical risk aggregation results identify extreme dependence structures
(comonotonicity/countermonotonicity) as optimizers for various functionals
under equal-marginal constraints. The \textit{Lorenz} setting suggests an
analogous optimization program: identify objectives $\mathcal{J}(F)$ built
from \textit{Lorenz} \textit{surfaces} for which fixed points (including
comonotonic/countermonotonic corners and independence) arise as extremizers.
Natural candidates are \textit{Lorenz}-type integral functionals 
\begin{equation}
\mathcal{J}(F)=\int_{[0,1]^{d}}w(u)\,\Phi \!(L_{F}(u))\,du\qquad \text{or}%
\qquad \mathcal{J}(F)=d(L_{F},G),  \label{2087}
\end{equation}%
possibly under constraints reflecting market calibration (moments, marginal
tail behavior, or transportation budgets). In this view, the universal
attractor $G$ provides a canonical candidate for \textquotedblleft
balanced\textquotedblright\ dependence, while corner fixed points represent
maximally concentrated configurations. Notably, independence is itself a
fixed point of the \textit{Lorenz operator} in several settings, raising the
question of which \textit{Lorenz}-based objectives it may optimize (e.g.\
dispersion-maximizing criteria) under economically meaningful constraints.
Even partial results in this direction would connect iterated \textit{Lorenz}
\textit{dynamics} to extremal dependence theory.

\subsubsection{Connections to unbalanced optimal transport and insurance
risk allocation}

The map $F\mapsto L_{n}^{F}$ can be interpreted as a structured mass
redistribution on $[0,1]^{d}$ toward the explicit product-form benchmark $G$%
. This naturally suggests links to (possibly unbalanced) \textit{optimal
transport }(OT): we may compare $F$ and $G$ via an optimal transport\textit{%
\ }cost to obtain alternative multivariate concentration measures, or
interpret the iteration path $\{L_{n}^{F}\}$ as a canonical transport-like
deformation. In insurance and systemic risk applications, this viewpoint can
support capital allocation: the coordinates that most influence the movement
of $L_{n}^{F}$ toward $G$ (as quantified by sensitivities of $\mathcal{C}%
_{w,\cdot }$) may be interpreted as the principal contributors to
concentration and thus to capital requirements.

\medskip The above directions are intended as a compact roadmap. They all
leverage the same conceptual ingredient: iterated \textit{Lorenz transforms}
provide a canonical, robustness-promoting balancing procedure with an
explicit universal limit, enabling both summary indices and stress-transform
constructions.

\section{Conclusion}

The presented results provide a generalization of our earlier ones from \cite%
{[3]} to the multivariate case. Our simulations show a relatively quick
convergence with $15-20$ iterations being enough for most of the cases%
\footnote{{\footnotesize Source code:
https://github.com/vjord-research/source-code/tree/main/lorenz-curves/paper2}%
}. The paper provided further interesting and unexplored properties of the 
\textit{Lorenz curve}s and the results shown can find multiple applications
as discussed in \cite{[3]}. Additionally, they have relation to many
problems of pure dependence character.

\bigskip

\bigskip

\bigskip

\textbf{Declarations:}

\textbf{Chronology (ICMJE rules compatibility): }{\small \textit{Theorem} %
\ref{main-th} was first brought to the author's attention by Zvetan Ignatov
in the summer of 2007. It subsequently remained an open problem for many
years, and to the author's knowledge, no viable solution was published or
discussed during that time. The only known reference to the problem appears
in the PhD thesis of Ismat Ibrahim \cite{[41]}, where the conjecture is
mentioned briefly as a tangential topic. That work discusses only the
trivial case of an independent bivariate starting distribution and includes
a remark that informally questions the conjecture's validity in the general
case. The author first found strong evidence for the conjecture's
validity---under the assumption of a density---in the summer of 2023 after a
period of focused research. The author is grateful to Zvetan Ignatov for his
encouragement during the initial development of the solution. The
methodology, its implementation, and the writing of this paper, including
any potential errors or omissions, are the sole responsibility of the author.%
}

\newpage

\begin{flushleft}
{\Large Appendix A}
\end{flushleft}

In this appendix, we prove several technical results supporting the \textit{%
Fr\'{e}chet-Hoeffding bounds} calculations from the main text.

\begin{claim}
\label{fh-claim1} The validity of the following equation holds 
\begin{equation}
\frac{\dint\nolimits_{0}^{x_{1}}\dint%
\nolimits_{0}^{x_{2}}F_{1}^{-1}(u_{1})F_{2}^{-1}(u_{2})dMin(u_{1},u_{2})}{%
\mu _{12}^{F_{+}}}=\frac{\dint\nolimits_{0}^{x_{1}\wedge
x_{2}}F_{1}^{-1}(u)F_{2}^{-1}(u)du}{\dint%
\nolimits_{0}^{1}F_{1}^{-1}(u)F_{2}^{-1}(u)du}.  \label{22.1}
\end{equation}
\end{claim}

\begin{proof}
We have%
\begin{eqnarray}
&&\dint\nolimits_{0}^{x_{1}}\dint%
\nolimits_{0}^{x_{2}}F_{1}^{-1}(u_{1})F_{2}^{-1}(u_{2})dMin(u_{1},u_{2}) 
\notag \\
&=&\dint\nolimits_{0}^{x_{1}}\dint%
\nolimits_{0}^{x_{2}}F_{1}^{-1}(u_{1})F_{2}^{-1}(u_{2})d(u_{1}1_{\left\{
u_{1}\leq u_{2}\right\}
})+\dint\nolimits_{0}^{x_{1}}\dint%
\nolimits_{0}^{x_{2}}F_{1}^{-1}(u_{1})F_{2}^{-1}(u_{2})d(u_{2}1_{\left\{
u_{2}\leq u_{1}\right\} })  \notag \\
&=&\dint\nolimits_{0}^{x_{1}}\dint%
\nolimits_{0}^{x_{2}}F_{1}^{-1}(u_{1})F_{2}^{-1}(u_{2})(\frac{\partial }{%
\partial u_{2}}1_{\left\{ u_{1}\leq u_{2}\right\}
})du_{1}du_{2}+\dint\nolimits_{0}^{x_{1}}\dint%
\nolimits_{0}^{x_{2}}F_{1}^{-1}(u_{1})F_{2}^{-1}(u_{2})(u_{1}\frac{\partial
^{2}}{\partial u_{1}\partial u_{2}}1_{\left\{ u_{1}\leq u_{2}\right\}
})du_{1}du_{2}  \notag \\
&&+\dint\nolimits_{0}^{x_{1}}\dint%
\nolimits_{0}^{x_{2}}F_{1}^{-1}(u_{1})F_{2}^{-1}(u_{2})(\frac{\partial }{%
\partial u_{1}}1_{\left\{ u_{2}\leq u_{1}\right\}
})du_{1}du_{2}+\dint\nolimits_{0}^{x_{1}}\dint%
\nolimits_{0}^{x_{2}}F_{1}^{-1}(u_{1})F_{2}^{-1}(u_{2})(u_{2}\frac{\partial
^{2}}{\partial u_{1}\partial u_{2}}1_{\left\{ u_{2}\leq u_{1}\right\}
})du_{1}du_{2}  \notag \\
&=&\dint\nolimits_{0}^{x_{1}}\dint%
\nolimits_{0}^{x_{2}}F_{1}^{-1}(u_{1})F_{2}^{-1}(u_{2})\delta
(u_{2}-u_{1})du_{1}du_{2}+\dint\nolimits_{0}^{x_{1}}\dint%
\nolimits_{0}^{x_{2}}F_{1}^{-1}(u_{1})F_{2}^{-1}(u_{2})\delta
(u_{1}-u_{2})du_{1}du_{2}  \notag \\
&&+\dint\nolimits_{0}^{x_{1}}\dint%
\nolimits_{0}^{x_{2}}F_{1}^{-1}(u_{1})F_{2}^{-1}(u_{2})u_{1}\delta
^{^{\prime
}}(u_{1}-u_{2})du_{1}du_{2}+\dint\nolimits_{0}^{x_{1}}\dint%
\nolimits_{0}^{x_{2}}F_{1}^{-1}(u_{1})F_{2}^{-1}(u_{2})u_{2}\delta
^{^{\prime }}(u_{2}-u_{1})du_{1}du_{2}.  \label{22.2}
\end{eqnarray}

For the first two terms in (\ref{22.2}), we can use the indefinite integral
properties of the \textit{Dirac function} when being an integrand, as well
as its translation property. We get%
\begin{eqnarray}
&&\frac{\dint\nolimits_{0}^{x_{1}}F_{1}^{-1}(u_{1})\left(
\dint\nolimits_{0}^{x_{2}}F_{2}^{-1}(u_{2})\delta (u_{2}-u_{1})du_{2}\right)
du_{1}}{\mu _{12}^{F_{+}}}  \notag \\
&=&\frac{\dint%
\nolimits_{0}^{x_{1}}F_{1}^{-1}(u_{1})F_{2}^{-1}(u_{1})H(x_{2}-u_{1})du_{1}}{%
\mu _{12}^{F_{+}}}=\frac{\dint%
\nolimits_{0}^{x_{1}}F_{1}^{-1}(u_{1})F_{2}^{-1}(u_{1})1_{\left\{ u_{1}\leq
x_{2}\right\} }du_{1}}{\mu _{12}^{F_{+}}}  \notag \\
&=&\frac{\dint\nolimits_{0}^{x_{1}\wedge
x_{2}}F_{1}^{-1}(u_{1})F_{2}^{-1}(u_{1})du_{1}}{\mu _{12}^{F_{+}}},
\label{22}
\end{eqnarray}

and analogously 
\begin{equation}
\frac{\dint\nolimits_{0}^{x_{2}}F_{2}^{-1}(u_{2})\left(
\dint\nolimits_{0}^{x_{1}}F_{1}^{-1}(u_{1})\delta (u_{1}-u_{2})du_{1}\right)
du_{2}}{\mu _{12}^{F_{+}}}=\frac{\dint\nolimits_{0}^{x_{1}\wedge
x_{2}}F_{1}^{-1}(u_{2})F_{2}^{-1}(u_{2})du_{2}}{\mu _{12}^{F_{+}}},
\label{23}
\end{equation}

where $H(.)$ is the Heaviside function: $H(x)=\left\{ 
\begin{array}{c}
1,x>0 \\ 
0,x\leq 0%
\end{array}%
\right. .$

For the last two terms in (\ref{22.2}), we get%
\begin{eqnarray}
&&\frac{\dint\nolimits_{0}^{x_{1}}\dint%
\nolimits_{0}^{x_{2}}F_{1}^{-1}(u_{1})F_{2}^{-1}(u_{2})u_{1}\delta
^{^{\prime }}(u_{1}-u_{2})du_{1}du_{2}}{\mu _{12}^{F_{+}}}+\frac{%
\dint\nolimits_{0}^{x_{1}}\dint%
\nolimits_{0}^{x_{2}}F_{1}^{-1}(u_{1})F_{2}^{-1}(u_{2})u_{2}\delta
^{^{\prime }}(u_{2}-u_{1})du_{1}du_{2}}{\mu _{12}^{F_{+}}}  \notag \\
&=&\frac{\dint\nolimits_{0}^{x_{1}}\dint\nolimits_{0}^{x_{2}}\left(
u_{1}-u_{2}\right) F_{1}^{-1}(u_{1})F_{2}^{-1}(u_{2})\delta ^{^{\prime
}}(u_{1}-u_{2})du_{1}du_{2}}{\mu _{12}^{F_{+}}}=-\frac{\dint%
\nolimits_{0}^{x_{1}\wedge x_{2}}F_{1}^{-1}(u_{2})F_{2}^{-1}(u_{2})du_{2}}{%
\mu _{12}^{F_{+}}},  \label{23a}
\end{eqnarray}

where $\delta ^{^{\prime }}(.)$ denotes the derivative of the \textit{Dirac
function}. We also applied the distributional identity $\delta ^{^{\prime
}}(x)x=-$ $\delta (x)$ and then completed the computation using the
derivation from (\ref{22}).

Plugging (\ref{22}), (\ref{23}), and (\ref{23a}) into (\ref{22.2}), and
making the same computation for the denominator (i.e., taking the numerator
with $x_{1}=1$ and $x_{2}=1$), we finally get%
\begin{equation}
L_{F_{+}}(x_{1},x_{2})=\frac{\dint\nolimits_{0}^{x_{1}\wedge
x_{2}}F_{1}^{-1}(u)F_{2}^{-1}(u)du}{\mu _{12}^{F_{+}}}=\frac{%
\dint\nolimits_{0}^{x_{1}\wedge x_{2}}F_{1}^{-1}(u)F_{2}^{-1}(u)du}{%
\dint\nolimits_{0}^{1}F_{1}^{-1}(u)F_{2}^{-1}(u)du}.  \label{24}
\end{equation}
\end{proof}

\begin{claim}
\label{fh-claim2} If $L_{F}(x_{1},x_{2})=L_{+}^{F}(x_{1},x_{2})$, then $%
F(x_{1},x_{2})=F_{+}(x_{1},x_{2}).$
\end{claim}

\begin{proof}
We may note that we can write the implication in terms of copulas. Namely,
we have to prove that it holds%
\begin{eqnarray}
&&\dint\nolimits_{0}^{x_{1}}\dint%
\nolimits_{0}^{x_{2}}F_{1}^{-1}(u_{1})F_{2}^{-1}(u_{2})dC(u_{1},u_{2})
\label{105} \\
&=&Min\left[ \dint\nolimits_{0}^{x_{1}}\dint%
\nolimits_{0}^{1}F_{1}^{-1}(u_{1})F_{2}^{-1}(u_{2})dC(u_{1},u_{2}),\dint%
\nolimits_{0}^{1}\dint%
\nolimits_{0}^{x_{2}}F_{1}^{-1}(u_{1})F_{2}^{-1}(u_{2})dC(u_{1},u_{2})\right]
\text{ }  \notag \\
&\Rightarrow &C(x_{1},x_{2})=Min(x_{1},x_{2}),  \label{105a}
\end{eqnarray}

where $C(x_{1},x_{2})$ is the copula of $F(x_{1},x_{2})$ with marginals $%
F_{1}(x_{1})$ and $F_{2}(x_{2})$. Additionally, let's denote 
\begin{eqnarray}
A(x_{1},x_{2})
&=&\dint\nolimits_{0}^{x_{1}}\dint%
\nolimits_{0}^{x_{2}}F_{1}^{-1}(u_{1})F_{2}^{-1}(u_{2})dC(u_{1},u_{2}) 
\notag \\
B_{1}(x_{1})
&=&\dint\nolimits_{0}^{x_{1}}\dint%
\nolimits_{0}^{1}F_{1}^{-1}(u_{1})F_{2}^{-1}(u_{2})dC(u_{1},u_{2})  \notag \\
B_{2}(x_{2})
&=&\dint\nolimits_{0}^{1}\dint%
\nolimits_{0}^{x_{2}}F_{1}^{-1}(u_{1})F_{2}^{-1}(u_{2})dC(u_{1},u_{2}).
\label{105aa}
\end{eqnarray}

Intuitively, similarly to the direct implication from the main text, we can
notice that the joint measure $dC(x_{1},x_{2})$ concentrates its mass along
a one-dimensional subset of the square. We expect that the only way for the
minimum--representation (\ref{105}) to hold for all $x_{1}$ and $x_{2}$ is
that the support of the copula measure lies along the set where the
\textquotedblleft marginal integrals\textquotedblright\ match---that is,
along the diagonal $x_{1}=x_{2}$. We will make this observation formal
resorting to the theory of distributions, but now unlike the case of \textit{%
Claim }\ref{fh-claim1} above and \textit{Claim }{\footnotesize \ref%
{fh-claim3} }below, where test functions can be avoided\footnote{%
{\footnotesize This is a popular shortcut calculus relying solely on
properties of the \textit{Dirac delta function}, as discussed in the cited
references on distributions. It avoids the complexity associated with
employing test functions.}}, here it is necessary to use their toolkit.

As no differentiability of $C$ is assumed, we will work with \textit{%
mollified versions} of $A(x_{1},x_{2}),$ $B_{1}(x_{1})$, and $B_{2}(x_{2})$.
Let $\left\{ \phi _{n}\right\} _{n\geq 1}$ be a sequence of smooth functions
such that for each $n\in N$, $\phi _{n}:R\rightarrow \lbrack 0,+\infty )$, $%
\phi _{n}(u)=0$ for $\left\vert u\right\vert >\frac{1}{n}$, $\int_{-\frac{1}{%
n}}^{\frac{1}{n}}\phi _{n}(u)du=1$, and as $n\rightarrow \infty $, $\phi
_{n} ${} converges (in the sense of distributions) to the \textit{Dirac
delta function}. For an interior point $x_{1}\in \lbrack \frac{1}{n},1-\frac{%
1}{n}] $ and any $x_{2}\in (0,1)$, define the \textit{smeared derivative}
(see \cite[Section 5.3]{[20]}, \cite[Appendix C5]{[20]}, \cite[Chapter 9]%
{[19]}, and \cite[Chapter 1]{[21]}) in the $x_{1}$-direction by%
\begin{equation}
D_{n,1}(x_{1},x_{2})=\int_{-\frac{1}{n}}^{\frac{1}{n}}\phi _{n}(u)\frac{%
A(x_{1}+u,x_{2})-A(x_{1},x_{2})}{u}du.  \label{106}
\end{equation}

This quantity is a \textit{mollified version} of the difference quotient. In
the limit as $n\rightarrow +\infty $, if a weak derivative $%
D_{1}(x_{1},x_{2})$ of $A(x_{1},x_{2})$ with respect to $x_{1}$ exists, then%
\begin{equation}
D_{1}(x_{1},x_{2})=\underset{n\rightarrow +\infty }{\lim }%
D_{n,1}(x_{1},x_{2})  \label{107}
\end{equation}

in the sense of distributions. A similar expression may be written for the
derivative in the $x_{2}$-direction.

Assume by employing contradiction argument that the copula measure $%
dC(u_{1},u_{2})$ assigns a positive mass to a set off the diagonal. In
particular, there exist numbers $0\leq a<b\leq 1$, $\delta >0$, and $%
\varepsilon >0$ (small) such that the set 
\begin{equation}
E_{\varepsilon ,\delta }\subset \left\{ (u_{1},u_{2})\in \lbrack
0,1]^{2}:a\leq u_{1}\leq a+\varepsilon ,b\leq u_{2}\leq b+\varepsilon ,\text{%
and }u_{2}\geq u_{1}+\delta \right\}  \label{107a}
\end{equation}

has a strictly positive $C-$measure, i.e.,%
\begin{equation}
C(E_{\varepsilon ,\delta })>0.  \label{108a}
\end{equation}

This set is a small patch of size $\varepsilon $ in both directions, located
in a region where $u_{2}${} exceeds $u_{1}$ by at least $\delta $. Now,
choose a point $(x_{1}^{\ast },x_{2}^{\ast })$ so that: (i) $x_{1}^{\ast }$
lies in the interval $[a,a+\varepsilon ],$ and (ii) $x_{2}^{\ast }$ is
chosen so that $x_{2}^{\ast }<b$. By the definition of $B_{1}${} and $A$ in (%
\ref{105aa}), $B_{1}(x_{1}^{\ast })$ will include the contribution from the
patch $E_{\varepsilon ,\delta }$ (since $E_{\varepsilon ,\delta }\subset
\lbrack 0,x_{1}^{\ast }]\times \lbrack 0,1]$) while the integration
rectangle of $A(x_{1}^{\ast },x_{2}^{\ast })$ ($[0,x_{1}^{\ast }]\times
\lbrack 0,x_{2}^{\ast }]$) does not capture any of the mass in $%
E_{\varepsilon ,\delta }${} (because $x_{2}^{\ast }<b$). In other words, we
have%
\begin{equation}
A(x_{1}^{\ast },x_{2}^{\ast })<B_{1}(x_{1}^{\ast }).  \label{108aa}
\end{equation}

But according to the assumed equality%
\begin{equation}
A(x_{1},x_{2})=\min \left[ B_{1}(x_{1}),B_{2}(x_{2})\right]  \label{109a}
\end{equation}

if we are in the region where $B_{1}(x_{1})\leq B_{2}(x_{2})$ (which we can
ensure by an appropriate choice of $x_{2}^{\ast }${}), then we should have%
\begin{equation}
A(x_{1}^{\ast },x_{2}^{\ast })=B_{1}(x_{1}^{\ast }).  \label{109aa}
\end{equation}

Thus, the presence of the off--diagonal patch $E_{\varepsilon ,\delta }${}
produces a discrepancy (see \textit{Figure A1} for illustration). 
\begin{figure}[tbph]
\caption{Figure A1: Integration regions}
\label{fig:figure12}\centering
\includegraphics[width=0.45\linewidth]{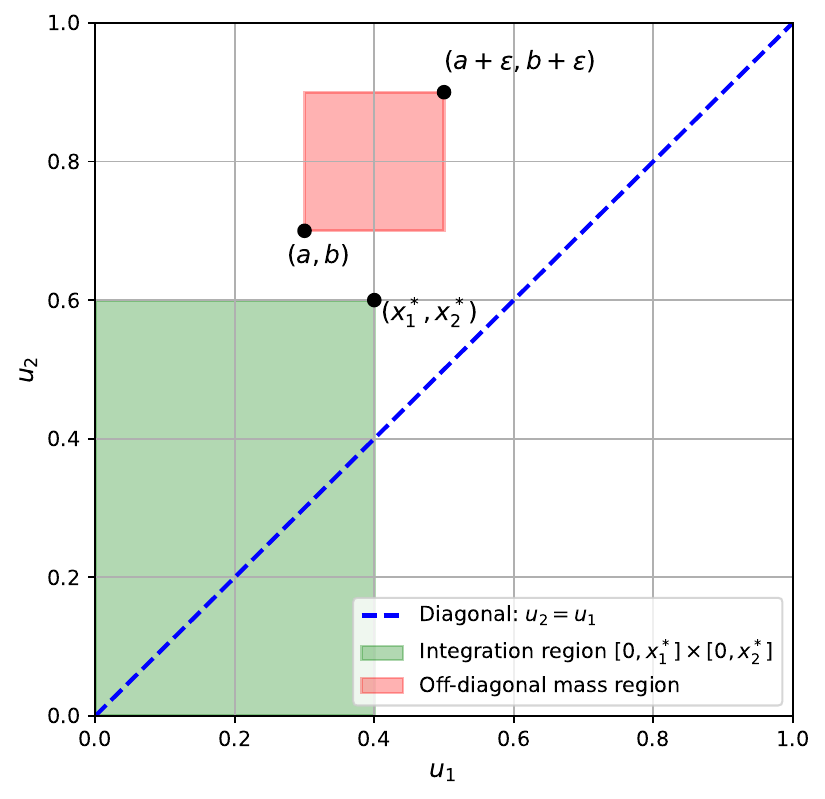} 
% }{}{Figure}{\special{language "Scientific Word";type "GRAPHIC";maintain-aspect-ratio TRUE;display "USEDEF";valid_file "T";width 3.2673in;height 3.1644in;depth 0pt;original-width 4.8715in;original-height 4.715in;cropleft "0";croptop "1";cropright "1";cropbottom "0";tempfilename 'SYFU8612.bmp';tempfile-properties "XPR";}
\end{figure}

It can be reinforced further. Consider the \textit{smeared derivative} in
the $x_{1}$--direction, defined for interior $x_{1}\in \lbrack \frac{1}{n},1-%
\frac{1}{n}]$ by (\ref{106}). Because the integration defining $%
A(x_{1},x_{2})$ stops at $x_{2}$, if we fix $x_{2}=x_{2}^{\ast }$ (with $%
x_{2}^{\ast }<b$) and take $x_{1}=x_{1}^{\ast }${} as above, then for any
small increment $s>0$ the difference%
\begin{equation}
A(x_{1}^{\ast }+s,x_{2}^{\ast })-A(x_{1}^{\ast },x_{2}^{\ast })  \label{110}
\end{equation}

captures only the increase in mass from the rectangle $[0,x_{1}^{\ast
}+s]\times \lbrack 0,x_{2}^{\ast }]$. Since the off--diagonal patch $%
E_{\varepsilon ,\delta }${} (with size $\varepsilon $) lies entirely outside 
$[0,x_{1}^{\ast }+s]\times \lbrack 0,x_{2}^{\ast }]$, the incremental mass
added here is smaller than the incremental mass in%
\begin{equation}
B_{1}(x_{1}^{\ast }+s)-B_{1}(x_{1}^{\ast }),  \label{140}
\end{equation}

which integrates over the full vertical range $[0,1]$ and thus captures the
mass from $E_{\varepsilon ,\delta }.$ In the limit as $n\rightarrow +\infty $%
, the \textit{smeared derivative} (which approximates the weak derivative of 
$A$ with respect to $x_{1}$) satisfies%
\begin{equation}
\underset{n\rightarrow +\infty }{\lim }D_{n,1}(x_{1}^{\ast },x_{2}^{\ast })<%
\frac{dB_{1}}{dx_{1}}(x_{1}^{\ast })  \label{141}
\end{equation}

since the increment in $A$ is deficient by at least the contribution of the
off--diagonal patch of size controlled by $\varepsilon $ and $\delta $. This
strict inequality contradicts the required equality of derivatives (which
would hold if $A(x_{1},x_{2})=B_{1}(x_{1})$ in the region under
consideration). We can say that effectively the \textit{smeared derivative} $%
D_{n,1}(x_{1}^{\ast },x_{2}^{\ast })$ detects the discrepancy (\ref{108aa}),
leading to the inequality (\ref{141}). This violates the assumed equality of
the integrals (and their derivatives) and forces the conclusion that no
assumed $\varepsilon $--patch (or off--diagonal mass) can exist. In turn,
the only possibility is that the measure $dC(u_{1},u_{2})$ is concentrated
on the diagonal $\{(u,u):u\in \lbrack 0,1]\}$, which implies the validity of%
\begin{equation}
C(x_{1},x_{2})=Min(x_{1},x_{2}).  \label{142}
\end{equation}
\end{proof}

\begin{claim}
\label{fh-claim3} The validity of the following equation holds 
\begin{equation}
\frac{\dint\nolimits_{0}^{x_{1}}\dint%
\nolimits_{0}^{x_{2}}F_{1}^{-1}(u_{1})F_{2}^{-1}(u_{2})dMax(u_{1}+u_{2}-1,0)%
}{\mu _{12}^{F_{-}}}=\frac{\dint%
\nolimits_{1-x_{2}}^{x_{1}}F_{1}^{-1}(u)F_{2}^{-1}(1-u)du}{%
\dint\nolimits_{0}^{1}F_{1}^{-1}(u)F_{2}^{-1}(1-u)du}.  \label{35a}
\end{equation}
\end{claim}

\begin{proof}
We have%
\begin{eqnarray}
&&\frac{\dint\nolimits_{0}^{x_{1}}\dint%
\nolimits_{0}^{x_{2}}F_{1}^{-1}(u_{1})F_{2}^{-1}(u_{2})dMax(u_{1}+u_{2}-1,0)%
}{\mu _{12}^{F_{-}}}  \notag \\
&=&\frac{\dint\nolimits_{0}^{x_{1}}\dint%
\nolimits_{0}^{x_{2}}F_{1}^{-1}(u_{1})F_{2}^{-1}(u_{2})d(\left(
u_{1}+u_{2}-1\right) 1_{\left\{ u_{1}+u_{2}\geq 1\right\} })}{\mu
_{12}^{F_{-}}}  \label{35.2} \\
&=&P_{1}+P_{2}-P_{3},  \notag
\end{eqnarray}

where by $P_{i},$ $i=1,2,3,$ we denoted%
\begin{eqnarray}
P_{1} &=&\frac{\dint\nolimits_{0}^{x_{1}}\dint%
\nolimits_{0}^{x_{2}}F_{1}^{-1}(u_{1})F_{2}^{-1}(u_{2})(\frac{\partial }{%
\partial u_{2}}1_{\left\{ u_{2}\geq 1-u_{1}\right\} }+u_{1}\frac{\partial
^{2}}{\partial u_{1}\partial u_{2}}1_{\left\{ u_{2}\geq 1-u_{1}\right\}
})du_{1}du_{2}}{\mu _{12}^{F_{-}}}  \notag \\
P_{2} &=&\frac{\dint\nolimits_{0}^{x_{1}}\dint%
\nolimits_{0}^{x_{2}}F_{1}^{-1}(u_{1})F_{2}^{-1}(u_{2})(\frac{\partial }{%
\partial u_{1}}1_{\left\{ u_{1}\geq 1-u_{2}\right\} }+u_{2}\frac{\partial
^{2}}{\partial u_{1}\partial u_{2}}1_{\left\{ u_{1}\geq 1-u_{2}\right\}
})du_{1}du_{2}}{\mu _{12}^{F_{-}}}  \notag \\
P_{3} &=&\frac{\dint\nolimits_{0}^{x_{1}}\dint%
\nolimits_{0}^{x_{2}}F_{1}^{-1}(u_{1})F_{2}^{-1}(u_{2})(\frac{\partial ^{2}}{%
\partial u_{1}\partial u_{2}}1_{\left\{ u_{2}\geq 1-u_{1}\right\}
})du_{1}du_{2}}{\mu _{12}^{F_{-}}}.  \label{31}
\end{eqnarray}

Let $P_{4}=P_{1}+P_{2}$. We get%
\begin{eqnarray}
\mu _{12}^{F_{-}}P_{4}\text{ } &=&  \notag \\
&=&\dint\nolimits_{0}^{x_{1}}\dint%
\nolimits_{0}^{x_{2}}F_{1}^{-1}(u_{1})F_{2}^{-1}(u_{2})(\frac{\partial }{%
\partial u_{1}}1_{\left\{ u_{1}\geq 1-u_{2}\right\} }+\frac{\partial }{%
\partial u_{2}}1_{\left\{ u_{2}\geq 1-u_{1}\right\} }+(u_{1}+u_{2})\frac{%
\partial ^{2}}{\partial u_{1}\partial u_{2}}1_{\left\{ u_{1}+u_{2}\geq
1\right\} })du_{1}du_{2}  \notag \\
&=&\dint\nolimits_{0}^{x_{1}}\dint%
\nolimits_{0}^{x_{2}}F_{1}^{-1}(u_{1})F_{2}^{-1}(u_{2})(2\delta
(u_{1}+u_{2}-1)+(u_{1}+u_{2})\delta ^{^{\prime
}}(u_{1}+u_{2}-1))du_{1}du_{2}.  \label{32}
\end{eqnarray}

In the last expression, $\delta ^{^{\prime }}(.)$ denotes again the
derivative of the \textit{Dirac function} and a further resort to its
special properties is needed. We begin with the substitution $u_{1}+u_{2}=u$
and proceed from there%
\begin{eqnarray}
\mu _{12}^{F_{-}}P_{4}
&=&\dint\nolimits_{0}^{x_{1}}\dint%
\nolimits_{u_{1}}^{u_{1}+x_{2}}F_{1}^{-1}(u_{1})F_{2}^{-1}(u-u_{1})(2\delta
(u-1)+u\delta ^{^{\prime }}(u-1))du_{1}du  \notag \\
&=&2\dint\nolimits_{0}^{x_{1}}\dint%
\nolimits_{u_{1}}^{u_{1}+x_{2}}F_{1}^{-1}(u_{1})F_{2}^{-1}(u-u_{1})\delta
(u-1)du_{1}du+\dint\nolimits_{0}^{x_{1}}\dint%
\nolimits_{u_{1}}^{u_{1}+x_{2}}F_{1}^{-1}(u_{1})F_{2}^{-1}(u-u_{1})u\delta
^{^{\prime }}(u-1)du_{1}du  \notag \\
&=&2\dint\nolimits_{0}^{x_{1}}F_{1}^{-1}(u_{1})\left(
\dint\nolimits_{u_{1}}^{u_{1}+x_{2}}F_{2}^{-1}(u-u_{1})\delta (u-1)du\right)
du_{1}+\dint\nolimits_{0}^{x_{1}}F_{1}^{-1}(u_{1})\left(
\dint\nolimits_{u_{1}}^{u_{1}+x_{2}}F_{2}^{-1}(u-u_{1})u\delta ^{^{\prime
}}(u-1)du\right) du_{1}  \notag \\
&=&2\dint\nolimits_{0}^{x_{1}}F_{1}^{-1}(u_{1})F_{2}^{-1}(1-u_{1})\left(
H(u_{1}+x_{2}-1)-H(u_{1}-1)\right)
du_{1}-\dint\nolimits_{1-x_{2}}^{x_{1}}F_{1}^{-1}(u_{1})\left[ \frac{%
\partial }{\partial u}[F_{2}^{-1}(u-u_{1})u]\mid _{u=1}\right] du_{1}  \notag
\\
&=&2\dint%
\nolimits_{1-x_{2}}^{x_{1}}F_{1}^{-1}(u_{1})F_{2}^{-1}(1-u_{1})du_{1}-\dint%
\nolimits_{1-x_{2}}^{x_{1}}F_{1}^{-1}(u_{1})\left[ \frac{\partial }{\partial
u}[F_{2}^{-1}(u-u_{1})u]\mid _{u=1}\right] du_{1}  \notag \\
&=&2\dint%
\nolimits_{1-x_{2}}^{x_{1}}F_{1}^{-1}(u_{1})F_{2}^{-1}(1-u_{1})du_{1}-\dint%
\nolimits_{1-x_{2}}^{x_{1}}F_{1}^{-1}(u_{1})\left[ \frac{\partial }{\partial
u}F_{2}^{-1}(u-u_{1})\mid _{u=1}\right] du_{1}-\dint%
\nolimits_{1-x_{2}}^{x_{1}}F_{1}^{-1}(u_{1})F_{2}^{-1}(1-u_{1})du_{1}  \notag
\\
&=&\dint%
\nolimits_{1-x_{2}}^{x_{1}}F_{1}^{-1}(u_{1})F_{2}^{-1}(1-u_{1})du_{1}-\dint%
\nolimits_{1-x_{2}}^{x_{1}}F_{1}^{-1}(u_{1})\left[ \frac{\partial }{\partial
u}F_{2}^{-1}(u-u_{1})\mid _{u=1}\right] du_{1}.  \label{33}
\end{eqnarray}

Here, in addition to the indefinite integral and translation properties of
the \textit{Dirac delta function} used earlier in (\ref{22.2}), we have also
applied the general property $\dint\nolimits_{a}^{b}h(x)\delta ^{^{\prime
}}(x-c)dx=-h^{^{\prime }}(c)$ for any regular function $h(x)$ and constants $%
a,b$, and $c.$ Additionally, when applying this, we account for the
necessary change in the limits of integration.

Having established (\ref{33}), it is easy to see that for $P_{3}$, we have%
\begin{eqnarray}
P_{3} &=&\frac{\dint\nolimits_{0}^{x_{1}}\dint%
\nolimits_{0}^{x_{2}}F_{1}^{-1}(u_{1})F_{2}^{-1}(u_{2})(\frac{\partial ^{2}}{%
\partial u_{1}\partial u_{2}}1_{\left\{ u_{2}\geq 1-u_{1}\right\}
})du_{1}du_{2}}{\mu _{12}^{F_{-}}}  \notag \\
&=&\frac{\dint\nolimits_{0}^{x_{1}}\dint%
\nolimits_{u_{1}}^{u_{1}+x_{2}}F_{1}^{-1}(u_{1})F_{2}^{-1}(u-u_{1})\delta
^{^{\prime }}(u-1)du_{1}du}{\mu _{12}^{F_{-}}}=-\frac{\dint%
\nolimits_{1-x_{2}}^{x_{1}}F_{1}^{-1}(u_{1})\left[ \frac{\partial }{\partial
u}F_{2}^{-1}(u-u_{1})\mid _{u=1}\right] du_{1}}{\mu _{12}^{F_{-}}}.
\label{33a}
\end{eqnarray}

Thus substituting (\ref{33}) and (\ref{33a}) into (\ref{35.2}), we get%
\begin{equation}
L_{F_{-}}(x_{1},x_{2})=\frac{\dint\nolimits_{0}^{x_{1}}\dint%
\nolimits_{0}^{x_{2}}F_{1}^{-1}(u_{1})F_{2}^{-1}(u_{2})dMax(u_{1}+u_{2}-1,0)%
}{\mu _{12}^{F_{-}}}=\frac{\dint%
\nolimits_{1-x_{2}}^{x_{1}}F_{1}^{-1}(u)F_{2}^{-1}(1-u)du}{%
\dint\nolimits_{0}^{1}F_{1}^{-1}(u)F_{2}^{-1}(1-u)du}.  \label{35.3}
\end{equation}
\end{proof}

\begin{claim}
\label{fh-claim4} If $L_{F}(x_{1},x_{2})=L_{-}^{F}(x_{1},x_{2})$, then $%
F(x_{1},x_{2})=F_{-}(x_{1},x_{2}).$
\end{claim}

\begin{proof}
We may note that we can write the implication in terms of copulas. Namely,
we have to prove that it holds%
\begin{eqnarray}
&&\dint\nolimits_{0}^{x_{1}}\dint%
\nolimits_{0}^{x_{2}}F_{1}^{-1}(u_{1})F_{2}^{-1}(u_{2})dC(u_{1},u_{2})
\label{35b} \\
&=&Max\left[ \dint\nolimits_{0}^{x_{1}}\dint%
\nolimits_{0}^{1}F_{1}^{-1}(u_{1})F_{2}^{-1}(u_{2})dC(u_{1},u_{2})+\dint%
\nolimits_{0}^{1}\dint%
\nolimits_{0}^{x_{2}}F_{1}^{-1}(u_{1})F_{2}^{-1}(u_{2})dC(u_{1},u_{2})-1,0%
\right] \text{ }  \notag \\
&\Rightarrow &C(x_{1},x_{2})=Max(x_{1}+x_{2}-1,0),  \label{35bb}
\end{eqnarray}

where $C(x_{1},x_{2})$ is the copula of $F(x_{1},x_{2})$ with marginals $%
F_{1}(x_{1})$ and $F_{2}(x_{2})$. The proof can be carried out in a similar
manner to \textit{Claim }{\footnotesize \ref{fh-claim2}}.
\end{proof}

\pagebreak

\begin{flushleft}
{\Large Appendix B}\hfill
\end{flushleft}

In this appendix, we prove several auxiliary claims and lemmas. Some of them
will be useful in \textit{Section 4}, others in \textit{Appendix C} and the
proof of the main theorem. Many have also a standalone and supplementary
character to better understand the problems we face.

We start with the proof of several inequalities between the marginals of the
distributions participating in the \textit{Fr\'{e}chet-Hoeffding bounds} of $%
L_{F}(x_{1},x_{2})$ when it is viewed as a distribution function. Although
we will make referrals to definitions and results from stochastic orders
theory based on standard references such as \cite{[27]}, \cite{[28]}, \cite%
{[25]}, \cite{[5]}, and \cite{[24]}, the appendix is largely self-contained.

Let's consider the two copulas $C_{1}$ and $C_{2},$ subject to \textit{%
concordance order} $\leq _{c}$ (or \textit{Fr\'{e}chet-Hoeffding order}, see
for details \cite[Definition 2.8.1]{[25]} or with a nuance \cite[Definition
3.8.5]{[28]}). Concretely, we will impose the relation $C_{1}\leq _{c}$ $%
C_{2}$ if and only if $W\leq $ $C_{1}\leq C_{2}\leq M$ holds, where again we
use the standard definitions, thus by $W$ and $M$, we denote the
countermonotonic and the comonotonic copulas respectively. Now consider the
distribution functions $F_{C_{1}}(x_{1},x_{2})$ and $F_{C_{2}}(x_{1},x_{2})$
induced by the copulas $C_{1}$ and $C_{2}$. From (\ref{1}), for the
integrals forming the distribution functions and their marginals, for $i=1,2$%
, we have the representations 
\begin{eqnarray}
L^{F_{C_{i}}}(F_{1}(s_{1}),F_{2}(s_{2})) &=&\frac{\dint\nolimits_{0}^{s_{1}}%
\dint\nolimits_{0}^{s_{2}}u_{1}u_{2}dF_{C_{i}}(u_{1},u_{2})}{\mu ^{F_{C_{i}}}%
}=\frac{E^{F_{C_{i}}}(X_{1}X_{2}1_{\{X_{1}\leq s_{1}\}}1_{\{X_{2}\leq
s_{2}\}})}{E^{F_{C_{i}}}(X_{1}X_{2})}  \label{220} \\
L_{1}^{F_{C_{i}}}(F_{1}(s_{1})) &=&\frac{\dint\nolimits_{0}^{s_{1}}\dint%
\nolimits_{0}^{+\infty }u_{1}u_{2}1_{\{u_{1}\leq
s_{1}\}}dF_{C_{i}}(u_{1},u_{2})}{\mu ^{F_{C_{i}}}}=\frac{%
E^{F_{C_{i}}}(X_{1}X_{2}1_{\{X_{1}\leq s_{1}\}})}{E^{F_{C_{i}}}(X_{1}X_{2})}
\label{221} \\
L_{2}^{F_{C_{i}}}(F_{2}(s_{2})) &=&\frac{\dint\nolimits_{0}^{+\infty
}\dint\nolimits_{0}^{s_{2}}u_{1}u_{2}1_{\{u_{2}\leq
s_{2}\}}dF_{C_{i}}(u_{1},u_{2})}{\mu ^{F_{C_{i}}}}=\frac{%
E^{F_{C_{i}}}(X_{1}X_{2}1_{\{X_{2}\leq s_{2}\}})}{E^{F_{C_{i}}}(X_{1}X_{2})},
\label{222}
\end{eqnarray}

where the random vector $(X_{1},X_{2})$ has a d.f. $F_{C_{i}}(x_{1},x_{2})$
for the different cases $i=1,2$. Our aim is to see how the \textit{%
concordance order} $C_{1}\leq _{c}$ $C_{2}$ affects these three functionals.
We start with a simpler set of inequalities at the level of (un-normalized)
expectations.

\begin{claim}
\label{com-claim1} Let $W\leq C_{1}\leq C_{2}\leq M$. Then for all real $%
s,s_{1},s_{2}$,%
\begin{eqnarray}
E^{F_{C_{1}}}(X_{1}X_{2}1_{\{X_{1}\geq s_{1}\}}1_{\{X_{2}\geq s_{2}\}})
&\leq &E^{F_{C_{2}}}(X_{1}X_{2}1_{\{X_{1}\geq s_{1}\}}1_{\{X_{2}\geq
s_{2}\}})  \label{223.1} \\
E^{F_{C_{1}}}(X_{1}X_{2}1_{\{X_{1}\geq s\}}) &\leq
&E^{F_{C_{2}}}(X_{1}X_{2}1_{\{X_{1}\geq s\}})  \label{223.2} \\
E^{F_{C_{1}}}(X_{1}X_{2}1_{\{X_{2}\geq s\}}) &\leq
&E^{F_{C_{2}}}(X_{1}X_{2}1_{\{X_{2}\geq s\}})  \label{224} \\
E^{F_{C_{1}}}(X_{1}X_{2}) &\leq &E^{F_{C_{2}}}(X_{1}X_{2}).  \label{225}
\end{eqnarray}
\end{claim}

\begin{proof}
The inequalities follow after a direct application of the \textit{upper
ortant order} ($uo$), or more generally the \textit{supermodular order} ($sm$%
), to appropriate functionals based on the original work of \cite{[34]}, 
\cite{[30]}, and \cite{[13]}, receiving review and extensions in \cite{[29]}%
, \cite{[26]}, and \cite{[31]}, among others as well as getting a
comprehensive modern treatment in \cite[Chapter 9.A.4]{[24]}, \cite[Chapter 6%
]{[33]}, and especially the encyclopedic \cite[Chapter 6]{[5]}. We will use
mainly the latter in exposition of the proof. Let first remind some basic
definitions and properties of the stochastic orders we will use\footnote{%
{\footnotesize We will have a slight deviation from the notation used in the
claim.}}.

Start with the general \textit{integral order} $\prec _{\mathcal{F}}$.
Following \cite[Definition 3.30]{[5]}, the order $\prec _{\mathcal{F}}$ can
be defined on the space of probability measures $\mathcal{M}^{1}(E,\mathfrak{%
U})$ by 
\begin{equation}
P\prec _{\mathcal{F}}Q\text{ if }\int fdP\leq \int fdQ  \label{225.1}
\end{equation}

for all integrable $f\in \mathcal{F}$ for which both integrals are finite,
where: (i) $\mathcal{F}$ is the subclass of $\mathfrak{U}$-measurable
real-valued functions on $\mathit{E}$, chosen according to the order being
considered, (ii) $E$ is the sample space, (iii) $\mathfrak{U}$ is the family
of measurable sets on $E$, and (iv) $\mathcal{M}^{1}$ is the set of all
probability measures on the measurable space $(\mathit{E},$ $\mathfrak{U}$$)$
which have finite first moment.

Second, two random vectors $X$ and $Y$ in $R^{n}$ we have $X\leq _{lo}Y$ if
for their distribution functions holds $F_{X}\leq F_{Y}$ and $X\leq _{uo}Y$
if for their survival functions holds $\overline{F}_{X}\leq \overline{F}_{Y}$%
. Alternatively, we can also denote the orders by $F_{X}\leq _{lo}F_{Y}$ and 
$F_{X}\leq _{uo}F_{Y}$ (see \cite[Definition 6.1]{[5]}). Then we can notice
that the \textit{concordance order} between $X\leq _{c}Y$ can be viewed also
as a situation when both $X\leq _{lo}Y$ and $X\leq _{uo}Y$ hold (see again 
\cite[Definition 6.1]{[5]}). This gives the obvious implications 
\begin{equation}
X\leq _{c}Y\Longrightarrow X\leq _{lo}Y\text{ and }X\leq
_{c}Y\Longrightarrow X\leq _{uo}Y.  \label{225.2}
\end{equation}

Intuitively, the relation between the copulas dependence has strong
influence both on the upper and the lower tails of the distributions of $X$
and $Y.$ For the bivariate case we have the equivalence 
\begin{equation}
X\leq _{c}Y\Longleftrightarrow X\leq _{lo}Y\Longleftrightarrow X\leq _{uo}Y.
\label{225.3}
\end{equation}

Third, based on \cite[Definition 6.6]{[5]}, define for the functions $%
f:R^{n}\longrightarrow R$ the difference operator $\Delta _{i}^{\varepsilon
} $, $\varepsilon >0,1\leq i\leq n$ by $\Delta _{i}^{\varepsilon
}f(x)=f(x+\varepsilon e_{i})-f(x)$, where $e_{i}$ is the $i$-th unit vector.
Then we can consider the class $\mathcal{F}_{\Delta }$ of \textquotedblleft $%
\Delta $-monotone" functions on $R^{n}$. They are defined as the functions $%
f $ for which for any subset $J=\{i_{1},...,i_{k}\}\subset \{1,...,n\}$ and
any $\varepsilon _{1},...,\varepsilon _{k}>0$ holds 
\begin{equation}
\Delta _{i_{1}}^{\varepsilon _{1}}...\Delta _{i_{k}}^{\varepsilon _{k}}\geq
0,  \label{225.4}
\end{equation}

or (by \cite[Remark 6.7]{[5]}) if $f$ $\ $is differentiable, holds the
easier to check condition 
\begin{equation}
\frac{\partial ^{k}f}{\partial x_{i_{1}}...\partial x_{i_{k}}}\geq 0\text{
for }\forall k\leq n\text{ and }i_{1}<...<i_{k}.  \label{225.6}
\end{equation}

This prompts to define the \textit{integral order} $\leq _{\Delta }$
generated by $\mathcal{F}_{\Delta }$ (i.e. comparing integral transforms of $%
X$ and $Y$ based on $\Delta $-monotone functions) by posing 
\begin{equation}
\leq _{\Delta }\text{=}\leq _{\mathcal{F}_{\Delta }}.  \label{225.7}
\end{equation}

An important result is that we have by \cite{[13]} and \cite[Theorem 6.8]%
{[5]} that

\begin{equation}
X\leq _{uo}Y\Longleftrightarrow X\leq _{\mathcal{F}_{\Delta }}Y.
\label{225.8}
\end{equation}

Fourth, the class $\mathcal{F}_{\Delta }$ can be strengthened based on \cite[%
Definition 6.12]{[5]}. This is done by considering the class $\mathcal{F}%
_{sm}$ of \textquotedblleft supermodular" functions $f$ which for all $1\leq
i\leq j\leq n$ and $\varepsilon $, $\delta >0$ obey 
\begin{equation}
\Delta _{i}^{\varepsilon }\Delta _{j}^{\varepsilon }f(x)\geq 0\text{ }%
\forall x\in R^{n},  \label{225.9}
\end{equation}

or (by \cite[Remark 6.13]{[5]}) if $f$ $\ $is differentiable, holds the
easier to check condition 
\begin{equation}
\frac{\partial ^{2}f}{\partial x_{i}\partial x_{j}}\geq 0\text{ for }\forall
i<j\text{ and }\forall x\in R^{n}.  \label{225.9a}
\end{equation}

Analogously to the previous case, we can define the integral order $\leq
_{sm}$generated by $\mathcal{F}_{sm}$ (i.e. the \textit{\textquotedblleft
supermodular\textquotedblright\ order}) by posing 
\begin{equation}
\leq _{sm}\text{=}\leq _{\mathcal{F}_{sm}}.  \label{225.10}
\end{equation}

By definition $\mathcal{F}_{\Delta }\subset \mathcal{F}_{sm}$ leading to the
comparison with respect to the $\leq _{sm}$ order being stronger than the
comparison with respect to the \textit{ortant orders}. An important result
holds for the case of $n=2$ when restricted\footnote{{\footnotesize For }$%
n\geq 3${\footnotesize \ and not restricting to the \textit{Fr\'{e}chet class%
}, by \cite{[34]}, \cite{[30]}, and \cite[Theorem 6.14]{[5]}, we have only }$%
X\leq _{sm}X^{c}.$} to the \textit{Fr\'{e}chet class} $\mathcal{F}%
(F_{1},F_{2})$ (i.e. $X=(X_{1},X_{2})$, $Y=(Y_{1},Y_{2})$ have identical
marginals $F_{1}$ and $F_{2}$). There we have by \cite{[35]} and \cite[%
Theorem 6.15]{[5]} that%
\begin{equation}
X\leq _{sm}Y\text{ }\Longleftrightarrow \text{ }X\leq _{lo}Y\text{ }%
\Longleftrightarrow \text{ }X\leq _{uo}Y  \label{225.11a}
\end{equation}%
\begin{equation}
X_{c}\text{ }\leq _{sm}\text{ }Y\text{ }\leq _{sm}\text{ }X^{c},
\label{225.12}
\end{equation}

where $X^{c}=(F^{-1}(U),F_{2}^{-1}(U))$ and $%
X_{c}=(F^{-1}(U),F_{2}^{-1}(1-U))$ are the comonotonic and the
countermonotonic vectors respectively.

The exposition above helps to position the problem we need to prove better
both within the literature and the theory of stochastic orders which is
important for its better understanding. Returning to our direct notation,
providing a solution based on the stated theorems is not difficult. Since we
posed \textit{concordant order} between the copulas $C_{1}$ and $C_{2}$, due
to (\ref{225.2}), this also implies $uo$ order between them. So, from $%
C_{1}\leq _{uo}C_{2}$ and under equal marginals also $F_{C_{1}}\leq
_{uo}F_{C_{2}}$, follows $F_{C_{1}}\leq _{\mathcal{F}_{\Delta }}F_{C_{2}}$.
The latter means that we can apply the two probability distributions $%
F_{C_{1}}$ and $F_{C_{2}}$ on any integrands which are $\Delta $-monotone
and preserve the order. Using (\ref{225.6}) we can directly check that the
functions $x_{1}x_{2}1_{\{x_{1}\geq s_{1}\}}1_{\{x_{2}\geq
s_{2}\}},x_{1}x_{2}1_{\{x_{1}\geq s\}},x_{1}x_{2}1_{\{x_{2}\geq s\}},$ and $%
x_{1}x_{2}$ generating the expectations in (\ref{223.1}), (\ref{223.2}), (%
\ref{224}), and (\ref{225}) respectively are $\Delta $-monotone. So the
validity of the claim follows. We may note that we do not need to restrict
ourselves to (i) $n=2$ or (ii) \textit{supermodular order} and use the
alternative theorems stated. Our proof works for the case $n\geq 3.$ Also,
the functions under scope although being also supermodular is not of
relevance since it is enough that they are $\Delta $-monotone. Finally,
setting the claim for copulas is more convenient since this allows to focus
on the dependence. The equal marginals are implied.
\end{proof}

Although \textit{Claim} \ref{com-claim1} is already useful, it is
insufficient for directly comparing the normalized functionals (\ref{220})-(%
\ref{222}), because it yields inequalities for both numerator and
denominator. We therefore seek a stronger comparison, at the level of
ratios. To this end we introduce a tilting (change-of-measure) argument and
first work under the stronger \textit{likelihood ratio order}. We then
discuss how much can be recovered under the initial concordance assumption.

We recall the following standard stochastic orders; see \cite{[24]} for a
standard reference as well as discussion on the particular usage below.

\begin{definition}
\label{lr} Let $X,Y$ be $R^{2}$-valued random vectors with densities $%
f_{X},f_{Y}$ with respect to a common dominating measure. We say that $X$ is
smaller than $Y$ in \textit{likelihood ratio order}, and write $X\leq _{%
\mathrm{lr}}Y$, if the density ratio $\frac{f_{Y}}{f_{X}}$ is coordinatewise
non-decreasing; that is, 
\begin{equation}
\frac{f_{Y}(x_{1},x_{2})}{f_{X}(x_{1},x_{2})}\quad \text{is non-decreasing
in each coordinate $x_{i}$, $i=1,2.$}  \label{800}
\end{equation}
\end{definition}

\begin{remark}
We use a coordinatewise monotonicity condition on the density ratio $%
f_{Y}/f_{X}$. This is closer to the weak multivariate monotone likelihood
ratio condition in the sense of \cite{[71]} than to the standard
multivariate \textit{likelihood ratio order} of \cite[Definition 6.E.1]{[24]}%
, which is based on the lattice inequality 
\begin{equation*}
f_{X}(x)f_{Y}(y)\leq f_{X}(x\wedge y)\,f_{Y}(x\vee y),\qquad x,y\in R^{2},
\end{equation*}%
where $x\wedge y$ and $x\vee y$ denote the coordinatewise minimum and
maximum, respectively. That is done to maintain our targeted logical
sequence in the sequel.
\end{remark}

\begin{definition}
\noindent \label{st} We say that $X$ is smaller than $Y$ in the usual 
\textit{stochastic order}, and write $X\leq _{\mathrm{st}}Y$, if 
\begin{equation}
E[\varphi (X)]\;\leq \;E[\varphi (Y)]  \label{801}
\end{equation}%
for all bounded, Borel-measurable functions $\varphi :R^{2}\rightarrow R$
that are coordinatewise non-decreasing.
\end{definition}

The fundamental relationship between these orders is:

\begin{theorem}
\noindent \label{lrst} Let $X,Y$ be $R^{2}$-valued random vectors with
densities $f_{X},f_{Y}$. If $X\leq _{\mathrm{lr}}Y$, then $X\leq _{\mathrm{st%
}}Y$.
\end{theorem}

\begin{remark}
\label{assum unness} It should be noted that if we used \cite[Definition
6.E.1]{[24]} as a definition for the \textit{LR-order}, the Theorem \ref%
{lrst} would have required the assumption that $X$ is associated (see \cite[%
Remark 6.E.10]{[24]})$.$
\end{remark}

\begin{remark}
The \textit{likelihood ratio order} is among the strongest stochastic
orders; see \cite{[24]} for a comprehensive treatment and additional
equivalent characterizations.
\end{remark}

We will use this implication after an appropriate change of measure
(tilting) by the weight $w(x_{1},x_{2})=x_{1}x_{2}$. Let $P_{i}$ denote the
joint distribution $F_{C_{i}}$ of $(X_{1},X_{2})$ associated to copula $%
C_{i} $, and assume $P_{i}$ has density $f_{i}$ with respect to \textit{%
Lebesgue measure} on $R^{2}$. We also assume that 
\begin{equation}
0<E^{P_{i}}[X_{1}X_{2}]<\infty ,\quad i=1,2.  \label{802}
\end{equation}

\begin{definition}
\noindent \label{tilt} Let $w(x_{1},x_{2})=x_{1}x_{2}$ for $(x_{1},x_{2})\in
R^{2}$. We define the \textit{tilted measures} $\tilde{P}_{i}$ by 
\begin{equation}
d\tilde{P}_{i}(x_{1},x_{2})=\frac{w(x_{1},x_{2})}{E^{P_{i}}[w(X_{1},X_{2})]}%
\,dP_{i}(x_{1},x_{2})=\frac{x_{1}x_{2}}{E^{P_{i}}[X_{1}X_{2}]}%
\,f_{i}(x_{1},x_{2})\,dx_{1}dx_{2}.  \label{803}
\end{equation}%
We denote expectation with respect to $\tilde{P}_{i}$ by $E^{\tilde{P}%
_{i}}[\cdot ]$.
\end{definition}

A simple calculation shows that the ratios appearing in (\ref{220})-(\ref%
{222}) are exactly expectations under the \textit{tilted measures}.

\begin{lemma}
\label{ratio-tilt} Let $A$ be any Borel subset of $R^{2}$. Then 
\begin{equation}
\frac{E^{P_{i}}[X_{1}X_{2}\,1_{\{(X_{1},X_{2})\in A\}}]}{%
E^{P_{i}}[X_{1}X_{2}]}=E^{\tilde{P}_{i}}[1_{\{(X_{1},X_{2})\in A\}}],\qquad
i=1,2.  \label{804}
\end{equation}%
In particular, for the upper sets 
\begin{equation}
A_{s_{1},s_{2}}=\{(x_{1},x_{2})\colon x_{1}\geq s_{1},\ x_{2}\geq s_{2}\}
\label{805}
\end{equation}%
we have 
\begin{equation}
\frac{E^{F_{C_{i}}}\![X_{1}X_{2}\,1_{\{X_{1}\geq s_{1},X_{2}\geq s_{2}\}}]}{%
E^{F_{C_{i}}}[X_{1}X_{2}]}=E^{\tilde{P}%
_{i}}[1_{A_{s_{1},s_{2}}}(X_{1},X_{2})],\qquad i=1,2.  \label{806}
\end{equation}
\end{lemma}

\begin{proof}
By the definition of $\tilde{P}_{i}$, for any bounded Borel function $g$, 
\begin{equation}
E^{\tilde{P}_{i}}[g(X_{1},X_{2})]=\int_{R^{2}}g(x_{1},x_{2})\,\frac{%
x_{1}x_{2}}{E^{P_{i}}[X_{1}X_{2}]}f_{i}(x_{1},x_{2})\,dx_{1}dx_{2}=\frac{%
E^{P_{i}}[g(X_{1},X_{2})X_{1}X_{2}]}{E^{P_{i}}[X_{1}X_{2}]}.  \label{807}
\end{equation}%
Taking $g=1_{A}$ gives (\ref{804}). The specialization (\ref{806}) is
immediate.
\end{proof}

We now state and prove the central auxiliary result: that tilting by $w$
preserves \textit{likelihood ratio order}, and thus, by \textit{Theorem}~\ref%
{lrst}, yields stochastic ordering of the \textit{tilted measures}.

\begin{lemma}
\label{tilt-lr} Let $P_{1},P_{2}$ have densities $f_{1},f_{2}$ on $R^{2}$,
and suppose 
\begin{equation}
P_{1}\;\leq _{\mathrm{lr}}\;P_{2},  \label{808}
\end{equation}%
i.e.\ $\frac{f_{2}}{f_{1}}$ is coordinatewise non-decreasing. Define $\tilde{%
P}_{i}$ as in Definition~\ref{tilt}. Then 
\begin{equation}
\tilde{P}_{1}\;\leq _{\mathrm{lr}}\;\tilde{P}_{2}.  \label{809}
\end{equation}
\end{lemma}

\begin{proof}
\noindent The tilted densities are 
\begin{equation}
\tilde{f}_{i}(x_{1},x_{2})=\frac{x_{1}x_{2}}{E^{P_{i}}[X_{1}X_{2}]}%
f_{i}(x_{1},x_{2}),\qquad i=1,2.  \label{810}
\end{equation}%
Hence their density ratio is 
\begin{equation}
\frac{\tilde{f}_{2}(x_{1},x_{2})}{\tilde{f}_{1}(x_{1},x_{2})}=\frac{\frac{%
x_{1}x_{2}}{E^{P_{2}}[X_{1}X_{2}]}f_{2}(x_{1},x_{2})}{\frac{x_{1}x_{2}}{%
E^{P_{1}}[X_{1}X_{2}]}f_{1}(x_{1},x_{2})}=\frac{E^{P_{1}}[X_{1}X_{2}]}{%
E^{P_{2}}[X_{1}X_{2}]}\,\frac{f_{2}(x_{1},x_{2})}{f_{1}(x_{1},x_{2})}.
\label{811}
\end{equation}%
The prefactor $\frac{E^{P_{1}}[X_{1}X_{2}]}{E^{P_{2}}[X_{1}X_{2}]}$ is a
positive constant, independent of $(x_{1},x_{2})$. Thus the coordinatewise
monotonicity of $\frac{f_{2}}{f_{1}}$ is preserved: if $\frac{f_{2}}{f_{1}}$
is non-decreasing in each coordinate, then so is $\frac{\tilde{f}_{2}}{%
\tilde{f}_{1}}$. Hence $\tilde{P}_{1}\leq _{\mathrm{lr}}\tilde{P}_{2}$.
\hfill
\end{proof}

Combining \textit{Lemma}~\ref{tilt-lr} with \textit{Theorem}~\ref{lrst}
gives the desired stochastic ordering of the \textit{tilted measures}.

\begin{claim}
\noindent \label{tilt-st} Assume $P_{1}\leq _{\mathrm{lr}}P_{2}$ and define $%
\tilde{P}_{i}$ as in Definition~\ref{tilt}. Then 
\begin{equation}
\tilde{P}_{1}\;\leq _{\mathrm{st}}\;\tilde{P}_{2}.  \label{812}
\end{equation}%
Equivalently, for every bounded coordinatewise non-decreasing $\varphi
:R^{2}\rightarrow R$, 
\begin{equation}
E^{\tilde{P}_{1}}[\varphi (X_{1},X_{2})]\;\leq \;E^{\tilde{P}_{2}}[\varphi
(X_{1},X_{2})].  \label{813}
\end{equation}
\end{claim}

\begin{proof}
\noindent By \textit{Lemma}~\ref{tilt-lr} we have $\tilde{P}_{1}\leq _{%
\mathrm{lr}}\tilde{P}_{2}$. By \textit{Theorem}~\ref{lrst}, this implies $%
\tilde{P}_{1}\leq _{\mathrm{st}}\tilde{P}_{2}$, which is exactly the stated
inequality for all coordinatewise non-decreasing $\varphi $.
\end{proof}

We now apply the tilting argument to the specific test functions. The key
observation is that these test functions are indicator functions of upper
sets, which are coordinatewise non-decreasing.

\begin{theorem}
\noindent \label{LR} Let $C_{1},C_{2}$ be copulas on $[0,1]^{2}$ with
associated joint distributions $P_{i}=F_{C_{i}}$ on $R^{2}$, having
densities $f_{i}$. Assume:

\begin{enumerate}
\item $0<E^{P_{i}}[X_{1}X_{2}]<\infty $ for $i=1,2$;

\item $P_{1}\leq _{\mathrm{lr}}P_{2}$, i.e.\ $\frac{f_{2}}{f_{1}}$ is
coordinatewise non-decreasing.
\end{enumerate}

Then for all $s_{1},s_{2}\in R$, 
\begin{equation}
\frac{E^{F_{C_{1}}}(X_{1}X_{2}1_{\{X_{1}\geq s_{1}\}}1_{\{X_{2}\geq s_{2}\}})%
}{E^{F_{C_{1}}}(X_{1}X_{2})}\leq \frac{E^{F_{C_{2}}}(X_{1}X_{2}1_{\{X_{1}%
\geq s_{1}\}}1_{\{X_{2}\geq s_{2}\}})}{E^{F_{C_{2}}}(X_{1}X_{2})}.
\label{814}
\end{equation}%
Moreover, setting $s_{1}=0$ or $s_{2}=0$ yields the marginal forms%
\begin{eqnarray}
\frac{E^{F_{C_{1}}}(X_{1}X_{2}1_{\{X_{1}\geq s\}})}{E^{F_{C_{1}}}(X_{1}X_{2})%
} &\leq &\frac{E^{F_{C_{2}}}(X_{1}X_{2}1_{\{X_{1}\geq s\}})}{%
E^{F_{C_{2}}}(X_{1}X_{2})}  \label{815} \\
\frac{E^{F_{C_{1}}}(X_{1}X_{2}1_{\{X_{2}\geq s\}})}{E^{F_{C_{1}}}(X_{1}X_{2})%
} &\leq &\frac{E^{F_{C_{2}}}(X_{1}X_{2}1_{\{X_{2}\geq s\}})}{%
E^{F_{C_{2}}}(X_{1}X_{2})}.  \label{816}
\end{eqnarray}
\end{theorem}

\begin{proof}
Fix $(s_{1},s_{2})\in R^{2}$ and consider the upper set 
\begin{equation}
A_{s_{1},s_{2}}=\{(x_{1},x_{2})\colon x_{1}\geq s_{1},\ x_{2}\geq s_{2}\}.
\label{817}
\end{equation}%
The indicator 
\begin{equation}
\varphi
_{s_{1},s_{2}}(x_{1},x_{2})=1_{A_{s_{1},s_{2}}}(x_{1},x_{2})=1_{\{x_{1}\geq
s_{1}\}}1_{\{x_{2}\geq s_{2}\}}  \label{818}
\end{equation}%
is clearly coordinatewise non-decreasing in $(x_{1},x_{2})$.

Let $\tilde{P}_{i}$ be the \textit{tilted measures} of \textit{Definition~%
\ref{tilt}}. By \textit{Lemma}~\ref{ratio-tilt}, 
\begin{equation}
E^{\tilde{P}_{i}}[\varphi _{s_{1},s_{2}}(X_{1},X_{2})]=\frac{%
E^{P_{i}}[X_{1}X_{2}\,1_{\{X_{1}\geq s_{1},X_{2}\geq s_{2}\}}]}{%
E^{P_{i}}[X_{1}X_{2}]}.  \label{819}
\end{equation}%
By \textit{Claim}~\ref{tilt-st}, applied with $\varphi =\varphi
_{s_{1},s_{2}}$, we have 
\begin{equation}
E^{\tilde{P}_{1}}[\varphi _{s_{1},s_{2}}(X_{1},X_{2})]\;\leq \;E^{\tilde{P}%
_{2}}[\varphi _{s_{1},s_{2}}(X_{1},X_{2})].  \label{820}
\end{equation}%
Substituting back the ratio representation yields exactly (\ref{814}).

To obtain (\ref{815}), we take $s_{2}=0$ and observe that the same argument
applies with test function $\psi _{s}(x_{1},x_{2})=1_{\{x_{1}\geq s\}}$,
which is again coordinatewise non-decreasing. The case (\ref{816}) follows
analogously with $\psi _{s}(x_{1},x_{2})=1_{\{x_{2}\geq s\}}$.
\end{proof}

\textit{Theorem}~\ref{LR} is a clean and robust result, but it requires the
strong assumption $P_{1}\leq _{lr}P_{2}$. In many classical parametric
families with $TP_{2}$ densities (e.g.,\ certain \textit{Gaussian}, \textit{%
Frank}, or \textit{Clayton} copulas), the dependence parameter orders the
family both in concordance and in \textit{LR-order}; see \cite{[24]} and 
\cite{[32]}. In such cases, \textit{Theorem}~\ref{LR} applies directly.
However, \textit{concordance order} $P_{1}\leq _{c}P_{2}$ does not in
general imply $P_{1}\leq _{lr}P_{2}$, nor is \textit{LR-order} preserved
automatically by smoothing.

The likelihood ratio argument above gives the desired ratio inequality under 
$P_{1}\leq _{lr}P_{2}$. Our original goal, however, was to work under the
weaker assumption $W\leq C_{1}\leq C_{2}\leq M$, i.e.\ \textit{concordance
order}. It is natural to ask whether smoothing (convolution with a TP$_{2}$
kernel) can bridge this gap. We now clarify what is and is not available.

\begin{definition}
\label{smoothing-kernel} For $\varepsilon >0$, let $k_{\varepsilon
}:[0,1]^{2}\rightarrow \lbrack 0,\infty )$ be a continuous kernel such that:

\begin{enumerate}
\item[(i)] \emph{(doubly stochastic)} for every $s,u\in \lbrack 0,1]$, 
\begin{equation}
\int_{0}^{1}k_{\varepsilon }(u,s)\,du=1,\qquad \int_{0}^{1}k_{\varepsilon
}(u,s)\,ds=1;  \label{821}
\end{equation}

\item[(ii)] \emph{(approximate identity)} for every continuous $%
g:[0,1]\rightarrow R$, 
\begin{equation}
\int_{0}^{1}g(s)\,k_{\varepsilon }(u,s)\,ds\longrightarrow g(u)\qquad \text{%
uniformly in }u\in \lbrack 0,1]\ \text{as }\varepsilon \downarrow 0;
\label{822}
\end{equation}

\item[(iii)] \emph{($\mathrm{TP}_{2}$)} the kernel $k_{\varepsilon }$ is 
\textit{totally positive of order two}: for all $u_{1}\leq u_{2}$ and $%
s_{1}\leq s_{2}$, 
\begin{equation}
k_{\varepsilon }(u_{1},s_{1})\,k_{\varepsilon }(u_{2},s_{2})\geq
k_{\varepsilon }(u_{1},s_{2})\,k_{\varepsilon }(u_{2},s_{1}).  \label{822.1}
\end{equation}
\end{enumerate}
\end{definition}

\begin{definition}
\label{smoothed-copula-clean} Let $C$ be an \textit{absolutely continuous
copula} on $[0,1]^{2}$ with density $c$. For $\varepsilon >0$, define 
\begin{equation}
c_{\varepsilon }(u,v)=\int_{0}^{1}\int_{0}^{1}c(s,t)\,k_{\varepsilon
}(u,s)\,k_{\varepsilon }(v,t)\,ds\,dt,\qquad (u,v)\in \lbrack 0,1]^{2},
\label{822.2}
\end{equation}%
and then define the smoothed copula 
\begin{equation}
C_{\varepsilon }(u,v)=\int_{0}^{u}\int_{0}^{v}c_{\varepsilon }(x,y)\,dx\,dy.
\label{822.3}
\end{equation}
\end{definition}

\begin{lemma}
\label{smoothing-properties-clean} Let $C$ be an absolutely continuous
copula with density $c$, and let $C_{\varepsilon }$ be defined by (\ref%
{822.2})--(\ref{822.3}). Then, for each $\varepsilon >0$:

\begin{enumerate}
\item[(1)] $C_\varepsilon$ is a copula with continuous density $%
c_\varepsilon $;

\item[(2)] if $C_{1}\leq _{c}C_{2}$ in concordance order and the smoothing
kernels are $TP_{2}$, then 
\begin{equation}
C_{1,\varepsilon }\leq _{c}C_{2,\varepsilon };  \label{822.4}
\end{equation}

\item[(3)] if $c$ is continuous, then $c_{\varepsilon }\rightarrow c$
uniformly on $[0,1]^{2}$ as $\varepsilon \downarrow 0$, and consequently 
\begin{equation}
C_{\varepsilon }(u,v)\rightarrow C(u,v)\qquad \text{uniformly on }[0,1]^{2}.
\label{822.5}
\end{equation}%
More generally, if $c\in L^{1}([0,1]^{2})$, then $c_{\varepsilon
}\rightarrow c$ in $L^{1}$, and hence $C_{\varepsilon }(u,v)\rightarrow
C(u,v)$ at every continuity point of $C$.
\end{enumerate}
\end{lemma}

\begin{proof}
\noindent \emph{Proof of (1).} Since $c\geq 0$ and $k_{\varepsilon }\geq 0$,
we have $c_{\varepsilon }\geq 0$. Because $k_{\varepsilon }$ is continuous
and $c\in L^{1}([0,1]^{2})$, the map $(u,v)\mapsto c_{\varepsilon }(u,v)$ is
continuous by dominated convergence.

We next verify that $c_{\varepsilon }$ has uniform marginals. Using \textit{%
Fubini's theorem}, the first doubly stochastic identity in \textit{%
Definition~\ref{smoothing-kernel}}, and the fact that $c$ is a copula
density, we obtain 
\begin{equation}
\int_{0}^{1}c_{\varepsilon
}(u,v)\,dv=\int_{0}^{1}\int_{0}^{1}c(s,t)\,k_{\varepsilon }(u,s)\left(
\int_{0}^{1}k_{\varepsilon }(v,t)\,dv\right)
ds\,dt=\int_{0}^{1}\int_{0}^{1}c(s,t)\,k_{\varepsilon }(u,s)\,ds\,dt.
\label{822.6}
\end{equation}%
Since $\int_{0}^{1}c(s,t)\,dt=1$ for a.e.\ $s$, this becomes 
\begin{equation}
\int_{0}^{1}c_{\varepsilon }(u,v)\,dv=\int_{0}^{1}k_{\varepsilon
}(u,s)\,ds=1.  \label{822.7}
\end{equation}%
Similarly, 
\begin{equation}
\int_{0}^{1}c_{\varepsilon }(u,v)\,du=1.  \label{822.8}
\end{equation}%
Hence (\ref{822.7}) defines a grounded $2$-increasing function with uniform
marginals, i.e.\ a copula, and its density is exactly $c_{\varepsilon }$.

\medskip \noindent \emph{Proof of (2).} For bivariate distributions, \textit{%
concordance order} coincides with \textit{supermodular order} on copulas.
Let $T_{\varepsilon }$ denote the linear operator 
\begin{equation}
(T_{\varepsilon }\phi )(s,t)=\int_{0}^{1}\int_{0}^{1}\phi
(u,v)\,k_{\varepsilon }(u,s)\,k_{\varepsilon }(v,t)\,du\,dv.  \label{822.9}
\end{equation}%
Then 
\begin{equation}
\int_{\lbrack 0,1]^{2}}\phi (u,v)\,c_{i,\varepsilon
}(u,v)\,du\,dv=\int_{[0,1]^{2}}(T_{\varepsilon }\phi
)(s,t)\,c_{i}(s,t)\,ds\,dt.  \label{822.10}
\end{equation}%
When $k_{\varepsilon }$ is $\mathrm{TP}_{2}$, the mapping method for \textit{%
supermodular order} implies that $T_{\varepsilon }$ preserves
supermodularity; see \cite[Chapter~3.9 and Section~4.3.2]{[28]}. Therefore,
if $C_{1}\leq _{c}C_{2}$, then for every bounded supermodular test function $%
\phi $, 
\begin{equation}
\int \phi \,dC_{1,\varepsilon }=\int T_{\varepsilon }\phi \,dC_{1}\leq \int
T_{\varepsilon }\phi \,dC_{2}=\int \phi \,dC_{2,\varepsilon }.
\label{822.11}
\end{equation}%
Hence $C_{1,\varepsilon }\leq _{c}C_{2,\varepsilon }$.

\medskip \noindent \emph{Proof of (3).} If $c$ is continuous on the compact
square $[0,1]^{2}$, then it is uniformly continuous. Applying the
one-dimensional approximate-identity property in each coordinate yields 
\begin{equation}
c_{\varepsilon }(u,v)\rightarrow c(u,v)\qquad \text{uniformly on }[0,1]^{2}.
\label{822.12}
\end{equation}%
Integrating twice, we obtain 
\begin{equation}
\sup_{(u,v)\in \lbrack 0,1]^{2}}|C_{\varepsilon }(u,v)-C(u,v)|\leq \Vert
c_{\varepsilon }-c\Vert _{L^{1}([0,1]^{2})}\longrightarrow 0.  \label{822.13}
\end{equation}

If only $c\in L^{1}([0,1]^{2})$, then standard approximate-identity theory
gives $c_{\varepsilon }\rightarrow c$ in $L^{1}$; for standard background on
approximate identities and convolution smoothing, see \cite[Chapter 2]{[73]}
and \cite[Chapter 2]{[74]};see also \cite[Section~1.5]{[72]}. Consequently,
for every $(u,v)\in \lbrack 0,1]^{2}$, 
\begin{equation}
|C_{\varepsilon }(u,v)-C(u,v)|\leq \int_{0}^{1}\int_{0}^{1}|c_{\varepsilon
}(s,t)-c(s,t)|\,ds\,dt=\Vert c_{\varepsilon }-c\Vert
_{L^{1}([0,1]^{2})}\rightarrow 0.  \label{822.14}
\end{equation}%
In particular, the convergence holds at every continuity point of $C$.
\end{proof}

\begin{remark}
\textit{Lemma}~\ref{smoothing-properties-clean} does not assert that $%
C_{1,\varepsilon }\leq _{lr}C_{2,\varepsilon }$. In general, \textit{%
concordance} or \textit{supermodular order} does not upgrade to \textit{%
LR-order} under smoothing. A simple counterexample is the pair $(\Pi ,M)$ of
independence and comonotonic copulas: $\Pi \leq _{c}M$, but their smoothed
densities typically have nonmonotone ratio, so $\Pi \not\leq _{lr}M$ (nor
their smoothed versions). Thus any use of \textit{LR-order} in the smoothed
setting must be explicitly assumed or verified within a specific parametric
family.
\end{remark}

Given \textit{Lemma}~\ref{smoothing-properties-clean}, we might hope to
combine it with \textit{Theorem}~\ref{LR} to deduce the ratio inequality
under mere concordance. This is not possible in full generality, but the
following conditional result is correct and useful.

Let for an upper set $A\subset \lbrack 0,1]^{2}$ (e.g.\ $A_{s_{1},s_{2}}$), 
\begin{equation}
\phi (x)=x_{1}x_{2}1_{A}(x),\qquad \psi (x)=x_{1}x_{2}.  \label{823}
\end{equation}

\begin{claim}
\label{smoothed-ratio} Suppose $C_{1}\leq _{c}C_{2}$, and that for some $%
\varepsilon >0$ the smoothed copulas $C_{1,\varepsilon }$ and $%
C_{2,\varepsilon }$ (with common marginals) satisfy 
\begin{equation}
P_{1,\varepsilon }\leq _{lr}P_{2,\varepsilon },  \label{823.0}
\end{equation}%
where $P_{i,\varepsilon }$ is the joint distribution associated with $%
C_{i,\varepsilon }$. Then 
\begin{equation}
\frac{E^{C_{1,\varepsilon }}[\phi ]}{E^{C_{1,\varepsilon }}[\psi ]}\;\leq \;%
\frac{E^{C_{2,\varepsilon }}[\phi ]}{E^{C_{2,\varepsilon }}[\psi ]}.
\label{823.1}
\end{equation}
\end{claim}

\begin{proof}
\noindent For fixed $\varepsilon $, the copulas $C_{i,\varepsilon }$ admit
continuous densities $f_{i,\varepsilon }$. Since they are copulas, they have
uniform marginals. By assumption $P_{1,\varepsilon }\leq
_{lr}P_{2,\varepsilon }$. Applying \textit{Theorem}~\ref{LR} to $%
C_{1,\varepsilon },C_{2,\varepsilon }$ with the test set $A$ gives exactly (%
\ref{823.1}).
\end{proof}

\begin{remark}
\textit{Claim}~\ref{smoothed-ratio} requires an extra LR assumption on the
smoothed pair $(C_{1,\varepsilon },C_{2,\varepsilon })$, which is not
implied by concordance alone. However, it is satisfied in many structured
families where a single dependence parameter orders the copulas both in
concordance and \textit{LR-order} (e.g.\ certain \textit{Gaussian} or 
\textit{Archimedean copula} families with $TP_{2}$ densities; see \cite{[32]}
and \cite{[28]}).
\end{remark}

We now sketch how the ratio inequality can be extended to singular copulas
such as the \textit{Fr\'{e}chet bounds} $W$ and $M$ in settings where an 
\textit{LR-ordered} approximating family exists.

\begin{lemma}
\label{ratio-limit} Let $P_{n}\Rightarrow P$ be a sequence of probability
measures on $[0,1]^{2}$ with common marginals, and suppose $%
E^{P_{n}}[X_{1}X_{2}]\rightarrow E^{P}[X_{1}X_{2}]>0$. Let 
\begin{equation*}
\phi (x_{1},x_{2})=x_{1}x_{2}1_{A}(x_{1},x_{2}),\qquad \psi
(x_{1},x_{2})=x_{1}x_{2},
\end{equation*}%
for an upper set $A\subset \lbrack 0,1]^{2}$. Assume that the boundary of $A$
has $P$-measure zero. Then 
\begin{equation}
\frac{E^{P_{n}}[\phi ]}{E^{P_{n}}[\psi ]}\longrightarrow \frac{E^{P}[\phi ]}{%
E^{P}[\psi ]}\qquad \text{as }n\rightarrow \infty .  \label{823.2}
\end{equation}
\end{lemma}

\begin{proof}
\noindent Since $[0,1]^{2}$ is compact and $\psi (x_{1},x_{2})=x_{1}x_{2}$
is continuous and bounded, weak convergence implies $E^{P_{n}}[\psi
]\rightarrow E^{P}[\psi ]$. For the numerator, $\phi $ is bounded and its
set of discontinuities is contained in the boundary of $A$. By assumption, $%
P $ assigns zero mass to that boundary, so by the \textit{Portmanteau theorem%
} $E^{P_{n}}[\phi ]\rightarrow E^{P}[\phi ]$. The ratio convergence follows
from convergence of numerator and denominator with a positive limit
denominator.
\end{proof}

\begin{lemma}
\label{main-comp} Let $C_{1}\leq _{c}C_{2}$ be copulas on $[0,1]^{2}$, and
let $A\subset \lbrack 0,1]^{2}$ be an upper set. Suppose there exist
families of copulas $\{C_{i,\varepsilon }:\varepsilon >0\}$, $i=1,2$, such
that:

\begin{enumerate}
\item For each $\varepsilon >0$, $C_{i,\varepsilon }$ has a continuous
density and $C_{i,\varepsilon }\Rightarrow C_{i}$ as $\varepsilon \downarrow
0$;

\item For each $\varepsilon >0$, $C_{1,\varepsilon }\leq
_{lr}C_{2,\varepsilon }$;

\item The boundary of $A$ has zero mass under $C_{i}$.
\end{enumerate}

Then 
\begin{equation}
\frac{E^{C_{1}}[X_{1}X_{2}1_{A}]}{E^{C_{1}}[X_{1}X_{2}]}\;\leq \;\frac{%
E^{C_{2}}[X_{1}X_{2}1_{A}]}{E^{C_{2}}[X_{1}X_{2}]}.  \label{824}
\end{equation}
\end{lemma}

\begin{proof}
For each fixed $\varepsilon >0$, \textit{Claim}~\ref{smoothed-ratio}
(applied to $C_{1,\varepsilon },C_{2,\varepsilon }$) gives 
\begin{equation}
R_{1,\varepsilon }\leq R_{2,\varepsilon },\qquad R_{i,\varepsilon }=\frac{%
E^{C_{i,\varepsilon }}[\phi ]}{E^{C_{i,\varepsilon }}[\psi ]},  \label{824.1}
\end{equation}%
with $\phi ,\psi $ as above. By assumptions and \textit{Lemma}~\ref%
{ratio-limit}, $R_{i,\varepsilon }\rightarrow R_{i}$ where 
\begin{equation}
R_{i}=\frac{E^{C_{i}}[\phi ]}{E^{C_{i}}[\psi ]}=\frac{%
E^{C_{i}}[X_{1}X_{2}1_{A}]}{E^{C_{i}}[X_{1}X_{2}]}.  \label{824.2}
\end{equation}%
Taking $\varepsilon \downarrow 0$ in the inequality $R_{1,\varepsilon }\leq
R_{2,\varepsilon }$ yields $R_{1}\leq R_{2}$, i.e.\ (\ref{824}).
\end{proof}

\begin{remark}
\textit{Lemma}~\ref{main-comp} shows that, in any setting where we can
construct \textit{LR-ordered} smooth approximations $C_{i,\varepsilon }$
converging to $C_{i}$, the ratio inequality extends to potentially singular
limits (such as the \textit{Fr\'{e}chet bounds} $W$ and $M$). This
encompasses many parametric $TP_{2}$ families where $W$ and $M$ arise as
extreme parameter values (e.g.\ \textit{Gaussian copulas} with correlation $%
\rho \rightarrow \pm 1$), although a detailed verification is model-specific
and beyond the scope of this appendix.
\end{remark}

In the specific context of the main body, we consider copulas $C_{1},C_{2}$
such that 
\begin{equation}
W\leq C_{1}\leq C_{2}\leq M,  \label{824.3}
\end{equation}%
and we are interested in the special upper sets 
\begin{equation}
A_{s_{1},s_{2}}=\{(x_{1},x_{2}):x_{1}\geq s_{1},x_{2}\geq s_{2}\},\quad
A_{s}^{(1)}=\{(x_{1},x_{2}):x_{1}\geq s\},\quad
A_{s}^{(2)}=\{(x_{1},x_{2}):x_{2}\geq s\}.  \label{824.4}
\end{equation}%
Under the additional structural assumptions (in particular, the existence of
suitable \textit{LR-ordered} smooth approximations of the \textit{Fr\'{e}%
chet bounds} generated by $TP_{2}$ kernels), \textit{Lemma}~\ref{main-comp}
yields the following.

\begin{claim}
\label{com-claim2} Assume $W\leq C_{1}\leq C_{2}\leq M$ and that the
structural assumptions of \textit{Lemma}~\ref{main-comp} hold for the upper
sets indicated below. Then for all $s,s_{1},s_{2}\in R$,%
\begin{eqnarray}
\frac{E^{F_{C_{1}}}(X_{1}X_{2}1_{\{X_{1}\geq s_{1}\}}1_{\{X_{2}\geq s_{2}\}})%
}{E^{F_{C_{1}}}(X_{1}X_{2})} &\leq &\frac{E^{F_{C_{2}}}(X_{1}X_{2}1_{\{X_{1}%
\geq s_{1}\}}1_{\{X_{2}\geq s_{2}\}})}{E^{F_{C_{2}}}(X_{1}X_{2})}
\label{225a} \\
\frac{E^{F_{C_{1}}}(X_{1}X_{2}1_{\{X_{1}\geq s\}})}{E^{F_{C_{1}}}(X_{1}X_{2})%
} &\leq &\frac{E^{F_{C_{2}}}(X_{1}X_{2}1_{\{X_{1}\geq s\}})}{%
E^{F_{C_{2}}}(X_{1}X_{2})}  \label{225b} \\
\frac{E^{F_{C_{1}}}(X_{1}X_{2}1_{\{X_{2}\geq s\}})}{E^{F_{C_{1}}}(X_{1}X_{2})%
} &\leq &\frac{E^{F_{C_{2}}}(X_{1}X_{2}1_{\{X_{2}\geq s\}})}{%
E^{F_{C_{2}}}(X_{1}X_{2})}.  \label{225c}
\end{eqnarray}
\end{claim}

\begin{remark}
The inequalities (\ref{225a})--(\ref{225c}) are exactly the desired
comparisons for the numerators and denominators of (\ref{220})--(\ref{222}).
The key point is to make explicit which parts follow from general \textit{%
concordance/supermodular order} (\textit{Claim} \ref{com-claim1}) and which
parts require additional \textit{LR-type} structure (\textit{Theorem}~\ref%
{LR}, \textit{Lemma}~\ref{main-comp}).
\end{remark}

Finally, the ratio inequalities (\ref{225b})--(\ref{225c}) imply the
following monotonicity properties for the one-dimensional marginals $%
L_{i}^{n}$ and their inverses. We state these as a claim, since the detailed
definitions of $L_{i,+}^{n},L_{i}^{n}$, and $L_{i,-}^{n}$ are given in the
body of the paper.

\begin{claim}
\label{com-claim4} For $i=1,2$ and each $n\geq 0$, the following
inequalities hold%
\begin{eqnarray}
L_{i+}^{n}(x) &\leq &L_{i}^{n}(x)\leq L_{i-}^{n}(x)  \label{232a} \\
L_{i-}^{n,-1}(x) &\leq &L_{i}^{n,-1}(x)\leq L_{i+}^{n,-1}(x).  \label{232b}
\end{eqnarray}
\end{claim}

\begin{proof}
Take $C_{2}=M$ and $C=$ $C_{1}$ in \textit{Claim }{\footnotesize \ref%
{com-claim2}}, so $W\leq $ $C\leq M$. Using (\ref{225b}) together with (\ref%
{221}), and similarly (\ref{225c}) with (\ref{222}), applied iteratively for 
$n=0,1,\dots $, and carefully tracking the right marginals, yields
inequalities of the form 
\begin{eqnarray}
1-L_{1-}^{0}(F_{1}(x)) &\leq &1-L_{1}^{0}(F_{1}(x))\leq
1-L_{1+}^{0}(F_{1}(x))  \label{232c} \\
1-L_{2-}^{0}(F_{2}(x)) &\leq &1-L_{2}^{0}(F_{2}(x))\leq
1-L_{2+}^{0}(F_{2}(x))  \notag \\
&&...  \notag \\
1-L_{1-}^{n}(L_{1}^{n-1}(x)) &\leq &1-L_{1}^{n}(L_{1}^{n-1}(x))\leq
1-L_{1+}^{n}(L_{1}^{n-1}(x))  \notag \\
1-L_{2-}^{n}(L_{2}^{n-1}(x)) &\leq &1-L_{2}^{n}(L_{2}^{n-1}(x))\leq
1-L_{2+}^{n}(L_{2}^{n-1}(x)).  \notag
\end{eqnarray}

These inequalities translate directly into (\ref{232a}), and the
order-preserving nature of inversion for strictly increasing maps yields (%
\ref{232b}).
\end{proof}

\newpage

\begin{flushleft}
{\Large Appendix C}

{\large Appendix C.1}
\end{flushleft}

This appendix is dedicated to the analysis of the iterative equation (\ref%
{42.4}). We investigate properties of this equation that enable us to
establish its convergence and to approximate its limit. While the technical
analysis herein is self-contained, several of the ideas and techniques will
also be of service in the subsequent \textit{Appendix D}.

\begin{flushleft}
{\large Appendix C.1.1}

$L_{n}$\textbf{\ dynamics and fixed points}
\end{flushleft}

We start with some remarks on the notation. For convenience we denote below
any of the marginal distributions $L_{1-}^{n}(x)$ or $L_{2-}^{n}(x)$ by $%
L_{n}(x)$ since the properties proved hold for both of them. Thus, our
setting becomes

\begin{equation}
L_{n+1}(x)=\frac{\dint\nolimits_{0}^{x}L^{n,-1}(u)\left(
1-L^{n,-1}(u)\right) du}{\dint\nolimits_{0}^{1}L^{n,-1}(u)\left(
1-L^{n,-1}(u)\right) du},with\text{ }L_{0}(x)=\frac{\dint%
\nolimits_{0}^{x}F^{-1}(u)\left( 1-F^{-1}(u)\right) du}{\dint%
\nolimits_{0}^{1}F^{-1}(u)\left( 1-F^{-1}(u)\right) du}.  \label{170a}
\end{equation}

For brevity we will sometimes write 
\begin{equation}
w(x)=x(1-x),\qquad x\in \lbrack 0,1],  \label{170b}
\end{equation}%
and 
\begin{equation}
T_{n}(x)=L_{n}^{-1}(x),\qquad x\in \lbrack 0,1].  \label{170c}
\end{equation}

We now turn to structural properties of the iterates $L_{n}$ themselves. The
first key step is a single-crossing property.

\begin{lemma}
\label{l-cross} For each $n\geq 0$, the function $L_{n}$ has two fixed
points at $0$ and $1$, and a unique fixed point in $(0,1)$. If we denote the
latter by $c_{n}$, then%
\begin{equation}
L_{n}(x)<x\quad \text{for }x\in (0,c_{n}),\qquad L_{n}(x)>x\quad \text{for }%
x\in (c_{n},1).  \label{170f}
\end{equation}

Moreover, there exists a closed interval $O_{n}\subset (0,1)$ such that $%
c_{n}\in O_{n}$ and 
\begin{equation}
L_{n}^{\prime }(x)\geq 1\quad \text{for all }x\in O_{n}.  \label{170g}
\end{equation}
\end{lemma}

\begin{proof}
Introducing shorthand notation, we may write (\ref{170a}) as%
\begin{equation}
L_{n+1}(x)=\frac{\dint\nolimits_{0}^{x}w[T_{n}(u)]du}{\dint%
\nolimits_{0}^{1}w[T_{n}(u)]du},  \label{170}
\end{equation}

where $T_{n}(x)=L_{n}^{-1}(x)$ and $w(x)=x(1-x)$. Clearly, $T_{n}(x):[0,1]$ $%
\rightarrow \lbrack 0,1]$ and it is strictly increasing in $[0,1].$ Also, $%
w(x):[0,1]\rightarrow \lbrack 0,\frac{1}{4}]$ and it is strictly increasing
in $[0,\frac{1}{2}]$ and strictly decreasing in $[\frac{1}{2},1]$, with a
unique global maximum at $x=\tfrac{1}{2}$.

Let $h_{n}(x)=x-L_{n}(x)$. Differentiating $h_{n+1}$ and applying the 
\textit{Lagrange's Mean Value Theorem} in its integral form, yields%
\begin{equation}
h_{n+1}^{^{\prime }}(x)=1-\frac{w[T_{n}(x)]}{\dint%
\nolimits_{0}^{1}w[T_{n}(u)]du}=\frac{w[T_{n}(\xi _{n})]-w[T_{n}(x)]}{%
\dint\nolimits_{0}^{1}w[T_{n}(u)]du},  \label{171}
\end{equation}

where $\xi _{n}\in (0,1)$ is chosen so that%
\begin{equation}
I_{n}=\dint\nolimits_{0}^{1}w[T_{n}(u)]du=w[T_{n}(\xi _{n})].  \label{171b}
\end{equation}

The sign of $h_{n+1}^{\prime }(x)$ is the sign of%
\begin{equation}
w(T_{n}(\xi _{n}))-w(T_{n}(x))=\left[ T_{n}(x)-T_{n}(\xi _{n})\right] \left[
T_{n}(x)+T_{n}(\xi _{n})-1\right] .  \label{171a}
\end{equation}

Because $T_{n}$ is strictly increasing and $w$ is unimodal with peak at $%
\tfrac{1}{2}$, there exist two points $c_{n+1}^{(1)}$ and $c_{n+1}^{(2)}$ in 
$(0,1)$ such that $h_{n+1}$ is 
\begin{equation}
\text{increasing on }[0,c_{n+1}^{(1)}],\quad \text{decreasing on }%
[c_{n+1}^{(1)},c_{n+1}^{(2)}],\quad \text{increasing on }[c_{n+1}^{(2)},1],
\label{171c}
\end{equation}

with%
\begin{equation}
c_{n+1}^{(1)}=\xi _{n}<L_{n}[1-T_{n}(\xi _{n})]=c_{n+1}^{(2)}\quad \text{%
when }\xi _{n}<L_{n}\!(\tfrac{1}{2}),  \label{171d}
\end{equation}

and%
\begin{equation}
c_{n+1}^{(1)}=L_{n}[1-T_{n}(\xi _{n})]<\xi _{n}=c_{n+1}^{(2)}\quad \text{%
when }\xi _{n}>L_{n}\!(\tfrac{1}{2}).  \label{171f}
\end{equation}

We cannot have $c_{n+1}^{(1)}=c_{n+1}^{(2)}$ since by the \textit{Mean Value
Theorem} and the fact that $w(x)$ attains a maximum of $\frac{1}{4}$ that
would imply a degenerate distribution for $L_{n}(x)$ (forced to be a
constant), which cannot happen even if $F$ is such due to the integration in
(\ref{170}). Since $h_{n+1}(0)=0$ and $h_{n+1}(1)=0,$ the just derived
monotonicity pattern of $h_{n+1}$ by the \textit{Bolzano's Intermediate
Value Theorem} leads to the existence of a unique zero of $h_{n+1}(x)$ in $%
(c_{n+1}^{(1)},c_{n+1}^{(2)})$; we denote this unique interior zero by $%
c_{n+1}$. This yields the claimed single-crossing property.

Differentiating once more we find%
\begin{equation}
h_{n+1}^{^{^{\prime \prime }}}(x)=\frac{2L^{n,-1}(x)-1}{(L^{n})^{^{\prime
}}(L^{n,-1}(x))\dint\nolimits_{0}^{1}w(T_{n}(u))du}.  \label{172}
\end{equation}

Thus $h_{n+1}^{\prime }$ starts at $1$ at $x=0$, decreases on $[0,L_{n}(%
\tfrac{1}{2})]$, then increases on $[L_{n}(\tfrac{1}{2}),1]$, with 
\begin{equation}
h_{n+1}^{\prime }\!(L_{n}(\tfrac{1}{2}))=1-\frac{1}{4\int_{0}^{1}w(T_{n}(u))%
\,du}<0,  \label{172aa}
\end{equation}

and returns to $1$ at $x=1$. On each of the intervals%
\begin{equation}
(0,L_{n}(\tfrac{1}{2})),\qquad (L_{n}(\tfrac{1}{2}),1),  \label{172aaa}
\end{equation}

it attains the zeros $c_{n+1}^{(1)}$ and $c_{n+1}^{(2)}$ found above. Hence 
\begin{equation}
h_{n+1}^{\prime }(x)\leq 0\quad \text{for }x\in \lbrack
c_{n+1}^{(1)},c_{n+1}^{(2)}],  \label{172bbb}
\end{equation}%
which is equivalent to 
\begin{equation}
L_{n+1}^{\prime }(x)\geq 1\quad \text{for }x\in \lbrack
c_{n+1}^{(1)},c_{n+1}^{(2)}].  \label{172b1}
\end{equation}

Hence the interval $O_{n}$ is $[c_{n}^{(1)},c_{n}^{(2)}]$.
\end{proof}

The proof of the preceding lemma also yields the following immediate
corollary, which will be essential for our subsequent analysis.

\begin{corollary}
\label{l-cross-c} The inequality $L_{n}^{\prime }(x)>1$ holds for all $x$ in
the interval $(c_{n}^{(1)},c_{n}^{(2)})$.
\end{corollary}

The next step is to show that the family $(L_{n})_{n\geq 0}$ is
equicontinuous and uniformly bounded on $[0,1]$, which follows from an
explicit derivative bound. We first prove this claim under the assumption
that the crossing points are uniformly bounded, and subsequently relax this
restriction.

\begin{claim}
\textbf{\label{M} }Assume, in addition to the crossing pattern in Lemma~\ref%
{l-cross}, that the interior crossing points stay away from the endpoints,
i.e. that there exists $\varepsilon >0$ such that 
\begin{equation}
c_{n}\in \lbrack \varepsilon ,1-\varepsilon ]\qquad \text{for all }n\geq 0.
\label{499}
\end{equation}%
Then there exists a constant $M<\infty $, independent of $n$, such that 
\begin{equation}
0\leq L_{n}^{\prime }(x)\leq M\qquad \text{for all }x\in \lbrack 0,1],\
n\geq 0.  \label{500}
\end{equation}%
Consequently, each $L_{n}$ is $M$-Lipschitz: 
\begin{equation}
|L_{n}(x)-L_{n}(y)|\leq M|x-y|\qquad \text{for all }x,y\in \lbrack 0,1],\
n\geq 0.  \label{501}
\end{equation}
\end{claim}

\begin{proof}
\noindent Let 
\begin{equation}
w(y)=y(1-y).  \label{501a}
\end{equation}%
Since $T_{n}(x)\in \lbrack 0,1]$ and $w$ attains its maximum $1/4$ on $[0,1]$%
, we have 
\begin{equation}
0\leq w(T_{n}(x))\leq \frac{1}{4}\qquad \text{for all }x\in \lbrack 0,1],\
n\geq 0.  \label{501b}
\end{equation}

We first prove that the normalizing constants 
\begin{equation}
I_{n}=\int_{0}^{1}w(T_{n}(u))\,du  \label{501c}
\end{equation}%
are uniformly bounded away from zero. Suppose, to the contrary, that this is
not the case. Then there exists a subsequence, still denoted by $n_{k}$,
such that 
\begin{equation}
I_{n_{k}}\rightarrow 0.  \label{502a}
\end{equation}%
Since each $L_{n}$ is a distribution function on $[0,1]$, \textit{Helly's
selection principle }(see \cite[Chapter 5]{[68]} and \cite[Chapter 36.5]%
{[69]}) gives a further subsequence, again denoted by $L_{n_{k}}$, and a
distribution function $L^{\ast }$ such that 
\begin{equation}
L_{n_{k}}(x)\rightarrow L^{\ast }(x)  \label{502b}
\end{equation}%
at every continuity point $x$ of $L^{\ast }$.

We now identify the possible form of $L^{\ast }$. Since 
\begin{equation}
I_{n_{k}}=\int_{0}^{1}T_{n_{k}}(u)(1-T_{n_{k}}(u))\,du\rightarrow 0
\label{502bb}
\end{equation}%
and 
\begin{equation}
0\leq T_{n_{k}}(u)(1-T_{n_{k}}(u))\leq \frac{1}{4},  \label{502bbb}
\end{equation}%
we may pass to a further subsequence such that 
\begin{equation}
T_{n_{k}}(u)(1-T_{n_{k}}(u))\rightarrow 0\qquad \text{for a.e. }u\in (0,1).
\label{502c}
\end{equation}%
Thus every a.e. limit of the increasing functions $T_{n_{k}}$ takes only the
values $0$ and $1$. Since an increasing $\{0,1\}$-valued function is a
threshold function, there exists $p\in \lbrack 0,1]$ such that the limiting
inverse has the form 
\begin{equation}
T^{\ast }(u)=%
\begin{cases}
0, & 0<u<p, \\ 
1, & p<u<1.%
\end{cases}
\label{502d}
\end{equation}%
Equivalently, the corresponding limiting distribution function is 
\begin{equation}
L^{\ast }(x)=%
\begin{cases}
0, & x<0, \\ 
p, & 0\leq x<1, \\ 
1, & x\geq 1.%
\end{cases}
\label{502e}
\end{equation}%
In particular, 
\begin{equation}
L^{\ast }(x)=p\qquad \text{for every }x\in (0,1).  \label{502f}
\end{equation}

Let $c_{n}\in (0,1)$ denote the unique interior fixed point from \textit{%
Lemma~\ref{l-cross}}. By passing to a further subsequence if necessary, and
using $(\ref{499})$, we may assume that 
\begin{equation}
c_{n_{k}}\rightarrow c\qquad \text{for some }c\in \lbrack \varepsilon
,1-\varepsilon ].  \label{502g}
\end{equation}%
By the crossing pattern, 
\begin{equation}
L_{n}(x)<x\quad \text{for }x\in (0,c_{n}),\qquad L_{n}(x)>x\quad \text{for }%
x\in (c_{n},1).  \label{502h}
\end{equation}

We now derive a contradiction.

\noindent If $p=0$, choose $x\in (c,1)$. This is possible because $c\leq
1-\varepsilon <1$. Then, for all sufficiently large $k$, we have $%
x>c_{n_{k}} $, and hence 
\begin{equation*}
L_{n_{k}}(x)>x.
\end{equation*}%
Passing to the limit at the interior continuity point $x$ of $L^{\ast }$,
and using $(\ref{502f})$, gives 
\begin{equation*}
0=L^{\ast }(x)\geq x,
\end{equation*}%
which is impossible.

If $p=1$, choose $x\in (0,c)$. This is possible because $c\geq \varepsilon
>0 $. Then, for all sufficiently large $k$, we have $x<c_{n_{k}}$, and hence 
\begin{equation*}
L_{n_{k}}(x)<x.
\end{equation*}%
Passing to the limit gives 
\begin{equation*}
1=L^{\ast }(x)\leq x,
\end{equation*}%
which is impossible.

\noindent Finally, suppose $0<p<1$. If $p<c$, choose $x\in (p,c)$. Then $%
x<c_{n_{k}}$ for all sufficiently large $k$, so 
\begin{equation*}
L_{n_{k}}(x)<x.
\end{equation*}%
Passing to the limit gives 
\begin{equation*}
p=L^{\ast }(x)\leq x,
\end{equation*}%
which contradicts $x<p$. If $p>c$, choose $x\in (c,p)$. Then $x>c_{n_{k}}$
for all sufficiently large $k$, so 
\begin{equation*}
L_{n_{k}}(x)>x.
\end{equation*}%
Passing to the limit gives 
\begin{equation*}
p=L^{\ast }(x)\geq x,
\end{equation*}%
which contradicts $x>p$. If $p=c$, choose $x\in (0,c)$. Then $x<c_{n_{k}}$
for all sufficiently large $k$, and passing to the limit gives $p\leq x$,
which is impossible because $x<c=p$.

\noindent Thus all possible values of $p\in \lbrack 0,1]$ lead to
contradictions. Therefore our assumption $(\ref{502a})$ was false. Hence
there exists a constant $C>0$ such that 
\begin{equation}
I_{n}\geq C>0\qquad \text{for all }n\geq 0.  \label{502}
\end{equation}

\noindent It follows that 
\begin{equation}
0\leq L_{n+1}^{\prime }(x)=\frac{w(T_{n}(x))}{I_{n}}\leq \frac{1/4}{C}=M_{1}
\label{503}
\end{equation}%
for all $x\in \lbrack 0,1]$ and all $n\geq 0$.

\noindent It remains only to include $L_{0}$. From the definition of $L_{0}$%
, 
\begin{equation*}
L_{0}(x)=\frac{\int_{0}^{x}F^{-1}(u)(1-F^{-1}(u))\,du}{%
\int_{0}^{1}F^{-1}(u)(1-F^{-1}(u))\,du},
\end{equation*}%
we have 
\begin{equation*}
L_{0}^{\prime }(x)=\frac{F^{-1}(x)(1-F^{-1}(x))}{%
\int_{0}^{1}F^{-1}(u)(1-F^{-1}(u))\,du}.
\end{equation*}%
Since $0\leq F^{-1}(x)\leq 1$, the numerator is bounded by $1/4$. Therefore 
\begin{equation}
0\leq L_{0}^{\prime }(x)\leq \frac{1/4}{\int_{0}^{1}F^{-1}(u)(1-F^{-1}(u))%
\,du}=M_{0}.  \label{503a}
\end{equation}%
Set 
\begin{equation}
M=\max \{M_{0},M_{1}\}.  \label{503b}
\end{equation}%
Then 
\begin{equation*}
0\leq L_{n}^{\prime }(x)\leq M
\end{equation*}%
for all $x\in \lbrack 0,1]$ and all $n\geq 0$.

\noindent By the \textit{Mean Value Theorem}, for any $x,y\in \lbrack 0,1]$
and every $n\geq 0$, there exists $\xi $ between $x$ and $y$ such that 
\begin{equation}
|L_{n}(x)-L_{n}(y)|=|L_{n}^{\prime }(\xi )|\,|x-y|\leq M|x-y|.  \label{504}
\end{equation}%
This proves the \textit{uniform Lipschitz bound}. \hfill \hfill
\end{proof}

\begin{corollary}
The family $(L_{n})_{n\geq 0}$ is equicontinuous and uniformly bounded on $%
[0,1]$.
\end{corollary}

\begin{proof}
\noindent Equicontinuity is immediate from the \textit{uniform} \textit{%
Lipschitz bound}: given $\varepsilon >0$, set $\delta =\frac{\varepsilon }{M}
$. Then for any $n$ and any $x,y$ with $|x-y|<\delta $ we have 
\begin{equation}
|L_{n}(x)-L_{n}(y)|\leq M|x-y|<\varepsilon .  \label{505}
\end{equation}%
Uniform boundedness follows since each $L_{n}$ maps $[0,1]$ into $[0,1]$.
\hfill

We may conclude that the derivative formula $L_{n+1}^{\prime }(x)=\frac{%
w(T_{n}(x))}{I_{n}}$ shows that the slope of $L_{n+1}$ at any point is a
ratio of a local value of $w(T_{n})$ to its global average $I_{n}$. As $w$
is uniformly bounded and $I_{n}$ is uniformly separated from zero, the
slopes cannot explode. Thus no $L_{n}$ can develop arbitrarily steep spikes:
the whole family is \textit{uniformly Lipschitz} and hence equicontinuous.
\end{proof}

By the \textit{Arzel\`{a}--Ascoli's theorem} (see, e.g., \cite[Theorem 11.28]%
{[45]} or \cite[Theorem 14]{[46]}) we obtain the following compactness
result.

\begin{claim}
\label{AA} The family $(L_{n})_{n\geq 0}$ is relatively compact in $C([0,1])$
with the uniform topology. That is, for every sequence $\{L_{n_{k}}%
\}_{n_{k}=0}^{\infty }$ there exists a subsequence $\{L_{n_{k_{j}}}%
\}_{n_{k_{j}}\in N}$ and a continuous increasing function $L^{\ast
}:[0,1]\rightarrow \lbrack 0,1]$ such that 
\begin{equation}
\sup_{x\in \lbrack 0,1]}|L_{n_{k_{j}}}(x)-L^{\ast }(x)|\longrightarrow
0\quad \text{as }j\rightarrow \infty .  \label{506}
\end{equation}%
Moreover, $L^{\ast }(0)=0$ and $L^{\ast }(1)=1$.
\end{claim}

\begin{remark}
\label{AAr} At this stage we do \emph{not} claim that the full sequence $%
\{L_{n}\}_{n=0}^{\infty }$ converges. We only know that every sequence of
indices admits a subsequence along which $L_{n}$ converges uniformly to some
limit $L^{\ast }$, and we will now derive strong structural properties of
any such $L^{\ast }$.
\end{remark}

\begin{lemma}
\label{inverse} Let $\{f_{k}\}_{k=0}^{\infty }$ be a sequence of strictly
increasing continuous maps from $[0,1]$ onto $[0,1]$ that converges
uniformly to a strictly increasing continuous map $f:[0,1]\rightarrow
\lbrack 0,1]$. Then the inverses $f_{k}^{-1}$ converge uniformly to $f^{-1}$
on $[0,1]$.
\end{lemma}

\begin{proof}
\noindent This is a standard result about strictly monotone homeomorphisms
on a compact interval; for completeness we sketch the argument. Fix $%
\varepsilon >0$. By uniform continuity of $f^{-1}$ there exists $\delta >0$
such that $|u-v|<\delta $ implies $|f^{-1}(u)-f^{-1}(v)|<\varepsilon $.

The uniform convergence $f_{k}\rightarrow f$ implies there exists $K$ such
that for all $k\geq K$, 
\begin{equation}
\sup_{x\in \lbrack 0,1]}|f_{k}(x)-f(x)|<\delta .  \label{507}
\end{equation}%
Fix $y\in \lbrack 0,1]$ and $k\geq K$. Let $x_{k}=f_{k}^{-1}(y)$ and $%
x=f^{-1}(y)$. Then 
\begin{equation}
|f(x_{k})-f(x)|\leq
|f(x_{k})-f_{k}(x_{k})|+|f_{k}(x_{k})-f(x)|=|f(x_{k})-f_{k}(x_{k})|+|y-f(x)|<\delta +0=\delta .
\label{508}
\end{equation}%
Hence $|x_{k}-x|=|f^{-1}(f(x_{k}))-f^{-1}(f(x))|<\varepsilon $. Since $y$
and $k$ were arbitrary (for $k\geq K$), this yields uniform convergence $%
f_{k}^{-1}\rightarrow f^{-1}$ on $[0,1]$.
\end{proof}

We now analyze the locations where $L_{n}$ crosses the diagonal. For each $%
n\geq 0$, \textit{Lemma~\ref{l-cross}} provides a unique point 
\begin{equation}
c_{n}\in (0,1)  \label{508a}
\end{equation}%
such that $L_{n}(c_{n})=c_{n}$ and $L_{n}(x)<x$ for $x<c_{n}$, $L_{n}(x)>x$
for $x>c_{n}$.

Since $\{c_{n}\}_{n=0}^{\infty }\subset \lbrack 0,1]$ is bounded, it admits
convergent subsequences. Our next goal is to derive, subsequence by
subsequence, a mean-value identity satisfied by any such limit point.

\begin{claim}
\label{subseq-MV} Let $\{n_{k}\}_{k=1}^{\infty }$ be a strictly increasing
sequence of indices with $n_{k}\rightarrow \infty $. Suppose that 
\begin{equation}
c_{n_{k}}\longrightarrow c^{\ast }\in (0,1)  \label{508b}
\end{equation}%
along this subsequence. Then there exists a further subsequence (still
denoted $n_{k}$) and a strictly increasing continuous map $L^{\ast
}:[0,1]\rightarrow \lbrack 0,1]$ with inverse $T^{\ast }=(L^{\ast })^{-1}$
such that 
\begin{equation*}
L_{n_{k}-1}\rightarrow L^{\ast },\qquad T_{n_{k}-1}\rightarrow T^{\ast
}\quad \text{uniformly on }[0,1],
\end{equation*}%
and $c^{\ast }$ satisfies the mean-value equation 
\begin{equation}
\int_{0}^{c^{\ast }}w(T^{\ast }(u))\,du=c^{\ast }\int_{0}^{1}w(T^{\ast
}(u))\,du.  \label{509}
\end{equation}
\end{claim}

\begin{proof}
\noindent Fix a subsequence $\{n_{k}\}_{k=0}^{\infty }$ with $%
c_{n_{k}}\rightarrow c\in (0,1)$. Since $n_{k}\rightarrow \infty $, the
shifted indices $m_{k}=n_{k}-1$ also satisfy $m_{k}\rightarrow \infty $.

By \textit{Claim~\ref{AA}}, the family $(L_{n})_{n\geq 0}$ is relatively
compact in $C([0,1])$. Hence from the sequence $\{L_{m_{k}}\}_{k=1}^{\infty
} $ we can extract a further subsequence (not relabelled) such that 
\begin{equation}
L_{m_{k}}\rightarrow L^{\ast }\quad \text{uniformly on }[0,1]  \label{510}
\end{equation}%
for some continuous increasing $L^{\ast }:[0,1]\rightarrow \lbrack 0,1]$
with $L^{\ast }(0)=0$, $L^{\ast }(1)=1$. By \textit{Lemma~\ref{inverse}},
the inverses 
\begin{equation}
T_{m_{k}}=L_{m_{k}}^{-1}  \label{511}
\end{equation}%
then converge uniformly to $T^{\ast }=(L^{\ast })^{-1}$ on $[0,1]$.

For each $k$ we have $L_{n_{k}}(c_{n_{k}})=c_{n_{k}}$. By the
representation~(\ref{170}) applied with $n=m_{k}=n_{k}-1$, 
\begin{equation}
L_{n_{k}}(x)=\frac{\int_{0}^{x}w(T_{n_{k}-1}(u))\,du}{%
\int_{0}^{1}w(T_{n_{k}-1}(u))\,du}=\frac{\int_{0}^{x}w(T_{m_{k}}(u))\,du}{%
\int_{0}^{1}w(T_{m_{k}}(u))\,du},  \label{512}
\end{equation}%
where $T_{m_{k}}=L_{m_{k}}^{-1}$. Evaluating at $x=c_{n_{k}}$ and using $%
L_{n_{k}}(c_{n_{k}})=c_{n_{k}}$ gives 
\begin{equation}
c_{n_{k}}=\frac{\int_{0}^{c_{n_{k}}}w(T_{m_{k}}(u))\,du}{%
\int_{0}^{1}w(T_{m_{k}}(u))\,du}.  \label{513}
\end{equation}%
Equivalently, 
\begin{equation}
\int_{0}^{c_{n_{k}}}w(T_{m_{k}}(u))\,du=c_{n_{k}}\int_{0}^{1}w(T_{m_{k}}(u))%
\,du.  \label{514}
\end{equation}

By uniform convergence $T_{m_{k}}\rightarrow T^{\ast }$ and continuity of $w$%
, we have uniform convergence 
\begin{equation*}
w\circ T_{m_{k}}\rightarrow w\circ T^{\ast }\quad \text{on }[0,1].
\end{equation*}%
Thus, for any fixed $x\in \lbrack 0,1]$, 
\begin{equation}
\int_{0}^{x}w(T_{m_{k}}(u))\,du\longrightarrow \int_{0}^{x}w(T^{\ast
}(u))\,du,\quad \int_{0}^{1}w(T_{m_{k}}(u))\,du\longrightarrow
\int_{0}^{1}w(T^{\ast }(u))\,du.  \label{515}
\end{equation}%
Moreover, $c_{n_{k}}\rightarrow c$ by assumption.

We now pass to the limit in~(\ref{513}). For the left-hand side we use
continuity of 
\begin{equation}
x\mapsto \int_{0}^{x}w(T^{\ast }(u))\,du  \label{516}
\end{equation}%
and the convergence $c_{n_{k}}\rightarrow c^{\ast }$ together with~(\ref{515}%
). For the right-hand side, combine $c_{n_{k}}\rightarrow c$ with~(\ref{515}%
). The limit of~(\ref{514}) as $k\rightarrow \infty $ is therefore 
\begin{equation}
\int_{0}^{c^{\ast }}w(T^{\ast }(u))\,du=c^{\ast }\int_{0}^{1}w(T^{\ast
}(u))\,du,  \label{517}
\end{equation}%
which is exactly~(\ref{509}). \medskip
\end{proof}

\begin{remark}
Claim~\ref{subseq-MV} is purely subsequential: we fix one convergent
subsequence $\{c_{n_{k}}\}_{k=1}^{\infty }$, and from it we extract a
companion subsequence of $\{L_{n_{k}-1}\}_{k=1}^{\infty }$ with limit $%
L^{\ast }$. There is no claim that different subsequences of $%
\{c_{n}\}_{n=0}^{\infty }$ must share the same limit $c^{\ast }$, nor that
different subsequential limits $L^{\ast }$ coincide.
\end{remark}

To understand the mean-value equation~(\ref{509}), we analyze the shape of
the weight function 
\begin{equation}
w^{\ast }(u)=w(T^{\ast }(u)),\qquad u\in \lbrack 0,1].  \label{518}
\end{equation}

\begin{lemma}
\label{unimodal-subseq} Let $L^{\ast }$ and $T^{\ast }=(L^{\ast })^{-1}$ be
as in Lemma~\ref{subseq-MV}, and define 
\begin{equation}
w^{\ast }(u)=w(T^{\ast }(u))=T^{\ast }(u)(1-T^{\ast }(u)).  \label{519}
\end{equation}%
Then:

\begin{enumerate}
\item $w^{\ast }$ is continuous on $[0,1]$, satisfies $w^{\ast }(0)=w^{\ast
}(1)=0$, and is strictly positive on $(0,1)$;

\item there exists a unique $u_{0}\in (0,1)$ such that $T^{\ast }(u_{0})=%
\tfrac{1}{2}$ and 
\begin{equation}
w^{\ast }\text{ is strictly increasing on }[0,u_{0}],\quad w^{\ast }\text{
is strictly decreasing on }[u_{0},1].  \label{520}
\end{equation}%
In particular, $w^{\ast }$ is strictly unimodal on $[0,1]$ with a single
peak at $u_{0}$.
\end{enumerate}
\end{lemma}

\begin{proof}
Since $L^{\ast }$ is continuous, strictly increasing, and maps $[0,1]$ onto $%
[0,1]$, its inverse $T^{\ast }$ is also continuous, strictly increasing,
with $T^{\ast }(0)=0$ and $T^{\ast }(1)=1$. The function $w(y)=y(1-y)$ is
continuous on $[0,1]$, equals $0$ at $y=0$ and $y=1$, and is strictly
positive on $(0,1)$.

Thus $w^{\ast }=w\circ T^{\ast }$ is continuous on $[0,1]$, satisfies 
\begin{equation}
w^{\ast }(0)=w(0)=0,\quad w^{\ast }(1)=w(1)=0,  \label{521}
\end{equation}%
and is strictly positive on $(0,1)$ because $T^{\ast }(u)\in (0,1)$ for $%
u\in (0,1)$ and $w(y)>0$ for $y\in (0,1)$. This proves (1).

For (2), note that $T^{\ast }$ is strictly increasing and onto $[0,1]$, so
there exists a unique $u_{0}\in (0,1)$ such that $T^{\ast }(u_{0})=\tfrac{1}{%
2}$. Since $T^{\ast }$ is strictly increasing, we have 
\begin{equation}
T^{\ast }(u)<\tfrac{1}{2}\quad \text{for }u<u_{0},\qquad T^{\ast }(u)>\tfrac{%
1}{2}\quad \text{for }u>u_{0}.  \label{522}
\end{equation}%
The kernel $w$ is strictly increasing on $[0,\tfrac{1}{2}]$ and strictly
decreasing on $[\tfrac{1}{2},1]$. Therefore, for $u_{1}<u_{2}\leq u_{0}$ we
have 
\begin{equation}
T^{\ast }(u_{1})<T^{\ast }(u_{2})\leq \tfrac{1}{2}\ \Longrightarrow \
w(T^{\ast }(u_{1}))<w(T^{\ast }(u_{2})),  \label{850}
\end{equation}%
so $w^{\ast }$ is strictly increasing on $[0,u_{0}]$. Similarly, for $%
u_{0}\leq u_{1}<u_{2}$ we have 
\begin{equation}
\tfrac{1}{2}\leq T^{\ast }(u_{1})<T^{\ast }(u_{2})\ \Longrightarrow \
w(T^{\ast }(u_{1}))>w(T^{\ast }(u_{2})),  \label{851}
\end{equation}%
so $w^{\ast }$ is strictly decreasing on $[u_{0},1]$. Thus $w^{\ast }$ is
strictly unimodal with a unique maximum at $u_{0}$. \medskip
\end{proof}

With this shape information we can now analyze the mean-value equation~(\ref%
{517}) for a general strictly unimodal weight.

\begin{claim}
\label{MV-unique-subseq} Let $f:[0,1]\rightarrow \lbrack 0,\infty )$ be
continuous, strictly positive on $(0,1)$, and strictly unimodal in the sense
that there exists $u_{0}\in (0,1)$ such that $f$ is strictly increasing on $%
[0,u_{0}]$ and strictly decreasing on $[u_{0},1]$. Define 
\begin{equation}
F(c)=\int_{0}^{c}f(u)\,du-c\int_{0}^{1}f(u)\,du,\qquad 0\leq c\leq 1.
\label{852}
\end{equation}%
Then the equation 
\begin{equation}
\int_{0}^{c}f(u)\,du=c\int_{0}^{1}f(u)\,du  \label{853}
\end{equation}%
has at most one solution $c\in (0,1)$.
\end{claim}

\begin{proof}
The function $F$ defined in~(\ref{852}) is continuously differentiable, with 
\begin{equation}
F(0)=0,\quad F(1)=0,\qquad F^{\prime }(c)=f(c)-\int_{0}^{1}f(u)\,du=f(c)-C,
\label{854}
\end{equation}%
where $C=\int_{0}^{1}f(u)\,du>0$. Thus~(\ref{853}) is equivalent to $F(c)=0$.

By strict unimodality, the graph of $f$ increases up to a single maximum and
then decreases. Therefore the horizontal line $y=C$ can intersect the graph
of $f$ in at most two points in $(0,1)$. In other words, the equation 
\begin{equation}
f(c)=C  \label{855}
\end{equation}%
has at most two solutions in $(0,1)$. Equivalently, $F^{\prime }(c)=0$ has
at most two zeros in $(0,1)$.

Suppose, by contradiction, that $F(c)=0$ has two distinct solutions in $%
(0,1) $, say $0<c_{1}<c_{2}<1$, in addition to $c=0$ and $c=1$. By \textit{%
Rolle's Theorem}, there exist points 
\begin{equation}
c_{1}^{\prime }\in (0,c_{1}),\quad c_{2}^{\prime }\in (c_{1},c_{2}),\quad
c_{3}^{\prime }\in (c_{2},1)  \label{856}
\end{equation}%
such that 
\begin{equation}
F^{\prime }(c_{1}^{\prime })=F^{\prime }(c_{2}^{\prime })=F^{\prime
}(c_{3}^{\prime })=0.  \label{857}
\end{equation}%
Thus $F^{\prime }$ would have at least three distinct zeros in $(0,1)$,
contradicting the fact that $F^{\prime }$ can have at most two zeros. Hence $%
F(c)=0$ has at most one solution in $(0,1)$. \medskip
\end{proof}

We can now combine \textit{Claims~\ref{subseq-MV}, \ref{unimodal-subseq},
and \ref{MV-unique-subseq}} to characterise the subsequential limits of $%
\{c_{n}\}_{n=0}^{\infty }$.

\begin{corollary}
\label{subseq-c} Let $\{n_{k}\}_{k=0}^{\infty }$ be a strictly increasing
sequence with $n_{k}\rightarrow \infty $, such that $c_{n_{k}}\rightarrow
c^{\ast }\in (0,1)$. Then there exists a subsequential limit $T^{\ast }$ of
the inverses $\{T_{n_{k}-1}\}_{k=0}^{\infty }$ such that $c^{\ast }$ is the
unique solution in $(0,1)$ of 
\begin{equation}
\int_{0}^{c^{\ast }}w(T^{\ast }(u))\,du=c^{\ast }\int_{0}^{1}w(T^{\ast
}(u))\,du.  \label{858}
\end{equation}%
In particular, for this fixed subsequence and its associated limit $T^{\ast
} $, the interior crossing points $c_{n_{k}}$ cannot converge to two
different values.
\end{corollary}

\begin{proof}
By \textit{Claim~\ref{subseq-MV}}, there exists a further subsequence of $%
\{n_{k}\}_{k=0}^{\infty }$ and a limit $T^{\ast }$ such that~(\ref{517})
holds. By \textit{Lemma~\ref{unimodal-subseq}}, the function 
\begin{equation}
f(u)=w(T^{\ast }(u))  \label{859}
\end{equation}%
is continuous, strictly positive on $(0,1)$, and strictly unimodal.
Therefore \textit{Claim~\ref{MV-unique-subseq}} applies and shows that the
mean-value equation~(\ref{858}) has at most one solution $c^{\ast }\in (0,1)$%
.

Since $c_{n_{k}}\rightarrow c^{\ast }$ and~(\ref{517}) holds with this $c$,
we see that $c^{\ast }$ is indeed a solution of~(\ref{858}); by uniqueness,
no other limit in $(0,1)$ is possible for this subsequence and this $T^{\ast
}$. \medskip
\end{proof}

\begin{remark}
\textit{Corollary~\ref{subseq-c}} should be read carefully. It says: Fix a
convergent subsequence $\{c_{n_{k}}\}_{k=0}^{\infty }$ with limit $c^{\ast }$%
. \ Extract from $\{n_{k}\}_{k\geq 0}$ a further subsequence along which $%
T_{n_{k}-1}$ converges to some $T^{\ast }$. For this particular $T^{\ast }$,
the mean-value equation~(\ref{858}) has a unique interior solution, and the
limit of $\{c_{n_{k}}\}_{k=0}^{\infty }$ must be that solution.

What it does not claim is that:

\begin{itemize}
\item different subsequences $\{n_{k}\}_{k\geq 0}$ and $\{m_{\ell }\}_{l\geq
0}$ produce the same limit $T^{\ast }$; or

\item the corresponding mean-value points $c^{\ast }$ must coincide across
all subsequences.
\end{itemize}
\end{remark}

\begin{corollary}
\label{cross-limit} It holds%
\begin{equation*}
L^{\ast }(x)<x\ (x<c^{\ast }),\qquad L^{\ast }(x)>x\ (x>c^{\ast }).
\end{equation*}

as well as 
\begin{equation*}
T^{\ast }(x)>x\ (x<c^{\ast }),\qquad T^{\ast }(x)<x\ (x>c^{\ast }).
\end{equation*}
\end{corollary}

\begin{proof}
Fix $\varepsilon >0$ small. For $x\leq c^{\ast }-\varepsilon $ and $k$ large
we have $x<c_{n_{k}}$, hence by \textit{Lemma~\ref{l-cross}} we have $%
L_{n_{k}}(x)<x$; passing to the limit gives $L^{\ast }(x)\leq x$. Similarly,
for $x\geq c^{\ast }+\varepsilon $ and $k$ large we have $x>c_{n_{k}}$ and
thus $L_{n_{k}}(x)>x$, so $L^{\ast }(x)\geq x$. Letting $\varepsilon
\downarrow 0$ and using continuity of $L^{\ast }$ yields $L^{\ast }(x)<x$
for $x<c^{\ast }$ and $L^{\ast }(x)>x$ for $x>c^{\ast }$.
\end{proof}

\begin{remark}
At this stage of the argument, it is still logically possible that different
subsequences of $\{c_{n}\}_{n=0}^{\infty }$ converge to different limits $%
c^{\ast (1)},c^{\ast (2)},\dots $, each associated with its own
subsequential limit $T^{\ast (i)},i=1,2,...$ and its own mean-value
equation. The task of the subsequent analysis is precisely to exclude this
possibility and show that all such constants must in fact coincide.
\end{remark}

\begin{flushleft}
\textbf{Symmetry defect of }$L_{n}$
\end{flushleft}

In this section, we characterize the symmetry of $L_{n}$ with respect to $%
\frac{1}{2}$. We begin with several general observations before proceeding
to more complex derivations.

\begin{lemma}
\label{defect-inverse-sign} Let $L:[0,1]\rightarrow \lbrack 0,1]$ be a
continuous strictly increasing map with inverse $T=L^{-1}$. Define the
symmetry defects 
\begin{equation}
D_{L}(x)=L(x)+L(1-x)-1,\qquad x\in \lbrack 0,1],  \label{1200}
\end{equation}%
and 
\begin{equation}
D_{T}(u)=T(u)+T(1-u)-1,\qquad u\in \lbrack 0,1].  \label{1201}
\end{equation}%
Then the following implications hold 
\begin{equation}
D_{L}(x)\leq 0\ \forall x\in \lbrack 0,1]\quad \Longrightarrow \quad
D_{T}(u)\geq 0\ \forall u\in \lbrack 0,1],  \label{1202}
\end{equation}%
and 
\begin{equation}
D_{L}(x)\geq 0\ \forall x\in \lbrack 0,1]\quad \Longrightarrow \quad
D_{T}(u)\leq 0\ \forall u\in \lbrack 0,1].  \label{1203}
\end{equation}%
In particular, whenever $D_{L}$ has a constant sign on $[0,1]$, $D_{T}$ has
the opposite constant sign on $[0,1]$.
\end{lemma}

\begin{proof}
\noindent Assume first that 
\begin{equation}
L(x)+L(1-x)\leq 1\qquad \forall x\in \lbrack 0,1].  \label{1204}
\end{equation}%
Fix $u\in \lbrack 0,1]$ and set $x=T(u)$, so that $u=L(x)$. Then (\ref{1204}%
) gives 
\begin{equation}
L(1-x)\leq 1-L(x)=1-u.  \label{1205}
\end{equation}%
Since $T$ is increasing, 
\begin{equation}
1-x=T(L(1-x))\leq T(1-u).  \label{1206}
\end{equation}%
Therefore 
\begin{equation}
T(u)+T(1-u)=x+T(1-u)\geq x+(1-x)=1,  \label{1207}
\end{equation}%
that is, 
\begin{equation}
D_{T}(u)=T(u)+T(1-u)-1\geq 0.  \label{1208}
\end{equation}%
This proves (\ref{1202}). The proof of (\ref{1203}) is identical, with all
inequalities reversed.
\end{proof}

Going back to the usual indexed notation we have the following:\hfill

\begin{lemma}
\label{defect-shape} Assume that the initial defect 
\begin{equation}
D_{0}(x)=L_{0}(x)+L_{0}(1-x)-1  \label{1209}
\end{equation}%
has a constant sign on $[0,1]$. Then, for every $n\geq 0$, the defect 
\begin{equation}
D_{n}(x)=L_{n}(x)+L_{n}(1-x)-1,\qquad x\in \lbrack 0,1],  \label{1210}
\end{equation}%
has a constant sign on $[0,1]$, the sign alternates with $n$, and $|D_{n}|$
is nondecreasing on $[0,\tfrac{1}{2}]$.
\end{lemma}

\begin{proof}
\noindent For brevity write 
\begin{equation}
S_{n}(x)=T_{n}(x)+T_{n}(1-x)-1,\qquad x\in \lbrack 0,1],  \label{1211}
\end{equation}%
where again $T_{n}=L_{n}^{-1}$. By \textit{Lemma~\ref{defect-inverse-sign}}, 
$S_{n}$ has the opposite constant sign to $D_{n}$.

Differentiating 
\begin{equation}
D_{n+1}(x)=L_{n+1}(x)+L_{n+1}(1-x)-1  \label{1212}
\end{equation}%
and using 
\begin{equation}
L_{n+1}^{\prime }(x)=\frac{w(T_{n}(x))}{\int_{0}^{1}w(T_{n}(u))\,du},\qquad
w(y)=y(1-y),  \label{1213}
\end{equation}%
we obtain 
\begin{align}
D_{n+1}^{\prime }(x)& =L_{n+1}^{\prime }(x)-L_{n+1}^{\prime }(1-x)
\label{1214} \\
& =\frac{w(T_{n}(x))-w(T_{n}(1-x))}{\int_{0}^{1}w(T_{n}(u))\,du}.  \notag
\end{align}%
Since 
\begin{equation}
w(a)-w(b)=(a-b)(1-a-b),  \label{1215}
\end{equation}%
we get 
\begin{equation}
D_{n+1}^{\prime }(x)=\frac{(T_{n}(1-x)-T_{n}(x))(T_{n}(x)+T_{n}(1-x)-1)}{%
\int_{0}^{1}w(T_{n}(u))\,du}.  \label{1216}
\end{equation}%
Hence, for $x\in \lbrack 0,\tfrac{1}{2}]$, because $T_{n}$ is increasing, 
\begin{equation}
T_{n}(1-x)-T_{n}(x)\geq 0,  \label{1217}
\end{equation}%
and therefore 
\begin{equation}
\func{sgn}(D_{n+1}^{\prime }(x))=\func{sgn}(S_{n}(x)),\qquad x\in \lbrack 0,%
\tfrac{1}{2}].  \label{1218}
\end{equation}

Now assume inductively that $D_{n}$ has a constant sign on $[0,1]$. Then $%
S_{n}$ has the opposite constant sign on $[0,1]$, so by (\ref{1218}) the
derivative $D_{n+1}^{\prime }$ has that opposite sign on $[0,\tfrac{1}{2}]$.
Since 
\begin{equation}
D_{n+1}(0)=L_{n+1}(0)+L_{n+1}(1)-1=0,  \label{1219}
\end{equation}

for every $x\in \lbrack 0,\tfrac{1}{2}]$ we have 
\begin{equation}
D_{n+1}(x)=D_{n+1}(0)+\int_{0}^{x}D_{n+1}^{\prime
}(t)\,dt=\int_{0}^{x}D_{n+1}^{\prime }(t)\,dt.  \label{1219.1}
\end{equation}%
Thus, if $D_{n+1}^{\prime }$ has a constant sign on $[0,\tfrac{1}{2}]$, then 
$D_{n+1}$ has the same constant sign on $[0,\tfrac{1}{2}]$. By symmetry, 
\begin{equation}
D_{n+1}(1-x)=D_{n+1}(x),\qquad x\in \lbrack 0,1],  \label{1220}
\end{equation}%
hence $D_{n+1}$ has that constant sign on all of $[0,1]$. This proves that
the sign alternates.

Moreover, on $[0,\tfrac{1}{2}]$, the sign-corrected function 
\begin{equation}
\varepsilon _{n}\,D_{n}(x),\qquad \varepsilon _{n}=-\func{sgn}(D_{n}(x))\in
\{-1,+1\},  \label{1221}
\end{equation}%
has nonnegative derivative by (\ref{1218}). Hence $\varepsilon _{n}D_{n}$ is
nondecreasing on $[0,\tfrac{1}{2}]$, i.e. 
\begin{equation}
|D_{n}(x)|\ \text{is nondecreasing on }[0,\tfrac{1}{2}].  \label{1222}
\end{equation}%
The induction is complete. \hfill
\end{proof}

\begin{claim}
\label{An-monotone} Assume the hypotheses of \textit{Lemma~\ref{defect-shape}%
}. Define 
\begin{equation}
A_{n}=\int_{0}^{1}|D_{n}(x)|\,dx,  \label{1223}
\end{equation}%
and let 
\begin{equation}
E_{n}=\int_{0}^{1}(1-L_{n}(x))\,dx  \label{1224}
\end{equation}%
be the expectation of the random variable with distribution function $L_{n}$%
. Then 
\begin{equation}
A_{n}=|1-2E_{n}|  \label{1225}
\end{equation}%
for every $n$, and the sequence $\{A_{n}\}_{n\geq 0}$ is nonincreasing: 
\begin{equation}
A_{n+1}\leq A_{n}\qquad \text{for all }n\geq 0.  \label{1226}
\end{equation}
\end{claim}

\begin{proof}
Since $D_{n}$ has a constant sign on $[0,1]$ by \textit{Lemma~\ref%
{defect-shape}}, we have 
\begin{equation}
A_{n}=\left\vert \int_{0}^{1}D_{n}(x)\,dx\right\vert .  \label{1227}
\end{equation}%
Also, 
\begin{align}
\int_{0}^{1}D_{n}(x)\,dx& =\int_{0}^{1}(L_{n}(x)+L_{n}(1-x)-1)\,dx  \notag \\
& =2\int_{0}^{1}L_{n}(x)\,dx-1  \notag \\
& =1-2E_{n},  \label{1228}
\end{align}%
because 
\begin{equation}
E_{n}=\int_{0}^{1}(1-L_{n}(x))\,dx.  \label{1229}
\end{equation}%
This proves (\ref{1225}).

We now prove (\ref{1226}). For brevity write 
\begin{equation}
L=L_{n},\qquad T=T_{n},\qquad g(x)=1-L(x),\qquad E=E_{n}.  \label{1230}
\end{equation}%
Also set as before 
\begin{equation}
I=\int_{0}^{1}w(T(u))\,du.  \label{1231}
\end{equation}%
From the iterative equation defining $L_{n+1}$, 
\begin{equation}
L_{n+1}(x)=\frac{\int_{0}^{x}w(T(u))\,du}{I},  \label{1232}
\end{equation}

we have%
\begin{align}
E_{n+1}& =\int_{0}^{1}(1-L_{n+1}(x))\,dx=1-\int_{0}^{1}L_{n+1}(x)\,dx  \notag
\\
& =1-\frac{1}{I}\int_{0}^{1}\left( \int_{0}^{x}w(T(u))\,du\right) dx  \notag
\\
& =1-\frac{1}{I}\int_{0}^{1}\left( \int_{u}^{1}dx\right) w(T(u))\,du  \notag
\\
& =1-\frac{1}{I}\int_{0}^{1}(1-u)w(T(u))\,du.  \label{1233}
\end{align}%
Therefore 
\begin{equation}
I\,(E+E_{n+1}-1)=IE-\int_{0}^{1}(1-u)w(T(u))\,du.  \label{1234}
\end{equation}

We next rewrite $I$ and the numerator in (\ref{1233}) in terms of $g$. Using
the change of variables $u=L(x)$, so that $du=L^{\prime }(x)\,dx$ and $%
T(u)=x $, we obtain 
\begin{equation}
I=\int_{0}^{1}x(1-x)L^{\prime }(x)\,dx=\int_{0}^{1}(1-2x)g(x)\,dx.
\label{1235}
\end{equation}%
Similarly, 
\begin{align}
\int_{0}^{1}(1-u)w(T(u))\,du& =\int_{0}^{1}g(x)\,x(1-x)L^{\prime }(x)\,dx 
\notag \\
& =\frac{1}{2}\int_{0}^{1}(1-2x)g(x)^{2}\,dx.  \label{1236}
\end{align}

Now define on $[0,\tfrac{1}{2}]$ 
\begin{equation}
r(x)=g(x)+g(1-x)-1=-D_{n}(x),  \label{1237}
\end{equation}%
\begin{equation}
d(x)=g(x)-g(1-x),  \label{1238}
\end{equation}%
and 
\begin{equation}
p(x)=(1-2x)\,d(x).  \label{1239}
\end{equation}%
Since $g$ is decreasing, $d(x)\geq 0$ on $[0,\tfrac{1}{2}]$. Moreover, if $%
0\leq x\leq y\leq \tfrac{1}{2}$, then 
\begin{equation}
d(x)-d(y)=(g(x)-g(y))+(g(1-y)-g(1-x))\geq 0,  \label{1240}
\end{equation}%
so $d$ is nonincreasing on $[0,\tfrac{1}{2}]$. Therefore $p$ is nonnegative
and nonincreasing on $[0,\tfrac{1}{2}]$.

Pairing $x$ with $1-x$ in (\ref{1235}) gives 
\begin{equation}
I=\int_{0}^{1/2}p(x)\,dx.  \label{1241}
\end{equation}%
Likewise, pairing $x$ with $1-x$ in (\ref{1236}) yields 
\begin{align}
\int_{0}^{1}(1-u)w(T(u))\,du& =\frac{1}{2}%
\int_{0}^{1/2}(1-2x)(g(x)^{2}-g(1-x)^{2})\,dx  \notag \\
& =\frac{1}{2}\int_{0}^{1/2}p(x)(1+r(x))\,dx.  \label{1242}
\end{align}%
Furthermore, 
\begin{equation}
E-\frac{1}{2}=\int_{0}^{1/2}r(x)\,dx.  \label{1243}
\end{equation}

We now substitute (\ref{1241}), (\ref{1242}) and the identity 
\begin{equation}
E=\frac{1}{2}+\int_{0}^{1/2}r(x)\,dx  \label{1244}
\end{equation}%
into (\ref{1234}). This gives 
\begin{align}
I\,(E+E_{n+1}-1)& =I\left( \frac{1}{2}+\int_{0}^{1/2}r(x)\,dx\right) -\frac{1%
}{2}\int_{0}^{1/2}p(x)(1+r(x))\,dx  \notag \\
& =\frac{1}{2}I+I\int_{0}^{1/2}r(x)\,dx-\frac{1}{2}\int_{0}^{1/2}p(x)\,dx-%
\frac{1}{2}\int_{0}^{1/2}p(x)r(x)\,dx.  \label{1245}
\end{align}%
Using (\ref{1241}), the two scalar terms cancel 
\begin{equation}
\frac{1}{2}I-\frac{1}{2}\int_{0}^{1/2}p(x)\,dx=0.  \label{1246}
\end{equation}%
Hence 
\begin{equation}
I\,(E+E_{n+1}-1)=\left( \int_{0}^{1/2}r(x)\,dx\right) \left(
\int_{0}^{1/2}p(x)\,dx\right) -\frac{1}{2}\int_{0}^{1/2}p(x)r(x)\,dx.
\label{1247}
\end{equation}

Let 
\begin{equation}
\sigma =\func{sgn}\!\left( E-\frac{1}{2}\right) .  \label{1248}
\end{equation}

Since $r=-D_{n}$ and $D_{n}$ has a constant sign, $r$ also has a constant
sign on $[0,\tfrac{1}{2}]$. Moreover, 
\begin{equation}
E-\frac{1}{2}=\int_{0}^{1/2}r(x)\,dx.  \label{1248.1}
\end{equation}%
Hence, if $E\neq \tfrac{1}{2}$, the sign of $r$ coincides with 
\begin{equation}
\sigma =\func{sgn}\!\left( E-\frac{1}{2}\right) ,  \label{1248.2}
\end{equation}%
and therefore 
\begin{equation}
\sigma r(x)=|r(x)|\qquad \text{for all }x\in \lbrack 0,\tfrac{1}{2}].
\label{1248.3}
\end{equation}%
If $E=\tfrac{1}{2}$, then $\int_{0}^{1/2}r(x)\,dx=0$; since $r$ has a
constant sign, this implies $r\equiv 0$, and the argument is trivial. Thus
the function 
\begin{equation}
q(x)=\sigma r(x)=|r(x)|  \label{1249}
\end{equation}%
is nonnegative on $[0,\tfrac{1}{2}]$. By \textit{Lemma~\ref{defect-shape}}, $%
|D_{n}|$ is nondecreasing on $[0,\tfrac{1}{2}]$, hence $q=|r|=|D_{n}|$ is
also nondecreasing on $[0,\tfrac{1}{2}]$.

Multiplying (\ref{1247}) by $\sigma $ we obtain 
\begin{equation}
I\,\sigma (E+E_{n+1}-1)=\left( \int_{0}^{1/2}q(x)\,dx\right) \left(
\int_{0}^{1/2}p(x)\,dx\right) -\frac{1}{2}\int_{0}^{1/2}p(x)q(x)\,dx.
\label{1250}
\end{equation}

Now $p$ is nonnegative and nonincreasing on $[0,\tfrac{1}{2}]$, whereas $q$
is nonnegative and nondecreasing on $[0,\tfrac{1}{2}]$. Hence the \textit{%
reverse Chebyshev inequality }(see \cite{[66]} and \cite[Chapter 9]{[32]})
on the interval $[0,\tfrac{1}{2}]$ yields 
\begin{equation}
\int_{0}^{1/2}p(x)q(x)\,dx\leq 2\left( \int_{0}^{1/2}p(x)\,dx\right) \left(
\int_{0}^{1/2}q(x)\,dx\right) .  \label{1251}
\end{equation}%
Substituting (\ref{1251}) into (\ref{1250}) gives 
\begin{equation}
I\,\sigma (E+E_{n+1}-1)\geq 0.  \label{1252}
\end{equation}%
Since $I>0$, we conclude that 
\begin{equation}
\left( E_{n}-\frac{1}{2}\right) \left( E_{n}+E_{n+1}-1\right) \geq 0.
\label{1253}
\end{equation}

Finally, by \textit{Lemma~\ref{defect-shape}}, the sign of $D_{n}$
alternates with $n$, hence the sign of $r=-D_{n}$, and therefore also the
sign of $E_{n}-\frac{1}{2}$, alternates with $n$. Thus 
\begin{equation}
|E_{n}-\tfrac{1}{2}|-|E_{n+1}-\tfrac{1}{2}|=\sigma \left(
E_{n}+E_{n+1}-1\right) \geq 0.  \label{1254}
\end{equation}%
Therefore 
\begin{equation}
|E_{n+1}-\tfrac{1}{2}|\leq |E_{n}-\tfrac{1}{2}|,  \label{1255}
\end{equation}%
and by (\ref{1227}), 
\begin{equation}
A_{n+1}=2|E_{n+1}-\tfrac{1}{2}|\leq 2|E_{n}-\tfrac{1}{2}|=A_{n}.
\label{1256}
\end{equation}%
This proves (\ref{1226}). \hfill
\end{proof}

\begin{corollary}
\label{An-strict} Assume the hypotheses of \textit{Claim~\ref{An-monotone}}.
Fix $n\geq 0$ and define, on $[0,\tfrac{1}{2}]$, 
\begin{equation}
p_{n}(x)=(1-2x)(L_{n}(1-x)-L_{n}(x)),  \label{1257}
\end{equation}%
and 
\begin{equation}
q_{n}(x)=|D_{n}(x)|=|L_{n}(x)+L_{n}(1-x)-1|.  \label{1258}
\end{equation}%
If

\begin{enumerate}
\item $A_{n}>0$,

\item $p_{n}$ is strictly decreasing on $[0,\tfrac{1}{2}]$,

\item $q_{n}$ is strictly increasing on $[0,\tfrac{1}{2}]$,
\end{enumerate}

then 
\begin{equation}
\left( E_{n}-\frac{1}{2}\right) \left( E_{n}+E_{n+1}-1\right) >0,
\label{1259}
\end{equation}%
and consequently 
\begin{equation}
A_{n+1}<A_{n}.  \label{1260}
\end{equation}%
In particular, if the above three conditions hold for every index $n$ with $%
A_{n}>0$, then $\{A_{n}\}_{n\geq 0}$ is strictly decreasing until it reaches 
$0$.
\end{corollary}

\begin{proof}
\noindent We use the notation introduced in the proof of \textit{Claim~\ref%
{An-monotone}}. There we obtained 
\begin{equation}
I\,\sigma (E_{n}+E_{n+1}-1)=\left( \int_{0}^{1/2}q_{n}(x)\,dx\right) \left(
\int_{0}^{1/2}p_{n}(x)\,dx\right) -\frac{1}{2}\int_{0}^{1/2}p_{n}(x)q_{n}(x)%
\,dx,  \label{1261}
\end{equation}%
where 
\begin{equation}
I=\int_{0}^{1}w(T_{n}(u))\,du>0,\qquad \sigma =\func{sgn}\!\left( E_{n}-%
\frac{1}{2}\right) .  \label{1262}
\end{equation}

By assumption, $p_{n}$ is nonnegative and strictly decreasing on $[0,\tfrac{1%
}{2}]$, while $q_{n}$ is nonnegative and strictly increasing on $[0,\tfrac{1%
}{2}]$. Moreover, $A_{n}>0$ implies 
\begin{equation}
\int_{0}^{1/2}q_{n}(x)\,dx>0,  \label{1263}
\end{equation}%
so $q_{n}$ is not identically zero.

Hence the \textit{reverse Chebyshev inequality} is strict 
\begin{equation}
\int_{0}^{1/2}p_{n}(x)q_{n}(x)\,dx<2\left( \int_{0}^{1/2}p_{n}(x)\,dx\right)
\left( \int_{0}^{1/2}q_{n}(x)\,dx\right) .  \label{1264}
\end{equation}%
Substituting (\ref{1264}) into (\ref{1261}) yields 
\begin{equation}
I\,\sigma (E_{n}+E_{n+1}-1)>0.  \label{1265}
\end{equation}%
Since $I>0$, we conclude that 
\begin{equation}
\left( E_{n}-\frac{1}{2}\right) \left( E_{n}+E_{n+1}-1\right) >0.
\label{1266}
\end{equation}

Finally, as in the proof of \textit{Claim~\ref{An-monotone}}, the sign of $%
E_{n}-\tfrac{1}{2}$ alternates with $n$, and 
\begin{equation}
A_{n}=2\left\vert E_{n}-\frac{1}{2}\right\vert .  \label{1267}
\end{equation}%
Therefore (\ref{1266}) is equivalent to 
\begin{equation}
\left\vert E_{n+1}-\frac{1}{2}\right\vert <\left\vert E_{n}-\frac{1}{2}%
\right\vert ,  \label{1268}
\end{equation}%
which in turn gives 
\begin{equation}
A_{n+1}<A_{n}.  \label{1269}
\end{equation}%
This proves the claim. \hfill
\end{proof}

\begin{remark}
\label{strict-mon-remark} The additional assumptions in \textit{Corollary~%
\ref{An-strict}} are natural from the point of view of the geometry of the
defect. Indeed, on $[0,\tfrac{1}{2}]$ the quantity 
\begin{equation}
q_{n}(x)=|D_{n}(x)|=|L_{n}(x)+L_{n}(1-x)-1|  \label{1270}
\end{equation}%
measures the magnitude of the defect, and the inductive sign pattern proved
in \textit{Lemma~\ref{defect-shape}} shows that $q_{n}$ is nondecreasing on $%
[0,\tfrac{1}{2}]$. Thus the additional strict monotonicity assumption on $%
q_{n}$ simply excludes the degenerate situation in which the defect remains
locally flat on a nontrivial subinterval.

Likewise, 
\begin{equation}
p_{n}(x)=(1-2x)(L_{n}(1-x)-L_{n}(x)),\qquad x\in \lbrack 0,\tfrac{1}{2}],
\label{1271}
\end{equation}%
is the geometric weight that appears after symmetrizing the identities for $%
I_{n}$ and $E_{n+1}$ on the half-interval. The factor $1-2x$ decreases
strictly from $1$ to $0$, while $L_{n}(1-x)-L_{n}(x)\geq 0$ measures the
separation between the two symmetric branches of $L_{n}$. Hence the
condition that $p_{n}$ be strictly decreasing means that this weighted
branch-separation does not contain a plateau compensating the decay of $1-2x$%
.

Under these two strict half-interval shape conditions, the reverse \textit{%
Chebyshev inequality} used in the proof \textit{Claim~\ref{An-monotone}}
becomes strict, and therefore the decay of $A_{n}$ is strict whenever $%
A_{n}>0$. Numerically, these hypotheses are precisely what is observed in
the examples considered in this appendix: the defect profile $q_{n}$ grows
strictly from the endpoint $0$ to its extremal value at $x=\tfrac{1}{2}$,
whereas the weighted branch-separation $p_{n}$ decays strictly to $0$ as $%
x\uparrow \tfrac{1}{2}$.
\end{remark}

\begin{flushleft}
\textbf{Vanishing of the symmetry defect of }$L_{n}$\textbf{\ in the
one-sign regime }
\end{flushleft}

We now demonstrate the convergence of the symmetry defect to zero, given
certain assumptions to be detailed in the next section. This convergence
ensures the symmetry of the subsequential limits about $\frac{1}{2}$.

\begin{claim}
\label{An-vanishing} Assume that there exists an index $N\geq 0$ such that,
for every $n\geq N$, the defect 
\begin{equation}
D_{n}(x)=L_{n}(x)+L_{n}(1-x)-1  \label{1271.1}
\end{equation}%
has a constant sign on $[0,1]$. Then 
\begin{equation}
A_{n}=\int_{0}^{1}|D_{n}(x)|\,dx\longrightarrow 0\qquad \text{as }%
n\rightarrow \infty .  \label{1272}
\end{equation}
\end{claim}

\begin{proof}
\noindent Fix $n\geq N$, and define as before on $[0,\tfrac{1}{2}]$ 
\begin{equation}
q_{n}(x)=|D_{n}(x)|,  \label{1273}
\end{equation}%
\begin{equation}
d_{n}(x)=L_{n}(1-x)-L_{n}(x),  \label{1274}
\end{equation}%
and 
\begin{equation}
p_{n}(x)=(1-2x)\,d_{n}(x).  \label{1275}
\end{equation}

Since $D_{n}$ has a constant sign for every $n\geq N$, the same argument as
in \textit{Lemma~\ref{defect-shape}} shows that $q_{n}=|D_{n}|$ is
nonnegative and nondecreasing on $[0,\tfrac{1}{2}]$ for every $n\geq N$. On
the other hand, the proof of \textit{Claim~\ref{An-monotone}} shows that $%
d_{n}$ is nonnegative and nonincreasing on $[0,\tfrac{1}{2}]$, hence so is $%
p_{n}$.

Moreover, 
\begin{equation}
A_{n}=2\int_{0}^{1/2}q_{n}(x)\,dx.  \label{1276}
\end{equation}

We now extract from the proof of \textit{Claim~\ref{An-monotone}} an exact
representation of $A_{n+1}$ as a weighted average of $q_{n}$. Indeed, for $%
n\geq N$, equations (\ref{1261}) and (\ref{1225}) give 
\begin{equation}
I_{n}\,\sigma _{n}(E_{n}+E_{n+1}-1)=\left( \int_{0}^{1/2}q_{n}(x)\,dx\right)
\left( \int_{0}^{1/2}p_{n}(x)\,dx\right) -\frac{1}{2}%
\int_{0}^{1/2}p_{n}(x)q_{n}(x)\,dx,  \label{1277}
\end{equation}%
where 
\begin{equation}
I_{n}=\int_{0}^{1}w(T_{n}(u))\,du=\int_{0}^{1/2}p_{n}(x)\,dx>0  \label{1278}
\end{equation}%
and 
\begin{equation}
\sigma _{n}=\func{sgn}\!\left( E_{n}-\frac{1}{2}\right) .  \label{1279}
\end{equation}%
Since $D_{n}$ has a constant sign and the sign alternates with $n$, 
\begin{equation}
\sigma _{n}(E_{n}+E_{n+1}-1)=\left\vert E_{n}-\frac{1}{2}\right\vert
-\left\vert E_{n+1}-\frac{1}{2}\right\vert =\frac{A_{n}-A_{n+1}}{2}.
\label{1280}
\end{equation}%
Using also 
\begin{equation}
\int_{0}^{1/2}q_{n}(x)\,dx=\frac{A_{n}}{2},  \label{1281}
\end{equation}%
\begin{equation}
A_{n+1}=\frac{\int_{0}^{1/2}p_{n}(x)q_{n}(x)\,dx}{\int_{0}^{1/2}p_{n}(x)\,dx}%
,\qquad n\geq N.  \label{1282}
\end{equation}

Next, introduce the two probability measures on $[0,\tfrac{1}{2}]$ 
\begin{equation}
d\nu (x)=2\,dx,  \label{1283}
\end{equation}%
and 
\begin{equation}
d\lambda (x)=4(1-2x)\,dx.  \label{1284}
\end{equation}%
Then (\ref{1276}) becomes 
\begin{equation}
A_{n}=\int_{0}^{1/2}q_{n}\,d\nu .  \label{1285}
\end{equation}

We claim that 
\begin{equation}
A_{n+1}\leq \int_{0}^{1/2}q_{n}\,d\lambda \leq A_{n},\qquad n\geq N.
\label{1286}
\end{equation}

For the first inequality, apply the same \textit{reverse Chebyshev} argument
as in the proof of \textit{Claim~\ref{An-monotone}} on $[0,\tfrac{1}{2}]$ to
the nonincreasing function $d_{n}$ and the nondecreasing function $q_{n}$
with respect to the base measure $(1-2x)\,dx$ to get 
\begin{equation}
\frac{\int_{0}^{1/2}(1-2x)d_{n}(x)q_{n}(x)\,dx}{\int_{0}^{1/2}(1-2x)d_{n}(x)%
\,dx}\leq \frac{\int_{0}^{1/2}(1-2x)q_{n}(x)\,dx}{\int_{0}^{1/2}(1-2x)\,dx}.
\label{1287}
\end{equation}%
By (\ref{1275}) and (\ref{1282}), this is exactly 
\begin{equation}
A_{n+1}\leq 4\int_{0}^{1/2}(1-2x)q_{n}(x)\,dx=\int_{0}^{1/2}q_{n}\,d\lambda .
\label{1288}
\end{equation}%
For the second inequality, observe that the distribution functions of $\nu $
and $\lambda $ are 
\begin{equation}
F_{\nu }(x)=2x,\qquad F_{\lambda }(x)=4x-4x^{2},\qquad 0\leq x\leq \frac{1}{2%
},  \label{1289}
\end{equation}%
and 
\begin{equation}
F_{\lambda }(x)-F_{\nu }(x)=2x(1-2x)\geq 0.  \label{1290}
\end{equation}%
Thus $\lambda $ is stochastically smaller than $\nu $, and because $q_{n}$
is nondecreasing, 
\begin{equation}
\int_{0}^{1/2}q_{n}\,d\lambda \leq \int_{0}^{1/2}q_{n}\,d\nu =A_{n}.
\label{1291}
\end{equation}%
This proves (\ref{1286}).

By the same argument as in \textit{Claim~\ref{An-monotone}}, applied from
the index $N$ onward, the tail $\{A_{n}\}_{n\geq N}$ is nonincreasing. Hence
it converges: 
\begin{equation*}
A_{n}\downarrow a\qquad \text{for some }a\geq 0.
\end{equation*}%
By (\ref{1286}), 
\begin{equation}
\int_{0}^{1/2}q_{n}\,d\lambda \longrightarrow a.  \label{1292}
\end{equation}

Now the sequence $\{q_{n}\}_{n\geq N}$ consists of bounded nondecreasing
functions on $[0,\tfrac{1}{2}]$, with $0\leq q_{n}\leq 1$. By \textit{%
Helly's selection principle }there exists a subsequence $\{q_{n_{k}}\}_{k%
\geq 1}$ and a bounded nondecreasing function $q$ on $[0,\tfrac{1}{2}]$ such
that 
\begin{equation*}
q_{n_{k}}(x)\longrightarrow q(x)
\end{equation*}%
for every continuity point of $q$, hence almost everywhere on $[0,\tfrac{1}{2%
}]$. By dominated convergence applied to (\ref{1285}) and (\ref{1292}), 
\begin{equation}
\int_{0}^{1/2}q\,d\nu =a,\qquad \int_{0}^{1/2}q\,d\lambda =a.  \label{1293}
\end{equation}%
Therefore 
\begin{equation}
\int_{0}^{1/2}q\,d\nu -\int_{0}^{1/2}q\,d\lambda =0.  \label{1294}
\end{equation}

Since $q$ is nondecreasing, its distributional derivative $dq$ is a
nonnegative \textit{Stieltjes measure}. Integration by parts for monotone
functions yields 
\begin{equation}
\int_{0}^{1/2}q\,d\nu -\int_{0}^{1/2}q\,d\lambda =\int_{[0,1/2]}(F_{\lambda
}(x)-F_{\nu }(x))\,dq(x).  \label{1295}
\end{equation}%
But 
\begin{equation}
F_{\lambda }(x)-F_{\nu }(x)=2x(1-2x)>0\qquad \text{for }0<x<\frac{1}{2},
\label{1296}
\end{equation}%
and $dq\geq 0$, so (\ref{1294}) implies 
\begin{equation}
dq=0\qquad \text{on }(0,\tfrac{1}{2}).  \label{1296.1}
\end{equation}%
Hence $q$ is constant on $[0,\tfrac{1}{2}]$. Since 
\begin{equation}
q_{n}(0)=|D_{n}(0)|=0\qquad \text{for every }n,  \label{1296.2}
\end{equation}%
we have $q(0)=0$, and therefore 
\begin{equation}
q\equiv 0.  \label{196.21}
\end{equation}%
Thus (\ref{1293}) gives $a=0$.

We have shown that every convergent subsequence of $\{A_{n}\}_{n>N}$ has
limit $0$. Since $\{A_{n}\}_{n>N}$ itself converges, its limit must be $0$,
i.e. 
\begin{equation}
A_{n}\longrightarrow 0.  \label{196.22}
\end{equation}%
This proves (\ref{1272}). \hfill
\end{proof}

\begin{corollary}
\label{subseq-symmetric} Assume the hypotheses of \textit{Claim~\ref%
{An-vanishing}}. Let $\{n_{k}\}_{k\geq 1}$ be any sequence of integers with $%
n_{k}\rightarrow \infty $, and suppose that 
\begin{equation}
L_{n_{k}}(x)\longrightarrow L^{\ast }(x)  \label{1296.3}
\end{equation}%
uniformly on $[0,1]$. Then 
\begin{equation}
L^{\ast }(x)+L^{\ast }(1-x)-1=0\qquad \text{for every }x\in \lbrack 0,1].
\label{1298}
\end{equation}

In particular, under any of the sufficient hypotheses on the starting law $F$
given earlier in this appendix for which the one-sign regime holds from the
start, Claim~\ref{An-vanishing} applies with $N=0$. Hence, by the
compactness result \textit{Claim~\ref{AA}}, every subsequential limit of $%
\{L_{n}\}_{n\geq 0}$ in $C([0,1])$ is symmetric.
\end{corollary}

\begin{proof}
By \textit{Claim~\ref{An-vanishing}}, 
\begin{equation}
A_{n}=\int_{0}^{1}|D_{n}(x)|\,dx\longrightarrow 0.  \label{1299}
\end{equation}%
Along the subsequence $(n_{k})$, define 
\begin{equation}
D_{n_{k}}(x)=L_{n_{k}}(x)+L_{n_{k}}(1-x)-1.  \label{1300}
\end{equation}%
Since $L_{n_{k}}\rightarrow L^{\ast }$ uniformly on $[0,1]$, we also have 
\begin{equation}
D_{n_{k}}\longrightarrow D^{\ast }  \label{1301}
\end{equation}%
uniformly on $[0,1]$, where 
\begin{equation}
D^{\ast }(x)=L^{\ast }(x)+L^{\ast }(1-x)-1.  \label{1302}
\end{equation}%
Therefore 
\begin{equation}
\int_{0}^{1}|D^{\ast }(x)|\,dx=\lim_{k\rightarrow \infty
}\int_{0}^{1}|D_{n_{k}}(x)|\,dx=\lim_{k\rightarrow \infty }A_{n_{k}}=0.
\label{1303}
\end{equation}%
Since $D^{\ast }$ is continuous on $[0,1]$, it follows that 
\begin{equation}
D^{\ast }(x)=0\qquad \forall x\in \lbrack 0,1].  \label{1303.1}
\end{equation}%
This proves (\ref{1298}). \hfill
\end{proof}

\begin{corollary}
\label{crossings-half} Assume the hypotheses of \textit{Claim~\ref%
{An-vanishing}}. Assume moreover the earlier one-crossing framework: for
each $n$, $L_{n}$ has a unique crossing point $c_{n}\in (0,1)$ with the
diagonal, 
\begin{equation}
L_{n}(c_{n})=c_{n},  \label{1304}
\end{equation}%
Then every cluster point of $(c_{n})$ is equal to $\tfrac{1}{2}$.
Consequently, 
\begin{equation}
c_{n}\longrightarrow \frac{1}{2}.  \label{1307}
\end{equation}
\end{corollary}

\begin{proof}
Let $c^{\ast }$ be any cluster point of $\{c_{n}\}_{n\geq 0}$. Then there
exists a subsequence $\{n_{k}\}_{k\geq 0}$ such that 
\begin{equation}
c_{n_{k}}\rightarrow c^{\ast }.  \label{1308}
\end{equation}%
By the compactness result \textit{Claim~\ref{AA}}, after passing to a
further subsequence if necessary, we may assume that 
\begin{equation}
L_{n_{k}}\rightarrow L^{\ast }  \label{1309}
\end{equation}%
uniformly on $[0,1]$, for some continuous increasing function $L^{\ast }$
with $L^{\ast }(0)=0$ and $L^{\ast }(1)=1$. By \textit{Corollary~\ref%
{subseq-symmetric}}, $L^{\ast }$ is symmetric: 
\begin{equation}
L^{\ast }(x)+L^{\ast }(1-x)-1=0\qquad \text{for every }x\in \lbrack 0,1].
\label{1310}
\end{equation}%
Moreover, the crossing relation passes to the limit. Indeed, 
\begin{equation}
L_{n_{k}}(c_{n_{k}})=c_{n_{k}},  \label{1311}
\end{equation}%
and 
\begin{equation}
|L^{\ast }(c^{\ast })-c^{\ast }|\leq |L^{\ast }(c^{\ast })-L^{\ast
}(c_{n_{k}})|+|L^{\ast }(c_{n_{k}})-L_{n_{k}}(c_{n_{k}})|+|c_{n_{k}}-c^{\ast
}|.  \label{1312}
\end{equation}%
The first term tends to $0$ by continuity of $L^{\ast }$, the second term
tends to $0$ by uniform convergence, and the third term tends to $0$ by
construction. Thus 
\begin{equation}
L^{\ast }(c^{\ast })=c^{\ast }.  \label{1313}
\end{equation}
\end{proof}

By symmetry, 
\begin{equation}
L^{\ast }(1-c^{\ast })=1-L^{\ast }(c^{\ast })=1-c^{\ast }.  \label{1314}
\end{equation}%
Thus both $c^{\ast }$ and $1-c^{\ast }$ are crossing points of $L^{\ast }$
with the diagonal. By the one-crossing assumption inherited in the limiting
framework, this forces 
\begin{equation}
c^{\ast }=1-c^{\ast },  \label{1314.1}
\end{equation}%
and hence 
\begin{equation}
c^{\ast }=\frac{1}{2}.  \label{1314.2}
\end{equation}%
Therefore every cluster point of $\{c_{n}\}_{n\geq 0}$ is equal to $\tfrac{1%
}{2}$, and (\ref{1307}) follows.

It is important to note that we can release some of the underlying
assumptions.

\begin{remark}
\noindent \textbf{\label{endpoint-separation-release} }In the proof of 
\textit{Claim~\ref{M}} at the start of the appendix, we relied on the
auxiliary assumption 
\begin{equation*}
c_{n}\in \lbrack \varepsilon ,1-\varepsilon ]\qquad (n\geq 0).
\end{equation*}%
We now show that, in the regimes in which \textit{Claim~\ref{An-vanishing}}
applies, this assumption is in fact not needed. The convergence%
\begin{equation}
A_{n}=\int_{0}^{1}|D_{n}(x)|\,dx\longrightarrow 0,\qquad
D_{n}(x)=L_{n}(x)+L_{n}(1-x)-1.  \label{505a}
\end{equation}

was proved in \textit{Claim~\ref{An-vanishing}} without using \textit{Claim~%
\ref{M}}, \textit{Arzel\`{a}--Ascoli} compactness, or uniform subsequential
convergence of $L_{n}$. Indeed, its proof uses only the one-sign structure
of $D_{n}$, the monotonicity properties of $D_{n}$, \textit{Helly selection}
for the auxiliary monotone functions, and dominated convergence. Therefore (%
\ref{505a}) may be used before the \textit{uniform Lipschitz} and \textit{%
Arzel\`{a}--Ascoli} compactness arguments.
\end{remark}

\begin{claim}
\noindent \textbf{\label{In-noncollapse-from-An} }Under the hypotheses of
Claim~\ref{An-vanishing}, the normalizing constants 
\begin{equation}
I_{n}=\int_{0}^{1}T_{n}(u)(1-T_{n}(u))\,du,\qquad T_{n}=L_{n}^{-1},
\label{505c}
\end{equation}%
are uniformly bounded away from zero. That is, there exists a constant $C>0$
such that 
\begin{equation}
I_{n}\geq C>0\qquad \text{for all }n\geq 0.  \label{505d}
\end{equation}%
Consequently, in the situation covered by Claim~\ref{An-vanishing}, the
assumption $(\ref{499})$ in Claim~\ref{M} may be omitted.
\end{claim}

\begin{proof}
\noindent Suppose, to the contrary, that the conclusion is false. Then there
exists a subsequence, still denoted by $n_{k}$, such that 
\begin{equation}
I_{n_{k}}\rightarrow 0.  \label{505e}
\end{equation}

Let $\mu _{n}$ be the probability measure on $[0,1]$ with distribution
function $L_{n}$. Since $[0,1]$ is compact, \textit{Helly's selection
principle}, or equivalently \textit{Prokhorov compactness }(see, e.g.,%
\textit{\ }\cite[Chapter 1]{[68]} and \cite[Chapter 23]{[75]}) on a compact
metric space, gives a further subsequence, again denoted by $n_{k}$, and a
probability measure $\mu ^{\ast }$ on $[0,1]$ such that 
\begin{equation}
\mu _{n_{k}}\Rightarrow \mu ^{\ast }.  \label{505f}
\end{equation}%
Let $L^{\ast }$ be the distribution function of $\mu ^{\ast }$. Then 
\begin{equation}
L_{n_{k}}(x)\rightarrow L^{\ast }(x)  \label{505g}
\end{equation}%
at every continuity point $x$ of $L^{\ast }$.

We next identify the possible weak limit. Since $T_{n}=L_{n}^{-1}$ is the
quantile function corresponding to $\mu _{n}$, we have 
\begin{equation}
I_{n}=\int_{0}^{1}T_{n}(u)(1-T_{n}(u))\,du=\int_{[0,1]}x(1-x)\,d\mu _{n}(x).
\label{505h}
\end{equation}%
The function $x\mapsto x(1-x)$ is continuous and bounded on $[0,1]$.
Therefore, from $(\ref{505f})$, 
\begin{equation}
\int_{\lbrack 0,1]}x(1-x)\,d\mu _{n_{k}}(x)\longrightarrow \int_{\lbrack
0,1]}x(1-x)\,d\mu ^{\ast }(x).  \label{505i}
\end{equation}%
Together with $(\ref{505e})$, this gives 
\begin{equation}
\int_{\lbrack 0,1]}x(1-x)\,d\mu ^{\ast }(x)=0.  \label{505j}
\end{equation}%
Since $x(1-x)>0$ for $x\in (0,1)$, the measure $\mu ^{\ast }$ is supported
on the two endpoint points $\{0,1\}$. Hence there exists $p\in \lbrack 0,1]$
such that 
\begin{equation}
\mu ^{\ast }=p\delta _{0}+(1-p)\delta _{1}.  \label{505k}
\end{equation}%
Equivalently, 
\begin{equation}
L^{\ast }(x)=%
\begin{cases}
0, & x<0, \\ 
p, & 0\leq x<1, \\ 
1, & x\geq 1.%
\end{cases}
\label{505l}
\end{equation}%
In particular, 
\begin{equation}
L^{\ast }(x)=p\qquad \text{for every }x\in (0,1).  \label{505m}
\end{equation}

We now use the vanishing of the symmetry defect. For every $x\in (0,1)$,
both $x$ and $1-x$ are continuity points of $L^{\ast }$, because $L^{\ast }$
is constant on $(0,1)$. Hence 
\begin{equation*}
L_{n_{k}}(x)\rightarrow p,\qquad L_{n_{k}}(1-x)\rightarrow p,
\end{equation*}%
and therefore 
\begin{equation}
D_{n_{k}}(x)=L_{n_{k}}(x)+L_{n_{k}}(1-x)-1\longrightarrow 2p-1\qquad (0<x<1).
\label{505n}
\end{equation}%
By \textit{Fatou's lemma} and $(\ref{1272})$, 
\begin{equation}
\int_{0}^{1}|2p-1|\,dx\leq \liminf_{k\rightarrow \infty
}\int_{0}^{1}|D_{n_{k}}(x)|\,dx=0.  \label{505o}
\end{equation}%
Thus 
\begin{equation}
2p-1=0,\qquad p=\frac{1}{2}.  \label{505p}
\end{equation}%
Consequently the only possible denominator-collapsing limit under $(\ref%
{1272})$ is 
\begin{equation}
\mu ^{\ast }=\frac{1}{2}\delta _{0}+\frac{1}{2}\delta _{1},  \label{505q}
\end{equation}%
or, equivalently, 
\begin{equation}
L^{\ast }(x)=\frac{1}{2}\qquad \text{for }0<x<1.  \label{505r}
\end{equation}

It remains to show that this is incompatible with the crossing pattern. Let $%
c_{n}\in (0,1)$ be the unique interior fixed point from \textit{Lemma~\ref%
{l-cross}}, so that 
\begin{equation}
L_{n}(x)<x\quad (0<x<c_{n}),\qquad L_{n}(x)>x\quad (c_{n}<x<1).  \label{505s}
\end{equation}%
Since $c_{n_{k}}\in \lbrack 0,1]$, by passing to a further subsequence if
necessary we may assume 
\begin{equation}
c_{n_{k}}\rightarrow c\qquad \text{for some }c\in \lbrack 0,1].  \label{505t}
\end{equation}

\noindent If $c<1$, choose 
\begin{equation*}
x\in (\max \{c,1/2\},1).
\end{equation*}%
Then $x>c_{n_{k}}$ for all sufficiently large $k$. By the crossing pattern, 
\begin{equation*}
L_{n_{k}}(x)>x
\end{equation*}%
for all sufficiently large $k$. Passing to the limit and using $(\ref{505r})$%
, we obtain 
\begin{equation*}
\frac{1}{2}\geq x,
\end{equation*}%
which contradicts the choice $x>1/2$.

\noindent If $c=1$, choose $x\in (0,1/2)$. Then $x<c_{n_{k}}$ for all
sufficiently large $k$. By the crossing pattern, 
\begin{equation*}
L_{n_{k}}(x)<x
\end{equation*}%
for all sufficiently large $k$. Passing to the limit and using $(\ref{505r})$%
, we obtain 
\begin{equation*}
\frac{1}{2}\leq x,
\end{equation*}%
which contradicts the choice $x<1/2$.

Both cases lead to contradictions. Hence the assumption $(\ref{505e})$ is
impossible. Therefore 
\begin{equation*}
\inf_{n\geq 0}I_{n}>0.
\end{equation*}%
This proves the claim. \hfill
\end{proof}

\begin{corollary}
\textbf{\label{M-unconditional-under-An} }In every regime of Appendix~C in
which Claim~\ref{An-vanishing} applies, the auxiliary assumption $(\ref{499}%
) $ in Claim~\ref{M} is superfluous. More precisely, under the hypotheses of
Claim~\ref{An-vanishing}, the conclusion of Claim~\ref{M} holds without
assuming 
\begin{equation*}
c_{n}\in \lbrack \varepsilon ,1-\varepsilon ].
\end{equation*}
\end{corollary}

\begin{proof}
\textit{\noindent Claim~\ref{In-noncollapse-from-An}} gives a constant $C>0$
such that 
\begin{equation*}
I_{n}\geq C>0\qquad (n\geq 0).
\end{equation*}%
Since 
\begin{equation*}
0\leq T_{n}(x)(1-T_{n}(x))\leq \frac{1}{4},
\end{equation*}%
the iteration formula yields 
\begin{equation*}
0\leq L_{n+1}^{\prime }(x)=\frac{T_{n}(x)(1-T_{n}(x))}{I_{n}}\leq \frac{1/4}{%
C}.
\end{equation*}%
The same direct estimate applies to $L_{0}$ from its defining formula. Hence
the derivative bound and the \textit{Lipschitz} conclusion of \textit{Claim~%
\ref{M}} follow without the $c_{n}\in \lbrack \varepsilon ,1-\varepsilon ]$
assumption. \hfill
\end{proof}

\begin{flushleft}
\textbf{Variation-diminishing structures} \textbf{and} \textbf{classes of
starting laws for which the symmetry-defect contraction applies}
\end{flushleft}

This section establishes sufficient conditions for the validity of the
constant sign assumptions regarding the initial defect. Consequently, these
also constitute sufficient conditions for the symmetry of the subsequential
limits of $L_{n}$ about $\frac{1}{2}$.

We first study the number of sign changes of the symmetry defect in a
broader context 
\begin{equation}
D_{n}(x)=L_{n}(x)+L_{n}(1-x)-1,\qquad x\in \lbrack 0,1],  \label{1315}
\end{equation}%
and we relate the sign-change dynamics to classical variation-diminishing
theory in the sense of Karlin. Then we formulate explicit sufficient
conditions on the starting distribution $F$ ensuring either immediate
entrance into the one-sign regime or a strict reduction of sign changes at
the first iteration.

\begin{definition}
\label{Z-signchanges} Let $f$ be a continuous real-valued function on an
interval $(a,b)$. We denote by 
\begin{equation}
Z(f;(a,b))  \label{1316}
\end{equation}%
the maximal integer $m\geq 0$ for which there exist points 
\begin{equation}
a<x_{1}<\cdots <x_{m+1}<b  \label{1317}
\end{equation}%
such that 
\begin{equation}
f(x_{i})f(x_{i+1})<0,\qquad i=1,\dots ,m.  \label{1318}
\end{equation}%
Thus $Z(f;(a,b))$ is the number of strict sign changes of $f$ on $(a,b)$.
Whenever the interval is $(0,\tfrac{1}{2})$ we write simply 
\begin{equation}
Z(f)=Z\!\left( f;\left( 0,\tfrac{1}{2}\right) \right) .  \label{1319}
\end{equation}
\end{definition}

\begin{remark}
\label{bibliographic-VD} The variation-diminishing theorem below for totally
nonnegative kernels goes back to the classical theory of total positivity; a
standard reference is \cite{[36]}, especially the material in Ch.~5. A
direct statistical treatment of variation-diminishing transformations is
given in \cite{[64]}. For modern background on totally positive matrices and
kernels, see \cite{[65]}.
\end{remark}

We start with an introductory lemma for the \textit{Volterra kernel}.

\begin{lemma}
\label{Volterra-kernel-TN} Let 
\begin{equation}
K(x,u)=1_{\{0\leq u\leq x\leq \tfrac{1}{2}\}},\qquad (x,u)\in \left[ 0,%
\tfrac{1}{2}\right] ^{2}.  \label{1320}
\end{equation}%
Then $K$ is a totally nonnegative kernel on $\left[ 0,\tfrac{1}{2}\right]
^{2}$. Equivalently, for every $m\geq 1$ and every strictly increasing
tuples 
\begin{equation}
0<x_{1}<\cdots <x_{m}\leq \frac{1}{2},\qquad 0<u_{1}<\cdots <u_{m}\leq \frac{%
1}{2},  \label{1321}
\end{equation}%
the determinant 
\begin{equation}
\det \!(K(x_{i},u_{j}))_{i,j=1}^{m}  \label{1321.1}
\end{equation}%
is nonnegative. In fact, each such determinant is either $0$ or $1$.
\end{lemma}

\begin{proof}
Fix $m$ and strictly increasing tuples $x_{1}<\cdots <x_{m}$ and $%
u_{1}<\cdots <u_{m}$. Set 
\begin{equation}
A_{ij}=K(x_{i},u_{j})=1_{\{u_{j}\leq x_{i}\}}.  \label{1322}
\end{equation}%
For each row index $i$, the sequence $A_{i1},\dots ,A_{im}$ consists of an
initial block of $1$'s followed by a block of $0$'s, because the $u_{j}$ are
increasing. Hence each row has the form 
\begin{equation}
(\underbrace{1,\dots ,1}_{k_{i}\ \text{entries}},0,\dots ,0)  \label{1323}
\end{equation}%
for some integer $k_{i}\in \{0,\dots ,m\}$. Since the $x_{i}$ are
increasing, the row parameters satisfy 
\begin{equation}
0\leq k_{1}\leq k_{2}\leq \cdots \leq k_{m}\leq m.  \label{1324}
\end{equation}

If $k_{i}=k_{i+1}$ for some $i$, then rows $i$ and $i+1$ are identical, so 
\begin{equation}
\det A=0.  \label{1325}
\end{equation}%
Therefore a nonzero determinant is possible only when 
\begin{equation}
k_{i}=i\qquad \text{for every }i=1,\dots ,m.  \label{1326}
\end{equation}%
In that case $A$ is lower triangular with all diagonal entries equal to $1$,
hence 
\begin{equation}
\det A=1.  \label{1327}
\end{equation}%
Thus every minor is either $0$ or $1$, in particular nonnegative.

Now we can formulate the \textit{Karlin's variation-diminishing theorem} for
the \textit{Volterra kernel.}\hfill
\end{proof}

\begin{theorem}
\label{VD-Volterra} Let 
\begin{equation}
(\mathcal{K}f)(x)=\int_{0}^{1/2}K(x,u)f(u)\,du=\int_{0}^{x}f(u)\,du,\qquad
0\leq x\leq \frac{1}{2},  \label{1328}
\end{equation}%
where $K$ is the kernel from \textit{Lemma~\ref{Volterra-kernel-TN}}. Then 
\begin{equation}
Z(\mathcal{K}f)\leq Z(f)  \label{1329}
\end{equation}%
for every continuous function $f$ on $[0,\tfrac{1}{2}]$.
\end{theorem}

\begin{proof}
By \textit{Lemma~\ref{Volterra-kernel-TN}}, the kernel $K$ is totally
nonnegative. The variation-diminishing theorem for totally nonnegative
kernels therefore applies; see \textit{Remark~\ref{bibliographic-VD}}. This
yields (\ref{1329}). \hfill
\end{proof}

\begin{claim}
\label{VD-applied} For every $n\geq 0$, define 
\begin{equation}
S_{n}(u)=T_{n}(u)+T_{n}(1-u)-1,\qquad 0\leq u\leq 1,  \label{1330}
\end{equation}%
and 
\begin{equation}
G_{n}(u)=w(T_{n}(u))-w(T_{n}(1-u)),\qquad w(y)=y(1-y),\qquad 0\leq u\leq 
\frac{1}{2}.  \label{1331}
\end{equation}%
Then the following hold.

\begin{enumerate}
\item For $0\leq x\leq \tfrac{1}{2}$, 
\begin{equation}
D_{n+1}(x)=\frac{1}{I_{n}}\int_{0}^{x}G_{n}(u)\,du,\qquad
I_{n}=\int_{0}^{1}w(T_{n}(v))\,dv>0.  \label{C1-Dn1-primitive}
\end{equation}

\item For $0<u<\tfrac{1}{2}$, 
\begin{equation}
G_{n}(u)=(T_{n}(1-u)-T_{n}(u))S_{n}(u),  \label{C1-Gn-factor}
\end{equation}%
hence 
\begin{equation}
Z(G_{n})=Z(S_{n}).  \label{C1-ZGn-ZSn}
\end{equation}

\item We have 
\begin{equation}
Z(S_{n})=Z(D_{n}).  \label{C1-ZSn-ZDn}
\end{equation}

\item Consequently, 
\begin{equation}
Z(D_{n+1})\leq Z(D_{n})\qquad \text{for every }n\geq 0.  \label{C1-ZDn-mon}
\end{equation}
\end{enumerate}
\end{claim}

\begin{proof}
For the first statement, recall that 
\begin{equation}
L_{n+1}(x)=\frac{\int_{0}^{x}w(T_{n}(u))\,du}{I_{n}}.  \label{1332}
\end{equation}%
Hence for $0\leq x\leq \tfrac{1}{2}$, 
\begin{align}
D_{n+1}(x)& =L_{n+1}(x)+L_{n+1}(1-x)-1  \label{1332.1} \\
& =\frac{1}{I_{n}}\left(
\int_{0}^{x}w(T_{n}(u))\,du+\int_{0}^{1-x}w(T_{n}(u))\,du-%
\int_{0}^{1}w(T_{n}(u))\,du\right)  \notag \\
& =\frac{1}{I_{n}}\left(
\int_{0}^{x}w(T_{n}(u))\,du-\int_{1-x}^{1}w(T_{n}(u))\,du\right) .  \notag
\end{align}%
Changing variables $v=1-u$ in the second integral, we obtain 
\begin{equation}
D_{n+1}(x)=\frac{1}{I_{n}}\int_{0}^{x}(w(T_{n}(u))-w(T_{n}(1-u)))\,du,
\label{1333}
\end{equation}%
which is (\ref{C1-Dn1-primitive}).

For the second statement, use 
\begin{equation}
w(a)-w(b)=(a-b)(1-a-b).  \label{1334}
\end{equation}%
With $a=T_{n}(u)$ and $b=T_{n}(1-u)$ this gives 
\begin{align}
G_{n}(u)& =w(T_{n}(u))-w(T_{n}(1-u))  \label{1335} \\
& =(T_{n}(u)-T_{n}(1-u))(1-T_{n}(u)-T_{n}(1-u))  \notag \\
& =(T_{n}(1-u)-T_{n}(u))(T_{n}(u)+T_{n}(1-u)-1),  \notag
\end{align}%
which is exactly (\ref{C1-Gn-factor}). Since $T_{n}$ is strictly increasing,
the factor $T_{n}(1-u)-T_{n}(u)$ is strictly positive on $(0,\tfrac{1}{2})$,
so $G_{n}$ and $S_{n}$ have the same sign pattern there. Hence (\ref%
{C1-ZGn-ZSn}) holds.

For the third statement, let $x\in (0,1)$ and set $u=L_{n}(x)$. Then 
\begin{equation}
S_{n}(u)=T_{n}(u)+T_{n}(1-u)-1=x+T_{n}(1-L_{n}(x))-1.  \label{1336}
\end{equation}%
Therefore 
\begin{equation}
S_{n}(L_{n}(x))=T_{n}(1-L_{n}(x))-(1-x).  \label{1337}
\end{equation}%
Because $T_{n}$ is increasing, the sign of this quantity is the sign of 
\begin{equation}
(1-L_{n}(x))-L_{n}(1-x)=-(L_{n}(x)+L_{n}(1-x)-1)=-D_{n}(x).  \label{1338}
\end{equation}%
Thus 
\begin{equation}
\func{sgn}(S_{n}(L_{n}(x)))=-\func{sgn}(D_{n}(x))\qquad (x\in (0,1),\
D_{n}(x)\neq 0).  \label{1339}
\end{equation}%
Since $L_{n}$ is a strictly increasing homeomorphism of $[0,1]$, composition
with $L_{n}$ preserves the number of sign changes. Hence (\ref{C1-ZSn-ZDn})
follows.

Finally, by \textit{Theorem~\ref{VD-Volterra}} and (\ref{C1-Dn1-primitive}), 
\begin{equation}
Z(D_{n+1})\leq Z(G_{n}).  \label{1340}
\end{equation}%
Using (\ref{C1-ZGn-ZSn}) and (\ref{C1-ZSn-ZDn}) we conclude that 
\begin{equation}
Z(D_{n+1})\leq Z(D_{n}),  \label{1341}
\end{equation}%
which is (\ref{C1-ZDn-mon}). \hfill
\end{proof}

\begin{remark}
\label{no-strict-from-kernel} The preceding argument gives only the
non-strict inequality 
\begin{equation}
Z(D_{n+1})\leq Z(D_{n}).  \label{1342}
\end{equation}%
It does not give the strict inequality 
\begin{equation}
Z(D_{n+1})<Z(D_{n})\qquad \text{whenever }Z(D_{n})>0.  \label{1343}
\end{equation}%
There are two reasons for this.

First, the \textit{Volterra kernel} $K(x,u)=1_{\{u\leq x\}}$ is totally
nonnegative, but it is not strictly totally positive. Indeed, for instance,
if 
\begin{equation}
0<u_{1}<u_{2}<x_{1}<x_{2}<\frac{1}{2},  \label{1344}
\end{equation}%
then 
\begin{equation}
\det 
\begin{pmatrix}
K(x_{1},u_{1}) & K(x_{1},u_{2}) \\ 
K(x_{2},u_{1}) & K(x_{2},u_{2})%
\end{pmatrix}%
=\det 
\begin{pmatrix}
1 & 1 \\ 
1 & 1%
\end{pmatrix}%
=0.  \label{1345}
\end{equation}%
Thus the strict versions of \textit{Karlin's theorems} do not apply directly
to this kernel.

Second, strict loss of sign changes is false in general for the \textit{%
Volterra operator} even on smooth inputs. For example, on $[0,\tfrac{1}{2}]$
define 
\begin{equation}
f(u)=u-\frac{1}{8}.  \label{1346}
\end{equation}%
Then $f$ is $C^{\infty }$ and has exactly one sign change: 
\begin{equation}
Z(f)=1.  \label{1347}
\end{equation}%
However 
\begin{equation}
(\mathcal{K}f)(x)=\int_{0}^{x}\left( u-\frac{1}{8}\right) \,du=\frac{x^{2}}{2%
}-\frac{x}{8}=x\left( \frac{x}{2}-\frac{1}{8}\right) ,  \label{1348}
\end{equation}%
which also has exactly one sign change on $(0,\tfrac{1}{2})$. Hence 
\begin{equation}
Z(\mathcal{K}f)=1=Z(f).  \label{1349}
\end{equation}

Therefore strict reduction of sign changes cannot be deduced from the kernel
alone. The weight $w(y)=y(1-y)$ does not alter this conclusion at the level
of the \textit{Karlin argument}, because $w$ enters only through the input
profile 
\begin{equation}
G_{n}(u)=w(T_{n}(u))-w(T_{n}(1-u)),  \label{1350}
\end{equation}%
while the transform from $G_{n}$ to $D_{n+1}$ remains the same \textit{%
Volterra operator}.
\end{remark}

\begin{claim}
\label{strict-drop-masses} Let $f$ be a continuous function on $[0,a]$, $a>0$%
, with finitely many sign changes on $(0,a)$. Suppose that the zeros of $f$
on $[0,a]$ satisfy 
\begin{equation*}
0=\xi _{0}<\xi _{1}<\cdots <\xi _{m}<\xi _{m+1}=a,
\end{equation*}%
and that $f$ has constant sign on each nodal interval 
\begin{equation*}
I_{k}=(\xi _{k-1},\xi _{k}),\qquad k=1,\dots ,m+1,
\end{equation*}%
with alternating signs from one interval to the next.

Define the nodal masses 
\begin{equation}
M_{k}=\int_{\xi _{k-1}}^{\xi _{k}}f(u)\,du,\qquad k=1,\dots ,m+1.
\label{1351}
\end{equation}%
Assume moreover that 
\begin{equation}
|M_{1}|\geq |M_{2}|\geq \cdots \geq |M_{m+1}|.  \label{1352}
\end{equation}%
Then the primitive 
\begin{equation*}
F(x)=\int_{0}^{x}f(u)\,du
\end{equation*}%
satisfies 
\begin{equation}
Z(F)<Z(f).  \label{1353}
\end{equation}
\end{claim}

\begin{proof}
Because the signs of $f$ alternate from lobe to lobe, the numbers $M_{k}$
alternate in sign as well. Suppose, for definiteness, that $M_{1}>0$; the
other case is identical after multiplying by $-1$. Then 
\begin{equation*}
M_{1}>0,\quad M_{2}<0,\quad M_{3}>0,\ \dots
\end{equation*}%
and the monotonicity (\ref{1352}) implies that the partial sums 
\begin{equation*}
S_{k}=M_{1}+\cdots +M_{k},\qquad k=1,\dots ,m+1,
\end{equation*}%
are all nonnegative. More precisely, each $S_{k}$ has the sign of $M_{1}$
unless an equality of consecutive absolute nodal masses produces a zero
partial sum. In either case, the sequence $\{S_{k}\}_{k\geq 0}$ cannot
alternate signs.

Now $F(\xi _{k})=S_{k}$. Since the values $F(\xi _{k})$ do not alternate
signs, the primitive $F$ cannot alternate sign across all nodal intervals of 
$f$. Therefore $F$ must have strictly fewer sign changes than $f$, proving (%
\ref{1353}). \hfill
\end{proof}

\begin{corollary}
\label{strict-drop-iterates} Fix $n\geq 0$. Let the zeros of 
\begin{equation}
G_{n}(u)=w(T_{n}(u))-w(T_{n}(1-u))  \label{1354}
\end{equation}%
on $(0,\tfrac{1}{2})$ be 
\begin{equation}
0=\xi _{0}<\xi _{1}<\cdots <\xi _{m}<\xi _{m+1}=\frac{1}{2},  \label{1355}
\end{equation}%
and define the corresponding nodal masses 
\begin{equation}
M_{k}^{(n)}=\int_{\xi _{k-1}}^{\xi _{k}}G_{n}(u)\,du,\qquad k=1,\dots ,m+1.
\label{1356}
\end{equation}%
If 
\begin{equation}
|M_{1}^{(n)}|\geq |M_{2}^{(n)}|\geq \cdots \geq |M_{m+1}^{(n)}|,
\label{1357}
\end{equation}%
then 
\begin{equation}
Z(D_{n+1})<Z(D_{n}).  \label{1358}
\end{equation}
\end{corollary}

\begin{proof}
By \textit{Claim~\ref{VD-applied}}, 
\begin{equation}
D_{n+1}(x)=\frac{1}{I_{n}}\int_{0}^{x}G_{n}(u)\,du\qquad \text{on }\left[ 0,%
\frac{1}{2}\right] ,  \label{1359}
\end{equation}%
and 
\begin{equation}
Z(G_{n})=Z(D_{n}).  \label{1360}
\end{equation}%
Since multiplication by the positive constant $I_{n}^{-1}$ does not affect
sign changes, \textit{Claim~\ref{strict-drop-masses}} applied to $f=G_{n}$
yields 
\begin{equation}
Z(D_{n+1})<Z(G_{n})=Z(D_{n}).  \label{1361}
\end{equation}
\end{proof}

The preceding corollary is not, by itself, a condition on the starting law.
Rather, it is a stepwise criterion: whenever the nodal masses of the current
profile $G_{n}$ are ordered as in (\ref{1352}), the next iterate loses at
least one sign change. Thus, if this criterion can be verified at every
iterate before the one-sign regime is reached, then the process enters the
one-sign regime after finitely many steps. If the condition fails at some
iterate, the corollary gives no conclusion at that step.\hfill

\begin{claim}
\label{start-law-HF} Let $Q=F^{-1}$ be the quantile function of the starting
distribution $F$ on $[0,1]$, and define on $[0,\tfrac{1}{2}]$ 
\begin{equation}
H_{F}(u)=Q(u)+Q(1-u)-1,  \label{1362}
\end{equation}%
\begin{equation}
W_{F}(u)=Q(1-u)-Q(u),  \label{1363}
\end{equation}%
and 
\begin{equation}
I_{F}=\int_{0}^{1}Q(t)(1-Q(t))\,dt.  \label{1364}
\end{equation}%
Assume that $I_{F}>0$. Then

\begin{enumerate}
\item $W_{F}(u)\geq 0$ and $W_{F}$ is nonincreasing on $[0,\tfrac{1}{2}]$;

\item the initial defect 
\begin{equation}
D_{0}(x)=L_{0}(x)+L_{0}(1-x)-1  \label{1365}
\end{equation}%
satisfies 
\begin{equation}
D_{0}(x)=\frac{1}{I_{F}}\int_{0}^{x}W_{F}(u)\,H_{F}(u)\,du,\qquad 0\leq
x\leq \frac{1}{2};  \label{1366}
\end{equation}

\item if, in addition, $Q$ is continuous and strictly increasing on $[0,1]$,
then 
\begin{equation}
Z(D_{0})\leq Z(H_{F}).  \label{1367}
\end{equation}
\end{enumerate}
\end{claim}

\begin{proof}
Since $Q$ is increasing, 
\begin{equation}
W_{F}(u)=Q(1-u)-Q(u)\geq 0\qquad \left( 0\leq u\leq \frac{1}{2}\right) .
\label{1368}
\end{equation}%
If $0\leq u<v\leq \tfrac{1}{2}$, then 
\begin{equation}
Q(1-u)\geq Q(1-v),\qquad Q(u)\leq Q(v),  \label{1369}
\end{equation}%
so 
\begin{equation}
W_{F}(u)=Q(1-u)-Q(u)\geq Q(1-v)-Q(v)=W_{F}(v).  \label{1370}
\end{equation}%
Thus $W_{F}$ is nonincreasing.

Next, by the definition of $L_{0}$, 
\begin{equation}
L_{0}(x)=\frac{\int_{0}^{x}Q(t)(1-Q(t))\,dt}{I_{F}}.  \label{1371}
\end{equation}%
Exactly as in the proof of \textit{Claim~\ref{VD-applied}}, for $0\leq x\leq 
\tfrac{1}{2}$, 
\begin{align}
D_{0}(x)& =\frac{1}{I_{F}}\left(
\int_{0}^{x}Q(u)(1-Q(u))\,du-\int_{1-x}^{1}Q(u)(1-Q(u))\,du\right)  \notag \\
& =\frac{1}{I_{F}}\int_{0}^{x}(Q(u)(1-Q(u))-Q(1-u)(1-Q(1-u)))\,du.
\label{1372}
\end{align}%
Using 
\begin{equation}
a(1-a)-b(1-b)=(b-a)(a+b-1),  \label{1373}
\end{equation}%
with $a=Q(u)$ and $b=Q(1-u)$, we obtain 
\begin{equation}
Q(u)(1-Q(u))-Q(1-u)(1-Q(1-u))=(Q(1-u)-Q(u))(Q(u)+Q(1-u)-1),  \label{1374}
\end{equation}%
which proves (\ref{1366}).

Assume now that $Q$ is continuous and strictly increasing on $[0,1]$. Then $%
H_{F}$ and $W_{F}$ are continuous on $[0,\tfrac{1}{2}]$, and 
\begin{equation}
W_{F}(u)>0\qquad \text{for every }0<u<\frac{1}{2}.  \label{1375}
\end{equation}%
Hence $W_{F}H_{F}$ is continuous on $[0,\tfrac{1}{2}]$, and by \textit{%
Theorem~\ref{VD-Volterra}} applied to 
\begin{equation}
D_{0}(x)=\frac{1}{I_{F}}\int_{0}^{x}(W_{F}H_{F})(u)\,du,  \label{1376}
\end{equation}%
we obtain 
\begin{equation}
Z(D_{0})\leq Z(W_{F}H_{F}).  \label{1377}
\end{equation}%
Since $W_{F}(u)>0$ on $(0,\tfrac{1}{2})$, the functions $W_{F}H_{F}$ and $%
H_{F}$ have the same sign changes on $(0,\tfrac{1}{2})$, and therefore 
\begin{equation}
Z(W_{F}H_{F})=Z(H_{F}).  \label{1378}
\end{equation}%
This proves (\ref{1367}). \hfill \hfill
\end{proof}

\begin{corollary}
\label{HF-one-sign} If 
\begin{equation}
H_{F}(u)=Q(u)+Q(1-u)-1  \label{1393}
\end{equation}%
has a constant sign on $[0,\tfrac{1}{2}]$, then $D_{0}$ has a constant sign
on $[0,1]$. Consequently, by \textit{Lemma~\ref{defect-shape}}, all later
defects $D_{n}$ have a constant sign on $[0,1]$, with alternating sign in $n$%
.
\end{corollary}

\begin{proof}
If $H_{F}$ has a constant sign, then so does $W_{F}H_{F}$ because $W_{F}\geq
0$ on $[0,\tfrac{1}{2}]$. Formula (\ref{1366}) then shows that $%
D_{0}^{\prime }$ has a fixed sign on $[0,\tfrac{1}{2}]$. Since $D_{0}(0)=0$
and $D_{0}(1-x)=D_{0}(x)$, it follows that $D_{0}$ has a constant sign on $%
[0,1]$. The persistence statement is then exactly \textit{Lemma~\ref%
{defect-shape}}. \hfill
\end{proof}

\begin{corollary}
\label{two-lobe} Assume that $Q$ is continuous and strictly increasing on $%
[0,1]$. Assume that $H_{F}$ has exactly one zero 
\begin{equation}
0<\xi <\frac{1}{2}  \label{1394}
\end{equation}%
on $(0,\tfrac{1}{2})$. If 
\begin{equation}
\left\vert \int_{0}^{\xi }W_{F}(u)H_{F}(u)\,du\right\vert \geq \left\vert
\int_{\xi }^{1/2}W_{F}(u)H_{F}(u)\,du\right\vert ,  \label{1395}
\end{equation}%
then $D_{0}$ has a constant sign on $[0,1]$.

A simpler sufficient condition is 
\begin{equation}
\int_{0}^{\xi }|H_{F}(u)|\,du\geq \int_{\xi }^{1/2}|H_{F}(u)|\,du,
\label{1396}
\end{equation}%
because $W_{F}$ is nonincreasing on $[0,\tfrac{1}{2}]$.
\end{corollary}

\begin{proof}
If $H_{F}$ has exactly one zero, then $W_{F}H_{F}$ has exactly two nodal
intervals on $(0,\tfrac{1}{2})$ because $W_{F}(u)>0$ for every $0<u<\tfrac{1%
}{2}$. Condition (\ref{1395}) says exactly that the first nodal mass
dominates the second in absolute value. Therefore \textit{Claim~\ref%
{strict-drop-masses}} implies that the primitive (\ref{1366}) has no sign
change on $(0,\tfrac{1}{2})$, i.e. $D_{0}$ has a constant sign there, and
hence on all of $[0,1]$ by symmetry.

To obtain (\ref{1396}) as a sufficient condition, note that $W_{F}$ is
nonincreasing on $[0,\tfrac{1}{2}]$. Thus 
\begin{equation}
\left\vert \int_{0}^{\xi }W_{F}(u)H_{F}(u)\,du\right\vert \geq W_{F}(\xi
)\int_{0}^{\xi }|H_{F}(u)|\,du,  \label{1397}
\end{equation}%
whereas 
\begin{equation}
\left\vert \int_{\xi }^{1/2}W_{F}(u)H_{F}(u)\,du\right\vert \leq W_{F}(\xi
)\int_{\xi }^{1/2}|H_{F}(u)|\,du.  \label{1398}
\end{equation}%
Hence (\ref{1396}) implies (\ref{1395}). \hfill
\end{proof}

\begin{corollary}
\label{general-lobes} Assume that $Q$ is continuous and strictly increasing
on $[0,1]$. Assume that the zeros of $H_{F}$ on $(0,\tfrac{1}{2})$ are 
\begin{equation}
0=\xi _{0}<\xi _{1}<\cdots <\xi _{m}<\xi _{m+1}=\frac{1}{2},  \label{1399}
\end{equation}%
and define the lobe areas 
\begin{equation}
A_{k}(F)=\int_{\xi _{k-1}}^{\xi _{k}}|H_{F}(u)|\,du,\qquad k=1,\dots ,m+1.
\label{1400}
\end{equation}%
If 
\begin{equation}
A_{1}(F)\geq A_{2}(F)\geq \cdots \geq A_{m+1}(F),  \label{1401}
\end{equation}%
then 
\begin{equation}
Z(D_{0})<Z(H_{F}).  \label{1402}
\end{equation}
\end{corollary}

\begin{proof}
Let 
\begin{equation}
M_{k}(F)=\int_{\xi _{k-1}}^{\xi _{k}}W_{F}(u)H_{F}(u)\,du  \label{1403}
\end{equation}%
be the weighted nodal masses of $W_{F}H_{F}$. Since $W_{F}$ is nonincreasing
on $[0,\tfrac{1}{2}]$, we have 
\begin{equation}
|M_{k}(F)|=\int_{\xi _{k-1}}^{\xi _{k}}W_{F}(u)|H_{F}(u)|\,du\geq W_{F}(\xi
_{k})\int_{\xi _{k-1}}^{\xi _{k}}|H_{F}(u)|\,du=W_{F}(\xi _{k})A_{k}(F),
\label{1404}
\end{equation}%
and also 
\begin{equation}
|M_{k+1}(F)|\leq W_{F}(\xi _{k})\int_{\xi _{k}}^{\xi
_{k+1}}|H_{F}(u)|\,du=W_{F}(\xi _{k})A_{k+1}(F).  \label{1405}
\end{equation}%
Therefore (\ref{1400}) implies 
\begin{equation}
|M_{1}(F)|\geq |M_{2}(F)|\geq \cdots \geq |M_{m+1}(F)|.  \label{1406}
\end{equation}%
Applying \textit{Claim~\ref{strict-drop-masses}} to 
\begin{equation}
f(u)=W_{F}(u)H_{F}(u),  \label{1407}
\end{equation}%
and using (\ref{1366}), yields 
\begin{equation}
Z(D_{0})<Z(W_{F}H_{F})=Z(H_{F}),  \label{1408}
\end{equation}%
because $W_{F}(u)>0$ for every $0<u<\tfrac{1}{2}$. \hfill
\end{proof}

\begin{claim}
\label{prob-meaning} Assume that $F$ is continuous and strictly increasing
on $[0,1]$, so that $Q=F^{-1}$ is its genuine inverse. Let 
\begin{equation}
F^{\sharp }(x)=1-F(1-x),\qquad x\in \lbrack 0,1],  \label{1409}
\end{equation}%
be the cdf of the reflected random variable $1-X$ when $X\sim F$. Then the
following are equivalent:

\begin{enumerate}
\item 
\begin{equation}
Q(u)+Q(1-u)-1\leq 0\qquad \forall u\in \lbrack 0,1];  \label{1410}
\end{equation}

\item 
\begin{equation}
F(x)\geq F^{\sharp }(x)=1-F(1-x)\qquad \forall x\in \lbrack 0,1];
\label{1411}
\end{equation}

\item 
\begin{equation}
X\preceq _{\mathrm{st}}1-X,  \label{1412}
\end{equation}%
that is, $X$ is stochastically dominated by its reflection.
\end{enumerate}

Likewise, 
\begin{equation}
Q(u)+Q(1-u)-1\geq 0\ \forall u  \label{1413}
\end{equation}%
is equivalent to 
\begin{equation}
F(x)\leq 1-F(1-x)\ \forall x,  \label{1414}
\end{equation}%
that is, $1-X\preceq _{\mathrm{st}}X$.
\end{claim}

\begin{proof}
We prove the first equivalence; the second is identical with all
inequalities reversed.

Assume first that 
\begin{equation}
Q(u)+Q(1-u)-1\leq 0\qquad \forall u\in \lbrack 0,1].  \label{1415}
\end{equation}%
Fix $x\in \lbrack 0,1]$ and put $u=F(x)$. Then $x=Q(u)$, and the assumed
inequality gives 
\begin{equation}
Q(1-u)\leq 1-Q(u)=1-x.  \label{1416}
\end{equation}%
Applying the increasing map $F$ to both sides, we obtain 
\begin{equation}
1-u\leq F(1-x).  \label{1417}
\end{equation}%
Since $u=F(x)$, this is 
\begin{equation}
F(x)\geq 1-F(1-x)=F^{\sharp }(x).  \label{1418}
\end{equation}

Conversely, assume 
\begin{equation}
F(x)\geq 1-F(1-x)\qquad \forall x\in \lbrack 0,1].  \label{1419}
\end{equation}%
Fix $u\in \lbrack 0,1]$ and put $x=Q(u)$. Then $u=F(x)$, so 
\begin{equation}
u\geq 1-F(1-x).  \label{1420}
\end{equation}%
Hence 
\begin{equation}
1-u\leq F(1-x).  \label{1421}
\end{equation}%
Applying the increasing map $Q=F^{-1}$, 
\begin{equation}
Q(1-u)\leq 1-x=1-Q(u),  \label{1422}
\end{equation}%
that is, 
\begin{equation}
Q(u)+Q(1-u)-1\leq 0.  \label{1423}
\end{equation}%
This proves the equivalence.

The equivalence with stochastic order is the standard definition: 
\begin{equation}
X\preceq _{\mathrm{st}}Y\quad \Longleftrightarrow \quad F_{X}(x)\geq
F_{Y}(x)\ \forall x.  \label{1424}
\end{equation}%
With $Y=1-X$, the cdf of $Y$ is exactly $F^{\sharp }$. \hfill
\end{proof}

\begin{corollary}
\label{density-conditions} Assume that $F$ has a continuous positive density 
$f$ on $(0,1)$, and define 
\begin{equation}
G_{F}(x)=F(x)+F(1-x)-1,\qquad x\in \lbrack 0,1].  \label{1425}
\end{equation}%
Then:

\begin{enumerate}
\item If $f(x)-f(1-x)$ changes sign at most once on $(0,1)$, then $G_{F}$
has a constant sign on $[0,1]$.

\item Consequently, by \textit{Claim~\ref{prob-meaning}}, 
\begin{equation}
H_{F}(u)=Q(u)+Q(1-u)-1  \label{1426}
\end{equation}%
has a constant sign on $[0,1]$, and by \textit{Corollary~\ref{HF-one-sign}}
the first defect $D_{0}$ has a constant sign on $[0,1]$.

\item A sufficient condition for the hypothesis in \emph{(1)} is that the
ratio 
\begin{equation}
x\longmapsto \frac{f(x)}{f(1-x)}  \label{1427}
\end{equation}%
be monotone on $(0,1)$.
\end{enumerate}
\end{corollary}

\begin{proof}
Since 
\begin{equation}
G_{F}(0)=0,\qquad G_{F}(1)=0,  \label{1428}
\end{equation}%
and 
\begin{equation}
G_{F}^{\prime }(x)=f(x)-f(1-x),  \label{1429}
\end{equation}%
the first claim is immediate: if $G_{F}^{\prime }$ changes sign at most
once, then $G_{F}$ is first monotone and then monotone in the opposite
direction; since it vanishes at both endpoints, it cannot cross the
horizontal axis in the interior, and therefore has a constant sign on $[0,1]$%
.

The second statement then follows from \textit{Claim~\ref{prob-meaning}} and 
\textit{Corollary~\ref{HF-one-sign}}.

For the third statement, observe that 
\begin{equation}
f(x)-f(1-x)=f(1-x)\left( \frac{f(x)}{f(1-x)}-1\right) .  \label{1430}
\end{equation}%
Since $f(1-x)>0$, the sign changes of $f(x)-f(1-x)$ are exactly the sign
changes of $\frac{f(x)}{f(1-x)}-1$. If the ratio is monotone, then this
difference changes sign at most once.
\end{proof}

\begin{claim}
\label{sufficient-conditions-summary} Let $Q=F^{-1}$ and $H_{F}$, $W_{F}$ be
as in \textit{Claim~\ref{start-law-HF}}. The following conditions give
explicit sufficient hypotheses on the starting law $F$: items \textup{(i)}, 
\textup{(ii)} and \textup{(iv)} imply that the initial defect $D_{0}$ has a
constant sign on $[0,1]$, whereas item \textup{(iii)} gives the strict
initial sign-change reduction 
\begin{equation}
Z(D_{0})<Z(H_{F}).  \label{1378.1}
\end{equation}

\begin{enumerate}
\item \textbf{Immediate one-sign regime.} If 
\begin{equation}
H_{F}(u)=Q(u)+Q(1-u)-1  \label{1379}
\end{equation}%
has a constant sign on $[0,\tfrac{1}{2}]$, then $D_{0}$ has a constant sign
on $[0,1]$.

\item \textbf{Two-lobe dominance.} Assume moreover that $Q$ is continuous
and strictly increasing on $[0,1]$. Suppose $H_{F}$ has exactly one zero 
\begin{equation}
0<\xi <\frac{1}{2}  \label{1380}
\end{equation}%
on $(0,\tfrac{1}{2})$. If 
\begin{equation}
\left\vert \int_{0}^{\xi }W_{F}(u)H_{F}(u)\,du\right\vert \geq \left\vert
\int_{\xi }^{1/2}W_{F}(u)H_{F}(u)\,du\right\vert ,  \label{1381}
\end{equation}%
then $D_{0}$ has a constant sign on $[0,1]$. A simpler sufficient condition
is 
\begin{equation}
\int_{0}^{\xi }|H_{F}(u)|\,du\geq \int_{\xi }^{1/2}|H_{F}(u)|\,du.
\label{1382}
\end{equation}

\item \textbf{General decreasing-lobe criterion.} Assume moreover that $Q$
is continuous and strictly increasing on $[0,1]$. Suppose the zeros of $%
H_{F} $ on $(0,\tfrac{1}{2})$ are 
\begin{equation}
0=\xi _{0}<\xi _{1}<\cdots <\xi _{m}<\xi _{m+1}=\frac{1}{2},  \label{1383}
\end{equation}%
and define the lobe areas 
\begin{equation}
A_{k}(F)=\int_{\xi _{k-1}}^{\xi _{k}}|H_{F}(u)|\,du,\qquad k=1,\dots ,m+1.
\label{1384}
\end{equation}%
If 
\begin{equation}
A_{1}(F)\geq A_{2}(F)\geq \cdots \geq A_{m+1}(F),  \label{1385}
\end{equation}%
then 
\begin{equation}
Z(D_{0})<Z(H_{F}).  \label{1386}
\end{equation}

\item \textbf{Density-level criterion.} If $F$ has a continuous positive
density $f$ on $(0,1)$ and the function 
\begin{equation}
x\mapsto f(x)-f(1-x)  \label{1387}
\end{equation}%
changes sign at most once on $(0,1)$, then $H_{F}$ has a constant sign on $%
[0,1]$, hence $D_{0}$ has a constant sign on $[0,1]$. A sufficient condition
for this is that 
\begin{equation}
x\longmapsto \frac{f(x)}{f(1-x)}  \label{1388}
\end{equation}%
be monotone on $(0,1)$.
\end{enumerate}
\end{claim}

\begin{proof}
Statement \textup{(i)} is exactly \textit{Corollary~\ref{HF-one-sign}}.

Statement \textup{(ii)} is exactly \textit{Corollary~\ref{two-lobe}}.

Statement \textup{(iii)} is exactly \textit{Corollary~\ref{general-lobes}}.

For \textup{(iv)}, define 
\begin{equation}
G_{F}(x)=F(x)+F(1-x)-1,\qquad x\in \lbrack 0,1].  \label{1389}
\end{equation}%
Then 
\begin{equation}
G_{F}(0)=0,\qquad G_{F}(1)=0,\qquad G_{F}^{\prime }(x)=f(x)-f(1-x).
\label{1390}
\end{equation}%
If $G_{F}^{\prime }$ changes sign at most once, then $G_{F}$ is first
monotone and then monotone in the opposite direction; since it vanishes at
both endpoints, it cannot cross the horizontal axis in the interior, and
therefore has a constant sign on $[0,1]$.

By Claim\textit{~\ref{prob-meaning}}, this is equivalent to $%
H_{F}(u)=Q(u)+Q(1-u)-1$ having a constant sign on $[0,1]$, and thus on $[0,%
\tfrac{1}{2}]$. Hence \textup{(i)} applies and yields that $D_{0}$ has a
constant sign.

Finally, if 
\begin{equation}
x\longmapsto \frac{f(x)}{f(1-x)}  \label{1391}
\end{equation}%
is monotone, then 
\begin{equation}
f(x)-f(1-x)=f(1-x)\left( \frac{f(x)}{f(1-x)}-1\right)  \label{1392}
\end{equation}%
changes sign at most once, because $f(1-x)>0$. Thus the first part of 
\textup{(iv)} applies. \hfill
\end{proof}

\begin{remark}
\label{sufficient-conditions-discussion} The sufficient conditions of 
\textit{Claim~\ref{sufficient-conditions-summary}} admit the following
interpretation.

\begin{enumerate}
\item The strongest condition is the one-sign condition on 
\begin{equation}
H_{F}(u)=F^{-1}(u)+F^{-1}(1-u)-1.  \label{1431}
\end{equation}%
Under the regularity assumptions of \textit{Claim~\ref{prob-meaning}}, this
is equivalent to a global stochastic ordering between the starting law and
its reflection.

\item The two-lobe condition is the sharpest practically useful criterion
when $H_{F}$ changes sign exactly once on $(0,\tfrac{1}{2})$. It says that
the first lobe, after weighting by 
\begin{equation}
W_{F}(u)=F^{-1}(1-u)-F^{-1}(u),  \label{1432}
\end{equation}%
dominates the second one. Because $W_{F}$ is decreasing, the simpler
unweighted area comparison 
\begin{equation}
\int_{0}^{\xi }|H_{F}(u)|\,du\geq \int_{\xi }^{1/2}|H_{F}(u)|\,du
\label{1433}
\end{equation}%
already suffices.

\item In the multiple-lobe case, the monotone decrease of the unweighted
lobe areas of $H_{F}$ implies a monotone decrease of the weighted nodal
masses of $W_{F}H_{F}$, and this is enough to force strict loss of sign
changes at the first step.

\item At the density level, one-sided skewness relative to the reflection,
as expressed by the one-crossing property of 
\begin{equation}
f(x)-f(1-x),  \label{1434}
\end{equation}%
or more concretely by monotonicity of the ratio 
\begin{equation}
f(x)/f(1-x),  \label{1435}
\end{equation}%
is sufficient to enter the one-sign regime immediately.
\end{enumerate}

These conditions are all sufficient, not necessary. They are intended as
checkable hypotheses on the starting law $F$ which guarantee that the
process either begins in the one-sign regime or strictly reduces the number
of sign changes at the first step.
\end{remark}

\begin{remark}
\label{what-is-proved-and-open} The results above establish the following
rigorous chain.

\begin{enumerate}
\item The defect iteration is variation diminishing in the sense that 
\begin{equation}
Z(D_{n+1})\leq Z(D_{n})\qquad \forall n\geq 0.  \label{1436}
\end{equation}

\item Strict reduction of sign changes does not follow from Karlin theory
alone, because the \textit{Volterra kernel} is totally nonnegative but not
strictly totally positive.

\item Nevertheless, strict reduction can be forced by explicit nodal-mass
conditions, either at a general iterate via \textit{Corollary~\ref%
{strict-drop-iterates}}, or directly at the first step from the starting law
via \textit{Claim~\ref{sufficient-conditions-summary}}.

\item If the process reaches the one-sign regime at some step $N$, then the
same argument as in \textit{Lemma~\ref{defect-shape}} implies that all
subsequent defects $D_{n}$, $n\geq N$, have a constant sign, alternating in $%
n$, and therefore the scalar symmetry defect 
\begin{equation}
A_{n}=\int_{0}^{1}|D_{n}(x)|\,dx  \label{1437}
\end{equation}%
is available for the contraction analysis developed above.
\end{enumerate}

What remains open, in full generality, is a theorem ensuring that every
starting law enters the one-sign regime after finitely many steps. The
results of the present subsection provide rigorous sufficient conditions on $%
F$ for that entrance to occur immediately or after a provable strict
reduction of sign changes at the first step.
\end{remark}

\begin{flushleft}
\textbf{Properties of the compound maps }$\Phi _{n}(x)$
\end{flushleft}

We consider now the compound maps built from the inverses:

\begin{equation}
\Phi _{n}(x)=L_{0}^{-1}\circ L_{1}^{-1}\circ \cdots \circ
L_{n}^{-1}(x),\qquad x\in \lbrack 0,1]  \label{172a}
\end{equation}

and list several important properties of them.

Each $\Phi _{n}$ is a continuous strictly increasing self-map of $[0,1]$.
The following single-crossing structure for $\Phi _{n}$ is inherited from
that of the $L_{n}$.

\begin{lemma}
\label{Phi} \bigskip For each $n\geq 0$, the function $\Phi _{n}(x)$has two
fixed points at $0$ and $1$, and a unique fixed point in $(0,1)$. If we
denote the latter by $\phi _{n}$ then%
\begin{equation}
\Phi _{n}(x)>x\quad \text{for }x\in (0,\phi _{n}),\qquad \Phi _{n}(x)<x\quad 
\text{for }x\in (\phi _{n},1).  \label{172a1}
\end{equation}

Moreover, there exists a closed interval $I_{n}\subset (0,1)$ such that $%
\phi _{n}\in I_{n}$ and \ \ \ 
\begin{equation}
\Phi _{n}^{^{\prime }}(x)\leq 1\quad \text{for all }x\in I_{n}.
\label{172a2}
\end{equation}
\end{lemma}

\begin{proof}
We can notice that similar behavior to the one described in this claim holds
for the inverse $L_{n}^{-1}(x)$ in terms of the crossing and its general
shape by considering the previous claim and the fact that $L_{n}^{-1}(x)$ is
a mirror image of $L_{n}(x)$ with respect to the diagonal $x$. The current
claim essentially says that the composite function $\Phi _{n}(x)$ inherits
very generally behavior of all the inverses $L_{n}(x)$. It is not at all
obvious how exactly this should happen, and the proof below serves for that
purpose.

We will make the proof by induction. Initially, we consider the composition $%
L^{0,-1}(L^{1,-1}(x))$. We introduce now the deviations $\overline{h}%
_{i}(x)=x-L^{i,-1}(x)$ for $i=0,1$ and the one for the composite function $%
\overline{h}_{\Phi _{1}}(x)=x-\Phi _{1}(x)=x-L^{0,-1}(L^{1,-1}(x)).$ We will
also keep to the notation for the respective fixed points as well as their
properties from the previous claims.

First, we will prove that $\Phi _{1}(x)$ has a unique fixed point in $(0,1)$%
. We notice that we can write%
\begin{equation}
\overline{h}_{\Phi
_{1}}(x)=x-L^{0,-1}(L^{1,-1}(x))=x-L^{1,-1}(x)+L^{1,-1}(x)-L^{0,-1}(L^{1,-1}(x))=%
\overline{h}_{1}(x)+\overline{h}_{0}(L^{1,-1}(x)).  \label{173}
\end{equation}

We can define now $\overline{c}_{0}$ by $\overline{c}_{0}=L_{1}(c_{0})$.
Then for $\overline{h}_{1}$ we have: 
\begin{equation}
\overline{h}_{0}(L^{1,-1}(x))>0\quad \text{for }L^{1,-1}(x)>c_{0}\quad
\Longleftrightarrow \quad \overline{h}_{0}(L^{1,-1}(x))>0\quad \text{for }%
x\in (\overline{c}_{0},1),  \label{174}
\end{equation}%
and 
\begin{equation}
\overline{h}_{0}(L^{1,-1}(x))<0\quad \text{for }L^{1,-1}(x)<c_{0}\quad
\Longleftrightarrow \quad \overline{h}_{0}(L^{1,-1}(x))<0\quad \text{for }%
x\in (0,\overline{c}_{0}).  \label{175}
\end{equation}%
Without loss of generality, we can assume that the unique zero of $\overline{%
h}_{1}$ satisfies $c_{1}<\overline{c}_{0}$, i.e., $c_{1}<L_{1}(c_{0})$ (the
proof for the other case is analogous).

Now we have the following analysis for $\overline{h}_{\Phi _{1}}(x)=%
\overline{h}_{1}(x)+\overline{h}_{0}(L^{1,-1}(x))$ depending on $x:$

\begin{itemize}
\item \textbf{On\ }$(0,c_{1})$\textbf{:} Since $\overline{h}_{1}(x)$ is
negative for $x<c_{1}$ by the crossing property, and since $\overline{h}%
_{0}(L^{1,-1}(x))$ is also negative by the crossing property (because $%
x<c_{1}<\overline{c}_{0}$ implies $L^{1,-1}(x)<L^{1,-1}(c_{1})\leq L^{1,-1}(%
\overline{c}_{0})=c_{0}$), the sum of these two terms must be negative.
Therefore,$\overline{h}_{\Phi _{1}}(x)=\overline{h}_{1}(x)+\overline{h}%
_{0}(L^{1,-1}(x))$ $<0$;

\item \textbf{At }$x=c_{1}$\textbf{:} By definition, $\overline{h}%
_{1}(c_{1})=0$. Since $c_{1}<\overline{c}_{0}$ implies $%
L^{1,-1}(c_{1})<c_{0} $, we have $\overline{h}_{0}(L^{1,-1}(x))<0$. Thus, $%
\overline{h}_{\Phi _{1}}(x)=\overline{h}_{1}(x)+\overline{h}%
_{0}(L^{1,-1}(x)) $ $<0;$

\item \textbf{On }$(c_{1},\overline{c}_{0})$\textbf{:} In this interval, the
two components of $\overline{h}_{\Phi _{1}}(x)$ have opposing signs. Since $%
x>c_{1}$, we have $\overline{h}_{1}(x)>0$; conversely, since $x<\overline{c}%
_{0}$, it follows that $\overline{h}_{0}(L^{1,-1}(x))<0$. The existence of a
root is guaranteed by the \textit{Intermediate Value Theorem}. As
established previously, $h_{\Phi _{1}}(c_{1})<0$. Furthermore, $h_{\Phi
_{1}}(\overline{c}_{0})>0$ (as we will show in the next section). Given this
change in sign, there must be at least one point $\phi _{1}\in $ $(c_{1},%
\overline{c}_{0})$ such that $\overline{h}_{\Phi _{1}}(\phi _{1})=0$. The
uniqueness of this root follows from the strict monotonicity of $\overline{h}%
_{\Phi _{1}}(x)$, a property inherited from $\overline{h}_{0}$ and $%
\overline{h}_{1}$. At this unique point $\phi _{1}$, the positive
contribution from $\overline{h}_{1}$ is precisely balanced by the negative
contribution from $\overline{h}_{0}$;

\item \textbf{At }$x=\overline{c}_{0}$\textbf{:} By definition, $\overline{h}%
_{0}(\overline{c}_{0})=0$. Also, for $x>c_{1}$ we already have $\overline{h}%
_{1}(x)>0$. Thus, $\overline{h}_{\Phi _{1}}(x)=\overline{h}_{1}(x)+\overline{%
h}_{0}(L^{1,-1}(x))$ $>0;$

\item \textbf{On }$(\overline{c}_{0},1)$\textbf{:} For $x>\overline{c}_{0}$,
we have $L^{1,-1}(x)>c_{0}$ and hence $\overline{h}_{0}(L^{1,-1}(x))>0$.
Also, for $x>c_{1}$ we already have $\overline{h}_{1}(x)>0$. Therefore, $%
\overline{h}_{\Phi _{1}}(x)=\overline{h}_{1}(x)+\overline{h}%
_{0}(L^{1,-1}(x))>0.$
\end{itemize}

Second, we will prove that $\Phi _{1}(x)>x$ holds in the interval $(0,\phi
_{1})$ and $\Phi _{1}(x)<x$ holds in the interval $(\phi _{1},1).$ For that
purpose, we want to show that the equation $\overline{h}_{\Phi
_{1}}^{^{\prime }}(x)=0$ has exactly two solutions in $(0,1)$.

Regarding the existence, we already know from above that $\overline{h}_{\Phi
_{1}}(\phi _{1})=0$. This allows to have:

\begin{itemize}
\item On the interval $[0,\phi _{1}]$, since $\overline{h}_{\Phi
_{1}}(0)=0\quad $and$\quad \overline{h}_{\Phi _{1}}(\phi _{1})=0,$ the 
\textit{Rolle's Theorem} guarantees the existence of some $\phi
_{1}^{(1)}\in (0,\phi _{1})$ such that $\overline{h}_{\Phi _{1}}^{^{\prime
}}(\phi _{1}^{(1)})=0;$

\item Similarly, on the interval $[\phi _{1},1]$, since $\overline{h}_{\Phi
_{1}}(\phi _{1})=0\quad $and$\quad \overline{h}_{\Phi _{1}}(1)=0$, there
exists $\phi _{1}^{(2)}\in (\phi _{1},1)$ such that $\overline{h}_{\Phi
_{1}}^{^{\prime }}(\phi _{1}^{(2)})=0.$
\end{itemize}

Regarding the uniqueness, we use a contradiction argument. Take first the
interval $(0,\phi _{1})$. Let's assume that there exist two distinct points $%
a$ and $b$ in it such that $\overline{h}_{\Phi _{1}}^{^{\prime }}(a)=%
\overline{h}_{\Phi _{1}}^{^{\prime }}(b)=0$ holds. Then again by the \textit{%
Rolle's Theorem}, there is a $c\in (a,b)$ with $\overline{h}_{\Phi
_{1}}^{^{\prime \prime }}(c)=0$. But we know that $\overline{h}_{\Phi
_{1}}^{^{\prime }}(x)=\overline{h}_{1}^{^{\prime }}(x)+\overline{h}%
_{0}^{\prime }(L^{1,-1}(x))\,[L^{1,-1}(x)]^{^{\prime }}$ and also $\overline{%
h}_{\Phi _{1}}^{^{\prime \prime }}(x)=\overline{h}_{1}^{^{\prime \prime
}}(x)+\overline{h}_{0}^{\prime \prime }(L^{1,-1}(x))\,\left[
[L^{1,-1}(x)]^{\prime }\right] ^{2}+\overline{h}_{0}^{\prime
}(L^{1,-1}(x))\,[L^{1,-1}(x)]^{^{\prime \prime }}$. Since each deviation
function $\overline{h}_{0}$ and $\overline{h}_{1}$ has exactly one local
maximum on $(0,c_{0})$ and $(0,c_{1})$, their derivatives are strictly
monotone in these intervals. This prevents $\overline{h}_{\Phi
_{1}}^{^{\prime \prime }}(x)$ to be zero and thus we get a contradiction.
Hence, there is at most one solution to $\overline{h}_{\Phi _{1}}^{^{\prime
}}(x)=0$ in $(0,\phi _{1})$. The argument for the interval $(\phi _{1},1)$
is the same.

A direct computation gives%
\begin{eqnarray}
\overline{h}_{\Phi _{1}}^{^{\prime }}(x) &=&1-(L^{0,-1})^{^{\prime
}}(L^{1,-1}(x))(L^{1,-1})^{\prime }(x)=1-\frac{1}{(L_{0}{}^{^{\prime
}})(L^{0,-1}(L^{1,-1}(x)))}\frac{1}{(L_{1}{}^{^{\prime }})(L^{1,-1}(x))} 
\notag \\
&=&1-\frac{\dint\nolimits_{0}^{1}w[F^{-1}(L^{0,-1}(L^{1,-1}(x)))]du}{%
w[F^{-1}(L^{0,-1}(L^{1,-1}(x)))]}\frac{\dint%
\nolimits_{0}^{1}w[L^{0,-1}(L^{1,-1}(x))]du}{w[L^{0,-1}(L^{1,-1}(x))]}.
\label{175.1}
\end{eqnarray}

Hence we have $\overline{h}_{\Phi _{1}}^{^{\prime }}(0)=-\infty $ and $%
\overline{h}_{\Phi _{1}}^{^{\prime }}(1)=-\infty $. We can also notice that
around $0.5$ there is mass concentration which produces at least one
positive value for $h_{\Phi _{1}}^{^{\prime }}(x)$. Therefore at $\phi
_{1}^{(1)}$ and $\phi _{1}^{(2)}$, the function $h_{\Phi _{1}}^{^{\prime
}}(x)$ changes its sign which determine the behavior $\Phi _{n}(x)$ and its
crossing properties. Namely, as hinted at the beginning, the positioning of
the $\Phi _{1}(x)$ above and below the diagonal is analogous to the
situation of $L^{n,-1}(x)$ for any $n$. The crossing pattern is: $\Phi
_{1}(x)>x$ for $x\in (0,\phi _{1})$ and $\Phi _{1}(x)<x$ for $x\in (\phi
_{1},1)\,$, while at each of these intervals the points $\phi _{1}^{(1)}$
and $\phi _{1}^{(2)}$ just produce a maximal distance between $\Phi _{1}(x)$
and the diagonal. Furthermore, the described behavior of $\overline{h}_{\Phi
_{1}}^{^{\prime }}(x)$ gives that $\overline{h}_{\Phi _{1}}^{^{\prime
}}(x)>0,$ and thus $\Phi _{1}^{^{\prime }}(x)<1$, hold for $x\in (\phi
_{1}^{(1)},\phi _{1}^{(2)})$ and $\overline{h}_{\Phi _{1}}^{^{\prime
}}(x)<0, $ and thus $\Phi _{1}^{^{\prime }}(x)>1,$ hold for $x\in \lbrack
0,\phi _{1}^{(1)})\cup (\phi _{1}^{(2)},1]$ with the corresponding zeros
attained at $\phi _{1}^{(1)}$ and $\phi _{1}^{(2)}$. So the interval $I_{1}$
is $[\phi _{1}^{(1)},\phi _{1}^{(2)}]$.

Now we can move to the case $n$. From $\Phi
_{n}(x)=L^{0,-1}(...(L^{n-1,-1}(L^{n,-1}(x)))$ we have $\Phi _{n}(x)=\Phi
_{n-1}(L^{n,-1}(x))$. Either inductively, or simply by change of notation,
we can notice that everything proved above also holds for $\Phi _{n}(x)$.
This is essentially due to the again appearing bifunctional composite form
of $\Phi _{n}(x)$ which was the case also for $\Phi
_{1}(x)=L^{0,-1}(L^{1,-1}(x))$. Exactly, it allows to make the induction
step by assuming for $n$ that $\Phi _{n}(x)$ obeys the conditions of the
claim and then prove them for $n+1$. The proof will verbally reiterate the
proof above for $n=1$ with a change of notation to reflect the fact that
here we work with $\Phi _{n-1}$ instead of $\Phi _{0}\equiv $ $L^{0,-1}$.

We may also finally notice that if we set $\Psi _{n}(x)=\Phi _{n}^{-1}(x)$
then $\Psi _{n}(x)=L_{n}(...(L_{1}(L_{0}(x)))$. The function $\Psi _{n}(x)$
resembles the crossing behavior of $L_{n}(x)$. Namely, $\Psi _{n}(x)$ has
two fixed points at $0$ and $1$. Additionally, it has a unique fixed point
in $(0,1)$. If we denote the latter by $\psi _{n}$, then $\Psi _{n}(x)<x$
holds in the interval $(0,\psi _{n})$ and $\Psi _{n}(x)>x$ holds in the
interval $(\psi _{n},1)$. There is also a closed interval $I_{n}\subset
(0,1) $ such that $\psi _{n}\in I_{n}$ and the derivative satisfies $\Psi
_{n}^{^{\prime }}(x)\geq 1$ for all $x\in I_{n}$.

The proof of the preceding lemma also yields the following immediate
corollary.
\end{proof}

\begin{corollary}
\label{Phi-c} The inequality $\Phi _{n}^{^{\prime }}(x)<1$ holds for all $x$
in the interval $(\phi _{n}^{(1)},\phi _{n}^{(2)})$.
\end{corollary}

We now relate the evolution of the $\Phi _{n}$ to that of the crossing
points $c_{n}$ from \textit{Lemma~\ref{l-cross}}.

\begin{lemma}
\label{dir} For every $n\geq 0$ and $x\in (0,1)$, 
\begin{equation}
sgn(\Phi _{n+1}(x)-\Phi _{n}(x))=sgn(T_{n+1}(x)-x)=sgn(c_{n+1}-x).
\label{526}
\end{equation}
\end{lemma}

\begin{proof}
\noindent By definition, 
\begin{equation}
\Phi _{n+1}(x)=\Phi _{n}(T_{n+1}(x)).  \label{526.1}
\end{equation}%
Since $\Phi _{n}$ is strictly increasing, the sign of $\Phi _{n+1}(x)-\Phi
_{n}(x)$ is the same as the sign of $T_{n+1}(x)-x$, i.e. 
\begin{equation}
sgn(\Phi _{n+1}(x)-\Phi _{n}(x))=sgn(T_{n+1}(x)-x).  \label{526.2}
\end{equation}

It remains to relate the sign of $T_{n+1}(x)-x$ to the position of $x$
relative to $c_{n+1}$. Because $T_{n+1}$ is the inverse of $L_{n+1}$ and $%
L_{n+1}$ is strictly increasing, we have 
\begin{equation}
T_{n+1}(x)>x\iff L_{n+1}(T_{n+1}(x))>L_{n+1}(x)\iff x>L_{n+1}(x).
\label{528}
\end{equation}%
By \textit{Lemma~\ref{l-cross}}, $L_{n+1}(x)<x$ for $x<c_{n+1}$ and $%
L_{n+1}(x)>x$ for $x>c_{n+1}$. Thus 
\begin{equation}
T_{n+1}(x)>x\iff x<c_{n+1}.  \label{529}
\end{equation}%
Similarly, $T_{n+1}(x)<x$ if and only if $x>c_{n+1}$. Therefore 
\begin{equation}
sgn(T_{n+1}(x)-x)=sgn(c_{n+1}-x),  \label{530}
\end{equation}%
which combined with the first identity gives (\ref{526}).
\end{proof}

We may observe that \textit{Lemma} \ref{dir} gives that as $n$ grows, the
evolution of $\Phi _{n}(x)$ is tightly controlled by the motion of the $%
c_{n} $'s. Next, we relate the latter to the $\phi _{n}$'s.

\begin{lemma}
\label{alignment} For each $n\geq 0$, the fixed point $\phi _{n+1}$ lies
between $\phi _{n}$ and $c_{n+1}$ 
\begin{equation}
\min \{\phi _{n},c_{n+1}\}\;\leq \;\phi _{n+1}\;\leq \;\max \{\phi
_{n},c_{n+1}\}.  \label{531}
\end{equation}
\end{lemma}

\begin{proof}
Define $k_{n+1}(x)=-\overline{h}_{\Phi _{n}}=\Phi _{n+1}(x)-x$. Then $%
k_{n+1} $ is continuous, strictly negative on $(\phi _{n+1},1)$ and strictly
positive on $(0,\phi _{n+1})$, and has a unique zero at $x=\phi _{n+1}$. We
evaluate $k_{n+1}$ at $\phi _{n}$ and $c_{n+1}$.

At $x=\phi _{n}$ we have 
\begin{equation}
k_{n+1}(\phi _{n})=\Phi _{n+1}(\phi _{n})-\phi _{n}=\Phi _{n+1}(\phi
_{n})-\Phi _{n}(\phi _{n}),  \label{532}
\end{equation}%
so by \textit{Lemma}~\ref{dir} with $x=\phi _{n}$, 
\begin{equation}
sgn(k_{n+1}(\phi _{n}))=sgn(\Phi _{n+1}(\phi _{n})-\Phi _{n}(\phi
_{n}))=sgn(c_{n+1}-\phi _{n}).  \label{533}
\end{equation}

At $x=c_{n+1}$ we get 
\begin{equation}
k_{n+1}(c_{n+1})=\Phi _{n+1}(c_{n+1})-c_{n+1}=\Phi
_{n}(T_{n+1}(c_{n+1}))-c_{n+1}=\Phi _{n}(c_{n+1})-c_{n+1}.  \label{534}
\end{equation}%
By the single-crossing property for $\Phi _{n}$ from \textit{Lemma} \ref{Phi}%
, we have 
\begin{equation}
sgn(k_{n+1}(c_{n+1}))=sgn(\Phi _{n}(c_{n+1})-c_{n+1})=sgn(\phi _{n}-c_{n+1}).
\label{535}
\end{equation}

Thus $k_{n+1}(\phi _{n})$ and $k_{n+1}(c_{n+1})$ have opposite signs. Since $%
k_{n+1}$ is continuous and has a unique zero $\phi _{n+1}$, it follows that $%
\phi _{n+1}$ lies strictly between $\phi _{n}$ and $c_{n+1}$, and this is
exactly (\ref{531}).
\end{proof}

\begin{remark}
\textit{Lemma}~\ref{alignment} formalizes the idea that the new fixed point
of the compound map $\Phi _{n+1}$ is \textquotedblleft
squeezed\textquotedblright\ between the previous compound fixed point $\phi
_{n}$ and the fresh base fixed point $c_{n+1}$. Over many iterations, this
repeated squeezing forces the sequences $\{\phi _{n}\}_{n=0}^{\infty }$ and $%
\{c_{n}\}_{n=0}^{\infty }$ to track each other closely.
\end{remark}

We have even a stronger result, which comes as a corollary of \textit{Lemma} %
\ref{Phi}.

\begin{corollary}
\label{Phi-strong-align} For every $n\geq 0$, the unique interior fixed
point $\phi _{n+1}\in (0,1)$ of $\Phi _{n+1}$ satisfies 
\begin{equation}
\min \{\,L_{n+1}(c_{n}),\,c_{n+1}\,\}\ \leq \ \phi _{n+1}\ \leq \ \max
\{\,L_{n+1}(c_{n}),\,c_{n+1}\,\}.  \label{536}
\end{equation}
\end{corollary}

\begin{proof}
\bigskip This is exactly the localization of the unique zero of the
deviation function $k_{n+1}(x)=\Phi _{n+1}(x)-x$ that is obtained inside the
proof of \textit{Lemma} \ref{Phi}. Indeed, in that proof we analyze the sign
contributions in $k_{n+1}$ and shows (after possibly exchanging the roles of
the two endpoints) that $k_{n+1}(c_{n+1})\quad $and$\quad
k_{n+1}(L_{n+1}(c_{n}))$ have opposite signs. Since $k_{n+1}$ is continuous
and has a unique zero in $(0,1)$ (namely $\phi _{n+1}$, by \textit{Lemma} %
\ref{Phi}), the \textit{Intermediate Value Theorem} implies that this zero
must lie between the two points $c_{n+1}$ and $L_{n+1}(c_{n})$. This yields (%
\ref{536}).
\end{proof}

\begin{flushleft}
\textbf{Local equicontinuity and subsequential uniform convergence of }$\Phi
_{n}(x)$
\end{flushleft}

We begin by establishing a basic contraction principle for increasing
self-maps with a single interior fixed point. The proof relies on two key
results: a \textit{Fejer-type inequality} (see \cite{[51]}) and a \textit{%
1-Lipschitz bound}.

\begin{lemma}
\label{Fejer} Let $f:[0,1]\rightarrow \lbrack 0,1]$ be continuous, strictly
increasing, with a unique interior fixed point $p\in (0,1)$ and
single-crossing pattern 
\begin{equation}
f(x)>x\ \ (x<p),\qquad f(x)<x\ \ (x>p).  \label{545}
\end{equation}%
Then:

\begin{enumerate}
\item For every $x\in \lbrack 0,1]$ 
\begin{equation}
|f(x)-p|\leq |x-p|,  \label{546}
\end{equation}

with strict inequality if $x\neq p$;

\item If $x,y$ lie on the same side of $p$ (both $\leq p$ or both $\geq p$),
then 
\begin{equation*}
|f(x)-f(y)|\leq |x-y|.
\end{equation*}
\end{enumerate}
\end{lemma}

\begin{proof}
\noindent If $x=p$ the inequality is trivial. Suppose $x>p$. Then $f(x)\in
\lbrack p,x)$, hence $0\leq f(x)-p<x-p$, so $|f(x)-p|<|x-p|$. If $x<p$, then 
$f(x)\in (x,p]$, so $0\leq p-f(x)<p-x$ and again $|f(x)-p|<|x-p|$. That
proves the first part of the lemma.

For the second part, suppose $p\leq x<y$; the case $x<y\leq p$ is analogous.
Then $f(x)\geq p$ and $f(y)\leq y$. Also $f$ increasing implies $f(x)\leq
f(y)$. Thus $0\leq f(y)-f(x)\leq y-x$. So 
\begin{equation}
|f(y)-f(x)|\leq |y-x|.  \label{546.1}
\end{equation}
\end{proof}

We can now establish subsequential convergence of the compound maps $\Phi
_{n}$.

\begin{lemma}
\label{Phi-1Lip} For every $n\geq 0$ and all $x,y\in \lbrack 0,1]$, 
\begin{equation}
|\Phi _{n}(x)-\Phi _{n}(y)|\leq |x-y|.  \label{860}
\end{equation}%
In particular, the family $(\Phi _{n})_{n\geq 0}$ is equicontinuous and
uniformly bounded on $[0,1]$.
\end{lemma}

\begin{proof}
Fix $n$. By \textit{Lemma~\ref{Phi}}, $\Phi _{n}$ is continuous, strictly
increasing, and has a unique interior fixed point $\phi _{n}$ with the
single-crossing pattern 
\begin{equation}
\Phi _{n}(x)>x\ (x<\phi _{n}),\qquad \Phi _{n}(x)<x\ (x>\phi _{n}).
\label{861}
\end{equation}%
Apply \textit{Lemma~\ref{Fejer}} with $f=\Phi _{n}$ and $p=\phi _{n}$. If $%
x,y$ lie on the same side of $\phi _{n}$, then \textit{Lemma~\ref{Fejer} }%
gives 
\begin{equation}
|\Phi _{n}(x)-\Phi _{n}(y)|\leq |x-y|.  \label{870}
\end{equation}%
If $x<\phi _{n}<y$, then by triangle inequality and \textit{Lemma~\ref{Fejer}
}on each side, 
\begin{equation}
|\Phi _{n}(x)-\Phi _{n}(y)|\leq |\Phi _{n}(x)-\Phi _{n}(\phi _{n})|+|\Phi
_{n}(\phi _{n})-\Phi _{n}(y)|\leq |x-\phi _{n}|+|y-\phi _{n}|=|x-y|.
\label{871}
\end{equation}%
This proves (\ref{860}). Uniform boundedness is clear since $\Phi
_{n}([0,1])\subset \lbrack 0,1]$. \medskip
\end{proof}

Next, we prove an auxiliary result using the fact that all the crossing
point of the subsequential limits of $L_{n}$ are the same.

\begin{claim}
\noindent \label{Ln-cn-shift} Assume that 
\begin{equation}
c_{n}\longrightarrow c^{\ast }\in (0,1).  \label{541a}
\end{equation}%
Then 
\begin{equation}
\left\vert L_{n+1}(c_{n})-c_{n+1}\right\vert \longrightarrow 0\qquad \text{%
as }n\rightarrow \infty .  \label{541}
\end{equation}
\end{claim}

\begin{proof}
Suppose, by contradiction, that the above conclusion is false. Then there
exist $\varepsilon _{0}>0$ and a strictly increasing sequence of indices $%
\{n_{k}\}_{k\geq 1}$ such that 
\begin{equation}
\left\vert L_{n_{k}+1}(c_{n_{k}})-c_{n_{k}+1}\right\vert \geq \varepsilon
_{0}\qquad \text{for all }k.  \label{542}
\end{equation}

By \textit{Claim \ref{AA}}, the family $\{L_{n}\}_{n\geq 0}$ is relatively
compact in $C([0,1])$. Hence, after passing to a further subsequence (not
relabeled), there exists a continuous increasing function $L^{\ast
}:[0,1]\rightarrow \lbrack 0,1]$ such that 
\begin{equation}
L_{n_{k}+1}\longrightarrow L^{\ast }\quad \text{uniformly on }[0,1].
\label{543}
\end{equation}

Since $c_{n}\rightarrow c^{\ast }$ by assumption, we also have 
\begin{equation}
c_{n_{k}}\longrightarrow c^{\ast },\qquad c_{n_{k}+1}\longrightarrow c^{\ast
}.  \label{544}
\end{equation}

For each $k$, the point $c_{n_{k}+1}$ is a fixed point of $L_{n_{k}+1}$,
hence 
\begin{equation}
L_{n_{k}+1}(c_{n_{k}+1})=c_{n_{k}+1}.  \label{544a}
\end{equation}%
Passing to the limit using (\ref{543}) yields 
\begin{equation}
L^{\ast }(c^{\ast })=c^{\ast }.  \label{544b}
\end{equation}

Again using uniform convergence and $c_{n_{k}}\rightarrow c^{\ast }$, we
obtain 
\begin{equation}
L_{n_{k}+1}(c_{n_{k}})\longrightarrow L^{\ast }(c^{\ast })=c^{\ast }.
\label{545c}
\end{equation}%
Together with $c_{n_{k}+1}\rightarrow c^{\ast }$, this implies 
\begin{equation}
\left\vert L_{n_{k}+1}(c_{n_{k}})-c_{n_{k}+1}\right\vert \longrightarrow 0,
\label{545d}
\end{equation}%
which contradicts (\ref{542}). Therefore, (\ref{541}) holds.
\end{proof}

We now use the convergence of $c_{n}$ and \textit{Lemma ~\ref{alignment}} to
deduce convergence of $\phi _{n}$.

\begin{claim}
\label{cluster} The cluster set of $\{\phi _{n}\}_{n\geq 0}$ coincides with
the cluster set of $\{c_{n}\}_{n\geq 0}$. In particular, if $%
c_{n}\rightarrow c^{\ast }\in (0,1)$, then 
\begin{equation}
\phi _{n}\rightarrow \phi ^{\ast }=c^{\ast }.  \label{536a}
\end{equation}
\end{claim}

\begin{proof}
The sequence $\{\phi _{n}\}_{n\geq 0}$ is contained in $[0,1]$, hence has at
least one cluster point. Let $\phi $ be an arbitrary cluster point of $%
\{\phi _{n}\}_{n\geq 0}^{\infty }$. Then there exists a strictly increasing
sequence of indices $\{n_{k}\}_{k\geq 1}$ such that 
\begin{equation}
\phi _{n_{k}}\rightarrow \phi .  \label{537}
\end{equation}

Since $c_{n}\rightarrow c^{\ast }$ by assumption, we also have 
\begin{equation}
c_{n_{k}+1}\rightarrow c^{\ast }.  \label{538}
\end{equation}

Next, from \textit{Corollary \ref{Phi-strong-align}, }we have%
\begin{equation}
\ \left\vert \phi _{n+1}-c_{n+1}\right\vert \ \leq \,\left\vert
L_{n+1}(c_{n})-\,c_{n+1}\right\vert \,\   \label{545e}
\end{equation}

and thus also%
\begin{equation}
\ \left\vert \phi _{n_{k}+1}-c_{n_{k}+1}\right\vert \ \leq \,\left\vert
L_{n_{k}+1}(c_{n_{k}})-\,c_{n_{k}+1}\right\vert \,.\   \label{545f}
\end{equation}

Passing to the limit along $k\rightarrow \infty $ and using (\ref{541}) from 
\textit{Claim \ref{Ln-cn-shift}} yields $\phi =c^{\ast }$. Consequently, the
sequence $\,\{\phi _{n}\}_{n=0}^{\infty }$ cannot oscillate between distinct
cluster points; it must converge to the unique point $c^{\ast }$. Thus, the
cluster set is the singleton $\{c^{\ast }\}$, and $\phi _{n}${}$%
\longrightarrow c^{\ast }$.
\end{proof}

Next, we establish the pointwise convergence of $\Phi _{n}(x)$ using a
monotonicity argument that relies on \textit{Lemma}~\ref{dir} and \textit{%
Lemma}~\ref{crossings-half}.

\begin{lemma}
\label{Phi-pointwise} For every fixed $x\in (0,1)$, the sequence $\{\Phi
_{n}(x)\}_{n=0}^{\infty }$ converges pointwise 
\begin{equation}
\lim_{n\rightarrow \infty }\Phi _{n}(x)=\Phi _{\infty }(x),  \label{546a}
\end{equation}

where we denoted the limit by $\Phi _{\infty }(x)$.
\end{lemma}

\begin{proof}
We consider two cases.

\medskip \noindent \textbf{Case 1:}\emph{\ $x\neq $ }$c^{\ast }$\emph{.} By 
\textit{Lemma}~\ref{dir}, 
\begin{equation}
sgn(\Phi _{n+1}(x)-\Phi _{n}(x))=sgn(c_{n+1}-x).  \label{547}
\end{equation}

Since by \textit{Lemma}~\ref{crossings-half} $c_{n}\rightarrow c^{\ast }$,
with $c^{\ast }=\frac{1}{2}$, there exists $N$ such that for all $n\geq N$, 
\begin{equation*}
\func{sgn}(c_{n+1}-x)=\func{sgn}(c^{\ast }-x),
\end{equation*}%
i.e.,\ the sign of $c_{n+1}-x$ does not change for $n\geq N$. For $n$
sufficiently large it is constant: positive if $x<c^{\ast }$ and negative if 
$x>c^{\ast }$. Thus, for large $n$, the sequence $\{\Phi
_{n}(x)\}_{n=0}^{\infty }$ is monotone. Being bounded in $[0,1]$, it must
converge. Let 
\begin{equation}
\lim_{n\rightarrow \infty }\Phi _{n}(x)=\Phi _{\infty }(x).  \label{547a}
\end{equation}%
We still need to identify $\Phi _{\infty }(x)$.

\medskip \noindent \textbf{Case 2:}\emph{\ $x=$ }$c^{\ast }$\emph{.}
Applying \textit{Lemma}~\ref{Fejer} to $f=\Phi _{n}$ and $p=\phi _{n}$ with $%
x=\phi ^{\ast }=c^{\ast }$ yields 
\begin{equation}
|\Phi _{n}(\phi ^{\ast })-\phi _{n}|\leq |\phi ^{\ast }-\phi _{n}|.
\label{548}
\end{equation}%
Hence 
\begin{equation}
|\Phi _{n}(c^{\ast })-c^{\ast }|=|\Phi _{n}(\phi ^{\ast })-\phi ^{\ast
}|\leq |\Phi _{n}(\phi ^{\ast })-\phi _{n}|+|\phi _{n}-\phi ^{\ast }|\leq
2|\phi _{n}-\phi ^{\ast }|.  \label{549}
\end{equation}%
Since $\phi _{n}\rightarrow \phi ^{\ast }$, the right-hand side tends to $0$%
, and thus $\Phi _{n}(c^{\ast })\rightarrow c^{\ast }.$

Combining the two cases shows that for every $x\in (0,1)$ the limit $%
\lim_{n\rightarrow \infty }\Phi _{n}(x)$ exists.
\end{proof}

\begin{remark}
\label{convrem}Far from $c^{\ast }$, the evolution of $\Phi _{n}(x)$ is
essentially monotone: once the crossing points $c_{n}$ have moved past $x$
in one direction, the increments $\Phi _{n+1}(x)-\Phi _{n}(x)$ keep the same
sign. At $x=\phi ^{\ast }$, the fixed points $\phi _{n}$ act as anchors:
each $\Phi _{n}$ pulls $x$ towards $\phi _{n}$, and as $\phi _{n}\rightarrow
\phi ^{\ast }$, this pull becomes more and more concentrated, forcing $\Phi
_{n}(\phi ^{\ast })$ to converge to $\phi ^{\ast }$ as well.
\end{remark}

\textit{Remark} \ref{convrem} leads us to conjecture a stronger result: (i)
the convergence of $\Phi _{n}(x)$ to $\Phi _{\infty }(x)$ is uniform, (ii) $%
c^{\ast }=\phi ^{\ast }=\frac{1}{2}$, (iii) $\Phi _{\infty }(x)$ is the
constant $\phi ^{\ast }$.

We now present a general argument for subsequential convergence that does
not rely on the previously established pointwise convergence, after which we
will establish the connection between these two modes of convergence within
the context of our problem.

\begin{lemma}
\label{Phi-AA-subseq} Let $\{n_{k}\}_{k\geq 0}$ be any strictly increasing
sequence with $n_{k}\rightarrow \infty $. Then there exists a further
subsequence (still denoted $\{n_{k}\}_{k\geq 0}$) and a map 
\begin{equation}
\Phi ^{\ast }:(0,1)\rightarrow (0,1)  \label{872}
\end{equation}%
such that for every $0<\delta <\tfrac{1}{2}$, 
\begin{equation}
\sup_{x\in \lbrack \delta ,\,1-\delta ]}\left\vert \Phi _{n_{k}}(x)-\Phi
^{\ast }(x)\right\vert \longrightarrow 0\qquad (k\rightarrow \infty ).
\label{873}
\end{equation}%
In particular, $\Phi ^{\ast }$ is continuous and increasing on $(0,1)$.
\end{lemma}

Moreover, the fixed points $\{\phi _{n_{k}}\}_{k\geq 0}$ have a further
subsequence (still denoted $\{\phi _{n_{k}}\}_{k\geq 0}$) converging to some 
$\phi ^{\ast }\in \lbrack 0,1]$. If in addition $\phi ^{\ast }\in (0,1)$,
then 
\begin{equation}
\Phi ^{\ast }(\phi ^{\ast })=\phi ^{\ast }.  \label{874}
\end{equation}

\begin{proof}
By \textit{Lemma \ref{Phi-1Lip}}, each $\Phi _{n}:[0,1]\rightarrow \lbrack
0,1]$ is $1$\textit{--Lipschitz}, hence the family $\{\Phi _{n}\}_{n\geq 0}$
is equicontinuous and uniformly bounded on $[0,1]$.

\smallskip \noindent \textbf{Step 1: Uniform convergence on one compact
subinterval.} Fix $0<\delta <\tfrac{1}{2}$ and set $I_{\delta }=[\delta
,1-\delta ]$. The restriction family $\{\Phi _{n_{k}}|_{I_{\delta
}}\}_{k\geq 0}$ is equicontinuous and uniformly bounded on the compact
interval $I_{\delta }$. By \textit{Arzel\`{a}--Ascoli}, there exists a
subsequence of indices $\{n_{k}^{(\delta )}\}_{k\geq 0}$ and a continuous
function $\Phi ^{(\delta )}:I_{\delta }\rightarrow \lbrack 0,1]$ such that 
\begin{equation}
\sup_{x\in I_{\delta }}\left\vert \Phi _{n_{k}^{(\delta )}}(x)-\Phi
^{(\delta )}(x)\right\vert \rightarrow 0\qquad (k\rightarrow \infty ).
\label{875}
\end{equation}%
Since each $\Phi _{n}$ is increasing, the uniform limit $\Phi ^{(\delta )}$
is also increasing.

\smallskip \noindent \textbf{Step 2: Choosing a nested exhaustion of $(0,1)$.%
} Define 
\begin{equation}
\delta _{m}=2^{-m}\qquad (m\geq 2),\qquad I_{m}=I_{\delta _{m}}=[\delta
_{m},\,1-\delta _{m}].  \label{875a}
\end{equation}%
Then $I_{m}$ are compact intervals with 
\begin{equation}
I_{2}\subset I_{3}\subset I_{4}\subset \cdots ,\qquad \bigcup_{m\geq
2}I_{m}=(0,1).  \label{875b}
\end{equation}%
(Indeed, for any $x\in (0,1)$ choose $m$ large so that $\delta
_{m}<x<1-\delta _{m}$.)

\smallskip \noindent \textbf{Step 3: Iterated extractions.} We now build a
chain of subsequences by induction on $m$.

\begin{itemize}
\item Start with the original subsequence $\{n_{k}^{(1)}\}_{k\geq
0}=\{n_{k}\}_{k\geq 0}$;

\item For $m=2$: apply Step 1 with $\delta =\delta _{2}$ to the sequence $%
\{\Phi _{n_{k}^{(1)}}\}_{k\geq 0}$ on $I_{2}$. Obtain a subsequence $%
\{n_{k}^{(2)}\}_{k\geq 0}\subset \{n_{k}^{(1)}\}_{k\geq 0}$ and a continuous
increasing limit $\Phi ^{(2)}:I_{2}\rightarrow \lbrack 0,1]$ such that 
\begin{equation}
\sup_{x\in I_{2}}\left\vert \Phi _{n_{k}^{(2)}}(x)-\Phi ^{(2)}(x)\right\vert
\rightarrow 0;  \label{875c}
\end{equation}

\item For $m=3$: apply Step 1 with $\delta =\delta _{3}$ to the sequence $%
\{\Phi _{n_{k}^{(2)}}\}_{k\geq 0}$ on $I_{3}$. Obtain a subsequence $%
\{n_{k}^{(3)}\}_{k\geq 0}\subset \{n_{k}^{(2)}\}_{k\geq 0}$ and a continuous
increasing limit $\Phi ^{(3)}:I_{3}\rightarrow \lbrack 0,1]$ such that 
\begin{equation}
\sup_{x\in I_{3}}\left\vert \Phi _{n_{k}^{(3)}}(x)-\Phi ^{(3)}(x)\right\vert
\rightarrow 0;  \label{875d}
\end{equation}

\item Continue inductively: having constructed $\{n_{k}^{(m-1)}\}_{k\geq 0}$%
, apply Step 1 with $\delta =\delta _{m}$ on $I_{m}$ to obtain a further
subsequence $\{n_{k}^{(m)}\}_{k\geq 0}\subset \{n_{k}^{(m-1)}\}_{k\geq 0}$
and a continuous increasing limit $\Phi ^{(m)}:I_{m}\rightarrow \lbrack 0,1]$
such that 
\begin{equation}
\sup_{x\in I_{m}}\left\vert \Phi _{n_{k}^{(m)}}(x)-\Phi ^{(m)}(x)\right\vert
\rightarrow 0.  \label{875e}
\end{equation}
\end{itemize}

Thus we have nested subsequences 
\begin{equation}
\{n_{k}^{(2)}\}\supset \{n_{k}^{(3)}\}\supset \cdots  \label{875f}
\end{equation}%
and uniform convergence on $I_{m}$ along $\{n_{k}^{(m)}\}$.

\smallskip \noindent \textbf{Step 4: The diagonal subsequence (Cantor
Diagonalization).}

The subsequences constructed above are chosen nested, that is, 
\begin{equation}
\{n_{k}^{(m+1)}\}_{k\geq 0}\subset \{n_{k}^{(m)}\}_{k\geq 0}\qquad \text{for
every }m\geq 2.  \label{875f1}
\end{equation}%
After discarding finitely many terms if necessary, define the diagonal
subsequence $\{n_{k}^{\mathrm{diag}}\}_{k\geq 0}$ by 
\begin{equation}
n_{k}^{\mathrm{diag}}=n_{k}^{(k+2)},\qquad k\geq 0.  \label{875g}
\end{equation}%
Fix $m\geq 2$. If $k\geq m-2$, then $k+2\geq m$, and by the nestedness of
the subsequences we have 
\begin{equation}
n_{k}^{\mathrm{diag}}=n_{k}^{(k+2)}\in \{n_{j}^{(k+2)}:j\geq 0\}\subset
\{n_{j}^{(m)}:j\geq 0\}.  \label{875h}
\end{equation}%
Thus, for every fixed $m\geq 2$, the tail of the diagonal subsequence $%
\{n_{k}^{\mathrm{diag}}\}_{k\geq 0}$ is a subsequence of $%
\{n_{k}^{(m)}\}_{k\geq 0}$. Therefore, using (\ref{875e}), we obtain 
\begin{equation}
\sup_{x\in I_{m}}\left\vert \Phi _{n_{k}^{\mathrm{diag}}}(x)-\Phi
^{(m)}(x)\right\vert \longrightarrow 0\qquad (k\rightarrow \infty ).
\label{875j}
\end{equation}%
Consequently, the diagonal subsequence converges uniformly on each compact
interval $I_{m}$. Since the intervals $I_{m}$ exhaust $(0,1)$, this gives
locally uniform convergence on $(0,1)$.

\smallskip \noindent \textbf{Step 5: Consistency of the limits on overlaps.}
Fix $m<\ell $. Then $I_{m}\subset I_{\ell }$. We claim that $\Phi ^{(\ell
)}=\Phi ^{(m)}$ on $I_{m}$. Indeed, along the diagonal subsequence we have
by (\ref{875e}) 
\begin{equation}
\Phi _{n_{k}^{\mathrm{diag}}}\rightarrow \Phi ^{(m)}\ \text{ uniformly on }%
I_{m},\qquad \Phi _{n_{k}^{\mathrm{diag}}}\rightarrow \Phi ^{(\ell )}\ \text{
uniformly on }I_{\ell },  \label{875k}
\end{equation}%
hence also uniformly on $I_{m}\subset I_{\ell }$. Uniform limits on the same
set are unique, so $\Phi ^{(\ell )}=\Phi ^{(m)}$ on $I_{m}$.

\textbf{Step 6: Definition of $\Phi ^{\ast }$ and proof of (\ref{873}).} By
Step 5, the family $\{\Phi ^{(m)}\}_{m\geq 2}$ is compatible on overlaps.
Define $\Phi ^{\ast }:(0,1)\rightarrow (0,1)$ by 
\begin{equation}
\Phi ^{\ast }(x)=\Phi ^{(m)}(x)\quad \text{whenever }x\in I_{m}.
\label{875l}
\end{equation}%
This is well-defined: if $x\in I_{m}\cap I_{\ell }$, Step 5 gives $\Phi
^{(m)}(x)=\Phi ^{(\ell )}(x)$.

Now fix an arbitrary $0<\delta <\tfrac{1}{2}$. Choose $m\geq 2$ such that $%
\delta _{m}\leq \delta $. Then $[\delta ,1-\delta ]\subset I_{m}$, and by (%
\ref{875j}), 
\begin{equation}
\sup_{x\in \lbrack \delta ,1-\delta ]}\left\vert \Phi _{n_{k}^{\mathrm{diag}%
}}(x)-\Phi ^{\ast }(x)\right\vert \leq \sup_{x\in I_{m}}\left\vert \Phi
_{n_{k}^{\mathrm{diag}}}(x)-\Phi ^{(m)}(x)\right\vert \rightarrow 0,
\label{875z}
\end{equation}%
which is exactly (\ref{873}) (after relabelling $n_{k}^{\mathrm{diag}}$ as $%
n_{k}$).

Continuity and monotonicity of $\Phi ^{\ast }$ on $(0,1)$ now follow because
on each $I_{m}$ the function $\Phi ^{\ast }$ agrees with the continuous
increasing function $\Phi ^{(m)}$.

\smallskip \noindent \textbf{Step 7: Subsequence convergence of fixed points
and the conditional fixed-point identity.} Since each $\phi _{n_{k}}\in
(0,1) $, the sequence $\{\phi _{n_{k}}\}$ is bounded in $[0,1]$, hence has a
convergent sub-subsequence (still denoted $\{\phi _{n_{k}}\}$) with $\phi
_{n_{k}}\rightarrow \phi ^{\ast }\in \lbrack 0,1]$.

Assume now that $\phi ^{\ast }\in (0,1)$. Choose $\delta >0$ such that $\phi
^{\ast }\in \lbrack \delta ,1-\delta ]$. Then for all large $k$, $\phi
_{n_{k}}\in \lbrack \delta ,1-\delta ]$ as well. By (\ref{873}), we have
uniform convergence $\Phi _{n_{k}}\rightarrow \Phi ^{\ast }$ on $[\delta
,1-\delta ]$. Therefore, 
\begin{equation}
\left\vert \Phi _{n_{k}}(\phi _{n_{k}})-\Phi ^{\ast }(\phi ^{\ast
})\right\vert \leq \left\vert \Phi _{n_{k}}(\phi _{n_{k}})-\Phi ^{\ast
}(\phi _{n_{k}})\right\vert +\left\vert \Phi ^{\ast }(\phi _{n_{k}})-\Phi
^{\ast }(\phi ^{\ast })\right\vert \rightarrow 0,  \label{875x}
\end{equation}%
using uniform convergence for the first term and continuity of $\Phi ^{\ast
} $ for the second. Since $\Phi _{n_{k}}(\phi _{n_{k}})=\phi _{n_{k}}$ for
all $k$, the left-hand side also equals $|\phi _{n_{k}}-\Phi ^{\ast }(\phi
^{\ast })|\rightarrow 0$, hence $\Phi ^{\ast }(\phi ^{\ast })=\phi ^{\ast }$%
, which is (\ref{874}).

We must note that the \textit{Cantor diagonalization} was needed because
uniform convergence on one compact interval does not automatically give a
single subsequence that converges uniformly on all compact intervals, and
the diagonal argument is exactly what makes one subsequence work on the
whole exhaustion of $(0,1)$.
\end{proof}

\begin{remark}
For every $n$ we have $\Phi _{n}(0)=0$ and $\Phi _{n}(1)=1$. We later prove
that $\Phi _{n}(x)\rightarrow \tfrac{1}{2}$ for each $x\in (0,1)$. Thus any $%
\Phi ^{\ast }$ limit map equals $\tfrac{1}{2}$ on $(0,1)$ but has endpoint
values $0,1$, hence it is discontinuous at $0$ and $1$. Therefore uniform
convergence on $[0,1]$ is impossible, and the correct compactness notion is
uniform convergence on compact subsets of $(0,1)$ (i.e. local uniform
convergence).
\end{remark}

By combining \textit{Lemma \ref{Phi-AA-subseq}} and \textit{Lemma \ref%
{Phi-pointwise}}, we obtain the following corollary.

\begin{corollary}
\label{uniform} The sequence $\{\Phi _{n}\}_{n\geq 0}$ converges to $\Phi
_{\infty }$ locally uniformly on $(0,1)$. Moreover, the fixed points $\phi
_{n}$ converge to $\phi ^{\ast }$. If $\phi ^{\ast }\in (0,1)$, then 
\begin{equation}
\Phi _{\infty }(\phi ^{\ast })=\phi ^{\ast }.  \label{875v}
\end{equation}
\end{corollary}

\begin{proof}
Fix any compact interval $[\delta ,1-\delta ]\subset (0,1)$. By \textit{%
Lemma~\ref{Phi-1Lip}}, the family $\{\Phi _{n}\}_{n\geq 0}$ is
equicontinuous and uniformly bounded on $[\delta ,1-\delta ]$.

Let $\{n_{k}\}_{k\geq 0}$ be any sequence with $n_{k}\rightarrow \infty $.
By \textit{Lemma~\ref{Phi-AA-subseq}}, there exists a subsequence (still
denoted $n_{k}$) and a continuous increasing function $\Phi ^{\ast }$ such
that 
\begin{equation}
\sup_{x\in \lbrack \delta ,1-\delta ]}|\Phi _{n_{k}}(x)-\Phi ^{\ast
}(x)|\rightarrow 0.  \label{875n}
\end{equation}

On the other hand, by \textit{Lemma~\ref{Phi-pointwise}}, the full sequence $%
\Phi _{n}(x)$ converges pointwise to $\Phi _{\infty }(x)$ for every $x\in
(0,1)$. Since uniform convergence implies pointwise convergence, we must
have 
\begin{equation}
\Phi ^{\ast }(x)=\Phi _{\infty }(x)\quad \text{for all }x\in \lbrack \delta
,1-\delta ].  \label{875m}
\end{equation}

Since the choice of subsequence was arbitrary, every subsequence admits a
further subsequence converging uniformly on $[\delta ,1-\delta ]$ to the
same limit $\Phi _{\infty }$. This implies that the full sequence $\Phi _{n}$
converges uniformly on $[\delta ,1-\delta ]$. As $\delta >0$ was arbitrary,
the convergence is locally uniform on $(0,1)$.

The convergence $\phi _{n}\rightarrow \phi ^{\ast }$ was established in 
\textit{Claim~\ref{cluster}}. If $\phi ^{\ast }\in (0,1)$, then by
continuity of $\Phi _{\infty }$, 
\begin{equation}
\Phi _{\infty }(\phi ^{\ast })=\lim_{n\rightarrow \infty }\Phi _{n}(\phi
_{n})=\lim_{n\rightarrow \infty }\phi _{n}=\phi ^{\ast }.  \label{876p}
\end{equation}
\end{proof}

\begin{flushleft}
\textbf{\ Convergence to a constant limit of }$\Phi _{n}$\textbf{\ }
\end{flushleft}

\noindent Our next goal is to further characterize the limits of the
sequences $\{\Phi _{n}\}_{n=0}^{\infty }$ and $\{\phi _{n}\}_{n\geq
0}^{\infty }$.

\begin{lemma}
\noindent \label{Phi-eq} Let $[\delta ,\eta ]\subset (0,1)$ be fixed. Then
for any sequence of indices $\{n_{k}\}_{k=0}^{\infty }$ with $%
n_{k}\rightarrow \infty ,$ there exists a subsequence (still denoted $n_{k}$%
) and a limit inverse map $T^{\ast }$ such that 
\begin{equation}
\Phi _{\infty }(x)=\Phi _{\infty }(T^{\ast }(x))\quad \text{for all }x\in
\lbrack \delta ,\eta ].  \label{583}
\end{equation}
\end{lemma}

\begin{proof}
By \textit{Claim~\ref{AA}}, from the sequence $\{n_{k}\}$ we may extract a
subsequence (not relabeled) such that 
\begin{equation}
L_{n_{k}}\rightarrow L^{\ast }\quad \text{uniformly on }[0,1],  \label{584}
\end{equation}%
and therefore 
\begin{equation}
T_{n_{k}}=L_{n_{k}}^{-1}\rightarrow T^{\ast }=(L^{\ast })^{-1}\quad \text{%
uniformly on }[0,1].  \label{585}
\end{equation}

For each $k$ and $x\in \lbrack \delta ,\eta ]$ we have the exact identity 
\begin{equation}
\Phi _{n_{k}}(x)=\Phi _{n_{k}-1}(T_{n_{k}}(x)).  \label{586}
\end{equation}

Since $T^{\ast }$ is continuous and maps $(0,1)$ into itself, the image $%
T^{\ast }([\delta ,\eta ])$ is a compact subset of $(0,1)$. Hence there
exist numbers $0<\delta ^{\prime }<\eta ^{\prime }<1$ such that 
\begin{equation}
T^{\ast }([\delta ,\eta ])\subset (\delta ^{\prime },\eta ^{\prime }).
\label{587}
\end{equation}

By uniform convergence $T_{n_{k}}\rightarrow T^{\ast }$ on $[0,1]$, there
exists $K$ such that for all $k\geq K$, 
\begin{equation}
T_{n_{k}}(x)\in \lbrack \delta ^{\prime },\eta ^{\prime }]\quad \text{for
all }x\in \lbrack \delta ,\eta ].  \label{588}
\end{equation}

By \textit{Corollary~\ref{uniform}}, $\Phi _{n}\rightarrow \Phi _{\infty }$
uniformly on $[\delta ^{\prime },\eta ^{\prime }]$. Therefore, for $x\in
\lbrack \delta ,\eta ]$, 
\begin{equation}
\Phi _{n_{k}-1}(T_{n_{k}}(x))\rightarrow \Phi _{\infty }(T^{\ast
}(x)),\qquad \Phi _{n_{k}}(x)\rightarrow \Phi _{\infty }(x).  \label{588a}
\end{equation}

Passing to the limit in the identity (\ref{586}) yields 
\begin{equation}
\Phi _{\infty }(x)=\Phi _{\infty }(T^{\ast }(x)),\qquad x\in \lbrack \delta
,\eta ].  \label{588b}
\end{equation}%
which proves the claim.
\end{proof}

We next consider the convergence of orbits under a cluster inverse $T^{\ast
} $ and then use it to show that $\Phi _{\infty }$ must be constant.

\begin{lemma}
\label{orbit} Let $T^{\ast }$ be a cluster inverse. Then $T^{\ast }$ has a
unique interior fixed point $c^{\ast }$, and for any $x\in (0,1)$, 
\begin{equation}
T^{\ast \circ m}(x)\longrightarrow c^{\ast }\quad \text{as }m\rightarrow
\infty .  \label{589}
\end{equation}
\end{lemma}

\begin{proof}
\noindent By \textit{Corollary} \ref{cross-limit}, the limit map $L^{\ast }$
has a unique interior fixed point $c^{\ast }$ with $L^{\ast }(x)<x$ for $%
x<c^{\ast }$ and $L^{\ast }(x)>x$ for $x>c^{\ast }$. Inverting, $T^{\ast
}(x)>x$ for $x<c^{\ast }$ and $T^{\ast }(x)<x$ for $x>c^{\ast }$. Thus the
orbit $\{T^{\ast \circ m}(x)\}_{m=0}^{\infty }$, i.e. $T^{\ast }(\overset{m-2%
}{\overbrace{...}}(T^{\ast }(x)))$ with $m\geq 0$, is monotone and bounded
for each $x$, hence convergent (see \cite[Chapter 3.2]{[18]} for details on
orbits' definition), and the limit must be the unique fixed point $c^{\ast }$%
.

Moreover, by the \textit{Mean Value Theorem} combined with \textit{Lemma} %
\ref{l-cross} and \textit{Corollary \ref{l-cross-c}}, there is a
neighborhood $U$ of $c^{\ast }$ and $q\in (0,1)$ such that 
\begin{equation*}
|T^{\ast }(y)-c^{\ast }|\leq q|y-c^{\ast }|\quad \text{for all }y\in U.
\end{equation*}%
For any $x\in (0,1)$, monotonicity of $T^{\ast }$ and its fixed-point
structure imply (see \cite[Chapter 5.2]{[18]} for details) that $T^{\ast
\circ m}(x)$ eventually enters $U$ (if $x<c^{\ast }$, the orbit increases
into $U$; if $x>c^{\ast }$, it decreases into $U$). Once inside $U$, the
contraction estimate ensures exponential convergence to $c^{\ast }$. Thus $%
T^{\ast \circ m}(x)\rightarrow c^{\ast }$ for all $x\in (0,1)$.
\end{proof}

\begin{lemma}
\label{Phi-const-direct} The limit map $\Phi _{\infty }$ is constant on $%
(0,1)$ and equal to $c^{\ast }$ 
\begin{equation}
\Phi _{\infty }(x)\equiv c^{\ast }=\frac{1}{2},\quad x\in (0,1).  \label{590}
\end{equation}
\end{lemma}

\begin{proof}
\noindent Fix an arbitrary compact interval $[\delta ,\eta ]\subset (0,1)$.
By \textit{Lemma}~\ref{uniform}, $\Phi _{n}\rightarrow \Phi _{\infty }$
uniformly on $[\delta ,\eta ]$.

Let $\{n_{k}\}_{k=0}^{\infty }$ be any sequence of indices with $%
n_{k}\rightarrow \infty $. By \textit{Claim} \ref{AA} we may extract a
subsequence (still denoted $n_{k}$) such that $L_{n_{k}}\rightarrow L^{\ast
} $ and $T_{n_{k}}\rightarrow T^{\ast }$ uniformly, with $T^{\ast }$ the
inverse of a limit $L^{\ast }$ satisfying (\ref{510}). By \textit{Lemma}~\ref%
{Phi-eq}, 
\begin{equation}
\Phi _{\infty }(x)=\Phi _{\infty }(T^{\ast }(x)),\quad x\in \lbrack \delta
,\eta ].  \label{591}
\end{equation}%
Iterating this relation, for all $m\in N$, 
\begin{equation}
\Phi _{\infty }(x)=\Phi _{\infty }(T^{\ast \circ m}(x)),\quad x\in \lbrack
\delta ,\eta ].  \label{592}
\end{equation}%
By \textit{Lemma}~\ref{orbit}, $T^{\ast \circ m}(x)\rightarrow c^{\ast }$
for every $x\in \lbrack \delta ,\eta ]$. By continuity of $\Phi _{\infty }$ (%
\textit{Corollary}~\textit{\ref{uniform}}), 
\begin{equation}
\Phi _{\infty }(x)=\lim_{m\rightarrow \infty }\Phi _{\infty }(T^{\ast \circ
m}(x))=\Phi _{\infty }(c^{\ast })  \label{593}
\end{equation}%
for all $x\in \lbrack \delta ,\eta ]$. Thus $\Phi _{\infty }$ is constant on 
$[\delta ,\eta ]$, with value $c=\Phi _{\infty }(c^{\ast })$.

Since $[\delta ,\eta ]\subset (0,1)$ was arbitrary, we conclude that $\Phi
_{\infty }(x)\equiv c$ on all of $(0,1)$. Finally, evaluating at $x=\phi
^{\ast }$ and using \textit{Lemma}~\ref{Phi-pointwise} (which gives $\Phi
_{n}(\phi ^{\ast })\rightarrow \phi ^{\ast }$), we obtain 
\begin{equation}
c=\Phi _{\infty }(\phi ^{\ast })=\lim_{n\rightarrow \infty }\Phi _{n}(\phi
^{\ast })=\phi ^{\ast }.  \label{594}
\end{equation}%
Hence $\Phi _{\infty }(x)\equiv \phi ^{\ast }$ for all $x\in (0,1)$.
\end{proof}

\begin{flushleft}
{\large Appendix C.1.2}
\end{flushleft}

In the previous appendix, we extensively analyzed the function $L_{n}(x)$
generated by the \textit{Fr\'{e}chet-Hoeffding} \textit{lower-bound}
iteration (\ref{170a}). We established several of its key properties and
characterized its asymptotic behavior. Specifically, we proved the
subsequential uniform convergence of the sequence $\{L_{n}(x)\}_{n\geq 0}$
on $[0,1]$, as well as the convergence of its crossing points, $c_{n}$, to $%
\frac{1}{2}$. We also considered the corresponding sequence of its iterated
inverses, $\{\Phi _{n}(x)\}_{n\geq 0}$, proving the convergence of their
crossing points, $\phi _{n}$, to $\frac{1}{2}$ and, more generally, the
locally uniform convergence of the sequence itself on $(0,1)$ to $\frac{1}{2}
$.

These results provide important insights into the mechanics of the iterative
equation (\ref{5.1}), particularly under conditions of extreme negative
dependence. Furthermore, many of the demonstrated techniques will prove
useful for analyzing the general $RR_{2}$ case and for making relevant
comparisons. Nevertheless, significant questions remain regarding the full
convergence of $\{L_{n}(x)\}_{n\geq 0}$ and the explicit form of its
subsequential or full limits.

In this appendix we complete the convergence analysis of the sequence $%
\{L_{n}\}_{n\geq 0}$. We assume throughout all the structural facts
established earlier in the paper, in particular those from \textit{%
Appendix~C.1.1} concerning monotonicity, convexity/concavity on the relevant
subintervals, and the already proved symmetry-defect convergence

\begin{equation}
\Vert L_{n}(x)+L_{n}(1-x)-1\Vert _{L^{\infty }([0,1])}\longrightarrow 0.
\label{1600}
\end{equation}

We also assume relative compactness of $\{L_{n}\}_{n\geq 0}$ in $C([0,1])$.

The proof is organized as follows.

\medskip \noindent \textbf{Outline:}

\begin{enumerate}
\item We first reduce the problem from the original sequence $%
\{L_{n}\}_{n\geq 0}$ to a one-sided profile $\{R_{n}\}_{n\geq 0}$ adapted to
symmetry. As every cluster point of $\{L_{n}\}_{n\geq 0}$ is symmetric, it
becomes uniquely encoded by its reduced profile.

\item We then pass from $\{R_{n}\}_{n\geq 0}$ to an autonomous dynamical
system on a new sequence $\{A_{n}\}_{n\geq 0}$ and compute the exact \textit{%
Fr\'{e}chet derivative} of the corresponding nonlinear map $T$.

\item The derivative formula yields an explicit pointwise coefficient 
\begin{equation}
K_{n}(x)=\frac{(1-\tau _{n}(x))p_{n}(x)+\tau _{n}(x)(1-p_{n}(x))}{M_{n}},
\label{1601}
\end{equation}%
whose strict inequality $K_{n}(x)<1$ is the exact contraction criterion.

\item The inequality $K_{n}(x)<1$ is proved by a complete branch analysis
according to whether $p_{n}(x)\gtrless \frac{1}{2}$.

\item Finally, we present two ways of closing the argument.

\begin{itemize}
\item[(i)] a direct orbitwise contraction proof, based on the tail
coefficients $q_{n}=\sup_{x}K_{n}(x)$;

\item[(ii)] a rigorous compactness/cluster-set proof, which is the argument
we regard as conceptually definitive.
\end{itemize}
\end{enumerate}

\begin{flushleft}
\textbf{The reduced profile }$R_{n}$\textbf{\ and why it is the right
variable}
\end{flushleft}

Define 
\begin{equation}
\gamma (x)=1-\sqrt{1-x},\qquad \eta (x)=2x-x^{2}=\gamma ^{-1}(x),\qquad x\in
\lbrack 0,1].  \label{1602}
\end{equation}%
Following the discussion in \textit{Appendix~C.1.1}, we define the reduced
profile of $L_{n}$ by 
\begin{equation}
R_{n}(x)=2L_{n}\!\left( \frac{1-\sqrt{1-x}}{2}\right) ,\qquad x\in \lbrack
0,1].  \label{1603}
\end{equation}

The map 
\begin{equation}
x\longmapsto \alpha (x)=\frac{1-\sqrt{1-x}}{2}  \label{1604}
\end{equation}%
sends $[0,1]$ onto $[0,\frac{1}{2}]$. Thus $R_{n}$ is exactly the left half
of $L_{n}$, reparametrized to the whole unit interval and rescaled by the
factor $2$. This is a useful variable because the symmetry relation 
\begin{equation}
L(x)=1-L(1-x)  \label{1605}
\end{equation}%
determines the whole function from its left half. In particular, if $L$ is
symmetric and $R$ is defined by (\ref{1603}), then $L$ can be reconstructed
from $R$ by 
\begin{equation}
L(x)=%
\begin{cases}
\dfrac{1}{2}\,R\!(4x(1-x)), & 0\leq x\leq \dfrac{1}{2}, \\[1.2ex] 
1-\dfrac{1}{2}\,R\!(4x(1-x)), & \dfrac{1}{2}\leq x\leq 1.%
\end{cases}
\label{1606}
\end{equation}%
Indeed, if $x\in \lbrack 0,1/2]$ then $4x(1-x)\in \lbrack 0,1]$ and 
\begin{equation}
\frac{1-\sqrt{1-4x(1-x)}}{2}=x,  \label{1607}
\end{equation}%
hence $R(4x(1-x))=2L(x)$. For $x\in \lbrack 1/2,1]$, symmetry yields 
\begin{equation}
L(x)=1-L(1-x)=1-\frac{1}{2}R\!(4x(1-x)).  \label{1608}
\end{equation}

Thus, once the limit is known to be symmetric, convergence of $R_n$ is
equivalent to convergence of $L_n$.

The reason this reduction is valid in the present problem is that the
symmetry defect has already been proved to vanish uniformly. Therefore every
subsequential limit of $(L_{n})$ is symmetric. This fact is precisely what
allows we to focus on the reduced profiles $R_{n}$ without losing
information.

\begin{lemma}
\label{R-convergence-implies-L-convergence} Assume that $R_{n}\rightarrow
R_{\infty }$ uniformly on $[0,1]$ and that (\ref{1600}) holds. Then $%
L_{n}\rightarrow L_{\infty }$ uniformly on $[0,1]$, where $L_{\infty }$ is
the symmetric function determined by (\ref{1606}) with $R=R_{\infty }$.
\end{lemma}

\begin{proof}
For $x\in \lbrack 0,\frac{1}{2}]$ we have exactly 
\begin{equation}
L_{n}(x)=\frac{1}{2}\,R_{n}\!(4x(1-x)).  \label{1609}
\end{equation}%
Hence 
\begin{equation}
\sup_{0\leq x\leq 1/2}|L_{n}(x)-L_{\infty }(x)|\leq \frac{1}{2}\Vert
R_{n}-R_{\infty }\Vert _{\infty }\longrightarrow 0.  \label{1610}
\end{equation}

Now let $x\in \lbrack \frac{1}{2},1]$. Set 
\begin{equation}
\varepsilon _{n}(x)=L_{n}(x)+L_{n}(1-x)-1.  \label{1611}
\end{equation}%
Then 
\begin{equation}
L_{n}(x)=1-L_{n}(1-x)+\varepsilon _{n}(x)=1-\frac{1}{2}R_{n}\!(4x(1-x))+%
\varepsilon _{n}(x),  \label{1612}
\end{equation}%
because $1-x\in \lbrack 0,\frac{1}{2}]$. Since $L_{\infty }$ is symmetric, 
\begin{equation}
L_{\infty }(x)=1-\frac{1}{2}R_{\infty }\!(4x(1-x)).  \label{1613}
\end{equation}%
Therefore 
\begin{equation}
|L_{n}(x)-L_{\infty }(x)|\leq \frac{1}{2}\Vert R_{n}-R_{\infty }\Vert
_{\infty }+\Vert \varepsilon _{n}\Vert _{\infty }.  \label{1614}
\end{equation}%
Taking the supremum over $x\in \lbrack \frac{1}{2},1]$ and using (\ref{1600}%
) gives 
\begin{equation}
\sup_{1/2\leq x\leq 1}|L_{n}(x)-L_{\infty }(x)|\leq \frac{1}{2}\Vert
R_{n}-R_{\infty }\Vert _{\infty }+\Vert L_{n}(\cdot )+L_{n}(1-\cdot )-1\Vert
_{\infty }\rightarrow 0.  \label{1615}
\end{equation}%
Combining the two half-interval estimates yields the desired uniform
convergence. \hfill
\end{proof}

\begin{flushleft}
\textbf{The autonomous map on }$A_{n}$
\end{flushleft}

The sequence $\{R_{n}\}_{n\geq 0}$ is not autonomous in its most immediate
form. The correct autonomous variable is obtained by integrating the inverse
profile.

Define 
\begin{equation}
A_{n}(x)=\frac{\int_{0}^{x}R_{n}^{-1}(s)\,ds}{\int_{0}^{1}R_{n}^{-1}(s)\,ds}%
,\qquad x\in \lbrack 0,1].  \label{1616}
\end{equation}%
Then, by the exact identity established in the discussion above, 
\begin{equation}
R_{n+1}(x)=A_{n}(\gamma (x)),\qquad x\in \lbrack 0,1].  \label{1617}
\end{equation}%
Equivalently, 
\begin{equation}
A_{n}(x)=R_{n+1}(\eta (x)),\qquad x\in \lbrack 0,1].  \label{1618}
\end{equation}

Now define the nonlinear map 
\begin{equation}
(TA)(x)=\frac{\int_{0}^{x}\eta (A^{-1}(s))\,ds}{\int_{0}^{1}\eta
(A^{-1}(s))\,ds},\qquad x\in \lbrack 0,1].  \label{1619}
\end{equation}%
Then from (\ref{1617}) and (\ref{1618}) we obtain 
\begin{equation}
A_{n+1}=T(A_{n}).  \label{1620}
\end{equation}%
Thus the convergence problem has been reduced to the iteration of the
autonomous map $T$.

For later use, introduce 
\begin{equation}
q_{n}=A_{n}^{-1},\qquad p_{n}(x)=\eta (q_{n}(x)),\qquad
M_{n}=\int_{0}^{1}p_{n}(s)\,ds,\qquad \tau _{n}(x)=A_{n+1}(x).  \label{1621}
\end{equation}%
By construction, 
\begin{equation}
p_{n}(x)=R_{n+1}^{-1}(x),\qquad \tau _{n}(x)=R_{n+2}(\eta (x)).  \label{1622}
\end{equation}

\begin{flushleft}
\textbf{Admissible profiles and basic properties}
\end{flushleft}

Let $\mathcal{A}$ denote the class of all functions $A\in C([0,1])$ such that

\begin{enumerate}
\item $A(0)=0$, $A(1)=1$;

\item $A$ is strictly increasing;

\item $A$ is convex;

\item $A(x)\le x$ on $[0,1]$.
\end{enumerate}

All iterates $A_{n}$ belong to $\mathcal{A}$ by the results already
established for $L_{n}$ and $R_{n}$ in \textit{Appendix~C.1.1}.

If $A\in \mathcal{A}$, define 
\begin{equation}
q=A^{-1},\qquad p(x)=\eta (q(x)),\qquad M=\int_{0}^{1}p(s)\,ds,\qquad \tau
=TA.  \label{1623}
\end{equation}%
Then:

\begin{enumerate}
\item $q$ is increasing and concave;

\item since $A(x)\le x$, one has $q(x)\ge x$;

\item since $\eta(t)=2t-t^2$ is increasing and concave on $[0,1]$, the
composition $p=\eta\circ q$ is increasing and concave;

\item $p(0)=0$, $p(1)=1$;

\item $p(x)\geq \eta (x)=2x-x^{2}$ on $[0,1]$, and therefore 
\begin{equation}
M=\int_{0}^{1}p(s)\,ds\geq \int_{0}^{1}(2s-s^{2})\,ds=\frac{2}{3}.
\label{1624}
\end{equation}
\end{enumerate}

We shall use (\ref{1624}) systematically in what follows.

\begin{flushleft}
\textbf{Fr\'{e}chet derivative of }$T$
\end{flushleft}

This is the common conceptual core of both proofs. Let $A\in \mathcal{A}$,
and define $q,p,M,\tau $ as above. Then 
\begin{equation}
\tau (x)=\frac{\int_{0}^{x}p(s)\,ds}{M}.  \label{1625}
\end{equation}

Differentiating (\ref{1619}) with respect to $A$ yields the exact \textit{Fr%
\'{e}chet derivative}.

\begin{lemma}
\label{frechet} Let $A\in \mathcal{A}$ and $h\in C([0,1])$. Then 
\begin{equation}
DT_{A}[h](x)=\frac{-\int_{0}^{q(x)}(2-2t)\,h(t)\,dt+\tau
(x)\int_{0}^{1}(2-2t)\,h(t)\,dt}{M}.  \label{1626}
\end{equation}%
Consequently, 
\begin{equation}
\Vert DT_{A}\Vert _{\infty \rightarrow \infty }\leq \sup_{x\in \lbrack 0,1]}%
\frac{(1-\tau (x))p(x)+\tau (x)(1-p(x))}{M}.  \label{1627}
\end{equation}%
For the indexed orbit this becomes 
\begin{equation}
K_{n}(x)=\frac{(1-\tau _{n}(x))p_{n}(x)+\tau _{n}(x)(1-p_{n}(x))}{M_{n}}%
,\qquad \Vert DT_{A_{n}}\Vert _{\infty \rightarrow \infty }\leq \sup_{x\in
\lbrack 0,1]}K_{n}(x).  \label{1628}
\end{equation}
\end{lemma}

\begin{proof}
\noindent Let $A_{\varepsilon }=A+\varepsilon h$ and $q_{\varepsilon
}=A_{\varepsilon }^{-1}$. The standard differentiation formula for inverses
gives 
\begin{equation}
\left. \frac{d}{d\varepsilon }\right\vert _{\varepsilon =0}q_{\varepsilon
}(x)=-h(q(x)).  \label{1629}
\end{equation}%
Since $p_{\varepsilon }(x)=\eta (q_{\varepsilon }(x))$, we obtain 
\begin{equation}
\left. \frac{d}{d\varepsilon }\right\vert _{\varepsilon =0}p_{\varepsilon
}(x)=(2-2q(x))\,\left. \frac{d}{d\varepsilon }\right\vert _{\varepsilon
=0}q_{\varepsilon }(x)=-(2-2q(x))h(q(x)).  \label{1630}
\end{equation}%
Writing (\ref{1619}) as 
\begin{equation}
T(A)(x)=\frac{N_{A}(x)}{D_{A}},\qquad N_{A}(x)=\int_{0}^{x}p(s)\,ds,\qquad
D_{A}=\int_{0}^{1}p(s)\,ds=M,  \label{1631}
\end{equation}%
and differentiating the quotient yields 
\begin{equation}
DT_{A}[h](x)=\frac{\delta N_{A}[h](x)}{M}-\frac{N_{A}(x)}{M^{2}}\delta
D_{A}[h].  \label{1632}
\end{equation}%
Since $q$ is increasing and $p=\eta \circ q$, the change of variables $%
s=A(t) $ gives 
\begin{equation}
\delta N_{A}[h](x)=-\int_{0}^{q(x)}(2-2t)h(t)\,dt,\qquad \delta
D_{A}[h]=-\int_{0}^{1}(2-2t)h(t)\,dt.  \label{1633}
\end{equation}%
Substituting this into the quotient rule yields (\ref{1626}).

To derive (\ref{1627}), fix $x\in \lbrack 0,1]$. Then from (\ref{1626}), 
\begin{equation}
|DT_{A}[h](x)|\leq \frac{1}{M}\left( \int_{0}^{q(x)}(2-2t)|h(t)|\,dt+\tau
(x)\int_{q(x)}^{1}(2-2t)|h(t)|\,dt\right) .  \label{1634}
\end{equation}%
Taking $\Vert h\Vert _{\infty }\leq 1$ and using 
\begin{equation}
\int_{0}^{q(x)}(2-2t)\,dt=\eta (q(x))=p(x),\qquad
\int_{q(x)}^{1}(2-2t)\,dt=1-p(x),  \label{1635}
\end{equation}%
we obtain 
\begin{equation}
|DT_{A}[h](x)|\leq \frac{(1-\tau (x))p(x)+\tau (x)(1-p(x))}{M}.  \label{1636}
\end{equation}%
Taking the supremum over $x$ proves (\ref{1627}). The indexed formula (\ref%
{1628}) is just the specialization $A=A_{n}$. \hfill
\end{proof}

\begin{remark}
\label{functional-analytic-input} The use of the \textit{Fr\'{e}chet
derivative} in the sequel relies on standard \textit{Banach-space} calculus
and fixed-point theory. In particular, we use the \textit{Banach-space
mean-value theorem} and the contraction principles in the form of \cite[%
Theorem~9.9 (mean value theorem), Theorem~9.27 (contraction principle),
Theorem~9.29 (uniform contraction principle)]{[64]}. The exact derivative
formula (\ref{1626}) is the bridge between the nonlinear dynamics of $T$ and
the contraction mechanism.
\end{remark}

\begin{flushleft}
\textbf{Branch criterion}
\end{flushleft}

By (\ref{1628}), to prove strict contractivity it is enough to show 
\begin{equation}
K_{n}(x)<1.  \label{1637}
\end{equation}%
Equivalently, 
\begin{equation}
(1-\tau _{n}(x))p_{n}(x)+\tau _{n}(x)(1-p_{n}(x))<M_{n}.  \label{1638}
\end{equation}%
Writing $p=p_{n}(x)$, $\tau =\tau _{n}(x)$, $M=M_{n}$, this becomes 
\begin{equation}
p+\tau (1-2p)<M.  \label{1639}
\end{equation}%
Therefore:

\begin{enumerate}
\item if $p>\frac{1}{2}$, then (\ref{1638}) is equivalent to 
\begin{equation}
\tau >\frac{p-M}{2p-1};  \label{1640}
\end{equation}

\item if $p<\frac{1}{2}$, then (\ref{1638}) is equivalent to 
\begin{equation}
\tau <\frac{M-p}{1-2p}.  \label{1641}
\end{equation}
\end{enumerate}

We now prove (\ref{1640}) and (\ref{1641}).

\begin{flushleft}
\textbf{The branch }$p_{n}(x)>\frac{1}{2}$
\end{flushleft}

Fix $n$ and $x\in \lbrack 0,1]$, and abbreviate 
\begin{equation}
p=p_{n}(x),\qquad M=M_{n},\qquad \tau =\tau _{n}(x).  \label{1642}
\end{equation}%
We assume 
\begin{equation}
p>\frac{1}{2}.  \label{1643}
\end{equation}%
We prove (\ref{1640}).

\begin{lemma}
\label{concavity-left} For $0\leq s\leq x$ we have 
\begin{equation}
p_{n}(s)\geq \frac{s}{x}\,p.  \label{1644}
\end{equation}%
Consequently, 
\begin{equation}
\int_{0}^{x}p_{n}(s)\,ds\geq \frac{xp}{2}.  \label{1645}
\end{equation}
\end{lemma}

\begin{proof}
\noindent Since $p_{n}$ is concave and $p_{n}(0)=0$, for every $s\in \lbrack
0,x]$ we can write 
\begin{equation*}
s=\frac{s}{x}x+\left( 1-\frac{s}{x}\right) 0.
\end{equation*}%
Concavity gives 
\begin{equation*}
p_{n}(s)\geq \frac{s}{x}p_{n}(x)+\left( 1-\frac{s}{x}\right) p_{n}(0)=\frac{s%
}{x}p.
\end{equation*}%
Integrating from $0$ to $x$ yields (\ref{1645}). \hfill
\end{proof}

\begin{lemma}
\label{concavity-right} For $x\leq s\leq 1$ we have 
\begin{equation}
p_{n}(s)\geq p+\frac{1-p}{1-x}(s-x).  \label{1646}
\end{equation}%
Consequently, 
\begin{equation}
\int_{x}^{1}p_{n}(s)\,ds\geq \frac{(1-x)(p+1)}{2}.  \label{1647}
\end{equation}
\end{lemma}

\begin{proof}
Again by concavity, the graph of $p_{n}$ lies above the chord joining $(x,p)$
and $(1,1)$. The equation of this chord is 
\begin{equation*}
\ell (s)=p+\frac{1-p}{1-x}(s-x),\qquad s\in \lbrack x,1].
\end{equation*}%
Therefore $p_{n}(s)\geq \ell (s)$ on $[x,1]$. Integrating, 
\begin{equation*}
\int_{x}^{1}p_{n}(s)\,ds\geq \int_{x}^{1}\ell (s)\,ds=p(1-x)+\frac{1-p}{1-x}%
\cdot \frac{(1-x)^{2}}{2}=\frac{(1-x)(p+1)}{2},
\end{equation*}%
which is (\ref{1647}). \hfill
\end{proof}

Combining (\ref{1645}) and (\ref{1647}), we obtain 
\begin{equation}
M=\int_{0}^{1}p_{n}(s)\,ds\geq \frac{xp}{2}+\frac{(1-x)(p+1)}{2}=\frac{1+p-x%
}{2}.  \label{1648}
\end{equation}%
Hence 
\begin{equation}
x\geq 1+p-2M.  \label{1649}
\end{equation}%
Since 
\begin{equation}
\tau =\frac{1}{M}\int_{0}^{x}p_{n}(s)\,ds,  \label{1650}
\end{equation}%
(\ref{1645}) and (\ref{1649}) imply 
\begin{equation}
\tau \geq \frac{xp}{2M}\geq \frac{p(1+p-2M)}{2M}.  \label{1651}
\end{equation}

Thus it is enough to prove 
\begin{equation}
\frac{p(1+p-2M)}{2M}>\frac{p-M}{2p-1}.  \label{1652}
\end{equation}%
After multiplying by the positive quantity $2M(2p-1)$, this is equivalent to 
\begin{equation}
G(p,M)=2p^{3}+(1-4M)p^{2}-p+2M^{2}>0.  \label{1653}
\end{equation}

We now check this carefully.

\begin{lemma}
\label{G-positive} If 
\begin{equation}
M\geq \frac{2}{3},\qquad p\in (M,1],  \label{1654}
\end{equation}%
then 
\begin{equation}
G(p,M)=2p^{3}+(1-4M)p^{2}-p+2M^{2}>0.  \label{1655}
\end{equation}
\end{lemma}

\begin{proof}
\noindent For fixed $M$, consider $G(\cdot ,M)$ on $[M,1]$. We have 
\begin{equation}
\partial _{p}G(p,M)=6p^{2}+2(1-4M)p-1,  \label{1656}
\end{equation}%
\begin{equation}
\partial _{pp}G(p,M)=12p+2-8M.  \label{1657}
\end{equation}%
Since $p\geq M\geq \frac{2}{3}$, 
\begin{equation}
\partial _{pp}G(p,M)\geq 12M+2-8M=4M+2>0.  \label{1658}
\end{equation}%
So $G(\cdot ,M)$ is strictly convex on $[M,1]$, hence has a unique minimizer
there. We must show that this minimum is positive.

\textbf{Case 1: $M\geq \frac{7}{8}$.} Then 
\begin{equation}
\partial _{p}G(1,M)=7-8M\leq 0.  \label{1659}
\end{equation}%
Since $\partial _{p}G(\cdot ,M)$ is increasing, it follows that 
\begin{equation}
\partial _{p}G(p,M)\leq 0\qquad \text{for all }p\in \lbrack M,1].
\label{1660}
\end{equation}%
Hence $G(\cdot ,M)$ is decreasing on $[M,1]$ and its minimum is attained at $%
p=1$: 
\begin{equation}
G(1,M)=2(1-M)^{2}>0.  \label{1661}
\end{equation}

\textbf{Case 2: $\frac{2}{3}\leq M<\frac{7}{8}$.} Now 
\begin{equation}
\partial _{p}G(M,M)=-2M^{2}+2M-1<0,\qquad \partial _{p}G(1,M)=7-8M>0.
\label{1662}
\end{equation}%
Thus the unique minimizer $p_{\ast }(M)$ lies in $(M,1)$ and is
characterized by 
\begin{equation}
\partial _{p}G(p_{\ast }(M),M)=0.  \label{1663}
\end{equation}%
Solving the quadratic equation gives 
\begin{equation}
p_{\ast }(M)=\frac{4M-1+\sqrt{16M^{2}-8M+7}}{6}.  \label{1664}
\end{equation}%
Define 
\begin{equation}
H(M)=G(p_{\ast }(M),M).  \label{1665}
\end{equation}%
A direct simplification yields 
\begin{equation}
H(M)=\frac{-64M^{3}+156M^{2}-48M+10-(16M^{2}-8M+7)^{3/2}}{54}.  \label{1666}
\end{equation}%
Differentiating twice one finds 
\begin{equation}
H^{\prime \prime }(M)=-\frac{4}{9\sqrt{16M^{2}-8M+7}}(64M^{2}-32M+16+(16M-13)%
\sqrt{16M^{2}-8M+7}).  \label{1667}
\end{equation}%
For $M\in \lbrack \frac{2}{3},\frac{7}{8}]$ the bracket is strictly
positive, hence $H^{\prime \prime }(M)<0$. Therefore $H$ is concave on $[%
\frac{2}{3},\frac{7}{8}]$, so its minimum on this interval is attained at
one of the endpoints.

At $M=\frac{7}{8}$ we have $p_{\ast }(\frac{7}{8})=1$, hence 
\begin{equation}
H\!\left( \frac{7}{8}\right) =G\!\left( 1,\frac{7}{8}\right) =\frac{1}{32}>0.
\label{1668}
\end{equation}%
At $M=\frac{2}{3}$ one obtains 
\begin{equation}
H\!\left( \frac{2}{3}\right) =\frac{766-79\sqrt{79}}{1458}>0.  \label{1669}
\end{equation}%
Thus $H(M)>0$ on $[\frac{2}{3},\frac{7}{8}]$.

Combining the two cases proves $G(p,M)>0$ for all $M\geq \frac{2}{3}$ and $%
p\in (M,1]$. \hfill
\end{proof}

By (\ref{1624}), the hypotheses of \textit{Lemma~\ref{G-positive}} are
satisfied. Consequently (\ref{1640}) holds whenever $p_{n}(x)>\frac{1}{2}$.

\begin{flushleft}
\textbf{The branch }$p_{n}(x)<\frac{1}{2}$
\end{flushleft}

We now prove (\ref{1641}). Fix again $n$ and $x\in \lbrack 0,1]$, and write 
\begin{equation}
p=p_{n}(x),\qquad M=M_{n},\qquad \tau =tau_{n}(x).  \label{1670}
\end{equation}%
Assume 
\begin{equation}
p<\frac{1}{2}.  \label{1671}
\end{equation}%
We need to show 
\begin{equation}
\tau <\frac{M-p}{1-2p}.  \label{1672}
\end{equation}

Since $A_{n+1}$ is convex, increasing, and satisfies $A_{n+1}(0)=0$, $%
A_{n+1}(1)=1$, it lies below the diagonal: 
\begin{equation}
\tau _{n}(x)=A_{n+1}(x)\leq x.  \label{1673}
\end{equation}%
Also 
\begin{equation}
p_{n}(x)=\eta (A_{n}^{-1}(x))\geq \eta (x)=2x-x^{2}.  \label{1674}
\end{equation}%
Since $\gamma =\eta ^{-1}$ is increasing, this implies 
\begin{equation}
x\leq \gamma (p)=1-\sqrt{1-p}.  \label{1675}
\end{equation}%
Hence 
\begin{equation}
\tau \leq x\leq \gamma (p).  \label{1676}
\end{equation}%
Thus it is enough to prove 
\begin{equation}
\gamma (p)<\frac{M-p}{1-2p}.  \label{1677}
\end{equation}%
Because $M\geq \frac{2}{3}$, it is enough to prove 
\begin{equation}
\gamma (p)<\frac{\frac{2}{3}-p}{1-2p},\qquad 0<p<\frac{1}{2}.  \label{1678}
\end{equation}%
Set $u=\sqrt{1-p}\in \left( \frac{1}{\sqrt{2}},1\right) $. Then $p=1-u^{2}$,
and the inequality becomes 
\begin{equation}
1-u<\frac{u^{2}-\frac{1}{3}}{2u^{2}-1}.  \label{1679}
\end{equation}%
Since $2u^{2}-1>0$, this is equivalent to 
\begin{equation}
6u^{3}-3u^{2}-3u+2>0.  \label{1680}
\end{equation}%
The cubic on the left has derivative 
\begin{equation}
18u^{2}-6u-3,  \label{1681}
\end{equation}%
whose unique zero in $\left( \frac{1}{\sqrt{2}},1\right) $ is 
\begin{equation}
u_{\ast }=\frac{1+\sqrt{7}}{6}.  \label{1682}
\end{equation}%
A direct evaluation gives 
\begin{equation}
6u_{\ast }^{3}-3u_{\ast }^{2}-3u_{\ast }+2=\frac{13}{9}-\frac{7\sqrt{7}}{18}%
>0.  \label{1683}
\end{equation}%
Since the cubic is positive at $u=1/\sqrt{2}$ and at $u=1$, it is positive
throughout $\left( \frac{1}{\sqrt{2}},1\right) $. This proves (\ref{1641}).

\begin{flushleft}
\textbf{Conclusion of the branch analysis and proof of the }$L_{n}$\textbf{\
convergence}
\end{flushleft}

We have proved:

\begin{claim}
\label{pointwise-Kn} There exists $N$ such that for every $n\geq N$ and
every $x\in \lbrack 0,1]$, 
\begin{equation}
K_{n}(x)<1.  \label{1684}
\end{equation}%
Equivalently, 
\begin{equation}
q_{n}=\sup_{x\in \lbrack 0,1]}K_{n}(x)<1\qquad (n\geq N).  \label{1685}
\end{equation}
\end{claim}

\begin{proof}
\noindent The two branch inequalities (\ref{1640}) and (\ref{1641}) show
that (\ref{1638}) holds for every $x$, hence $K_{n}(x)<1$. Continuity of $%
K_{n}$ on the compact interval $[0,1]$ implies $q_{n}<1$.
\end{proof}

We now can formulate the following

\begin{theorem}
The sequence $\{L_{n}\}_{n\geq 0}$ converges uniformly on $[0,1].$
\end{theorem}

\begin{proof}
Because $\{L_{n}\}_{n\geq 0}$ is relatively compact in $C([0,1])$, the
transformed sequence $\{R_{n}\}_{n\geq 0}$ is also relatively compact in $%
C([0,1])$ by (\ref{1603}). Then, using (\ref{1617}), 
\begin{equation}
A_{n}(x)=R_{n+1}(\eta (x)),  \label{1686}
\end{equation}%
so $\{A_{n}\}_{n\geq 0}$ is relatively compact as well.

Let $\Omega $ be the cluster set of $\{A_{n}\}_{n\geq 0}$ in $C([0,1])$: 
\begin{equation}
\Omega =\left\{ A\in C([0,1])\,:\,\exists n_{k}\rightarrow \infty \text{
with }A_{n_{k}}\rightarrow A\text{ uniformly on }[0,1]\right\} .
\label{1687}
\end{equation}%
Then $\Omega $ is nonempty and compact.

\textbf{Step 1: every cluster point is admissible.} Since each $A_{n}$
belongs to $\mathcal{A}$ and the defining properties of $\mathcal{A}$ are
preserved under uniform limits, every $A\in \Omega $ belongs to $\mathcal{A}$%
. In particular, $A$ is continuous, strictly increasing, convex, and
satisfies $A(0)=0$, $A(1)=1$.

\textbf{Step 2: continuity of the inverse and of $T$.} Let $A_{k}\rightarrow
A$ uniformly in $C([0,1])$, with $A_{k},A\in \mathcal{A}$. Since $A$ is
strictly increasing and continuous on $[0,1]$, its inverse $A^{-1}$ is
continuous. We claim that 
\begin{equation}
A_{k}^{-1}\rightarrow A^{-1}\qquad \text{uniformly on }[0,1].  \label{1688}
\end{equation}%
Indeed, fix $\varepsilon >0$. Since $A$ is strictly increasing and
continuous, 
\begin{equation}
m_{\varepsilon }=\min \{|A(u)-A(v)|:\ |u-v|\geq \varepsilon ,\ u,v\in
\lbrack 0,1]\}>0.  \label{1689}
\end{equation}%
If $\Vert A_{k}-A\Vert _{\infty }<m_{\varepsilon }/2$ and $y\in \lbrack 0,1]$%
, let $x=A^{-1}(y)$, $x_{k}=A_{k}^{-1}(y)$. If $|x_{k}-x|\geq \varepsilon $,
then by the definition of $m_{\varepsilon }$, 
\begin{equation}
|A(x_{k})-A(x)|\geq m_{\varepsilon }.  \label{1690}
\end{equation}%
But 
\begin{equation}
|A(x_{k})-A(x)|=|A(x_{k})-A_{k}(x_{k})|\leq \Vert A-A_{k}\Vert _{\infty }<%
\frac{m_{\varepsilon }}{2},  \label{1691}
\end{equation}%
a contradiction. Hence $|x_{k}-x|<\varepsilon $ uniformly in $y$, proving $%
A_{k}^{-1}\rightarrow A^{-1}$ uniformly.

It follows that the maps 
\begin{equation}
A\mapsto q=A^{-1},\qquad A\mapsto p=\eta \circ q,\qquad A\mapsto
M=\int_{0}^{1}p,\qquad A\mapsto \tau =TA  \label{1692}
\end{equation}%
are continuous on $\mathcal{A}$ in the sup norm. Consequently, the
coefficient 
\begin{equation}
K(A,x)=\frac{(1-\tau (x))p(x)+\tau (x)(1-p(x))}{M}  \label{1693}
\end{equation}%
is jointly continuous in $(A,x)\in \mathcal{A}\times \lbrack 0,1]$.

\textbf{Step 3: index-free form of the branch analysis.} Let $A\in \Omega $.
Define 
\begin{equation}
q=A^{-1},\qquad p=\eta (q),\qquad M=\int_{0}^{1}p,\qquad \tau =TA.
\label{1694}
\end{equation}%
Exactly the same arguments as in the indexed proof apply: $p$ is increasing
and concave, $M\geq 2/3$, and the two branch inequalities proved above show
that 
\begin{equation}
K(A,x)<1\qquad \text{for every }x\in \lbrack 0,1].  \label{1695}
\end{equation}%
Thus 
\begin{equation}
K(A,x)<1\qquad \forall (A,x)\in \Omega \times \lbrack 0,1].  \label{1696}
\end{equation}

\textbf{Step 4: compactness yields a uniform gap.} Since $\Omega \times
\lbrack 0,1]$ is compact and $K$ is continuous, the maximum 
\begin{equation}
q_{\ast }=\max_{A\in \Omega }\max_{x\in \lbrack 0,1]}K(A,x)  \label{1697}
\end{equation}%
is attained. Because $K(A,x)<1$ everywhere on $\Omega \times \lbrack 0,1]$,
we have 
\begin{equation}
q_{\ast }<1.  \label{1698}
\end{equation}

By continuity of $K$, there exists an open neighborhood $\mathcal{U}$ of $%
\Omega $ in $C([0,1])$ such that 
\begin{equation}
\sup_{A\in \mathcal{U}}\sup_{x\in \lbrack 0,1]}K(A,x)\leq q  \label{1699}
\end{equation}%
for some $q\in (q_{\ast },1)$.

\textbf{Step 5: the orbit eventually enters the contraction neighborhood.}
We claim that there exists $N_{0}$ such that 
\begin{equation}
A_{n}\in \mathcal{U}\qquad \text{for all }n\geq N_{0}.  \label{1700}
\end{equation}%
If not, there would exist a subsequence $A_{n_{k}}\notin \mathcal{U}$. By
relative compactness, this subsequence has a uniformly convergent
subsubsequence with limit in $\Omega $, contradicting the fact that $%
\mathcal{U}$ is a neighborhood of $\Omega $.

\textbf{Step 6: tail contraction and convergence.} For $n\geq N_{0}$, the
segment joining $A_{n}$ and $A_{n+1}$ lies in a convex neighborhood on which
the derivative bound is at most $q<1$. By the \textit{Banach-space} \textit{%
mean-value theorem}, 
\begin{equation}
\Vert A_{n+2}-A_{n+1}\Vert _{\infty }=\Vert T(A_{n+1})-T(A_{n})\Vert
_{\infty }\leq q\Vert A_{n+1}-A_{n}\Vert _{\infty }.  \label{1701}
\end{equation}%
By iteration, 
\begin{equation}
\Vert A_{n+k+1}-A_{n+k}\Vert _{\infty }\leq q^{k}\Vert A_{n+1}-A_{n}\Vert
_{\infty },\qquad k\geq 0.  \label{1702}
\end{equation}%
Summing the geometric series shows that $\{A_{n}\}_{n\geq 0}$ is \textit{%
Cauchy} in $C([0,1])$, hence converges uniformly to some $A_{\infty }\in
C([0,1])$.

Now (\ref{1617}) implies 
\begin{equation}
R_{n+1}(x)=A_{n}(\gamma (x))\longrightarrow A_{\infty }(\gamma
(x))=R_{\infty }(x)  \label{1703}
\end{equation}%
uniformly on $[0,1]$. Finally, \textit{Lemma~\ref%
{R-convergence-implies-L-convergence}} and the symmetry-gap convergence (\ref%
{1600}) imply that 
\begin{equation}
L_{n}\longrightarrow L_{\infty }  \label{1704}
\end{equation}%
uniformly on $[0,1]$.

This completes the proof.
\end{proof}

\begin{remark}
Alternatively, if we take a closer look at the preceeding proofs, we can
reason in the following way.

By Claim~\ref{pointwise-Kn}, for every $n\geq N$, 
\begin{equation}
q_{n}=\sup_{x\in \lbrack 0,1]}K_{n}(x)<1.  \label{1705}
\end{equation}%
The preceding algebra shows that equality in the two branch inequalities is
excluded throughout the tail. In particular, the tail coefficients cannot
accumulate at $1$. Hence there exists $q\in (0,1)$ and an index, still
denoted by $N$, such that 
\begin{equation}
q_{n}\leq q\qquad \text{for all }n\geq N.  \label{1706}
\end{equation}%
Now the Banach-space mean-value theorem and the contraction principle imply
that the tail map is uniformly contractive in the sup norm. Therefore $%
\{A_{n}\}_{n\geq 0}$ is a Cauchy sequence in $C([0,1])$ and converges
uniformly to a limit $A_{\infty }$. By (\ref{1617}), 
\begin{equation}
R_{n+1}(x)=A_{n}(\gamma (x))\longrightarrow A_{\infty }(\gamma
(x))=R_{\infty }(x)  \label{1707}
\end{equation}%
uniformly on $[0,1]$. Finally, Lemma~\ref%
{R-convergence-implies-L-convergence} and the symmetry-gap convergence (\ref%
{1600}) yield 
\begin{equation*}
L_{n}\longrightarrow L_{\infty }
\end{equation*}%
uniformly on $[0,1]$, where $L_{\infty }$ is the symmetric function
reconstructed from $R_{\infty }$ by (\ref{1606}).
\end{remark}

\begin{flushleft}
\textbf{Final conclusion for the \textit{Fr\'{e}chet-Hoeffding lower-bound}
iteration}
\end{flushleft}

We now turn to an advisable but yet not crucial for our analysis final
objective: identifying possible closed-form expressions for the limit of $%
\{L_{n}\}_{n\geq 0}$, a task clearly achieved for the \textit{upper-bound}
case in the forthcoming \textit{Appendix C.2}.

We have the following

\begin{lemma}
\label{simpl} The iterative equation (\ref{170a}) simplifies to a form of
the iterative equation (\ref{104}) from \cite{[3]}.
\end{lemma}

\begin{proof}
Before going to the proof of the proposition, we start with some
simplifications. Substitute a new function $H^{n}(x):[-\frac{1}{2},\frac{1}{2%
}]\rightarrow \lbrack 0,1]$ such that $H^{n,-1}(x)=\frac{1}{2}-L^{n,-1}(x)$.
Since a direct differentiation in (\ref{170a}) gives that $L^{n}(x)$ is
monotonous, and thus also $H^{n}(x)$ and $H^{n,-1}(x)$, inverting the
latter, gives $L^{n}(x)=H^{n}(\frac{1}{2}-x)$. Now we can get for (\ref{170a}%
) in terms of $H^{n}$ 
\begin{equation}
H^{n+1}(\frac{1}{2}-x)=\frac{\dint\nolimits_{0}^{x}\left[ \frac{1}{4}-\left(
H^{n,-1}(u\right) )^{2}\right] du}{\dint\nolimits_{0}^{1}\left[ \frac{1}{4}%
-\left( H^{n,-1}(u\right) )^{2}\right] du},\text{ with }H^{0}(\frac{1}{2}-x)=%
\frac{\dint\nolimits_{0}^{x}F^{-1}(u)\left( 1-F^{-1}(u)\right) du}{%
\dint\nolimits_{0}^{1}F^{-1}(u)\left( 1-F^{-1}(u)\right) du}.  \label{92}
\end{equation}

Next, to simplify further, we would like to substitute a new function $%
K^{n,-1}(x)$ for the integrand $\frac{1}{4}-\left( H^{n,-1}(x\right) )^{2}$.
We consider two cases based on the monotonicity of the function $\frac{1}{4}%
-\left( H^{n,-1}(x\right) )^{2}$, so that we can take its proper inverse and
thus find $K^{n}(x)$. That would allow to effectively see what new equation (%
\ref{92}) gets transformed to. A right control of the domain and range of $%
K^{n,-1}(x)$ is needed.

In the first case, substitute initially a new function $K_{1}^{n,-1}(x)$
such that $K_{1}^{n,-1}(x)=\frac{1}{4}-\left( H^{n,-1}(x\right) )^{2}$ and
require $K_{1}^{n,-1}(x)$ be decreasing. This monotonicity is valid only for 
$H^{n,-1}(x)\in \lbrack 0,\frac{1}{2}].$ For the latter interval, we get
that it holds $K_{1}^{n,-1}(x)\in \lbrack 0,\frac{1}{4}]$. Since $%
H^{n,-1}(x) $ is decreasing ($H^{n}(x)=L^{n}(\frac{1}{2}-x)$) and $%
H^{n,-1}(x)=\frac{1}{2}$ for $x=0$, for the range $[0,\frac{1}{2}]$ of $%
H^{n,-1}(x)$ we get the domain $x\in \lbrack 0,H^{n}(0)]$. Note that since
we do not know for which $x$ it holds $H^{n,-1}(x)=0,$ we have put in the
domain the point $H^{n}(0)\footnote{{\footnotesize For }$H^{n}(.)$, 
{\footnotesize respectively }$L^{n}(.)$, {\footnotesize the only known
points, independent from the starting distribution }$F(.),${\footnotesize \
are derived by the defining equations (\ref{170a}) and (\ref{92}) and they
are }$0${\footnotesize \ and }$1${\footnotesize \ for }$L^{n}(.)$%
{\footnotesize \ and }$\pm \frac{1}{2}${\footnotesize \ for }$H^{n}(.)$%
{\footnotesize . More precisely, we have }$L^{n}(0)=0${\footnotesize , }$%
L^{n}(1)=1${\footnotesize \ and }$H^{n}\left( -\frac{1}{2}\right) =1$%
{\footnotesize , }$H^{n}\left( \frac{1}{2}\right) =0.${\footnotesize \ }}$.
In terms of $K_{1}^{n,-1}(x)$, $H^{n,-1}(x)=0$ implies $K_{1}^{n,-1}(x)=%
\frac{1}{4}$ and so for the domain of $K_{1}^{n,-1}(x)$ we get $x\in \lbrack
0,K_{1}^{n}(\frac{1}{4})]$. Again, since we do not know for which $x$ it
holds $K_{1}^{n,-1}(x)=\frac{1}{4}$, we have put in the domain $K_{1}^{n}(%
\frac{1}{4})$. The derivations imply that effectively we can define $%
K_{1}^{n,-1}(x):[0,K_{1}^{n}(\frac{1}{4})]\rightarrow \lbrack 0,\frac{1}{4}%
], $or equivalently $K_{1}^{n}(x):[0,\frac{1}{4}]\rightarrow \lbrack
0,K_{1}^{n}(\frac{1}{4})],$ with $K_{1}^{n,-1}(x)=\frac{1}{4}-\left(
H^{n,-1}(x\right) )^{2}$. Additionally, for $x\in \lbrack 0,K_{1}^{n}(\frac{1%
}{4})]$ we have $H^{n,-1}(x)=\sqrt{\frac{1}{4}-K_{1}^{n,-1}(x)}$. If we take
inverses in the latter and consider the range of $\sqrt{\frac{1}{4}%
-K_{1}^{n,-1}(x)},$ for $x\in \lbrack 0,\frac{1}{2}]$ we have $%
H^{n}(x)=K_{1}^{n}(\frac{1}{4}-x^{2})$ as well as for $x\in \lbrack 0,\frac{1%
}{4}]$ we have $K_{1}^{n}(x)=H^{n}\left( \sqrt{\frac{1}{4}-x}\right) .$

In the second case, substitute initially a new function $K_{2}^{n,-1}(x)$
such that $K_{2}^{n,-1}(x)=\frac{1}{4}-\left( H^{n,-1}(x\right) )^{2}$ and
require $K_{2}^{n,-1}(x)$ be increasing. This monotonicity is valid only for 
$H^{n,-1}(x)\in \lbrack -\frac{1}{2},0].$ For the latter interval we get
that again it holds $K_{2}^{n,-1}(x)\in \lbrack 0,\frac{1}{4}]$. Since $%
H^{n,-1}(x)$ is decreasing ($H^{n}(x)=L^{n}(\frac{1}{2}-x)$) and $%
H^{n,-1}(x)=-\frac{1}{2}$ for $x=1$, for the range $[-\frac{1}{2},0]$ of $%
H^{n,-1}(x)$ we get the domain $x\in \lbrack H^{n}(0),1]$. Since we do not
know for which $x$ it holds $H^{n,-1}(x)=0,$ we have put in the domain the
point $H^{n}(0)$. In terms of $K_{2}^{n,-1}(x)$, $H^{n,-1}(x)=0$ again
implies $K_{2}^{n,-1}(x)=\frac{1}{4}$ and so for the domain of $%
K_{2}^{n,-1}(x)$ we get $x\in \lbrack K_{2}^{n}(\frac{1}{4}),1]$. Again,
since we do not know for which $x$ it holds $K_{2}^{n,-1}(x)=\frac{1}{4}$,
we have put in the domain $K_{2}^{n}(\frac{1}{4})$. The derivations imply
that effectively we can define $K_{2}^{n,-1}(x):[K_{2}^{n}(\frac{1}{4}%
),1]\rightarrow \lbrack 0,\frac{1}{4}],$or equivalently $K_{2}^{n}(x):[0,%
\frac{1}{4}]\rightarrow \lbrack K_{2}^{n}(\frac{1}{4}),1],$ with $%
K_{2}^{n,-1}(x)=\frac{1}{4}-\left( H^{n,-1}(x\right) )^{2}$. Additionally,
for $x\in \lbrack K_{2}^{n}(\frac{1}{4}),1]$ we have $H^{n,-1}(x)=-\sqrt{%
\frac{1}{4}-K_{2}^{n,-1}(x)}$. If we take inverses in the latter and
consider the range of $\sqrt{\frac{1}{4}-K_{2}^{n,-1}(x)},$ for $x\in
\lbrack -\frac{1}{2},0]$ we have $H^{n}(x)=K_{2}^{n}(\frac{1}{4}-x^{2})$ as
well as for $x\in \lbrack 0,\frac{1}{4}]$ we have $K_{2}^{n}(x)=H^{n}\left( -%
\sqrt{\frac{1}{4}-x}\right) .$

Now, we can rewrite (\ref{92}) in terms of $K^{n}(.)$. Before doing this, we
should note that effectively we have distinguished between the two cases
above based on the integrand in (\ref{92}). Within the process, we specially
considered the range and domain of $H^{n}(.)$ and $K^{n}(.)$ and of their
inverses. All that was done up to iteration $n$. Moving to iteration $n+1$,
however, we have to be careful. The range and domain of $H^{n+1}(.)$ and $%
K^{n+1}(.)$ on the left hand side of (\ref{92}) have to be consistent with
their counterparts on the right hand side of the equation coming from the
integrand in which inverses of $H^{n}(.)$ and $K^{n}(.)$ participate.

We consider two new cases based on the order of $\frac{1}{2}$ and $H^{n}(0)$%
: 1) $\frac{1}{2}\leq $ $H^{n}(0)$ (with $H^{n}(0)=K_{1}^{n}(\frac{1}{4}%
)=K_{2}^{n}(\frac{1}{4})=L^{n}(\frac{1}{2})$ clearly being valid as well).
Three sub-cases occur based on the domain: a) $x\in \lbrack 0,\frac{1}{2}]$,
b) $x\in \lbrack \frac{1}{2},H^{n}(0)]$, and c) $x\in \lbrack H^{n}(0),1]$;
and 2) $H^{n}(0)\leq \frac{1}{2}$ (with $H^{n}(0)=K_{1}^{n}(\frac{1}{4}%
)=K_{2}^{n}(\frac{1}{4})=L^{n}(\frac{1}{2})$ clearly being valid as well).
Three sub-cases occur again: a) $x\in \lbrack 0,H^{n}(0)]$, b) $x\in \lbrack
H^{n}(0),\frac{1}{2}]$, and c) $x\in \lbrack \frac{1}{2},1]$. Within each of
them, we shall consider the range consistency. For space consideration, we
won't show these details and we will go directly to the next step.

In that skipped analysis, it turns out that we are allowed not to
differentiate explicitly between $K_{1}(.)$ and $K_{2}(.)$ in each of the
two cases above across their sub-cases. This is due to the fact that we can
take the union of the the two cases and their sub-cases in terms of the
domain both of $K(.)$ (iteration $n+1)$ and $K^{-1}(.)$ (iteration $n).$
Since the latter was the distinguishing feature for the sub-cases, i.e. the
subscripts usage, writing the equations in simpler uniform way is possible%
\begin{equation}
K^{n+1}\left( x(1-x)\right) =\left\{ \frac{\dint%
\nolimits_{0}^{x}K^{n,-1}(u)du}{\dint\nolimits_{0}^{1}K^{n,-1}(u)du}%
,K^{0}(x(1-x))=\frac{\dint\nolimits_{0}^{x}F^{-1}(u)\left(
1-F^{-1}(u)\right) du}{\dint\nolimits_{0}^{1}F^{-1}(u)\left(
1-F^{-1}(u)\right) du},x\in \lbrack 0,1].\right.  \label{111}
\end{equation}

Effectively we had the substitutions%
\begin{eqnarray}
L^{n}(\frac{1}{2}-x) &=&H^{n}(x)=K^{n}(\frac{1}{4}-x^{2}),x\in \lbrack -%
\frac{1}{2},\frac{1}{2}]  \label{112} \\
L^{n}\left( x\right) &=&K^{n}(x(1-x)),x\in \lbrack 0,1].  \label{113}
\end{eqnarray}

We can further note that there seems to be a shortcut approach for going
directly from (\ref{92}) to (\ref{111}) by just posing initially the
(unindexed) $K^{n,-1}(x):[0,1]\rightarrow \lbrack 0,\frac{1}{4}]$ by $%
K^{n,-1}(x)=\frac{1}{4}-\left( H^{n,-1}(x\right) )^{2}$. Then we invert two
times\footnote{{\footnotesize First the compound function }$\left(
H^{n,-1}(x\right) )^{2}${\footnotesize \ and second }$H^{n,-1}(x)$%
{\footnotesize \ intself. In this way, we do not need to specify whether we
have }$H^{n,-1}(x)=-\sqrt{\frac{1}{4}-K^{n,-1}(x)}${\footnotesize \ or }$%
H^{n,-1}(x)=\sqrt{\frac{1}{4}-K^{n,-1}(x)}${\footnotesize \ at the first
step, since somehow the effect is cancelled out in the second one.\ }} and
get $H^{n}(x)=K^{n}(\frac{1}{4}-x^{2})$ and from the latter also $%
L^{n}(x)=H^{n}(\frac{1}{2}-x)=K^{n}(\frac{1}{4}-(\frac{1}{2}%
-x)^{2})=K^{n}(x(1-x))$. Yet, this approach is only heuristic, since it does
not consider well the domain and range effects we discussed. Ignoring them
is inappropriate since there could appear impossible cases and the union
argument above not to work as fluently as it did.

We proceed now with (\ref{111}). Based on the monotonicity of the function $%
x\longmapsto x(1-x)$ on $x\in \lbrack 0,1]$, we have two cases%
\begin{eqnarray}
L^{n}\left( x\right) &=&K_{\ast }^{n}(x(1-x)),x\in \lbrack 0,\frac{1}{2}]
\label{136} \\
L^{n}\left( x\right) &=&K_{\ast \ast }^{n}(x(1-x)),x\in \lbrack \frac{1}{2}%
,1],  \label{137}
\end{eqnarray}

where on $x\in \lbrack 0,\frac{1}{4}]$ holds%
\begin{eqnarray}
K_{\ast }^{n+1}\left( x\right) &=&\frac{\dint\nolimits_{0}^{\varphi _{\ast
}(x)}K_{\ast }^{n,-1}(u)du}{\underset{\mu ^{\ast }}{\underbrace{%
\dint\nolimits_{0}^{K_{\ast }^{n}(\frac{1}{4})}K_{\ast
}^{n,-1}(u)du+\dint\nolimits_{K_{\ast }^{n}(\frac{1}{4})}^{1}K_{\ast \ast
}^{n,-1}(u)du}}},K_{\ast }^{0}(x)=\frac{\dint\nolimits_{0}^{\varphi _{\ast
}(x)}F^{-1}(u)\left( 1-F^{-1}(u)\right) du}{\dint\nolimits_{0}^{1}F^{-1}(u)%
\left( 1-F^{-1}(u)\right) du}  \label{130} \\
K_{\ast \ast }^{n+1}\left( x\right) &=&\frac{\dint\nolimits_{0}^{\varphi
_{\ast \ast }(x)}K_{\ast \ast }^{n,-1}(u)du}{\underset{\mu ^{\ast \ast }}{%
\underbrace{\dint\nolimits_{0}^{K_{\ast \ast }^{n}(\frac{1}{4})}K_{\ast
}^{n,-1}(u)du+\dint\nolimits_{K_{\ast \ast }^{n}(\frac{1}{4})}^{1}K_{\ast
\ast }^{n,-1}(u)du}}},K_{\ast \ast }^{0}(x)=\frac{\dint\nolimits_{0}^{%
\varphi _{\ast \ast }(x)}F^{-1}(u)\left( 1-F^{-1}(u)\right) du}{%
\dint\nolimits_{0}^{1}F^{-1}(u)\left( 1-F^{-1}(u)\right) du}  \label{131}
\end{eqnarray}

We used that for $x\in \lbrack 0,\frac{1}{2}],$ the function $x(1-x)$ is
monotonically increasing, and for $x\in \lbrack \frac{1}{2},1]$ it is
monotonically decreasing, which allows us to invert it. This gives rise to
the functions: $\varphi _{\ast }(x):[0,\frac{1}{4}]\rightarrow \lbrack 0,%
\frac{1}{2}]$ satisfying $\varphi _{\ast }(x)=\frac{1-\sqrt{1-4x}}{2}$ and $%
\varphi _{\ast \ast }(x):[0,\frac{1}{4}]\rightarrow \lbrack \frac{1}{2},1]$
satisfying $\varphi _{\ast \ast }(x)=\frac{1+\sqrt{1-4x}}{2}$. Effectively,
we had the substitutions%
\begin{eqnarray}
K_{\ast }^{n}(x) &=&L^{n}\left( \frac{1-\sqrt{1-4x}}{2}\right) ,x\in \lbrack
0,\frac{1}{4}]  \label{133} \\
K_{\ast \ast }^{n}(x) &=&L^{n}\left( \frac{1+\sqrt{1-4x}}{2}\right) ,x\in
\lbrack 0,\frac{1}{4}].  \label{134}
\end{eqnarray}

Additionally, we have the special points: $K_{\ast }^{n}(0)=0,K_{\ast \ast
}^{n}(0)=1$, and $K_{\ast }^{n}(\frac{1}{4})=K_{\ast \ast }^{n}(\frac{1}{4}%
)=L^{n}(\frac{1}{2})$.

Equations (\ref{130}) and (\ref{131}) allow to find $L^{n}\left( x\right) $
for $x\in \lbrack 0,1]$ by uniting the two cases. We have to solve each of
the equations (\ref{130}) and (\ref{131}) in a stand-alone way at its domain
and find the $K_{\ast }^{n}(x)$ and $K_{\ast \ast }^{n}(x)$. Then for $x\in
\lbrack 0,\frac{1}{2}],$ we plug $x(1-x)$ into $K_{\ast }^{n}(x)$ to find $%
L^{n}\left( x\right) =K_{\ast }^{n}(x(1-x))$. For $x\in \lbrack \frac{1}{2}%
,1],$ we plug $x(1-x)$ into $K_{\ast \ast }^{n}(x)$ to find $L^{n}\left(
x\right) =K_{\ast \ast }^{n}(x(1-x))$.

We can observe that (\ref{130}) and (\ref{131}) share a similar functional
form with the only difference coming from the respective domains and the
upper integral limit functions. Additionally, we can notice the similarity
of the two equations to equation (\ref{104}), i.e., the main equation
handled in \cite{[3]}. Here the upper integral limits are the functions $%
\varphi _{\ast }(x)$ and $\varphi _{\ast \ast }(x)$, while in the
aforementioned cases is just $x$. \textit{\ }
\end{proof}

\begin{remark}
Even in the simplified form provided by \textit{Lemma~\ref{simpl}},
obtaining a closed-form expression for either the limit of \textit{(\ref%
{170a})} or the solution to the functional equation \textit{(\ref{510})}
remains a non-trivial task.
\end{remark}

\begin{remark}
Differentiating (\ref{130}) and (\ref{131}) and then substituting $h^{\ast
}(x):[0,\frac{1}{4}]\longrightarrow \lbrack 0,\frac{1}{4}]$ such that $%
h^{\ast }(x)=K_{\ast }^{n,-1}(\varphi _{\ast }(x))$, leads to the following
functional equation%
\begin{equation}
(h^{\ast })^{^{\prime }}(x)=\frac{\mu ^{\ast }}{h^{\ast }(h^{\ast }(x))}%
\frac{\varphi _{\ast }^{^{\prime }}(x)}{\varphi _{\ast }^{^{\prime
}}(h^{\ast }(x))},  \label{762}
\end{equation}

with $h^{\ast }(0)=0$, $h^{\ast }(\frac{1}{4})=\frac{1}{4}$, and $(h^{\ast
})^{^{\prime }}(\frac{1}{4})=4\mu ^{\ast }$.

Analogously by substituting $h^{\ast \ast }(x):[0,\frac{1}{4}%
]\longrightarrow \lbrack \frac{1}{4},1]$ such that $h^{\ast \ast
}(x)=K_{\ast \ast }^{n,-1}(\varphi _{\ast }(x))$ we get%
\begin{equation}
(h^{\ast \ast })^{^{\prime }}(x)=\frac{\mu ^{\ast \ast }}{h^{\ast \ast
}(h^{\ast \ast }(x))}\frac{\varphi _{\ast \ast }^{^{\prime }}(x)}{\varphi
_{\ast \ast }^{^{\prime }}(h^{\ast \ast }(x))},  \label{763}
\end{equation}

with $h^{\ast \ast }(0)=1$, $h^{\ast \ast }(\frac{1}{4})=\frac{1}{4}$, and $%
(h^{\ast \ast })^{^{\prime }}(\frac{1}{4})=4\mu ^{\ast \ast }$.

Under further regularity/dynamical conditions (e.g., $h^{\ast }$ and $%
h^{\ast \ast }$ analytic on $(0,\frac{1}{4}]$ (see \cite{[44]}), increasing
on $(0,\frac{1}{4}]$, or $\frac{1}{4}$ being attracting, the two
functional--differential equations typically become well-posed
existence/uniqueness problems. Yet, even the methods in \cite{[44]} do not
provide straightforward path to find a closed form solution.
\end{remark}

\begin{example}
The ansatz function $G_{-}(x)$, defined by 
\begin{equation}
G_{-}(x)=\left\{ 
\begin{array}{c}
\frac{1-\sqrt{1-[4x(1-x)]^{\frac{\sqrt{5}+1}{2}}}}{2},0\leq x\leq \frac{1}{2}
\\ 
\frac{1+\sqrt{1-[4x(1-x)]^{\frac{\sqrt{5}+1}{2}}}}{2},\frac{1}{2}<x\leq 1 \\ 
0,x<0 \\ 
1,x>1%
\end{array}%
\right.  \label{764}
\end{equation}

provides a strong approximation for the limiting analytic function. This
conclusion is based on heuristic reasoning, supported by empirical
experiments, that involves approximating the terms $\varphi _{\ast
}^{^{\prime }}(x)$, $\varphi _{\ast \ast }^{^{\prime }}(x)$, $\varphi _{\ast
}^{^{\prime }}(h^{\ast }(x))$, and $\varphi _{\ast \ast }^{^{\prime
}}(h^{\ast \ast }(x))$ and the ratios participating in \textit{(\ref{762})}
and \textit{(\ref{763}).}\hfill
\end{example}

We conclude the appendix with several plots illustrating the dynamics and
providing further visual intuition. We take as in \textit{Section 4.3} $%
F_{1} $ $\thicksim $ $Lognormal(0.2$, $0.5)$. \textit{Figures C1.1-2} plot
the evolution of the d.f.s of the marginals and that of the compounded
inverses. They are counterparts to the plots of Section 4.3.

\begin{center}
\begin{minipage}[t]{0.48\linewidth}\centering
    \includegraphics[width=\linewidth]{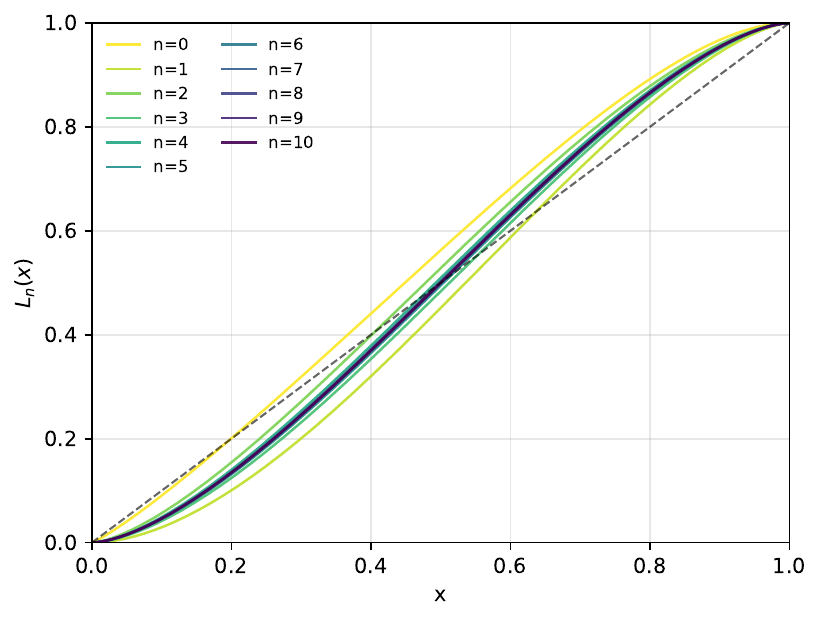}
\captionof{figure}[Figure C1.1:: $L^{n}(x)$ evolution]{Figure C1.1:\\ $L^{n}(x)$ evolution}
    \label{fig:C1_1}
  \end{minipage}\hfill 
\begin{minipage}[t]{0.48\linewidth}\centering
    \includegraphics[width=\linewidth]{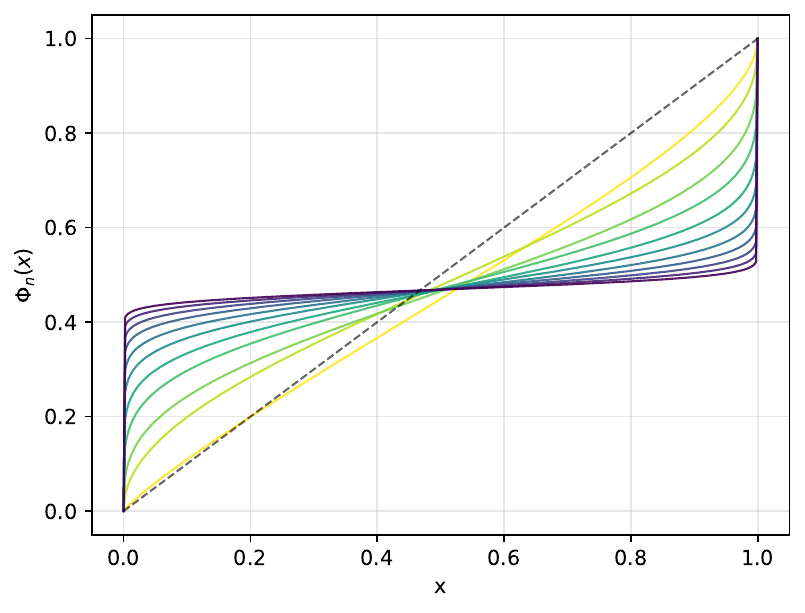}
\captionof{figure}[Figure C1.2:: Compounds of the inverse marginal d.f.s - $\Phi _{n}(x)$]{Figure C1.2:\\ Compounds of the inverse marginal d.f.s - $\Phi _{n}(x)$}
    \label{fig:C1_2}
  \end{minipage}
\end{center}

We can clearly see the difference. The crossings do not tend monotonically
to the right as the case of $RR_{2}$ density was, they rather oscillate
towards $0.5$. Exactly it is expected to be the limit of $\Phi _{n}(x)$ as
we discussed.

\begin{flushleft}
{\large Appendix C.2}
\end{flushleft}

Hereinafter, our focus is on the iterative equation (\ref{43}). For
convenience we denote below any of the marginal distributions $L_{1+}^{n}(x)$
or $L_{2+}^{n}(x)$ by $L_{n}(x)$ since the properties proved hold for both
of them. Thus, our setting becomes:

\begin{claim}
\label{major} Keeping to the posed assumptions in the main text, for the
iterative map%
\begin{equation}
L^{n+1}(x)=\frac{\dint\nolimits_{0}^{x}\left( L^{n,-1}(u)\right) ^{2}du}{%
\dint\nolimits_{0}^{1}\left( L^{n,-1}(u)\right) ^{2}du},\text{ with }%
L^{0}(x)=\frac{\dint\nolimits_{0}^{x}\left( F^{-1}(u)\right) ^{2}du}{%
\dint\nolimits_{0}^{1}\left( F^{-1}(u)\right) ^{2}du}  \label{561}
\end{equation}

a uniform convergence of $L^{n}(x)$ holds towards the distribution function $%
G_{+}(x)$%
\begin{equation}
G_{+}(x)=\left\{ 
\begin{array}{c}
x^{2},0\leq x\leq 1 \\ 
0,x<0 \\ 
1,x>1.%
\end{array}%
\right. .  \label{562}
\end{equation}
\end{claim}

\begin{proof}
We can do the proof by induction verbally following the logic of \cite{[3]}.
For space considerations it won't be presented. We will just point out that
again we can find majorizing polynomials such for $0\leq x\leq 1$ holds 
\begin{eqnarray}
0 &\leq &F(x)\leq 1  \label{563} \\
0 &\leq &L_{0}(x)\leq x  \notag \\
x^{3} &\leq &L_{1}(x)\leq x  \notag \\
x^{3} &\leq &L_{2}(x)\leq x^{\frac{5}{3}}  \notag \\
&&...  \notag \\
x^{\alpha _{n+1}} &\leq &L_{n}(x)\leq x^{\alpha _{n}},n\text{ - }odd  \notag
\\
x^{\alpha _{n}} &\leq &L_{n}(x)\leq x^{\alpha _{n+1}},n\text{ - }even,\text{ 
}  \notag
\end{eqnarray}

where the sequence $\alpha _{1}=1$, $\alpha _{2}=3$, $\alpha _{3}=\frac{5}{3}
$,... is given recursively by:%
\begin{equation}
\alpha _{n+1}=1+\frac{2}{\alpha _{n}}.  \label{564}
\end{equation}

The sequence (\ref{564}) is convergent (bounded and monotonous in even and
odd members, sharing a joint candidate limit). Solving the equation $\alpha
=1+\frac{2}{\alpha }$ gives us:%
\begin{equation}
\underset{n\rightarrow +\infty }{\lim }L_{n}(x)=x^{2}\text{ for }0\leq x\leq
1.  \label{565}
\end{equation}

The point-wise convergence of the sequence of functions $L_{n}(x)$ in the
unit interval (compact set), together with their continuity and monotonicity
in $x$, makes the continuity of the limiting function a necessary and
sufficient condition for uniform convergence in the unit interval. Thus the
sequence of functions $L_{n}(x)$ are not only point-wise convergent but also
uniformly convergent in $[0,1]$.
\end{proof}

\begin{remark}
For the \textit{Fr\'{e}chet-Hoeffding} upper-bound, we find that $L_{n}(x)$
has no crossing points in (0,1). This follows from the subdiagonal pattern
established by the polynomial majorization from \textit{Claim} \ref{major}
and is in contrast to the behavior observed for the lower-bound case in 
\textit{Appendix C.1}.
\end{remark}

\begin{lemma}
Under the \textit{Fr\'{e}chet-Hoeffding} \textit{upper-bound }iteration (\ref%
{562}), the compound maps 
\begin{equation}
\Phi _{n}=T_{0}\circ T_{1}\circ \cdots \circ T_{n},  \label{566}
\end{equation}%
where again $T_{n}$ denotes the inverse of $L_{n}(x)$, converge pointwise
and locally uniformly on $(0,1)$ to the constant map 
\begin{equation}
\Phi _{\infty }(x)=1\qquad x\in (0,1).  \label{567}
\end{equation}
\end{lemma}

\begin{proof}
From (\ref{563}) follows%
\begin{eqnarray}
x^{\frac{1}{\alpha _{n}}} &\leq &T_{n}(x)\leq x^{\frac{1}{\alpha _{n+1}}},n%
\text{ - }odd  \label{569} \\
x^{\frac{1}{\alpha _{n+1}}} &\leq &T_{n}(x)\leq x^{\frac{1}{\alpha _{n}}},n%
\text{ - }even.  \notag
\end{eqnarray}

We may note that all powers of $x$ are less than $1$. Therefore%
\begin{equation}
\Phi _{n}=T_{0}\circ T_{1}\circ \cdots \circ T_{n}\geq
\prod\nolimits_{j=0}^{n}x^{\frac{1}{\gamma _{j}}}=x^{\varepsilon
^{n}}\rightarrow 1,  \label{700}
\end{equation}

where $\exists $ $\varepsilon $ and $\gamma _{j}$ such that for $\forall $ $%
j $ holds $\alpha _{j}<\gamma _{j}<\varepsilon \in (0,1).$ Additionally, $%
T^{n}(x)\leq 1$ for all $n$ and thus $\Phi _{n}(x)\leq 1$ as well$.$ The
squeeze gives the result of the theorem.
\end{proof}

\bigskip \newpage

\begin{flushleft}
\bigskip {\Large Appendix D}
\end{flushleft}

In this appendix, we will prove that the following sequences of compound
functions 
\begin{eqnarray}
\Phi _{n}^{1}(x) &=&(L_{1}^{0,-1}\circ L_{1}^{1,-1}\circ \dots \circ
L_{1}^{n-1,-1}\circ L_{1}^{n,-1})(x)  \label{300} \\
\Phi _{n}^{2}(x) &=&(L_{2}^{0,-1}\circ L_{2}^{1,-1}\circ \dots \circ
L_{2}^{n-1,-1}\circ L_{2}^{n,-1})(x)  \notag
\end{eqnarray}

are uniformly convergent to constants for $x\in \lbrack 0,1]$.

We divide the proof in several steps:

In \textit{Appendix D.1}, we assume that the starting density $%
f_{12}(x_{1},x_{2})$ is $TP_{2}$. That is $x_{1}\rightarrow
E[X_{2}^{F}|X_{1}^{F}=x_{1}]$ is non-decreasing\footnote{{\footnotesize %
Without loss of generality }$x_{2}\rightarrow E[X_{1}^{F}|X_{2}^{F}=x_{2}]$%
{\footnotesize \ as well. These are standard properties as elaborated in 
\cite{[36]} and \cite{[38]}. More specialized ones can be traced in \cite%
{[32]}, \cite{[39]}, and \cite{[40]}.}}. Then we prove:

\begin{enumerate}
\item $TP_{2}$\emph{\ preservation: }$l_{12}^{n}(x_{1},x_{2})$ is $TP_{2}$
for $n\geq 0;$

\item \emph{Subdiagonality:} $L_{1}^{n}(x)$ and $L_{2}^{n}(x)$ lie strictly
below the diagonal $x$ $\rightarrow $ $x$ on $(0,1)$ for $n\geq 0;$

\item \emph{Uniform} c\emph{onvergence of }$\Phi _{n}^{i}(x)$\emph{, }$i=1,2$%
\emph{: }The following pointwise convergence holds%
\begin{equation}
\lim_{n\rightarrow +\infty }\Phi _{n}^{i}(x)=\left\{ 
\begin{array}{c}
1,\text{ }x>0 \\ 
0,\text{ }x=0.%
\end{array}%
\right.  \label{301}
\end{equation}%
Moreover, for every fixed $\delta $ $\in $ $(0,1)$, the convergence $\Phi
_{n}^{i}(x)\rightarrow 1$ is uniform in $x\in \lbrack \delta ,1]$.
\end{enumerate}

In \textit{Appendix D.2}, we assume that the starting density $%
f_{12}(x_{1},x_{2})$ is $RR_{2}$. That is $x_{1}\rightarrow
E[X_{2}^{F}|X_{1}^{F}=x_{1}]$ is non-increasing. We prove:

\begin{enumerate}
\item $(S)RR_{2}$\emph{\ preservation:} $l_{12}^{n}(x_{1},x_{2})$ is $%
(S)RR_{2}$ for $n\geq 0;$

\item \emph{Single-crossing property of}\textbf{\ }$L_{i}^{n}$ for $%
i=1,2:L_{i}^{n}(x)=x$ has at most one solution in $(0,1)$ for $i=1,2;$

\item \emph{Extension of the analysis from the Fr\'{e}chet--Hoeffding
lower-bound case to the general }$RR_{2}$\emph{\ setting: }isolating the
required additional assumptions, adapting the methodology from \textit{%
Appendix C}, and providing the necessary theoretical extensions;

\item \emph{Uniform} c\emph{onvergence of }$\Phi _{n}^{i}(x)$\emph{, }$i=1,2$%
\emph{: }The following pointwise convergence holds%
\begin{equation}
\lim_{n_{k}\rightarrow +\infty }\Phi _{n}^{i}(x)=\left\{ 
\begin{array}{c}
c,\text{ }x>0 \\ 
0,\text{ }x=0.%
\end{array}%
\right.  \label{302}
\end{equation}%
Moreover, for every fixed $\delta $ $\in $ $(0,1)$, the convergence $\Phi
_{n}^{i}(x)\rightarrow 1$ is uniform in $x\in \lbrack \delta ,1]$.
\end{enumerate}

In \textit{Appendix D.3}, we provide a summary of the results from \textit{%
Appendices D.1-D.2} to be used in the main text. In \textit{Appendix D.4, }%
we consider with the multivariate case.

\begin{flushleft}
{\large Appendix D.1}

\textbf{Notation and preliminaries}
\end{flushleft}

If as before $L_{n}$ is any bivariate d.f. on $[0,1]^{2}$ with density $%
l_{n}(x_{1},x_{2})=\frac{\partial ^{2}}{\partial x_{1}\partial x_{2}}%
L_{n}(x_{1},x_{2})$, let also following the notational logic from the main
text 
\begin{eqnarray}
I_{n}\;
&=&\;\int_{0}^{1}\int_{0}^{1}u_{1}u_{2}\,dL_{n}(u_{1},u_{2})\;=\;E[%
\,X_{1}^{L_{n}}X_{2}^{L_{n}}\,]  \label{303} \\
I_{F}
&=&\int_{0}^{1}\int_{0}^{1}u_{1}u_{2}\,dF(u_{1},u_{2})\;=\;E[%
\,X_{1}^{F}X_{2}^{F}\,]  \label{303.1}
\end{eqnarray}

and 
\begin{equation}
\mu _{1}^{n}\;=\;E[\,X_{1}^{L_{n}}\,],\quad \mu
_{2}^{n}\;=\;E[\,X_{2}^{L_{n}}\,].  \label{304}
\end{equation}

It is time also to give a formal definition for the \textit{total positivity
of order two participating in the formulation of the main theorem from the
main text.}

\begin{definition}
\label{def-tp2} A density $h(u,v)$ on $[0,1]^{2}$ is \textit{totally
positive of order two} ($TP_{2}$) if for all $0\leq u_{1}<u_{2}\leq 1$ and $%
0\leq v_{1}<v_{2}\leq 1$, 
\begin{equation}
h(u_{1},v_{1})\,h(u_{2},v_{2})\;\geq \;h(u_{1},v_{2})\,h(u_{2},v_{1}).
\label{307}
\end{equation}%
Equivalently (under standard regularity so that conditional expectation and
density exist), for $H$ a bivariate distribution with density $h$, the
conditional mean $x\mapsto m_{2|1}^{H}(x)=E[\,X_{2}^{H}\mid X_{1}^{H}=x\,]$
is non-decreasing.
\end{definition}

\begin{remark}
Under $TP_{2}$ the \textit{likelihood--ratio order} is monotone: for every $%
v_{1}<v_{2},$ the map%
\begin{equation}
u\mapsto \frac{h(u,v_{2})}{h(u,v_{1})}  \label{307.1}
\end{equation}

is non-decreasing in $u$. This can serve as an alternative definition for $%
TP_{2}.$
\end{remark}

\begin{flushleft}
\textbf{Main theorem (}$TP_{2}$\textbf{\ case)}
\end{flushleft}

We will prove that the map from (\ref{5.1}) preserves the $TP_{2}$ property.
Then will we prove in a main theorem of the appendix the subdiagonality.

\begin{lemma}
\label{tp-preservation} Let $h(u_{1},u_{2})$ be a $TP_{2}$ density on $%
[0,1]^{2}$ and let $A,B:[0,1]\rightarrow \lbrack 0,1]$ be two non-decreasing
differentiable bijections. Define 
\begin{equation*}
k(x_{1},x_{2})\;=\;A(x_{1})\,B(x_{2})\;h[A(x_{1}),\,B(x_{2})]\;A^{\prime
}(x_{1})\,B^{\prime }(x_{2}),\quad (x_{1},x_{2})\in \lbrack 0,1]^{2}.
\end{equation*}%
Then $k(x_{1},x_{2})$ is also $TP_{2}$ on $[0,1]^{2}$.
\end{lemma}

\begin{proof}
\noindent Since $A,B$ are non-decreasing, the map $(x_{1},x_{2})\mapsto
(u_{1},u_{2})=(A(x_{1}),B(x_{2}))$ preserves the \textit{rectangular ordering%
}. Thus for any $0\leq x_{1}<x_{1}^{\prime }\leq 1$ and $0\leq
x_{2}<x_{2}^{\prime }\leq 1$, we have 
\begin{equation}
A(x_{1})<A(x_{1}^{\prime }),\quad B(x_{2})<B(x_{2}^{\prime }).  \label{308}
\end{equation}%
Because $h$ is $TP_{2}$, 
\begin{equation}
h[A(x_{1}),B(x_{2})]h[A(x_{1}^{\prime }),B(x_{2}^{\prime })]\;\geq
\;h[A(x_{1}),B(x_{2}^{\prime })]\,h[A(x_{1}^{\prime }),B(x_{2})]  \label{309}
\end{equation}%
Multiplying both sides by the nonnegative factor 
\begin{equation}
A(x_{1})\,A(x_{1}^{\prime })\,A^{\prime }(x_{1})\,A^{\prime }(x_{1}^{\prime
})\;\times \;B(x_{2})\,B(x_{2}^{\prime })\,B^{\prime }(x_{2})\,B^{\prime
}(x_{2}^{\prime }),  \label{310}
\end{equation}%
we obtain exactly 
\begin{equation}
k(x_{1},x_{2})\,k(x_{1}^{\prime },x_{2}^{\prime })\;\geq
\;k(x_{1},x_{2}^{\prime })\,k(x_{1}^{\prime },x_{2}),  \label{311}
\end{equation}%
showing that $k$ is $TP_{2}$.
\end{proof}

\begin{theorem}
\label{subdiag} Assume the original density $f_{12}(u_{1},u_{2})$ is $TP_{2}$
on $[0,1]^{2}$. Then $l_{0}(x_{1},x_{2})$, the density of $L_{0}$, is $%
TP_{2} $. Moreover, for $n=0$ we have: 
\begin{equation}
L_{i}^{0}(x)\;\leq \;x,\quad \forall \,x\in (0,1),\quad i=1,2.  \label{312}
\end{equation}%
Consequently, for all $n\geq 1$, $L_{i}^{n}(x)<x$ on $(0,1)$.
\end{theorem}

\begin{proof}
\bigskip \textbf{(1) } A direct differentiation in (\ref{5.1}) gives:%
\begin{equation}
l_{n+1}(x_{1},x_{2})=\frac{L_{1}^{n,-1}(x_{1})L_{2}^{n,-1}(x_{2})}{%
\int_{0}^{1}\int_{0}^{1}u_{1}u_{2}dL_{n}(u_{1},u_{2})}%
l_{n}(L_{1}^{n,-1}(x_{1}),L_{2}^{n,-1}(x_{2}))(L_{1}^{n,-1})^{^{\prime
}}(x_{1})(L_{2}^{n,-1})^{^{\prime }}(x_{2}).  \label{313}
\end{equation}

Since $l_{n}$ ($f_{12}$) is $TP_{2}$ and $L_{i}^{-1}$ ($F_{i}^{-1}$) are
non-decreasing, the $TP_{2}$ preservation \textit{Lemma} {\footnotesize \ref%
{tp-preservation}} implies that $l_{0}(y_{1},y_{2})$ is $TP_{2}$ on $%
[0,1]^{2}$.

\medskip \noindent \textbf{(2) }We now prove that if $L_{n}$ has $TP_{2}$
density $l_{n}$, then $L_{i}^{n+1}(x)<x$ on $(0,1)$ for $i=1,2$. Fix $n\geq
0 $ and suppose $l_{n}$ is $TP_{2}$. Then the conditional mean 
\begin{equation}
m_{2\mid 1}^{n}(x_{1})\;=\;E[\,X_{2}^{L_{n}}\mid X_{1}^{L_{n}}=x_{1}]
\label{314}
\end{equation}%
is a (weakly) non-decreasing function of $x_{1}$ on $[0,1]$. Define 
\begin{equation}
g_{1}^{n}(x_{1})\;=\;x_{1}m_{2\mid
1}^{n}(x_{1})\;=\;x_{1}\,E[\,X_{2}^{L_{n}}\mid X_{1}^{L_{n}}=x_{1}].
\label{315}
\end{equation}%
Since $m_{2\mid 1}^{n}(x_{1})\geq 0$ and is non-decreasing, and $x_{1}\geq 0$%
, the product $g_{n}(x_{1})$ is also non-decreasing on $[0,1]$. Note
moreover $g_{n}(x_{1})$ is not almost surely constant under $L_{n}$, because 
$X_{1}^{L_{n}}$ has support $(0,1)$ and $m_{2\mid 1}^{n}$ varies strictly on
a set of positive measure ($TP_{2}$ densities on a rectangle cannot yield
a.s.\ constant conditional mean unless degenerate).

By the properties of the conditional expectations, we can write for the
first marginal 
\begin{equation}
L_{\,1}^{n+1}(x)\;=\;\frac{x\,E[\,X_{1}^{L_{n}}X_{2}^{L_{n}}\mid
X_{1}^{L_{n}}\leq L_{1}^{n,-1}(x)\,]}{E[\,X_{1}^{L_{n}}X_{2}^{L_{n}}]}\;=\;%
\frac{xE[g_{n}(X_{1}^{L_{n}})\mid X_{1}^{L_{n}}\leq L_{1}^{n,-1}(x)\,]}{%
E[g_{n}(X_{1}^{L_{n}})]}.  \label{316}
\end{equation}%
Since $g_{n}(X_{1}^{L_{n}})$ is non-decreasing in $X_{1}^{L_{n}}$ and $%
0<L_{1}^{n,-1}(x)<1$, it follows from the well-known fact that if $Y$ is a
real random variable and $\phi $ is non-decreasing, then for any $y_{0}$
with $P(Y\leq y_{0})\in (0,1)$, 
\begin{equation}
E[\phi (Y)\mid Y\leq y_{0}]\;\leq \;E[\phi (Y)].  \label{317}
\end{equation}%
Applying this with $Y=X_{1}^{L_{n}}$ under $L_{n}$ and $\phi =g$, we get 
\begin{equation}
E[g_{n}(X_{1}^{L_{n}})\mid X_{1}^{L_{n}}\leq L_{1}^{n,-1}(x)\,]\;\leq
\;E[g_{n}(X_{1}^{L_{n}})],  \label{318}
\end{equation}%
hence 
\begin{equation}
L_{1}^{n+1}(x)\;\leq \;x,\quad \forall \,x\in (0,1).  \label{319}
\end{equation}%
By symmetry (or by the identical argument with coordinates being swapped), 
\begin{equation}
L_{2}^{n+1}(x)\;\leq \;x,\quad \forall \,x\in (0,1).  \label{320}
\end{equation}%
This completes the proof with an obvious induction.
\end{proof}

\begin{theorem}
\label{Phi-conv-TP2} If the density $f_{12}$ is $TP_{2}$ (Totally Positive
of order 2), then the following hold for $i=1,2$:

\begin{enumerate}
\item The sequence of functions $\Phi _{n}^{i}$ converges pointwise to $1$
on $[0,1]$, that is, 
\begin{equation}
\lim_{n\rightarrow +\infty }\Phi _{n}^{i}(x)=\left\{ 
\begin{array}{c}
1,\text{ }x>0 \\ 
0,\text{ }x=0;%
\end{array}%
\right.  \label{320a}
\end{equation}

\item Furthermore, for any $\delta $ $\in (0,1)$, the convergence is uniform
on the interval $[\delta ,1].$
\end{enumerate}
\end{theorem}

\begin{proof}
By \textit{Theorem \ref{subdiag} we know that }$L_{i}^{n}(x)$ is subdiagonal 
\textit{for }$i=1,2$, i.e., $L_{i}^{n}(x)\leq x$. This combined with the
results from \textit{Appendix B,} and equation (\ref{232a}) in particular,
give even the more precise inequality:%
\begin{equation}
L_{i+}^{n}(x)\leq L_{i}^{n}(x)\leq \min [L_{i-}^{n}(x),x].  \label{402}
\end{equation}

However, folk knowledge based on basic geometry\footnote{{\footnotesize Can
easily be formalized by standard approximation theory.}} gives that due to
the subdiagonality, $L_{i}^{n}(x)\leq x^{\gamma n}$ holds for some $\gamma
_{n}>1$. This is enough for our purposes. So we get%
\begin{eqnarray}
L_{i+}^{n}(x) &\leq &L_{i}^{n}(x)\leq x^{\gamma n}  \label{403} \\
x^{\frac{1}{\gamma n}} &\leq &L_{i}^{n,-1}(x)\leq L_{i+}^{n,-1}(x).  \notag
\end{eqnarray}

Therefore%
\begin{equation}
\Phi _{n}^{i}(x)=L_{i}^{0,-1}(...(L_{i}^{n-1,-1}(L_{i}^{n,-1}(x)))\geq
\prod\nolimits_{j=0}^{n}x^{\frac{1}{\gamma _{j}}}=x^{\varepsilon
^{n}}\rightarrow 1,  \label{404}
\end{equation}

where $\exists $ $\varepsilon $ such that for $\forall $ $j$ holds $\gamma
_{j}<\varepsilon \in (0,1).$ Additionally, $L_{i}^{n,-1}(x)\leq 1$ for all $%
n $ and thus $\Phi _{n}^{i}(x)\leq 1$ as well$.$ The squeeze gives the
result of the theorem.
\end{proof}

\begin{flushleft}
{\large Appendix D.2}

\textbf{Notation and preliminaries}
\end{flushleft}

We keep to the notation of the previous appendix. We may note that we can
further write for the conditional mean in the form 
\begin{equation}
m_{2\mid 1}^{n}(x_{1})=\frac{\int_{0}^{1}u_{2}\,l_{n}(x_{1},u_{2})\,du_{2}}{%
l_{1}^{n}(x_{1})}.  \label{321}
\end{equation}

\begin{definition}
\label{def-rr2} A density $h(u,v)$ on $[0,1]^{2}$ is \textit{reverse regular
of order 2} ($RR_{2}$) if for all $0\leq u_{1}<u_{2}\leq 1$ and $0\leq
v_{1}<v_{2}\leq 1$, 
\begin{equation}
h(u_{1},v_{1})\,h(u_{2},v_{2})\;\leq \;h(u_{1},v_{2})\,h(u_{2},v_{1}).
\label{322}
\end{equation}%
Equivalently (under standard regularity so that conditional expectation and
density exist), for $H$ a bivariate distribution with density $h$, the
conditional mean $x\mapsto m_{2|1}^{H}(x)=E[\,X_{2}^{H}\mid X_{1}^{H}=x\,]$
is non-increasing.
\end{definition}

\begin{definition}
\label{def-srr2} A density $h(u,v)$ on $[0,1]^{2}$ is \textit{strictly
reverse regular of order 2} ($SRR_{2}$) if 
\begin{equation}
h(u_{1},v_{1})\,h(u_{2},v_{2})\;<\;h(u_{1},v_{2})\,h(u_{2},v_{1})
\label{322.1}
\end{equation}%
for all $0\leq u_{1}<u_{2}\leq 1$ and $0\leq v_{1}<v_{2}\leq 1$ such that
all four points lie in a region where $h>0$.

Equivalently (under standard regularity so that conditional expectation and
density exist), for $H$ a bivariate distribution with density $h$, the
conditional mean $x\mapsto m_{2|1}^{H}(x)=E[\,X_{2}^{H}\mid X_{1}^{H}=x\,]$
is strictly decreasing on every compact subinterval of $(0,1)$.
\end{definition}

\begin{remark}
Under $RR_{2}$ the \textit{likelihood--ratio order} is monotone: for every $%
v_{1}<v_{2},$ the map%
\begin{equation}
u\mapsto \frac{h(u,v_{2})}{h(u,v_{1})}  \label{307.2}
\end{equation}

is non-increasing in $u$. This can serve as an alternative definition for $%
RR_{2}.$
\end{remark}

\begin{remark}
Under $SRR_{2}$ the \textit{likelihood--ratio order} is strictly monotone:
for every $v_{1}<v_{2}$ and every compact interval $I\subset (0,1)$ with $%
\inf_{u\in I}h(u,v_{i})>0$ (no vanishing sections), $i=1,2$, the map%
\begin{equation}
u\mapsto \frac{h(u,v_{2})}{h(u,v_{1})}  \label{307.3}
\end{equation}

is strictly decreasing on $I$. This can serve as an alternative definition
for $SRR_{2}.$
\end{remark}

\begin{flushleft}
\textbf{Preservation of }$RR_{2}$\textbf{\ }
\end{flushleft}

We will prove that the map from (\ref{5.1}) preserves the $RR_{2}$ property
both in weak and strong sense.

\begin{lemma}
\label{rr2} Let $h(u,v)$ be $RR_{2}$. If $\phi ,\psi :[0,1]\rightarrow
\lbrack 0,1]$ are non-decreasing, then 
\begin{equation}
f(x,y)=h[\phi (x),\,\psi (y)]  \label{323}
\end{equation}%
is $RR_{2}$\noindent
\end{lemma}

\begin{proof}
Take $x_{1}<x_{2}$, $y_{1}<y_{2}$. Since $\phi ,\psi $ are non-decreasing, $%
\phi (x_{1})\leq \phi (x_{2}),\;\psi (y_{1})\leq \psi (y_{2})$. By $RR_{2}$
of $h$, 
\begin{equation}
h[\phi (x_{1}),\psi (y_{1})]\,h[\phi (x_{2}),\psi (y_{2})]\;\leq \;h[\phi
(x_{1}),\psi (y_{2})]\,h[\phi (x_{2}),\psi (y_{1})].  \label{324}
\end{equation}%
Hence 
\begin{equation}
f(x_{1},y_{1})\,f(x_{2},y_{2})\;\leq \;f(x_{1},y_{2})\,f(x_{2},y_{1}).
\label{325}
\end{equation}%
If $\phi (x_{1})=\phi (x_{2})$ or $\psi (y_{1})=\psi (y_{2})$, equality
holds. Thus $f$ is $RR_{2}$. \hfill
\end{proof}

\begin{lemma}
\label{ker} If $f(x,y)$ is $RR_{2}$ and $k_{1}(x)\geq 0$, $k_{2}(y)\geq 0$,
then 
\begin{equation}
H(x,y)=k_{1}(x)\,k_{2}(y)\,f(x,y)  \label{326}
\end{equation}%
is $RR_{2}$.
\end{lemma}

\begin{proof}
\noindent For $x_{1}<x_{2}$, $y_{1}<y_{2}$, let 
\begin{equation}
D=H(x_{1},y_{1})\,H(x_{2},y_{2})\;-\;H(x_{1},y_{2})\,H(x_{2},y_{1}).
\label{327}
\end{equation}%
Substitute $H$ 
\begin{equation}
D=k_{1}(x_{1})\,k_{1}(x_{2})\,k_{2}(y_{1})\,k_{2}(y_{2})[f(x_{1},y_{1})%
\,f(x_{2},y_{2})\;-\;f(x_{1},y_{2})\,f(x_{2},y_{1})].  \label{328}
\end{equation}%
Since each $k_{i}\geq 0$ and $f$ is $RR_{2}$, the bracket is $\leq 0$. Thus $%
D\leq 0$, proving $H$ is $RR_{2}$. \hfill
\end{proof}

\begin{theorem}
\label{rr2-pres} Let $F$ be a continuous bivariate d.f. on $[0,1]^{2}$ whose
density $f_{F}$ is $RR_{2}$. Define $L_{n}$ by (\ref{5.1}). Then each
density $l_{n}$ of $L_{n}$ is $RR_{2}$.
\end{theorem}

\begin{proof}
\noindent We prove by induction.

\medskip \noindent \textbf{Base Case ($n=0$).} From \textit{Appendix D1}, we
have 
\begin{equation}
l_{0}(x_{1},x_{2})=\frac{F_{1}^{-1}(x_{1})F_{2}^{-1}(x_{2})}{%
\int_{0}^{1}\int_{0}^{1}u_{1}u_{2}dF(u_{1},u_{2})}%
f_{12}(F_{1}^{-1}(x_{1}),F_{2}^{-1}(x_{2}))(F_{1}^{-1})^{^{\prime
}}(x_{1})(F_{2}^{-1})^{^{\prime }}(x_{2}).  \label{329}
\end{equation}

Let $\phi _{i}(x_{i})=F_{i}^{-1}(x_{i})$ (non-decreasing, $%
(F_{i}^{-1})^{\prime }(x_{i})\geq 0$). Then 
\begin{equation}
f_{F}^{\ast }(x_{1},x_{2})=f_{F}(\phi _{1}(x_{1}),\,\phi _{2}(x_{2}))
\label{330}
\end{equation}%
is $RR_{2}$ by \textit{Lemma} {\footnotesize \ref{rr2}}. Also set 
\begin{equation}
k_{i}(x_{i})=\phi _{i}(x_{i})\,(F_{i}^{-1})^{\prime }(x_{i})\;\geq 0.
\label{331}
\end{equation}%
Thus 
\begin{equation}
l_{0}(x_{1},x_{2})=\frac{1}{I_{F}}\;k_{1}(x_{1})\,k_{2}(x_{2})\,f_{F}^{\ast
}(x_{1},x_{2})  \label{332}
\end{equation}%
is $RR_{2}$ by \textit{Lemma} {\footnotesize \ref{ker}}.

\medskip \noindent \textbf{Inductive Step.} Suppose $l_{n}$ is $RR_{2}$.
From \textit{Appendix D1}, we have:%
\begin{equation}
l_{n+1}(x_{1},x_{2})=\frac{L_{1}^{n,-1}(x_{1})L_{2}^{n,-1}(x_{2})}{%
\int_{0}^{1}\int_{0}^{1}u_{1}u_{2}dL_{n}(u_{1},u_{2})}%
l_{n}(L_{1}^{n,-1}(x_{1}),L_{2}^{n,-1}(x_{2}))(L_{1}^{n,-1})^{^{\prime
}}(x_{1})(L_{2}^{n,-1})^{^{\prime }}(x_{2}).  \label{333}
\end{equation}

Set 
\begin{equation}
l_{n}^{\ast }(x_{1},x_{2})=l_{n}[a_{n}(x_{1}),\,b_{n}(x_{2})].  \label{334}
\end{equation}%
Since $l_{n}$ is $RR_{2}$ and $a_{n}$, $b_{n}$ are non-decreasing, $%
l_{n}^{\ast }$ is $RR_{2}$ by \textit{Lemma} {\footnotesize \ref{rr2}}. Also
define 
\begin{equation}
K_{1}(x_{1})=a_{n}(x_{1})\,a_{n}^{\prime
}(x_{1})=L_{1}^{n,-1}(x_{1})\,[L_{1}^{n,-1}]^{\prime }(x_{1})\;\geq 0,
\label{335}
\end{equation}%
\begin{equation}
K_{2}(x_{2})=b_{n}(x_{2})\,b_{n}^{\prime
}(x_{2})=L_{2}^{n,-1}(x_{2})\,[L_{2}^{n,-1}]^{\prime }(x_{2})\;\geq 0.
\label{336}
\end{equation}%
Then 
\begin{equation}
l_{n+1}(x_{1},x_{2})=\frac{1}{I_{n}}\,K_{1}(x_{1})\,K_{2}(x_{2})\;l_{n}^{%
\ast }(x_{1},x_{2})  \label{337}
\end{equation}%
is $RR_{2}$ by \textit{Lemma} {\footnotesize \ref{ker}}. This completes the
induction.
\end{proof}

Taking strict signs in the proof above, gives directly that the $SRR_{2}$
property is preserved across the iterations as well. We will need also the
following sufficient condition for $SRR_{2}$:

\begin{lemma}
\label{aux-lem} Let $f(x,v)$ be a family of probability densities on $%
\mathcal{X}\times \mathcal{V}$, where $\mathcal{X}$ and $\mathcal{V}$ are
intervals in $R$. Assume the following:

($RR_{2}$ Property) For any $v_{1},v_{2}\in \mathcal{V}$ with $v_{1}<v_{2}$,
the likelihood ratio $L(x)=f(x,v_{2})/f(x,v_{1})$ is non-increasing in $x$.

(Regularity and Strict Log-Concavity) For each fixed $v\in \mathcal{V}$, the
log-density $g(x,v)=\log f(x,v)$ is a real-analytic and strictly log-concave
function of $x$ on the interior of its support.

(Non-degeneracy) For every pair $v_{1}<v_{2}$, the function $x\mapsto
g(x,v_{2})-g(x,v_{1})$ is not constant over the interior of the common
support.

Then the kernel $f(x,v)$ has the strict $RR_{2}$ property ($SRR_{2}$). That
is, for any $v_{1}<v_{2}$, the ratio $L(x)$ is strictly decreasing in $x$
over the interior of their common support.
\end{lemma}

\begin{proof}
Let $g(x,v)=\log f(x,v)$. The proof proceeds by contradiction. Assume the
conclusion is false. Then for some $v_{1}<v_{2}$, the likelihood ratio $L(x)$
is not strictly decreasing. By condition (i), $L(x)$ is continuous and
non-increasing. The failure to be strictly decreasing implies there must
exist a non-degenerate interval $I\subset \mathcal{X}$ on which $L(x)$ is
constant. We may write: 
\begin{equation}
L(x)=\frac{f(x,v_{2})}{f(x,v_{1})}\equiv C\quad \text{for all }x\in I.
\label{338}
\end{equation}%
Taking the logarithm gives $d(x)\equiv \log C$, where $%
d(x)=g(x,v_{2})-g(x,v_{1})$. The critical step is to extend this local
constancy to a global one. By the regularity hypothesis, the functions $%
x\mapsto g(x,v_{1})$ and $x\mapsto g(x,v_{2})$ are real-analytic. Their
difference, $d(x)$, is therefore also a real-analytic function of $x$. We
have established that $d(x)$ is constant on the interval $I$. The \textit{%
Identity Theorem} for real-analytic functions states that if an analytic
function is constant on any open subset of its connected domain, it must be
constant over the entire domain. Therefore, the function $d(x)$ must be
constant for all $x$ in the interior of the common support. This, however,
is a direct contradiction of the non-degeneracy assumption. The initial
assumption that $L(x)$ is not strictly decreasing must be false. Thus, $L(x)$
is strictly decreasing over the interior of the common support.
\end{proof}

All above allows to conclude that the $SRR_{2}$ property is preserved across
iterations if the initial density $f_{F}$ is either $SRR_{2}$ or is an $%
RR_{2}$ density that also satisfies the regularity conditions of
log-concavity and non-degeneracy.\hfill

\begin{flushleft}
\textbf{Single-crossing property of }$L_{i}^{n}$ for $i=1,2$
\end{flushleft}

We first prove a claim and then the main lemma of the section.

\begin{claim}
\label{single-sol} The equation 
\begin{equation}
L_{1}^{n}(y)=\frac{1}{I_{n}}\int_{0}^{y}\int_{0}^{1}u_{1}u_{2}%
\,l_{n}(u_{1},u_{2})\,du_{2}\,du_{1}  \label{339}
\end{equation}%
has at most one solution for $y\in (0,1)$.
\end{claim}

\begin{proof}
\vspace{0pt} Let's define the function $H_{n}(y)$ as the difference between
the two sides of the equation (\ref{339}):%
\begin{equation}
H_{n}(y)=G_{n}(y)-L_{1}^{n}(y),  \label{339a}
\end{equation}

where: 
\begin{equation}
G_{n}(y)=\frac{1}{I_{n}}\int_{0}^{y}\int_{0}^{1}u_{1}u_{2}%
\,l_{n}(u_{1},u_{2})\,du_{2}\,du_{1}.  \label{344}
\end{equation}

We know $H_{n}(0)=H_{n}(1)=0$. If $H_{n}(y)$ had two or more roots in $(0,1)$%
, then by \textit{Rolle's Theorem}, its derivative $H_{n}(y)$ would have at
least two roots in $(0,1)$. Differentiating with respect to $y$ gives 
\begin{equation}
H_{n}^{\prime }(y)=\left( \frac{1}{I_{n}}\int_{0}^{1}yu_{2}\,l_{n}(y,u_{2})%
\,du_{2}\right) -\left( \int_{0}^{1}l_{n}(y,u_{2})\,du_{2}\right) .
\label{340}
\end{equation}%
The roots of $H_{n}^{\prime }(y)$ occur where the ratio of the two terms is $%
1$. This ratio is 
\begin{equation}
\lambda _{n}(y)=\frac{\frac{y}{I_{n}}\int_{0}^{1}u_{2}l_{n}(y,u_{2})du_{2}}{%
l_{1}^{n}(y)}=\frac{y}{I_{n}}E[X_{2}^{L_{n}}\mid X_{1}^{L_{n}}=y].
\label{341}
\end{equation}%
Here $I_{n}=E[X_{1}^{L_{n}}X_{2}^{L_{n}}]\in (0,1]$. The function $\lambda
_{n}(y)$ is a product of two continuous, positive functions:

\begin{itemize}
\item $f_{1}(y)=\frac{y}{I_{n}}$, which is strictly increasing;

\item $f_{2}(y)=E[X_{2}^{L_{n}}\mid X_{1}^{L_{n}}=y]$. If $l_{n}$ is $RR_{2}$%
, this conditional expectation is a non-increasing function of $y$.
\end{itemize}

The product of a strictly positive, strictly increasing function and a
positive, non-increasing function is unimodal. This implies $\lambda _{n}(y)$
is such as well.

A unimodal function can cross any horizontal line (like $y=1$) at most
twice. Thus, the equation $\lambda _{n}(y)=1$ has at most two solutions,
meaning $H_{n}^{\prime }(y)$ has at most two roots in $(0,1)$. Suppose for
contradiction that $H_{n}(y)$ has two interior roots, $y_{1}<y_{2}$. By 
\textit{Rolle's Theorem}, this would imply that $H_{n}^{\prime }(y)$ must
have at least three roots in $(0,1)$, a contradiction. Therefore, $H_{n}(y)$
can have at most one root in $(0,1)$. The proof is analogous for the other
marginal.
\end{proof}

\begin{lemma}
\label{one-sol} For each $n\geq 0$, the marginal equation $L_{i}^{n}(x)=x$
has at most one solution $c_{i}^{n}$ in $(0,1)$ for $i=1,2$.
\end{lemma}

\begin{proof}
\vspace{0pt} \noindent Let's fix a step $n$ and analyze the crossing
equation for the next iterate, $L_{1}^{n+1}(x)=x$. From the iterative
definition, with $x_{1}=x$ and $x_{2}=1$, we have $L_{2}^{n,-1}(1)=1$. The
equation becomes 
\begin{equation}
\frac{1}{I_{n}}\int_{0}^{L_{1}^{n,-1}(x)}\left(
\int_{0}^{1}u_{1}u_{2}\,l_{n}(u_{1},u_{2})\,du_{2}\right) \,du_{1}=x.
\label{342}
\end{equation}%
We perform a change of variables. Let $y=L_{1}^{n,-1}(x)$, which implies $%
x=L_{1}^{n}(y)$. Since $L_{1}^{n}$ is strictly increasing, this map is a
bijection from $(0,1)$ to $(0,1)$. Substituting $x=L_{1}^{n}(y)$ into the
equation gives 
\begin{equation}
\frac{1}{I_{n}}\int_{0}^{y}\int_{0}^{1}u_{1}u_{2}\,l_{n}(u_{1},u_{2})%
\,du_{2}\,du_{1}=L_{1}^{n}(y).  \label{343}
\end{equation}%
This equation is now entirely in terms of the distribution $L_{n}$. The
problem has been reduced to finding the number of solutions to this new
equation for $y\in (0,1)$. They are at most one due to the previous claim
and the $RR_{2}$ property preservation at each step of the iteration. The
proof is analogous for the other marginal.
\end{proof}

We check directly from the iterative equation (\ref{5.1}) that 
\begin{equation}
L_{1}^{n+1}[L_{1}^{n}(y)]=L_{1}^{n}(y)\quad \Longleftrightarrow \quad
G_{n}(y)=L_{1}^{n}(y).  \label{345}
\end{equation}%
Hence the next crossing is of the form $c_{1}^{n+1}=L_{1}^{n}(y_{n}^{\ast })$%
, where $y_{n}^{\ast }$ solves the above equation.

In the case where the equation $\lambda _{n}(y)=1$ has two solutions in $%
(0,1)$, which we denote by $y_{n}^{-}$ and $y_{n}^{+}$ , it follows that $%
y_{n}^{\ast }$ lies in the interval $(y_{n}^{-},y_{n}^{+})$. Consequently, $%
\lambda _{n}(y_{n}^{\ast })>1$, and the function's behavior is given by

\begin{equation}
\lambda _{n}<1\ \text{ on }(0,y_{n}^{-}),\quad \lambda _{n}>1\ \text{ on }%
(y_{n}^{-},y_{n}^{+}),\quad \lambda _{n}<1\ \text{ on }(y_{n}^{+},1).
\label{345a}
\end{equation}

In the case where the equation $\lambda _{n}(y)=1$ has a single solution in $%
(0,1)$, the point $y_{n}^{\ast }$ does not exist, and as a consequence, the
function $L_{1}^{n}(y)$ is subdiagonal.

\begin{flushleft}
\textbf{Compactness of }$L_{n}$
\end{flushleft}

Starting with several auxiliary results, we next provide \textit{Lipschitz
bounds} and establish the relative compactness of the family $(L_{n})_{n\geq
0}$. The first step is to prove that the normalizing denominators in the
bivariate iteration cannot collapse to zero.

The normalizing denominator of the $n$-th bivariate iterate is 
\begin{equation*}
I_{n}=\int_{0}^{1}\int_{0}^{1}u_{1}u_{2}\,dL_{n}(u_{1},u_{2})
\end{equation*}

We shall compare $I_{n}$ with the corresponding denominator along the
attainable \textit{Fr\'{e}chet--Hoeffding lower-bound} trajectory. Let, as
before, $L_{-}^{n}$ denote the $n$-th iterate generated by the lower
extremal initial law $F_{-}$. Thus $L_{-}^{n}$ is the $n$-th element of the
trajectory for which the dependence structure is the\textit{\ Fr\'{e}%
chet--Hoeffding lower-bound} at each step. We write its marginals as 
\begin{equation*}
L_{1-}^{n}(x)=L_{-}^{n}(x,1),\qquad L_{2-}^{n}(x)=L_{-}^{n}(1,x),
\end{equation*}%
and their generalized inverses as 
\begin{equation*}
T_{1-,n}=L_{1-}^{n,-1},\qquad T_{2-,n}=L_{2-}^{n,-1}.
\end{equation*}%
Define the lower-bound denominator 
\begin{equation}
I_{-,n}=\int_{0}^{1}\int_{0}^{1}u_{1}u_{2}\,dL_{-}^{n}(u_{1},u_{2}).
\label{345a23}
\end{equation}

We also need the following notation. Let $L_{W}^{n}$ denote the attainable 
\textit{Fr\'{e}chet--Hoeffding} lower element at level $n$. Equivalently, $%
L_{W}^{n}$ is obtained by replacing the copula at level $n$ by the lower 
\textit{Fr\'{e}chet--Hoeffding copula} $W$ inside the admissible \textit{%
Lorenz iteration}. Its denominator is 
\begin{equation}
I_{W,n}=\int_{0}^{1}\int_{0}^{1}u_{1}u_{2}\,dL_{W}^{n}(u_{1},u_{2}).
\label{345a24}
\end{equation}

The rigidity result for the \textit{Fr\'{e}chet--Hoeffding lower-bound}
case, \textit{Claim~\ref{fh-claim4}}, says that the \textit{Lorenz curve}
reaches its attainable \textit{Fr\'{e}chet--Hoeffding lower-bound} if and
only if the generating law itself is the lower extremal law. Consequently,
by induction along the iteration, the attainable lower element at level $n$
coincides with the lower extremal trajectory generated by $F_{-}$. In the
present notation, 
\begin{equation}
L_{W}^{n}=L_{-}^{n},\qquad n\geq 0.  \label{345a25}
\end{equation}%
Therefore the corresponding denominators are identical 
\begin{equation}
I_{W,n}=I_{-,n},\qquad n\geq 0.  \label{345a26}
\end{equation}

We first prove that the lower extremal denominators $I_{-,n}$ are uniformly
bounded away from zero.

\begin{lemma}
\noindent \textbf{\label{lbound}} The denominators of the \textit{Fr\'{e}%
chet--Hoeffding lower-bound} trajectory satisfy 
\begin{equation}
\inf_{n\geq 0}I_{-,n}>0.  \label{345a1}
\end{equation}%
\medskip
\end{lemma}

\begin{proof}
For the\textit{\ Fr\'{e}chet--Hoeffding lower-bound} trajectory we have,
from the countermonotone representation, 
\begin{equation}
I_{-,n}=\int_{0}^{1}T_{1-,n}(u)\,T_{2-,n}(1-u)\,du.  \label{345a27}
\end{equation}%
Moreover, the lower-bound identities established earlier give 
\begin{equation}
T_{1-,n}(u)+T_{2-,n}(1-u)=1,\qquad 0\leq u\leq 1.  \label{345a28}
\end{equation}%
Hence 
\begin{equation}
I_{-,n}=\int_{0}^{1}T_{1-,n}(u)(1-T_{1-,n}(u))\,du.  \label{345a29}
\end{equation}%
Equivalently, by using the second marginal, 
\begin{equation}
I_{-,n}=\int_{0}^{1}T_{2-,n}(u)(1-T_{2-,n}(u))\,du.  \label{345a30}
\end{equation}

Moreover, the first marginal of the next lower extremal iterate satisfies 
\begin{equation}
K_{n+1}(x)=\frac{\int_{0}^{x}T_{1-,n}(u)(1-T_{1-,n}(u))\,du}{I_{-,n}}.
\label{345a30f}
\end{equation}

Put 
\begin{equation}
w(x)=x(1-x).  \label{345a30g}
\end{equation}%
Then 
\begin{equation}
I_{-,n}=\int_{0}^{1}w(T_{n}(u))\,du.  \label{345a30h}
\end{equation}%
Since $K_{n+1}$ is absolutely continuous, its density is 
\begin{equation}
q_{n}(u)=T_{n+1}^{\prime }(u)=\frac{w(T_{n}(u))}{I_{-,n}}\quad \text{for
a.e. }u\in \lbrack 0,1].  \label{345a33}
\end{equation}%
The function $S_{n}$ is nondecreasing, while $w(x)=x(1-x)$ is increasing on $%
[0,\frac{1}{2}]$ and decreasing on $[\frac{1}{2},1]$. Therefore $q_{n}$ is a
unimodal probability density on $[0,1]$. Also, 
\begin{equation*}
0\leq w(x)\leq \frac{1}{4},
\end{equation*}%
so 
\begin{equation}
0\leq q_{n}(u)\leq \frac{1}{4I_{-,n}}.  \label{345a34}
\end{equation}%
Set 
\begin{equation}
M_{n}=\frac{1}{4I_{-,n}}.  \label{345a34a}
\end{equation}

We now use a simple layer-cake estimate for unimodal densities. Let $q$ be
any unimodal probability density on $[0,1]$ such that $0\leq q\leq M$. For $%
t\in \lbrack 0,M]$, define the superlevel set 
\begin{equation}
A_{t}=\{u\in \lbrack 0,1]:q(u)>t\}.  \label{345a35a}
\end{equation}%
Since $q$ is unimodal, each $A_{t}$ is an interval, possibly empty. Let 
\begin{equation}
\ell (t)=|A_{t}|.  \label{345a35b}
\end{equation}%
Then, by the layer-cake formula, 
\begin{equation}
\int_{0}^{M}\ell (t)\,dt=\int_{0}^{1}q(u)\,du=1.  \label{345a35}
\end{equation}%
For an interval $A\subset \lbrack 0,1]$ of length $r$, the quantity 
\begin{equation}
\int_{A}u(1-u)\,du  \label{345a35aa}
\end{equation}%
is minimized when $A$ is placed at one of the endpoints of $[0,1]$.
Therefore, with 
\begin{equation}
H(r)=\int_{0}^{r}u(1-u)\,du=\frac{r^{2}}{2}-\frac{r^{3}}{3},  \label{345a36}
\end{equation}%
we have 
\begin{equation}
\int_{A_{t}}u(1-u)\,du\geq H(\ell (t)).  \label{345a37}
\end{equation}%
Using again the layer-cake formula, 
\begin{eqnarray}
\int_{0}^{1}u(1-u)q(u)\,du &=&\int_{0}^{M}\int_{A_{t}}u(1-u)\,du\,dt  \notag
\\
&\geq &\int_{0}^{M}H(\ell (t))\,dt.  \label{345a38}
\end{eqnarray}%
Since 
\begin{equation}
H(r)=\frac{r^{2}}{2}-\frac{r^{3}}{3}\geq \frac{r^{2}}{6},\qquad 0\leq r\leq
1,  \label{345a38a}
\end{equation}%
we obtain 
\begin{eqnarray}
\int_{0}^{1}u(1-u)q(u)\,du &\geq &\frac{1}{6}\int_{0}^{M}\ell (t)^{2}\,dt 
\notag \\
&\geq &\frac{1}{6}\cdot \frac{\left( \int_{0}^{M}\ell (t)\,dt\right) ^{2}}{M}%
=\frac{1}{6M}  \label{345a39}
\end{eqnarray}%
where we used \textit{Cauchy's inequality} and (\ref{345a35}).

Applying this estimate to $q=q_{n}$ and $M=M_{n}=(4I_{-,n})^{-1}$, we get 
\begin{equation}
I_{-,n+1}=\int_{0}^{1}u(1-u)q_{n}(u)\,du\geq \frac{1}{6M_{n}}=\frac{2}{3}%
I_{-,n}.  \label{345a40}
\end{equation}

We also need the sharper estimate valid for small $I_{-,n}$. If $I_{-,n}\leq 
\frac{1}{16}$, then $M_{n}\geq 4$, hence $1/M_{n}\leq 1/4$. On the interval $%
[0,1/4]$, the function $H$ is convex and coincides with its lower convex
envelope on $[0,1]$. Therefore \textit{Jensen's inequality} applied to the
lower convex envelope of $H$, together with (\ref{345a35}), yields 
\begin{equation}
\int_{0}^{M}H(\ell (t))\,dt\geq MH\!\left( \frac{1}{M}\right) .
\label{345a41}
\end{equation}%
Consequently, for $I_{-,n}\leq \frac{1}{16}$, 
\begin{eqnarray}
I_{-,n+1} &\geq &M_{n}H\!\left( \frac{1}{M_{n}}\right) =\frac{1}{4I_{-,n}}%
\left[ \frac{(4I_{-,n})^{2}}{2}-\frac{(4I_{-,n})^{3}}{3}\right]  \notag \\
&=&2I_{-,n}-\frac{16}{3}I_{-,n}^{2}\geq \frac{5}{3}I_{-,n}.  \label{345a42}
\end{eqnarray}

We now prove that $I_{-,n}$ cannot approach zero. Suppose, by contradiction,
that 
\begin{equation}
\inf_{n\geq 0}I_{-,n}=0.  \label{345a42a}
\end{equation}%
Since $I_{-,0}>0$ under the standing non-degeneracy assumptions\footnote{$%
T_{0}${\footnotesize {} is continuous and nondegenerate, it cannot be
supported on }$\{0,1\}$ to make{\footnotesize .}}, we may choose $N\geq 1$
such that 
\begin{equation}
I_{-,N}<\min \left\{ \frac{1}{24},I_{-,0}\right\} .  \label{345a43}
\end{equation}%
From the rough estimate (\ref{345a40}), 
\begin{equation}
I_{-,N}\geq \frac{2}{3}I_{-,N-1},  \label{345a43a}
\end{equation}%
and hence 
\begin{equation}
I_{-,N-1}\leq \frac{3}{2}I_{-,N}<\frac{1}{16}.  \label{345a43b}
\end{equation}%
Therefore the sharper estimate (\ref{345a42}), applied at the index $N-1$,
gives 
\begin{equation}
I_{-,N}\geq \frac{5}{3}I_{-,N-1},  \label{345a43c}
\end{equation}%
or equivalently 
\begin{equation}
I_{-,N-1}\leq \frac{3}{5}I_{-,N}.  \label{345a43d}
\end{equation}%
In particular $I_{-,N-1}<1/16$, so the same argument can be repeated
backwards. Iterating, we get 
\begin{equation}
I_{-,0}\leq \left( \frac{3}{5}\right) ^{N}I_{-,N}<I_{-,N}.  \label{345a43e}
\end{equation}%
This contradicts $I_{-,N}<I_{-,0}$. Hence 
\begin{equation}
\inf_{n\geq 0}I_{-,n}>0.  \label{345a43f}
\end{equation}
\end{proof}

\begin{remark}
\noindent \textbf{\label{lbound-stronger-than-appendix-c}} \textit{Lemma~\ref%
{lbound}} gives a more general proof of the non-collapse of the \textit{Fr%
\'{e}chet--Hoeffding lower-bound} denominator than the argument in \textit{%
Appendix~C}. In \textit{Appendix~C}, the corresponding statement \textit{%
Claim~\ref{In-noncollapse-from-An}} was obtained under the additional
hypotheses used to prove the vanishing of the asymmetry functional $A_{n}$;
in particular, it relied on the assumptions entering \textit{Claim~\ref%
{An-vanishing}}. The proof above does not use those hypotheses. It does not
require a one-sign regime for the defect $L_{n}(x)+L_{n}(1-x)-1$.
\end{remark}

\begin{remark}
The assumption of strictly positive marginal densities---maintained in both
the main theorem and the extremal cases---is essential, as discussed in
detail in \textit{Appendix E}.\hfill
\end{remark}

We can now prove the desired denominator non-collapse for the original
bivariate iteration.

\begin{lemma}
\textbf{\label{denominator-noncollapse-2d}} The normalizing constants 
\begin{equation}
I_{n}=\int_{0}^{1}\int_{0}^{1}u_{1}u_{2}\,dL_{n}(u_{1},u_{2})
\label{345a43h}
\end{equation}%
are uniformly bounded away from zero. That is, 
\begin{equation}
\inf_{n\geq 0}I_{n}>0.  \label{345a222}
\end{equation}%
\medskip
\end{lemma}

\begin{proof}
Fix $n\geq 0$. Let $C_{n}$ be the copula of $L_{n}$. Following the notation
adopted in \textit{Appendix B}, let $F_{C_{n}}$ denote the bivariate
distribution associated with the $n$-th iterate and copula $C_{n}$. Then 
\begin{equation}
I_{n}=E^{F_{C_{n}}}(X_{1}X_{2}).  \label{345a3}
\end{equation}

Let $W$ denote the bivariate \textit{Fr\'{e}chet--Hoeffding lower-bound}
copula, 
\begin{equation}
W(u,v)=\max \{u+v-1,0\}.  \label{345a3a}
\end{equation}%
Let $F_{W}$ be the corresponding attainable lower-bound distribution at
level $n$, and let $L_{W}^{n}$ be its \textit{Lorenz curve}. By definition
of $I_{W,n}$, 
\begin{equation}
I_{W,n}=E^{F_{W}}(X_{1}X_{2}).  \label{345a44}
\end{equation}

Since $W\leq C_{n}\leq M$, \textit{Claim~\ref{com-claim1}}, applied with $%
C_{1}=W$ and $C_{2}=C_{n}$, gives 
\begin{equation}
E^{F_{W}}(X_{1}X_{2})\leq E^{F_{C_{n}}}(X_{1}X_{2}).  \label{345a4}
\end{equation}%
Using (\ref{345a3}) and (\ref{345a44}), this becomes 
\begin{equation}
I_{W,n}\leq I_{n}.  \label{345a45}
\end{equation}

By the rigidity of the attainable \textit{Fr\'{e}chet--Hoeffding lower-bound}
trajectory, see \textit{Claim~\ref{fh-claim4}}, we have 
\begin{equation}
L_{W}^{n}=L_{-}^{n}.  \label{345a45a}
\end{equation}%
Therefore 
\begin{equation}
I_{W,n}=I_{-,n}.  \label{345a46}
\end{equation}%
Combining (\ref{345a45}) and (\ref{345a46}), we obtain 
\begin{equation}
I_{n}\geq I_{-,n},\qquad n\geq 0.  \label{345a47}
\end{equation}

By \textit{Lemma~\ref{lbound}}, 
\begin{equation}
c_{-}=\inf_{n\geq 0}I_{-,n}>0.  \label{345a47a}
\end{equation}%
Hence 
\begin{equation}
I_{n}\geq I_{-,n}\geq c_{-},\qquad n\geq 0.  \label{345a47b}
\end{equation}%
Therefore 
\begin{equation}
\inf_{n\geq 0}I_{n}\geq c_{-}>0.  \label{345a47c}
\end{equation}%
This proves (\ref{345a222}).
\end{proof}

\begin{corollary}
\textbf{\label{srr2-denominator-noncollapse}} In particular, in the $SRR_{2}$
case the denominators of the bivariate iteration are uniformly bounded away
from zero 
\begin{equation}
\liminf_{n\rightarrow \infty }I_{n}>0.  \label{345a48}
\end{equation}%
\medskip
\end{corollary}

\begin{proof}
If the initial distribution is $SRR_{2}$, then the iterates considered in
the $SRR_{2}$ part of the proof are admissible bivariate distribution
functions, and their copulas $C_{n}$ satisfy the \textit{Fr\'{e}%
chet--Hoeffding bounds} 
\begin{equation}
W\leq C_{n}\leq M.  \label{345a48a}
\end{equation}%
Thus the comparison in \textit{Claim~\ref{com-claim1}} applies at every
step. The argument of \textit{Lemma~\ref{denominator-noncollapse-2d}} gives 
\begin{equation}
I_{n}\geq I_{W,n}=I_{-,n}.  \label{345a48b}
\end{equation}%
The lower extremal denominators satisfy 
\begin{equation}
\inf_{n\geq 0}I_{-,n}>0  \label{345a48c}
\end{equation}%
by \textit{Lemma~\ref{lbound}}. Therefore 
\begin{equation}
\liminf_{n\rightarrow \infty }I_{n}\geq \inf_{n\geq 0}I_{-,n}>0.
\label{345a48d}
\end{equation}
\end{proof}

\begin{lemma}
\textbf{\label{marginal-lipschitz-from-denominator} }\noindent Under the
conclusion of Lemma~\ref{denominator-noncollapse-2d}, the marginals of the
iterates $\{L_{n}^{F}\}_{n\geq 1}$ are uniformly Lipschitz. More precisely,
for $i=1,2$, 
\begin{equation}
L_{i}^{n+1}(b)-L_{i}^{n+1}(a)\leq \frac{1}{c_{F}}(b-a),\qquad 0\leq a\leq
b\leq 1,\quad n\geq 0.  \label{345a99}
\end{equation}
\end{lemma}

\begin{proof}
\noindent We prove the assertion for $i=1$; the proof for $i=2$ is
identical. Let $0\leq a\leq b\leq 1$, and set 
\begin{equation}
\alpha =L_{1}^{n,}{}^{-1}(a),\qquad \beta =L_{1}^{n,}{}^{-1}(b).
\label{345a10}
\end{equation}%
Using the definition of the \textit{Lorenz iteration} and then taking the
second coordinate equal to $1$, we have 
\begin{equation}
L_{1}^{n+1}(x)=L_{n+1}^{F}(x,1)=\frac{\int_{0}^{L_{1}^{n,}{}^{-1}(x)}%
\int_{0}^{1}u_{1}u_{2}\,dL_{n}(u_{1},u_{2})}{I_{n}}.  \label{345a11}
\end{equation}%
Therefore 
\begin{equation}
L_{1}^{n+1}(b)-L_{1}^{n+1}(a)=\frac{\int_{\alpha }^{\beta
}\int_{0}^{1}u_{1}u_{2}\,dL_{n}(u_{1},u_{2})}{I_{n}}.  \label{345a12}
\end{equation}%
Since $0\leq u_{1}u_{2}\leq 1$ on $[0,1]^{2}$, it follows that 
\begin{equation}
L_{1}^{n+1}(b)-L_{1}^{n+1}(a)\leq \frac{\int_{\alpha }^{\beta
}\int_{0}^{1}dL_{n}(u_{1},u_{2})}{I_{n}}.  \label{345a13}
\end{equation}%
The last integral is exactly the marginal mass assigned by $L_{1}^{n}$ to
the interval $[\alpha ,\beta ]$. Under the standing regularity assumptions
used in the paper, the marginals are continuous and strictly increasing, so
that 
\begin{equation}
L_{1}^{n}(\alpha )=a,\qquad L_{1}^{n}(\beta )=b.  \label{345a14}
\end{equation}%
Consequently, 
\begin{equation}
\int_{\alpha }^{\beta }\int_{0}^{1}dL_{n}(u_{1},u_{2})=L_{1}^{n}(\beta
)-L_{1}^{n}(\alpha )=b-a.  \label{345a15}
\end{equation}%
Combining (\ref{345a13}), (\ref{345a15}), and $I_{n}\geq c_{F}$, we obtain 
\begin{equation}
L_{1}^{n+1}(b)-L_{1}^{n+1}(a)\leq \frac{b-a}{I_{n}}\leq \frac{1}{c_{F}}(b-a).
\label{345a16}
\end{equation}%
This proves (\ref{345a99}) for $i=1$, and the same argument proves it for $%
i=2$.
\end{proof}

\begin{lemma}
\noindent \label{L-Lip} For each $n\geq 1$, the map $L_{n}:[0,1]^{2}%
\rightarrow \lbrack 0,1]$ is uniformly Lipschitz on $[0,1]^{2}$.
Consequently, since under the standing density assumption $L_{0}\in
C([0,1]^{2})$, the full family $(L_{n})_{n\geq 0}$ is relatively compact in $%
C([0,1]^{2})$ with the uniform topology.

In particular, every sequence $\{L_{n_{k}}\}_{k=0}^{\infty }$ with $%
n_{k}\rightarrow \infty $ admits a subsequence, again denoted by $%
\{L_{n_{k}}\}$, and a limit $L^{\ast }:[0,1]^{2}\rightarrow \lbrack 0,1]$
such that 
\begin{equation}
\sup_{(x_{1},x_{2})\in \lbrack 0,1]^{2}}\left\vert
L_{n_{k}}(x_{1},x_{2})-L^{\ast }(x_{1},x_{2})\right\vert \rightarrow 0\quad 
\text{as }k\rightarrow \infty .  \label{900}
\end{equation}%
Moreover, each such $L^{\ast }$ is a continuous distribution function on $%
[0,1]^{2}$, and its marginals 
\begin{equation}
L_{1}^{\ast }(x)=L^{\ast }(x,1),\qquad L_{2}^{\ast }(x)=L^{\ast }(1,x)
\label{5000}
\end{equation}%
are continuous strictly increasing distribution functions on $[0,1]$.
\end{lemma}

\begin{proof}
By \textit{Lemma~\ref{marginal-lipschitz-from-denominator}}, for $i=1,2$, 
\begin{equation}
L_{i}^{n+1}(b)-L_{i}^{n+1}(a)\leq \frac{1}{c_{F}}(b-a),\qquad 0\leq a\leq
b\leq 1,\quad n\geq 0.  \label{5001}
\end{equation}%
Equivalently, for every $n\geq 1$, $i=1,2$, and $x,y\in \lbrack 0,1]$, 
\begin{equation}
|L_{i}^{n}(x)-L_{i}^{n}(y)|\leq \frac{1}{c_{F}}|x-y|.  \label{5002}
\end{equation}

\noindent We now prove that the same estimate gives a \textit{uniform
Lipschitz bound} for the bivariate distribution functions $L_{n}$, $n\geq 1$%
. Let 
\begin{equation}
x=(x_{1},x_{2}),\qquad y=(y_{1},y_{2})  \label{5003}
\end{equation}%
be two arbitrary points of $[0,1]^{2}$. By adding and subtracting $%
L_{n}(y_{1},x_{2})$, we have 
\begin{equation}
|L_{n}(x_{1},x_{2})-L_{n}(y_{1},y_{2})|\leq
|L_{n}(x_{1},x_{2})-L_{n}(y_{1},x_{2})|+|L_{n}(y_{1},x_{2})-L_{n}(y_{1},y_{2})|.
\label{5004}
\end{equation}%
We estimate the two terms separately. Suppose first that $x_{1}\leq y_{1}$,
and let $(U_{n},V_{n})$ be a random vector with distribution function $L_{n}$%
. Then 
\begin{equation}
L_{n}(y_{1},x_{2})-L_{n}(x_{1},x_{2})=P(x_{1}<U_{n}\leq y_{1},\ V_{n}\leq
x_{2}).  \label{5005}
\end{equation}%
Hence 
\begin{equation}
0\leq L_{n}(y_{1},x_{2})-L_{n}(x_{1},x_{2})\leq P(x_{1}<U_{n}\leq
y_{1})=L_{1}^{n}(y_{1})-L_{1}^{n}(x_{1}).  \label{5006}
\end{equation}%
If $y_{1}\leq x_{1}$, the same argument with $x_{1}$ and $y_{1}$
interchanged gives 
\begin{equation}
|L_{n}(x_{1},x_{2})-L_{n}(y_{1},x_{2})|\leq
|L_{1}^{n}(x_{1})-L_{1}^{n}(y_{1})|.  \label{5007}
\end{equation}%
Therefore, in all cases, 
\begin{equation}
|L_{n}(x_{1},x_{2})-L_{n}(y_{1},x_{2})|\leq
|L_{1}^{n}(x_{1})-L_{1}^{n}(y_{1})|.  \label{5008}
\end{equation}%
Similarly, 
\begin{equation}
|L_{n}(y_{1},x_{2})-L_{n}(y_{1},y_{2})|\leq
|L_{2}^{n}(x_{2})-L_{2}^{n}(y_{2})|.  \label{5009}
\end{equation}%
Combining (\ref{5004}), (\ref{5008}), and (\ref{5009}), we obtain 
\begin{equation}
|L_{n}(x_{1},x_{2})-L_{n}(y_{1},y_{2})|\leq
|L_{1}^{n}(x_{1})-L_{1}^{n}(y_{1})|+|L_{2}^{n}(x_{2})-L_{2}^{n}(y_{2})|.
\label{5010}
\end{equation}%
Using (\ref{5002}), it follows that, for every $n\geq 1$, 
\begin{equation}
|L_{n}(x_{1},x_{2})-L_{n}(y_{1},y_{2})|\leq \frac{1}{c_{F}}\left(
|x_{1}-y_{1}|+|x_{2}-y_{2}|\right) .  \label{5011}
\end{equation}%
Thus the family $(L_{n})_{n\geq 1}^{\infty }$ is \textit{uniformly Lipschitz}
on $[0,1]^{2}$, and hence \textit{equicontinuous}.

Moreover, 
\begin{equation}
0\leq L_{n}(x_{1},x_{2})\leq 1\qquad \text{for all }(x_{1},x_{2})\in \lbrack
0,1]^{2}  \label{5012}
\end{equation}%
and all $n\geq 1$. Therefore the family $(L_{n})_{n\geq 1}$ is \textit{%
uniformly bounded} and equicontinuous. By the \textit{Arzel\`{a}--Ascoli's
theorem}, it is relatively compact in $C([0,1]^{2})$ with the uniform
topology.

\noindent It remains only to include $L_{0}$. Under the standing assumption
that the starting distribution $F$ admits a density and that $%
0<E(X_{1}X_{2})<\infty $, the \textit{tilted measure} 
\begin{equation}
dQ_{F}(x_{1},x_{2})=\frac{x_{1}x_{2}}{E(X_{1}X_{2})}\,dF(x_{1},x_{2})
\label{5013}
\end{equation}%
is absolutely continuous with respect to \textit{Lebesgue measure} and hence
has no atoms. Therefore the initial \textit{Lorenz curve} $L_{0}=L_{F}$ is
continuous on $[0,1]^{2}$. Adding the single continuous function $L_{0}$ to
the relatively compact tail family does not change relative compactness.
Hence the full sequence $\{L_{n}\}_{n\geq 0}$ is relatively compact in $%
C([0,1]^{2})$.

\noindent Now let $\{L_{n_{k}}\}_{k=0}^{\infty }$ be any sequence with $%
n_{k}\rightarrow \infty $. Since the family $\{L_{n}\}_{n\geq 0}$ is
relatively compact in $C([0,1]^{2})$, there exists a subsequence, again
denoted by $\{L_{n_{k}}\}$, and a function $L^{\ast }\in C([0,1]^{2})$ such
that (\ref{900}) holds.

\noindent Finally, $L^{\ast }$ is a distribution function. Indeed, each $%
L_{n_{k}}$ is coordinatewise nondecreasing, satisfies the boundary
conditions of a distribution function on $[0,1]^{2}$, and satisfies the
rectangle inequality 
\begin{equation}
L_{n_{k}}(b_{1},b_{2})-L_{n_{k}}(a_{1},b_{2})-L_{n_{k}}(b_{1},a_{2})+L_{n_{k}}(a_{1},a_{2})\geq 0
\label{5014}
\end{equation}%
whenever $0\leq a_{i}\leq b_{i}\leq 1$, $i=1,2$. Passing to the uniform
limit preserves these inequalities and boundary conditions. Hence $L^{\ast }$
is a bivariate distribution function on $[0,1]^{2}$. Standard arguments
(monotonicity and continuity are preserved under uniform convergence) show
that $L^{\ast }$ is also a distribution function with strictly increasing
continuous marginals $L_{i}^{\ast }.$\hfill
\end{proof}

\begin{flushleft}
\textbf{Equicontinuity and subsequential limits of }$\Phi _{n}^{i}$
\end{flushleft}

We now turn to the compound maps $\Phi _{n}^{i}$. The key point is that, for
each fixed $n$, the maps $L_{i}^{n}$ have the single-crossing property: they
have at most one crossing with the diagonal $x\mapsto x$ in $(0,1)$ or are
strictly subdiagonal. The following lemma, identical to \textit{Lemma}~%
\texttt{\textup{\ref{Fejer}}} in \textit{Appendix~C.1}, applies to any such
map.

\begin{lemma}
\label{FejerNew} Let $f:[0,1]\rightarrow \lbrack 0,1]$ be continuous and
strictly increasing, with a unique fixed point $p\in (0,1)$ and
single--crossing pattern 
\begin{equation}
f(x)>x\;\text{for }x<p,\qquad f(x)<x\;\text{for }x>p.  \label{915}
\end{equation}%
Then

\begin{enumerate}
\item[(i)] For every $x\in \lbrack 0,1]$, 
\begin{equation}
|f(x)-p|\leq |x-p|,  \label{916}
\end{equation}%
with strict inequality if $x\neq p$;

\item[(ii)] If $x,y$ lie on the same side of $p$ (both $\leq p$ or both $%
\geq p$), then 
\begin{equation}
|f(x)-f(y)|\leq |x-y|.  \label{917}
\end{equation}
\end{enumerate}
\end{lemma}

\begin{proof}
This is exactly the argument in \textit{Lemma}~\texttt{\textup{\ref{Fejer}}}
of \textit{Appendix~C.1} and does not depend on the special \textit{Fr\'{e}%
chet--Hoeffding lower--bound} kernel.
\end{proof}

From \textit{Lemma}~\ref{FejerNew}, we obtain equicontinuity of the family $%
\{\Phi _{n}^{i}\}$ on compact subintervals of $(0,1)$, as in \textit{Lemma}~%
\texttt{\textup{\ref{Phi-1Lip}}} of \textit{Appendix~C}.

\begin{lemma}
\label{Phi-equicont} Fix $0<\delta <\eta <1$ and $i\in \{1,2\}$. Then there
exists $N$ such that for all $n\geq N$ and all $x,y\in \lbrack \delta ,\eta
] $, 
\begin{equation}
|\Phi _{n}^{i}(x)-\Phi _{n}^{i}(y)|\leq |x-y|.  \label{918}
\end{equation}%
In particular, the tail family $\{\Phi _{n}^{i}\}_{n\geq N}$ is
equicontinuous and uniformly bounded on $[\delta ,\eta ]$.
\end{lemma}

\begin{proof}
The proof is identical to \textit{Lemma}~\texttt{\textup{\ref{Phi-1Lip}}} in 
\textit{Appendix~C.1}: once the unique fixed points of the compound maps are
close enough to their limit, the interval $[\delta ,\eta ]$ either lies on
one side of the fixed point or is split into two pieces that can be handled
separately using \textit{Lemma}~\ref{FejerNew} and the continuity of $\Phi
_{n}^{i}$ at its fixed point. The details are omitted.
\end{proof}

As a consequence, we have subsequential uniform convergence of $\Phi
_{n}^{i} $ on compact intervals.

\begin{lemma}
\label{Phi-subseq} Fix $i\in \{1,2\}$ and $0<\delta <\eta <1$. Then, for any
sequence of indices $\{n_{k}\}_{k\geq 0}$ with $n_{k}\rightarrow \infty $,
there exists a subsequence (still denoted $n_{k}$) and a continuous map $%
G_{i}:[\delta ,\eta ]\rightarrow \lbrack 0,1]$ such that 
\begin{equation}
\sup_{x\in \lbrack \delta ,\eta ]}|\Phi _{n_{k}}^{i}(x)-G_{i}(x)|\rightarrow
0\quad \text{as }k\rightarrow \infty .  \label{919}
\end{equation}
\end{lemma}

\begin{proof}
By \textit{Lemma}~\ref{Phi-equicont}, the tail family $\{\Phi
_{n}^{i}\}_{n\geq N}$ is equicontinuous and uniformly bounded on $[\delta
,\eta ]$ for some $N$. By the \textit{Arzel\`{a}--Ascoli's theorem}, any
sequence $\{\Phi _{n_{k}}^{i}\}_{k\geq 0}$ with $n_{k}\rightarrow \infty $
admits a subsequence converging uniformly on $[\delta ,\eta ]$ to a
continuous limit $G_{i}$.\hfill
\end{proof}

\begin{flushleft}
\textbf{An overview of the next steps}
\end{flushleft}

In \textit{Appendix~C.1}, after obtaining subsequential limits we could use
symmetry arguments to force all diagonal-crossing limits to coincide at $%
\frac{1}{2}$, and then deduce constancy of the limiting compound map. In the 
\textit{RR2 case} that symmetry is unavailable. A strong obstruction pops
up: even when $f_{n}\rightarrow f$ subsequentially, one-step relations $%
g_{n+1}=g_{n}\circ f_{n+1}$ do not pass to limits unless the indices
synchronize, i.e., we cannot take subsequential limits in (\ref{526.1}): $%
\Phi _{n+1}(x)=\Phi _{n}(T_{n+1}(x)).$

\medskip \noindent We take the following strategy:

\begin{enumerate}
\item From \emph{$\{\Phi _{n}^{i}\}$ }being relatively compact on\emph{\ $%
[\delta ,\eta ]$}, we know that it has cluster points;

\item Use the $m$--step identity 
\begin{equation}
\Phi _{n+m}^{i}=\Phi _{n}^{i}\circ T_{n+1}^{i}\circ \cdots \circ
T_{n+m}^{i}\qquad (m\geq 1)  \label{1000}
\end{equation}%
instead of the one-step identity (\ref{526.1});

\item Engineer a \textit{return-shift} index set: infinitely many $n$ and a
fixed $m$ such that \emph{both} $\Phi _{n}^{i}$ and $\Phi _{n+m}^{i}$
converge to the same limit $G$;

\item Extract a limit block map $Z_{m}^{i}$ for the composition $%
T_{n+1}^{i}\circ \cdots \circ T_{n+m}^{i}$. Then the limit of the block
identity becomes the invariance 
\begin{equation}
G_{i}=G_{i}\circ Z_{m}^{i};  \label{1001}
\end{equation}

\item Use the single-crossing dynamics of $Z_{m}^{i}$ (same qualitative type
as $T_{n}^{i}$) to force the only continuous solutions of $G=G\circ Z_{m}$
to be constants;

\item The formal goal is: For fixed $i\in \{1,2\}$ and $0<\delta <\eta <1$
to prove that each cluster point of $\{\Phi _{n}^{i}\}$ in $C([\delta ,\eta
])$ is a constant function on $[\delta ,\eta ]$. This weaker condition is
sufficient for our primary objective: proving the convergence of (\ref{5.1})
using (\ref{16})--(\ref{16.2}).
\end{enumerate}

\medskip Finally it is worth mentioning a \noindent conceptual link\textbf{.}
The recursion $\Phi _{n+1}^{i}(x)=\Phi _{n}^{i}(T_{n+1}^{i}(x))$ is an inner
composition (also called \textit{right composition}) of a moving family of
maps, in the sense of the \textit{Iterated Function System (IFS)} literature
(see, e.g., \cite{[52]}, \cite{[54]}, and \cite{[53]}). The viewpoint of
infinite compositions typically tracks orbit convergence for a fixed map or
a special set of maps. While this approach could be adapted to our problem,
we instead track orbit convergence for a time-inhomogeneous sequence of
maps, a method that necessitates a 'synchronization tool'. This tool is
provided by a recurrence principle from \textit{Ramsey theory} and
topological dynamics that pertains to \textit{infinite-dimensional
parallelepipeds (IP)} and \textit{finite sums (FS)} (see, e.g., \cite{[57]}
and \cite{[58]}; for more recent work, see \cite{[55]} and \cite{[56]}%
).\hfill

\begin{flushleft}
\textbf{FS/IP sets: the algebraic pairing mechanism for fixed shifts}
\end{flushleft}

To pass to limits in the $m$--step identity (\ref{1000}), we must compare
time $n$ with time $n+m$. Thus we want infinitely many paired indices $%
(n,n+m)$ that both lie inside the same convergence set. Finite-sums sets $%
\mathrm{FS}(\cdot )$ are designed precisely to provide such pairs: if $m$ is
chosen as the first generator, then adding $m$ toggles membership of that
generator and keeps us in the same $\mathrm{FS}$-set.

\noindent The formal goals of this section are:

\begin{enumerate}
\item Define $\mathrm{FS}(q_{1},q_{2},\dots )$ and fix the topological
conventions about 'convergence along a set';

\item Prove the deterministic pairing property: 
\begin{equation}
n\in \mathrm{FS}(q_{2},q_{3},\dots )\ \Longrightarrow \ n\in \mathrm{FS}%
(q_{1},q_{2},\dots )\ \text{and}\ n+q_{1}\in \mathrm{FS}(q_{1},q_{2},\dots );
\label{1002}
\end{equation}

\item Explain why we must require $n$ itself to lie in the convergence set.
\end{enumerate}

\begin{definition}
Let $X$ be a topological space and let $D=\{q_{1}<q_{2}<\cdots \}\subset N$
be infinite.. Define 
\begin{equation}
\mathrm{FS}(D)=\left\{ \sum_{j\in F}q_{j}:\ \emptyset \neq F\subset N\ \text{%
finite}\right\} .  \label{1003}
\end{equation}%
Every element is a finite sum of distinct generators, so there are no
infinite series.
\end{definition}

Since $\mathrm{FS}(q_{1},q_{2},\dots )$ is a set, not a sequence, we
interpret 
\begin{equation}
a_{n}\rightarrow a^{\ast }\quad \text{as }n\rightarrow \infty ,\ n\in 
\mathrm{FS}(q_{1},q_{2},\dots )  \label{1004}
\end{equation}%
to mean: if $\{i_{k}\}_{k\geq 1}$ is the increasing enumeration of $\mathrm{%
FS}(q_{1},q_{2},\dots )$, then $a_{i_{k}}\rightarrow a^{\ast }$ as $%
k\rightarrow \infty $.

\begin{lemma}
\medskip Define 
\begin{equation}
I=\mathrm{FS}(q_{1},q_{2},\dots ),\qquad I^{\prime }=\mathrm{FS}%
(q_{2},q_{3},\dots ),\qquad m=q_{1}.  \label{1005}
\end{equation}%
Then $I^{\prime }\subset I$ and for every $n\in I^{\prime }$ we have $n+m\in
I$.
\end{lemma}

\begin{proof}
\smallskip \noindent If $n\in I^{\prime }$, then $n=\sum_{j\in F}q_{j}$ for
some finite nonempty $F\subset \{2,3,\dots \}$. This already implies $n\in I$%
. Moreover $n+m=n+q_{1}=q_{1}+\sum_{j\in F}q_{j}=\sum_{j\in F\cup
\{1\}}q_{j}\in I$.
\end{proof}

\begin{remark}
\label{need-both-sides} In our RR2 case, we will only have convergence of $%
\Phi _{n}^{i}$ along some special index set $I$. To take limits in the block
identity 
\begin{equation}
\Phi _{n+m}^{i}=\Phi _{n}^{i}\circ (\cdots )  \label{1006}
\end{equation}%
we must control \emph{both} $\Phi _{n+m}^{i}$ (left side) and $\Phi _{n}^{i}$
(outer map on the right side). Thus we must ensure both indices $n$ and $n+m$
lie in the same convergence set. This is why we restrict $n$ to $I^{\prime }=%
\mathrm{FS}(q_{2},q_{3},\dots )$: then $n\in I$ and $n+m\in I$ automatically
for the fixed shift $m=q_{1}$.\hfill
\end{remark}

\begin{flushleft}
\textbf{The return--shift lemma: IP-limits in compact metric spaces}
\end{flushleft}

We cannot take subsequential limits in $\Phi _{n+1}(x)=\Phi _{n}(T_{n+1}(x))$
since by no means are $n_{k}$ and $n_{k}+1$ from the same subsequence $%
\{n_{k}\}_{k\geq 0}$. The correct principle is to upgrade the compactness by
leveraging \textit{Ramsey theory}---more concretely,\textit{\ }the work of 
\textit{Furstenberg--Weiss} \cite{[58]} and \textit{Hindman }\cite{[57]}%
)---to show that in a compact metric space, every sequence admits a
structured \textit{FS-subsequence} that converges in the strong \textit{IP}
sense.

The formal goals of this section are:

\begin{enumerate}
\item Introduce \textit{IP-convergence} and the \textit{Hindman space}
property;

\item State (and cite) the theorem: every compact metric space is \textit{%
Hindman space};

\item Restate it in the concrete metric form we need (the \textit{IP-limit}
existence theorem);

\item Deduce the specific \textit{return--shift lemma} used later.
\end{enumerate}

We use \cite{[55]} on \textit{Hindman spaces} for the basic definitions.

\begin{definition}
\label{IP-convergence1} A sequence $\{x_{n}\}_{n\in \mathrm{FS}(D)}$ is said
to be \textit{IP-convergent} to $x^{\ast }\in X$ if for every open
neighborhood $U\ni x^{\ast }$ there exists $L$ such that 
\begin{equation}
x_{n}\in U\qquad \text{for all }n\in \mathrm{FS}(q_{L},q_{L+1},q_{L+2},\dots
).  \label{1007.1}
\end{equation}
\end{definition}

If $(X,d)$ is a metric space, then (\ref{1007.1}) is equivalent to the
following $\varepsilon $-form:

\begin{definition}
\label{IP-convergence2}The sequence $\{x_{n}\}_{n\geq 1}$ \textit{%
IP-converges} to $x^{\ast }$ along $\mathrm{FS}(D)$ if and only if 
\begin{equation}
\forall \,\varepsilon >0\ \ \exists L\in N\ \ \text{s.t.}\ \ d(x_{n},x^{\ast
})<\varepsilon \quad \forall n\in \mathrm{FS}(q_{L},q_{L+1},q_{L+2},\dots ),
\label{1007}
\end{equation}%
because the open balls $B(x^{\ast },\varepsilon )$ form a neighborhood basis
at $x^{\ast }$.
\end{definition}

\begin{remark}
After discarding finitely many generators, all finite sums of the remaining
generators land inside a neighborhood $U$. This tail stability is the key
feature that makes \textit{IP-convergence} powerful for fixed-shift
arguments.
\end{remark}

\begin{definition}
\label{Hindman} A topological space $X$ is called \textit{Hindman} if for
every sequence $\{x_{n}\}_{n\geq 1}$ in $X$ there exists an infinite $%
D\subset N$ such that the restricted sequence $\{x_{n}\}_{n\in \mathrm{FS}%
(D)}$ is \textit{IP-convergent} in the sense of (\ref{1007.1}).
\end{definition}

\begin{remark}
\label{Lit} A foundational result of \textit{Furstenberg--Weiss} (recorded
as standard background in later works) states that every compact metric
space is \textit{Hindman}. The article \cite{[55]} explicitly notes this
fact and points to the relevant literature, e.g., \cite{[59]}.
\end{remark}

Next, we need to state the exact metric form we use. This is the form has to
be convenient and operational so that we can apply it to the function-valued
sequence $a_{n}=\Phi _{n}^{i}|_{[\delta ,\eta ]}$. It is simply (\ref{1007})
stated in metric terms.

\begin{theorem}
\label{IP-limit} Let $(X,d)$ be a compact metric space and $\{a_{n}\}_{n\geq
1}$ a sequence in $X$. Then there exist a strictly increasing sequence of
integers $q_{1}<q_{2}<\cdots $ and a point $a^{\ast }\in X$ such that for
every $\varepsilon >0$ there exists $L$ with 
\begin{equation}
d(a_{n},a^{\ast })<\varepsilon \qquad \text{for every }n\in \mathrm{FS}%
(q_{L},q_{L+1},q_{L+2},\dots ).  \label{1008}
\end{equation}%
Equivalently, the restricted sequence $\{a_{n}\}_{n\in \mathrm{FS}%
(q_{1},q_{2},\dots )}$ IP-converges to $a^{\ast }$ (in either the
neighborhood form (\ref{1007.1}) or the metric form (\ref{1007}).
\end{theorem}

\begin{proof}
This is a classical consequence of \textit{Hindman's finite sums theorem}
and its topological-dynamical refinements; see \cite{[55]} in survey context
and the cited results that all compact metric spaces are \textit{Hindman }as
well as \cite{[59]} for a focused attention. A comprehensive treatment of 
\textit{IP/FS }methods is given in the \textit{Hindman--Strauss} monograph 
\cite{[58]}.
\end{proof}

\noindent The \textit{Theorem~\ref{IP-limit}} is of service to \textit{RR2}
because it produces:

\begin{itemize}
\item a structured set of indices $I=\mathrm{FS}(q_1,q_2,\dots)$;

\item convergence along $I$ in the strong tail form (\ref{1008});

\item via the algebra of finite sums, a canonical fixed shift $m=q_{1}$ for
which there are infinitely many pairs $(n,n+m)$ with both indices in $I$.
\end{itemize}

This is exactly the synchronization we need to make (\ref{1006}) survive
taking limits.

\medskip We can move next with the following 'return-shift' lemma.

\begin{lemma}
\label{return-shift} Under the assumptions of Theorem~\ref{IP-limit}, define 
\begin{equation}
I=\mathrm{FS}(q_{1},q_{2},\dots ),\qquad I^{\prime }=\mathrm{FS}%
(q_{2},q_{3},\dots ),\qquad m=q_{1}.  \label{1009}
\end{equation}%
Then:

\begin{enumerate}
\item $a_{n}\rightarrow a^{\ast }$ as $n\rightarrow \infty $ along $n\in I$;

\item $a_{n+m}\rightarrow a^{\ast }$ as $n\rightarrow \infty $ along $n\in
I^{\prime }$;

\item for every $n\in I^{\prime }$, the pair $(n,n+m)$ lies in $I\times I$.
\end{enumerate}
\end{lemma}

\begin{proof}
\smallskip Item (3) is the deterministic algebra in the \textit{FS/IP} sets
section. Item (1): fix $\varepsilon >0$. By (\ref{1008}), choose $L$ such
that $d(a_{n},a^{\ast })<\varepsilon $ for every $n\in \mathrm{FS}%
(q_{L},q_{L+1},\dots )$. Now take $n\rightarrow \infty $ with $n\in I$.
Since $\mathrm{FS}(q_{1},\dots ,q_{L-1})$ is finite, all sufficiently large $%
n\in I$ must lie in the tail $\mathrm{FS}(q_{L},q_{L+1},\dots )$, hence $%
d(a_{n},a^{\ast })<\varepsilon $. This is $a_{n}\rightarrow a^{\ast }$ along 
$I$. Item (2): if $n\in I^{\prime }$, then $n+m\in I$ by item (3). Thus the
sequence $\{a_{n+m}\}_{n\in I^{\prime }}$ is a subsequence of $%
\{a_{k}\}_{k\in I}$, and by item (1) it converges to the same limit $a^{\ast
}$.\hfill
\end{proof}

\begin{flushleft}
\textbf{Application of the return--shift lemma to the function sequence }$%
\Phi _{n}^{i}|_{[\delta ,\eta ]}$
\end{flushleft}

We now specialize $\{a_{n}\}_{n\geq 1}$ to the function-valued sequence 
\begin{equation}
a_{n}=\Phi _{n}^{i}|_{[\delta ,\eta ]}\in C([\delta ,\eta ]).  \label{1010}
\end{equation}%
By \textit{Lemma~\ref{Phi-subseq}}, the family $\{\Phi _{n}^{i}\}_{n\geq 0}$
is relatively compact in $C([\delta ,\eta ])$ (uniform topology). Hence its
closure is a compact metric space. We may therefore apply \textit{Theorem~%
\ref{IP-limit}} to obtain an \textit{FS-indexed IP-limit} $G_{i}$ and a
fixed shift $m$ such that \emph{both} $\Phi _{n}^{i}$ and $\Phi _{n+m}^{i}$
converge to $G_{i}$ along the same tail \textit{FS-set}.

\medskip \noindent Our formal goal is to produce $m\geq 1$, a continuous
function $G_{i}$, and an infinite index set of the form $I^{\prime }=\mathrm{%
FS}(q_{2},q_{3},\dots )$ such that 
\begin{equation}
\Phi _{n}^{i}\rightarrow G_{i}\ \text{and}\ \Phi _{n+m}^{i}\rightarrow
G_{i}\quad \text{uniformly on }[\delta ,\eta ]\ \text{as }n\rightarrow
\infty ,\ n\in I^{\prime }.  \label{1011}
\end{equation}

\begin{theorem}
\label{Phi-returnshift} Fix $i\in \{1,2\}$ and $0<\delta <\eta <1$. There
exist integers $q_{1}<q_{2}<\cdots $, a shift $m=q_{1}$, an FS-set 
\begin{equation}
I=\mathrm{FS}(q_{1},q_{2},\dots ),\qquad I^{\prime }=\mathrm{FS}%
(q_{2},q_{3},\dots ),  \label{1012}
\end{equation}%
and a continuous function $G_{i}\in C([\delta ,\eta ])$ such that: 
\begin{equation}
\sup_{x\in \lbrack \delta ,\eta ]}|\Phi _{n}^{i}(x)-G_{i}(x)|\rightarrow
0\qquad \text{as }n\rightarrow \infty ,\ n\in I,  \label{1013}
\end{equation}%
and 
\begin{equation}
\sup_{x\in \lbrack \delta ,\eta ]}|\Phi _{n+m}^{i}(x)-G_{i}(x)|\rightarrow
0\qquad \text{as }n\rightarrow \infty ,\ n\in I^{\prime }.  \label{1014}
\end{equation}%
Moreover, for every $n\in I^{\prime }$, we have $(n,n+m)\in I\times I$.
\end{theorem}

\begin{proof}
\smallskip Let $X$ be the closure of $\{\Phi _{n}^{i}|_{[\delta ,\eta
]}:n\geq 0\}$ in $C([\delta ,\eta ])$ with the uniform metric $d_{\infty
}(f,g)=\sup_{x\in \lbrack \delta ,\eta ]}|f(x)-g(x)|$. By \textit{Lemma~\ref%
{Phi-subseq}}, $X$ is compact. Apply \textit{Theorem~\ref{IP-limit}} to the
sequence $a_{n}=\Phi _{n}^{i}|_{[\delta ,\eta ]}\in X$. We obtain $%
q_{1}<q_{2}<\cdots $ and $G_{i}\in X$ such that (\ref{1008}) holds in the
metric $d_{\infty }$. Then \textit{Lemma~\ref{return-shift}} yields (\ref%
{1013})--(\ref{1014}) with $m=q_{1}$.
\end{proof}

\begin{flushleft}
\textbf{The invariance }$G^{i}=G^{i}\circ Z^{i}_{m}$
\end{flushleft}

We now combine two ingredients:

\begin{itemize}
\item the return--shift convergence of $\Phi _{n}^{i}$ and $\Phi _{n+m}^{i}$
to the same $G^{i}$ along the same indices $n\in I^{\prime }$;

\item and the \textit{Arzel\`{a}--Ascoli's} compactness of the inverse
marginals $T_{n}^{i}$.
\end{itemize}

Because $m$ is fixed, we only need limits of finitely many shifts $%
T_{n+1}^{i},\dots ,T_{n+m}^{i}$, and therefore a standard (finite) diagonal
extraction provides a limit block map 
\begin{equation}
Z_{m}^{i}=\lim_{n\rightarrow \infty ,\,n\in I^{\prime }}T_{n+1}^{i}\circ
\cdots \circ T_{n+m}^{i}.  \label{1015}
\end{equation}%
Then the limit of the block identity becomes the invariance $%
G_{i}=G_{i}\circ Z_{m}^{i}$.

\medskip We aim at:

\begin{enumerate}
\item Extract (along the same indices $n\in I^{\prime }$) uniform limits of $%
T_{n+r}^{i}$ for $r=1,\dots ,m$;

\item Deduce uniform convergence of the block compositions $%
Z_{m,n}^{i}=T_{n+1}^{i}\circ \cdots \circ T_{n+m}^{i}$ to $Z_{m}^{i}$;

\item Pass to limits in the exact $m$--step identity (\ref{1006}) and obtain 
$G^{i}=G^{i}\circ Z_{m}^{i}$ on $[\delta ,\eta ]$.
\end{enumerate}

\medskip \noindent \textbf{Step 1: }Here we give compatible limits for
shifted inverse marginals. Consider the index sequence given by the
increasing enumeration of $I^{\prime }$. By the same compactness/diagonal
principles already used for $L_{n}$ and its marginals, we may pass to a
further subsequence of $I^{\prime }$ (still denoted $I^{\prime }$) such that
for each fixed $r\in \{1,2,\dots ,m\}$ there exists a continuous increasing
limit map $T_{r}^{i}$ with 
\begin{equation}
\sup_{x\in \lbrack 0,1]}\left\vert T_{n+r}^{i}(x)-T_{r}^{i}(x)\right\vert
\rightarrow 0\qquad \text{as }n\rightarrow \infty ,\ n\in I^{\prime }.
\label{1016}
\end{equation}

\medskip \noindent \textbf{Step 2: }We prove convergence of the finite block
composition\textbf{.} Define $Z_{m,n}^{i}=T_{n+1}^{i}\circ T_{n+2}^{i}\circ
\cdots \circ T_{n+m}^{i}$. Because $m$ is fixed and each factor converges
uniformly by (\ref{1016}), standard stability of finite compositions gives: 
\begin{equation}
\sup_{x\in \lbrack 0,1]}|Z_{m,n}^{i}(x)-Z_{m}^{i}(x)|\rightarrow 0\qquad
(n\rightarrow \infty ,\ n\in I^{\prime }),  \label{1017}
\end{equation}%
where we define the limit block map $Z_{m}^{i}=T_{1}^{i}\circ T_{2}^{i}\circ
\cdots \circ T_{m}^{i}$.

\medskip \noindent \textbf{Step 3:} Finally, we pass to the limit in the
block identity. For each $n\in I^{\prime }$ and each $x\in \lbrack \delta
,\eta ]$, the exact $m$--step identity (\ref{1017}) gives $\Phi
_{n+m}^{i}(x)=\Phi _{n}^{i}\!(Z_{m,n}^{i}(x))$. Let $n\rightarrow \infty $
along $n\in I^{\prime }$. By \textit{Theorem~\ref{Phi-returnshift}}, the
left-hand side converges uniformly on $[\delta ,\eta ]$ to $G^{i}(x)$, and
the outer function $\Phi _{n}^{i}$ converges uniformly to $G^{i}$ on $%
[\delta ,\eta ]$. By (\ref{1017}), the inner argument $Z_{m,n}(x)$ converges
uniformly to $Z_{m}^{i}(x)$. Hence we obtain, for all $x\in \lbrack \delta
,\eta ]$, 
\begin{equation}
G^{i}(x)=G\!^{i}(Z_{m}^{i}(x)).  \label{1018}
\end{equation}

So we get the following:

\begin{theorem}
\label{G-invariance} Along the return--shift extraction above, the
subsequential limit $G^{i}$ satisfies $G^{i}=G^{i}\circ Z_{m}^{i}$ on $%
[\delta ,\eta ]$.
\end{theorem}

\begin{flushleft}
\textbf{Single-crossing dynamics of }$Z_{m}$\textbf{\ and constancy of }$G$
\end{flushleft}

Iterating (\ref{1018}) yields 
\begin{equation}
G^{i}(x)=G^{i}(Z_{m}^{i,\circ k}(x))\qquad \text{for every }k\geq 1\text{
and }x\in \lbrack \delta ,\eta ].  \label{1019}
\end{equation}%
Hence, if the forward orbit $Z_{m}^{i,\circ k}(x)$ converges to a limit
point (an attractor), then continuity of $G^{i}$ forces $G^{i}(x)$ to equal $%
G^{i}$ evaluated at that attractor, independently of $x$. This is the
standard mechanism that invariant functions are constant along orbits from
one-dimensional dynamics (see\ also the orbit-collapse argument used in 
\textit{Appendix~C.1} based on (\ref{18})).

\medskip \noindent Our next goals are:

\begin{enumerate}
\item State the correct single-crossing pattern for $Z_{m}^{i}$ (the same as
the inverse marginals $T_{n}^{i}$);

\item Prove orbit convergence under that pattern (including
boundary/degenerate cases);

\item Combine orbit convergence with (\ref{1018}) to show $G^{i}$ is
constant on $[\delta ,\eta ]$.
\end{enumerate}

We start with some intuition. The map $Z_{m}^{i}$ arises as a subsequential
uniform limit of finite compositions of inverse marginals 
\begin{equation}
Z_{m,n}^{i}=T_{n+1}^{i}\circ \cdots \circ T_{n+m}^{i},  \label{1020}
\end{equation}%
and thus inherits the same qualitative monotone single-crossing pattern as
the $T_{n}^{i}$ (for large $n$) whenever those maps have the single-crossing
property. Now we record that pattern in a form suitable for a complete orbit
analysis.

\begin{corollary}
\label{Zm-singlecross} The map $Z_{m}^{i}:[0,1]\rightarrow \lbrack 0,1]$ is
continuous and strictly increasing, and exactly one of the following holds:

\begin{enumerate}
\item (\emph{Single-crossing (SC)}) There exists a unique fixed point $%
p_{m}^{i}\in (0,1)$ such that 
\begin{equation}
Z_{m}^{i}(x)>x\quad \text{for }x<p_{m}^{i},\qquad Z_{m}^{i}(x)<x\quad \text{%
for }x>p_{m}^{i}.  \label{1021}
\end{equation}

\item (\emph{Strictly overdiagonal (SO)}) $Z_{m}^{i}(x)>x$ for all $x\in
(0,1)$.
\end{enumerate}
\end{corollary}

\begin{proof}
The result is a direct consequence of \textit{Lemma~\ref{Phi}} from \textit{%
Appendix C.1.1.}
\end{proof}

We now formally prove that the orbits collapse to a constant. This contrasts
with the argument in (\ref{18}), which was purely intuitive and geometric
and lacked a rigorous algebraic basis. The analysis that follows provides
this formal proof.

\begin{lemma}
\label{Zm-orbits} Let $x\in (0,1)$ and set $x_{k}=Z_{m}^{i,\circ k}(x)$.

\begin{enumerate}
\item If (SC) holds, then 
\begin{equation}
\lim_{k\rightarrow \infty }Z_{m}^{i,\circ k}(x)=p_{m}^{i}\qquad \text{for
every }x\in (0,1),  \label{1022}
\end{equation}%
and the convergence is uniform on every compact subinterval $[\delta ,\eta
]\subset (0,1)$;

\item If (SO) holds, then $Z_{m}^{i,\circ k}(x)\uparrow 1$ as $k\rightarrow
\infty $ for every $x\in (0,1)$, and the convergence is uniform on every
compact $[\delta ,\eta ]\subset (0,1)$.
\end{enumerate}
\end{lemma}

\begin{proof}
Fix $x\in (0,1)$ and write $x_{k}=Z_{m}^{i,\circ k}(x)$.

\smallskip \emph{Case (SC).} We first show that the orbit is monotone and
trapped on the appropriate side of $p_{m}^{i}$.

\emph{\smallskip \noindent }\underline{Subcase (i): $x<p_{m}^{i}$.} By (\ref%
{1021}), $x_{1}=Z_{m}^{i}(x)>x=x_{0}$. Since $Z_{m}^{i}$ is increasing, if $%
x_{k}<p_{m}^{i}$ then $x_{k+1}=Z_{m}^{i}(x_{k})>x_{k}$ and also 
\begin{equation}
x_{k+1}=Z_{m}^{i}(x_{k})<Z_{m}^{i}(p_{m}^{i})=p_{m}^{i},  \label{1023}
\end{equation}%
so $x_{k+1}<p_{m}^{i}$. By induction, $(x_{k})$ is strictly increasing and
bounded above by $p_{m}^{i}$. Therefore $x_{k}\uparrow \ell $ for some $\ell
\leq p_{m}^{i}$. By continuity, 
\begin{equation}
Z_{m}^{i}(\ell )=Z_{m}^{i}\!\left( \lim_{k\rightarrow \infty }x_{k}\right)
=\lim_{k\rightarrow \infty }Z_{m}^{i}(x_{k})=\lim_{k\rightarrow \infty
}x_{k+1}=\ell ,  \label{1024}
\end{equation}%
so $\ell $ is a fixed point of $Z_{m}^{i}$. Under (SC) the fixed point in $%
(0,1)$ is unique and equals $p_{m}^{i}$, hence $\ell =p_{m}^{i}$. Thus $%
x_{k}\rightarrow p_{m}^{i}$ for all $x<p_{m}^{i}$.

\smallskip \noindent \underline{Subcase (ii): $x>p_{m}^{i}$.} By (\ref{1021}%
), $x_{1}=Z_{m}^{i}(x)<x=x_{0}$. If $x_{k}>p_{m}^{i}$, then $%
x_{k+1}=Z_{m}^{i}(x_{k})<x_{k}$ and also 
\begin{equation}
x_{k+1}=Z_{m}^{i}(x_{k})>Z_{m}^{i}(p_{m}^{i})=p_{m}^{i},  \label{1025}
\end{equation}%
so $x_{k+1}>p_{m}^{i}$. Thus $(x_{k})$ is strictly decreasing and bounded
below by $p_{m}$, hence $x_{k}\downarrow \ell $ for some $\ell \geq
p_{m}^{i} $. Continuity again implies $\ell $ is a fixed point, so $\ell
=p_{m}$. Hence $x_{k}\rightarrow p_{m}^{i}$ for all $x>p_{m}^{i}$.

Combining the two subcases, we conclude $Z_{m}^{i,\circ k}(x)\rightarrow
p_{m}^{i}$ for all $x\in (0,1)$.

\smallskip \noindent \underline{Uniformity on compact $[\delta ,\eta
]\subset (0,1)$.} Fix $0<\delta <\eta <1$. The iterates $f_{k}(x)=Z_{m}^{i,%
\circ k}(x)$ are continuous on $[\delta ,\eta ]$. Moreover, for each fixed $%
x<p_{m}^{i}$ the sequence $\{f_{k}(x)\}_{k\geq 0}$ is increasing in $k$, and
for each fixed $x>p_{m}^{i}$ it is decreasing in $k$. If $p_{m}\notin
\lbrack \delta ,\eta ]$, then the monotonicity direction is the same for all 
$x\in \lbrack \delta ,\eta ]$ and \textit{Dini's Theorem} yields uniform
convergence to the continuous limit $p_{m}^{i}$. If $p_{m}^{i}\in (\delta
,\eta )$, apply \textit{Dini} separately on $[\delta ,p_{m}^{i}]$ and on $%
[p_{m}^{i},\eta ]$, and use continuity at $p_{m}$ to glue the two uniform
limits (both equal $p_{m}^{i}$). This yields uniform convergence on $[\delta
,\eta ]$.

\smallskip \noindent \emph{Case (SO).} $Z_{m}^{i}(x)>x$ implies $%
\{x_{k}\}_{k\geq 0}$ is increasing, bounded above by $1$, hence converges to
a fixed point $\ell $. Under (SO) there is no fixed point in $(0,1)$, hence $%
\ell =1$. Uniformity follows by \textit{Dini}. \hfill
\end{proof}

Let $G^{i}$ be any subsequential limit obtained via the return--shift
extraction. We already proved in \textit{Theorem~\ref{G-invariance}} that $%
G^{i}$ satisfies $G^{i}=G^{i}\circ Z_{m}^{i}$ on $[\delta ,\eta ]$. By 
\textit{Lemma~\ref{Zm-orbits}}, every orbit $Z_{m}^{i,\circ k}(x)$ converges
to a limit point ($p_{m}^{i}$ in case (SC), or $1$ in the degenerate case).
Then (\ref{1012}) plus continuity forces $G^{i}$ to take the same value for
all $x$, hence $G^{i}$ is constant.

\begin{theorem}
\label{Phi-cluster-constant} Fix $i\in \{1,2\}$ and $0<\delta <\eta <1$. Let 
$G^{i}\in C([\delta ,\eta ])$ be any cluster point of $\{\Phi _{n}^{i}\}$
obtained along the return--shift extraction, so that $G^{i}$ satisfies $%
G^{i}(x)=G^{i}(Z_{m}^{i}(x))$ for the corresponding block limit map $Z_{m}$.
Then $G^{i}$ is constant on $[\delta ,\eta ]$.
\end{theorem}

\begin{proof}
From (\ref{1019}) we have $G^{i}(x)=G^{i}(Z_{m}^{i,\circ k}(x))$ for all $k$
and all $x\in \lbrack \delta ,\eta ]$. Fix $x\in \lbrack \delta ,\eta ]$. By 
\textit{Lemma~\ref{Zm-orbits}}, the orbit $Z_{m}^{i,\circ k}(x)$ converges
to a limit point $a^{i}\in \{p_{m}^{i},1\}$ (depending on whether we are in
(SC) or (OD)). Taking $k\rightarrow \infty $ in (\ref{1019}) and using
continuity of $G^{i}$ yields 
\begin{equation}
G^{i}(x)=\lim_{k\rightarrow \infty }G^{i}(Z_{m}^{i,\circ k}(x))=G^{i}(a^{i}).
\label{1026}
\end{equation}%
In case (SC), the limit point $a^{i}=p_{m}^{i}$ is the same for every $x\in
(0,1)$, hence $G^{i}(x)=G^{i}(p_{m}^{i})$ for all $x$ and $G^{i}$ is
constant. In case (OD), $a^{i}=1$ for every $x\in (0,1)$, hence $%
G^{i}(x)=G^{i}(1)$ for all $x$. In all cases, $G^{i}$ is constant on $%
[\delta ,\eta ]$. \hfill
\end{proof}

\begin{remark}
\label{convergence} Theorem~\ref{Phi-cluster-constant} provides exactly the
input required in the RR2 argument: every subsequential limit of $\Phi
_{n}^{i}$ on $[\delta ,\eta ]$ is a constant. When inserted into the bounds
following (\ref{84})-(\ref{87}) (where the dependence term cancels), this
forces the limiting copula to be the independence copula. A more formal
formulation is given in \textit{Appendix D.3}.
\end{remark}

\begin{flushleft}
\textbf{Further remarks}
\end{flushleft}

Having provided the supporting arguments for the main theorem, we conclude
with several further remarks.

As demonstrated throughout this work, our solution is fundamentally based on
proving that every subsequential limit of the sequence $\Phi _{n}^{i}$ on a
compact interval $[\delta ,\eta ]$ is constant. The key to establishing this
was the crossing pattern of the maps $T_{n}^{i}$ (and consequently of $%
Z_{m}^{i,\circ k}$). This structure, in turn, allowed us to leverage the
integral form of equation (\ref{5.1}) to construct a transport mechanism
that facilitates the necessary cancellations.

Furthermore, our approach not only establishes the convergence of the
iteration (\ref{5.1}) but also allows us to directly identify the limits:
the independence copula with power-law marginals. This distinction is
critical because a naive attempt to pass to the limit---even
subsequentially---in either (\ref{1000}) or (\ref{5.1}) would have failed,
as convergence was not guaranteed. Such a limit transition would have only
identified the fixed points of the iterations. These fixed points are merely
candidate limits; they are not guaranteed to be limit points or even members
of the cluster set, and the fundamental question of convergence would have
remained unresolved.

In light of the preceding discussion, it is worth noting a key distinction:
the problem of finding the fixed point of (\ref{5.1})---or equivalently, of (%
\ref{13}), assuming the existence of a density---is significantly simpler
than proving the convergence of the iteration itself. This fixed point may
serve as a candidate for the limit or be of interest as a standalone
mathematical problem. The claims presented below make this distinction
explicit.

\begin{claim}
\label{limit-FE} Suppose that $L_{n}\rightarrow L^{\ast }$ uniformly on $%
[0,1]^{2}$, with marginals $L_{i}^{n}\rightarrow L_{i}^{\ast }$ and inverse
marginals $T_{i}^{n}=(L_{i}^{n})^{-1}\rightarrow T_{i}^{\ast }=(L_{i}^{\ast
})^{-1}$ uniformly on $[0,1]$. Then $L^{\ast }$ satisfies 
\begin{equation}
L^{\ast }(x_{1},x_{2})=\frac{\int_{0}^{L_{1}^{\ast
,-1}(x_{1})}\int_{0}^{L_{2}^{\ast ,-1}(x_{2})}u_{1}u_{2}\,dL^{\ast
}(u_{1},u_{2})}{\int_{0}^{1}\int_{0}^{1}u_{1}u_{2}\,dL^{\ast }(u_{1},u_{2})}%
,\qquad (x_{1},x_{2})\in \lbrack 0,1]^{2}.  \label{902}
\end{equation}%
In particular, each marginal $L_{i}^{\ast }$ satisfies the one--dimensional
functional equation 
\begin{equation}
L_{i}^{\ast }(x)=\frac{1}{I^{\ast }}\int_{0}^{L_{i}^{\ast
,-1}(x)}k_{i}^{\ast }(u)\,du,\qquad 0\leq x\leq 1,  \label{903}
\end{equation}%
where 
\begin{equation}
I^{\ast }=\int_{0}^{1}\int_{0}^{1}u_{1}u_{2}\,dL^{\ast }(u_{1},u_{2}),\qquad
k_{i}^{\ast }(u)=u\,E[X_{j}^{L^{\ast }}\mid X_{i}^{L^{\ast
}}=u]l_{i}(u),\quad j\neq i.  \label{904}
\end{equation}
\end{claim}

\begin{proof}
The proof follows the scheme in \textit{Appendix~C}. For each $n$, we can
express 
\begin{equation}
L_{n+1}(x_{1},x_{2})=\frac{\int_{0}^{T_{1}^{n}(x_{1})}%
\int_{0}^{T_{2}^{n}(x_{2})}u_{1}u_{2}\,dL_{n}(u_{1},u_{2})}{%
\int_{0}^{1}\int_{0}^{1}u_{1}u_{2}\,dL_{n}(u_{1},u_{2})}.  \label{905}
\end{equation}%
By uniform convergence $L_{n}\rightarrow L^{\ast }$ and $T_{i}^{n}%
\rightarrow T_{i}^{\ast }$, and dominated convergence, we obtain 
\begin{equation}
\int_{0}^{T_{i}^{n}(x_{i})}\int_{0}^{T_{j}^{n}(x_{j})}u_{1}u_{2}%
\,dL_{n}(u_{1},u_{2})\rightarrow \int_{0}^{T_{i}^{\ast
}(x_{i})}\int_{0}^{T_{j}^{\ast }(x_{j})}u_{1}u_{2}\,dL^{\ast }(u_{1},u_{2}),
\label{906}
\end{equation}%
and similarly $I_{n}\rightarrow I^{\ast }$. Passing to the limit in the
identity for $L_{n+1}$ gives (\texttt{\textup{\ref{902})}}. The marginal
equation (\texttt{\textup{\ref{903})}} follows by integrating out the other
coordinate.
\end{proof}

Next, let $C^{\ast }$ be the copula associated with $L^{\ast }$, and $%
c^{\ast }$ its density. The functional equation (\texttt{\textup{\ref{902})}}%
can be rewritten in copula form. A direct differentiation (as in the main
text) yields the following.

\begin{claim}
\label{copula-FP} Let $L^{\ast }$ and $c^{\ast }$ be as above, and let $%
l_{i}^{\ast }$ be the marginal densities of $L^{\ast }$. Then the copula
density satisfies 
\begin{equation}
c^{\ast }(x_{1},x_{2})=\frac{1}{I^{\ast }}\frac{L_{1}^{\ast ,-1}(L_{1}^{\ast
,-1}(x_{1}))L_{2}^{\ast ,-1}(L_{2}^{\ast ,-1}(x_{2}))}{\,l_{1}^{\ast
}(L_{1}^{\ast ,-1}(x_{1}))\,l_{2}^{\ast }(L_{2}^{\ast ,-1}(x_{2}))}\,c^{\ast
}(L_{1}^{\ast ,-1}(x_{1}),L_{2}^{\ast ,-1}(x_{2})),\quad (x_{1},x_{2})\in
(0,1)^{2}.  \label{907}
\end{equation}%
Equivalently, $c^{\ast }$ is a fixed point of the implicit nonlinear
operator $\mathcal{T}$ on copulas defined in (\ref{13}) of the main
text.\hfill
\end{claim}

\begin{proof}
The derivation is exactly the same algebraic manipulation as in the
iterative case (equation (\ref{13}) of the main text), but applied to the
functional equation (\texttt{\textup{\ref{902})}}rather than to a single
step of the iteration. We omit the straightforward details.
\end{proof}

Iterating (\ref{907}) gives a compounding representation analogous to
equation (\ref{16}) in the main text, but with all step indices replaced by
a star.

\begin{claim}
\label{star-compounding} Let $T_{i}^{\ast }=L_{i}^{\ast ,-1}$. For each $%
n\geq 0$ there exist positive factors $\mathcal{I}^{\ast ,n}$, $\mathcal{P}%
_{i}^{\ast ,n}$, $\mathcal{Q}_{i}^{\ast ,n}$ such that 
\begin{equation}
c^{\ast }(x_{1},x_{2})=\mathcal{I}^{\ast ,n}\,\mathcal{D}^{\ast
,n}(x_{1},x_{2})\,\frac{\mathcal{P}_{1}^{\ast ,n}(x_{1})\,\mathcal{P}%
_{2}^{\ast ,n}(x_{2})}{Q\mathcal{\ }_{1}^{\ast ,n}(x_{1})\,Q\mathcal{\ }%
_{2}^{\ast ,n}(x_{2})},  \label{908}
\end{equation}%
where 
\begin{equation}
\mathcal{D}^{\ast ,n}(x_{1},x_{2})=c^{F}(T_{1}^{\ast \circ
(n+1)}(x_{1}),\,T_{2}^{\ast \circ (n+1)}(x_{2})),  \label{909}
\end{equation}%
and $\mathcal{P}_{i}^{\ast ,n},\mathcal{Q}_{i}^{\ast ,n}$ are products of
the form 
\begin{equation}
\mathcal{P}_{1}^{\ast ,n}(x_{1})=F_{1}^{-1}(T_{1}^{\ast \circ
(n+1)}(x_{1}))\prod_{k=0}^{n}T_{1}^{\ast }(T_{1}^{\ast \circ k}(x_{1})),Q%
\mathcal{\ }_{1}^{\ast ,n}(x_{1})=\prod_{k=0}^{n}l_{1}^{\ast }(T_{1}^{\ast
\circ k}(x_{1})),  \label{910}
\end{equation}%
and similarly for coordinate~2. The precise form of these factors is as in (%
\ref{16})--(\ref{16.2}) of the main text, with all iteration indices
replaced by~$\ast $.
\end{claim}

\begin{remark}
The representation (\ref{908}) shows that the dependence structure of $%
L^{\ast }$ is entirely encoded in the orbits of the inverse marginals $%
T_{i}^{\ast }$ under the copula $c^{F}$ of the original distribution~$F$ and
multiplicative factors built from $L_{i}^{\ast }$ and $l_{i}^{\ast }$.
\end{remark}

\begin{remark}
From a conceptual standpoint, the functional equation and the iteration are
two sides of the same coin: the iteration constructs approximations $L_{n}$
to any fixed point of the functional equation, while the functional equation
captures the limit shape seen by any convergent subsequence of $%
\{L_{n}\}_{n=0}^{\infty }$. Under $RR_{2}$ and suitable regularity, the
existence of a unique solution of the functional equation with good
dynamical properties is equivalent to the convergence of the full iteration
and the collapse of the compound maps $\Phi _{n}^{i}$ to constants.
Iterative approximations constitute a powerful method for solving functional
equations, as detailed in \cite{[47]} among others.
\end{remark}

\begin{flushleft}
{\large Appendix D.3}
\end{flushleft}

The results of \textit{Appendix C.1,} \textit{Appendix D.1} and \textit{%
Appendix D.2} allow us to formulate the following result:

\begin{theorem}
\label{phi-conv} If the density $f_{12}$ is $TP_{2}$ (Totally Positive of
order 2), then the following hold for $i=1,2$:

\begin{enumerate}
\item The sequence of functions $\Phi _{n}^{i}$ converges pointwise to $1$
on $[0,1]$, that is, 
\begin{equation}
\lim_{n\rightarrow +\infty }\Phi _{n}^{i}(x)=\left\{ 
\begin{array}{c}
1,\text{ }x>0 \\ 
0,\text{ }x=0%
\end{array}%
\right.  \label{400}
\end{equation}

\item Furthermore, for any $\delta $ $\in (0,1)$, the convergence is uniform
on the interval $[\delta ,1].$
\end{enumerate}

If the density $f_{12}$ is $SRR_{2}$ under the additional assumptions of
log-concavity and non-degeneracy, then the following hold for $i=1,2$:

\begin{enumerate}
\item For any sequence of indices $\{n_{k}\}_{k\geq 0}$ with $%
n_{k}\rightarrow \infty $ the sequence of functions $\Phi _{n_{k}}^{i}$
converges subsequentially to $\phi _{i}$ on $[0,1]$, that is, 
\begin{equation}
\lim_{k\rightarrow +\infty }\Phi _{n_{k}}^{i}(x)=\left\{ 
\begin{array}{c}
\phi _{i},\text{ }x\in (0,1) \\ 
0,\text{ }x=0 \\ 
1,\text{ }x=1%
\end{array}%
\right. ,  \label{401}
\end{equation}

where $\phi _{i}=c_{i}$ is the limit of the crossings of $L_{i}^{n}(x)$ with
the diagonal.

\item Furthermore, for every fixed $(\delta ,\eta )$ $\subset $ $(0,1)$, the
convergence to $\phi _{i}$ is uniform on the interval $[\delta ,\eta ]$.
\end{enumerate}
\end{theorem}

\begin{remark}
The analysis in the main body of the paper establishes that \textit{Theorem} %
\ref{main-th} is a direct consequence of \textit{Theorem} \ref{phi-conv}.
This relationship yields a stronger a posteriori conclusion. If the initial
density is either: (i) $SRR_{2},$ or (ii) $RR_{2}$ under the additional
assumptions of log-concavity and non-degeneracy, for $i=1,2$, the pointwise
limit is:%
\begin{equation}
\lim_{n\rightarrow +\infty }\Phi _{n}^{i}(x)=\left\{ 
\begin{array}{c}
1,\text{ }x>0 \\ 
0,\text{ }x=0%
\end{array}%
\right.  \label{401a}
\end{equation}

Furthermore, for any $\delta $ $\in (0,1)$, the convergence is uniform on
the interval $[\delta ,1].$
\end{remark}

\begin{flushleft}
{\large Appendix D.4}
\end{flushleft}

In this appendix, we give further conceptual details on \textit{Section 5}
and the proof of \textit{Theorem~\ref{Dd-main-long}. }We will first prove
that the following sequences of compound functions 
\begin{eqnarray}
\Phi _{n}^{1}(x) &=&(L_{1}^{0,-1}\circ L_{1}^{1,-1}\circ \dots \circ
L_{1}^{n-1,-1}\circ L_{1}^{n,-1})(x)  \label{942a} \\
\Phi _{n}^{2}(x) &=&(L_{2}^{0,-1}\circ L_{2}^{1,-1}\circ \dots \circ
L_{2}^{n-1,-1}\circ L_{2}^{n,-1})(x)  \notag
\end{eqnarray}

are uniformly subsequentially convergent to constants for $x\in \lbrack 0,1]$%
, which is enough to prove the theorem. Then we will point out the complete
uniform convergence.

\begin{flushleft}
\textbf{Multivariate }$TP_{2}$\textbf{\ and }$RR_{2}$\textbf{\ structures}
\end{flushleft}

In dimension $d>2$ there are several inequivalent ways to generalise the 2D
notions of $TP_{2}$ and $RR_{2}$. It is most natural to follow the
conditional--expectation viewpoint used in the 2D sections, and require
monotonicity in each coordinate of the full conditional mean of the
remaining coordinates.

\begin{definition}
\noindent \label{d-tp2-rr2} Let $H$ be a probability distribution on $%
[0,1]^{d}$ with density $h$. For $i\in \{1,\dots ,d\}$ define the
conditional mean 
\begin{equation}
m_{i}^{H}(u)\;=\;E(\prod_{j\neq i}X_{j}^{H}\,\mid X_{i}^{H}=u),\qquad u\in
(0,1),  \label{942}
\end{equation}%
where $X^{H}=(X_{1}^{H},\dots ,X_{d}^{H})$ has law $H$.

We say that $H$ is

\begin{itemize}
\item coordinatewise $TP_{2}$ if, for each $i$, the function $u\mapsto
m_{i}^{H}(u)$ is non--decreasing on $(0,1)$;

\item coordinatewise $RR_{2}$ if, for each $i$, the function $u\mapsto
m_{i}^{H}(u)$ is non--increasing on $(0,1)$.
\end{itemize}

If the inequalities are strict on every compact subinterval of $(0,1)$ where
the univariate density $h_{i}(u)$ is bounded away from $0$, we speak of
strict $TP_{2}$ or strict $RR_{2}$.
\end{definition}

In dimension $d=2$ we recover the familiar situation. There, for $H$ with
density $h$ we have 
\begin{equation}
m_{1}^{H}(u)\;=\;E[X_{2}^{H}\mid X_{1}^{H}=u],\qquad
m_{2}^{H}(u)\;=\;E[X_{1}^{H}\mid X_{2}^{H}=u],  \label{943}
\end{equation}%
and coordinatewise $TP_{2}$/$RR_{2}$ coincides with the usual $TP_{2}$/$%
RR_{2}$ conditions of \textit{Appendix~D.1--D.2} via the standard
characterisations in terms of conditional means and \textit{%
likelihood--ratio ordering} (see, e.g., \cite{[32]},\cite{[28]}, and \cite%
{[24]}).

In higher dimension there exist stronger notions such as \textit{%
multivariate total positivity of order two} ($MTP_{2}$) and its reverse form
($MRR_{2}$), defined in terms of the determinant sign of all $2\times 2$
marginals or via \textit{log--supermodularity} / \textit{log--submodularity}
of $h\footnote{{\footnotesize See\textit{\ Definition~\ref{MTP2-SMRR2}} from
the main text. For }$MTP_{2}${\footnotesize \ and the lattice inequality (%
\ref{2000a}) (log--supermodularity), see the foundational work of \cite{[32]}%
; see also \cite{[63]} for modern structural properties and support
conditions. The reverse inequality (\ref{2000b}) (log--submodularity), used
to define }$MRR_{2}${\footnotesize , is the natural sign-reversed analogue.}
\par
{}}$. Under suitable regularity\footnote{{\footnotesize Here
\textquotedblleft regularity\textquotedblright\ refers to assumptions
ensuring that the conditional laws }$Law(X_{-i}^{H}\mid X_{i}^{H}=u)$%
{\footnotesize \ admit a version depending measurably on }$u${\footnotesize %
\ and that the conditional expectations }$E[\varphi (X_{-i}^{H})\mid
X_{i}^{H}=u]${\footnotesize \ are well defined (for all }$u${\footnotesize \
in a set of full }$H_{i}${\footnotesize --measure) and finite for the class
of coordinatewise non--decreasing test functions }$\varphi ${\footnotesize \
under consideration. In the present paper it is sufficient to work in the
absolutely continuous setting with }$h>0${\footnotesize \ on }$(0,1)^{d}$%
{\footnotesize \ (so that regular conditional densities exist and the above
monotonicity statements can be interpreted pointwise for a.e.\ }$u$%
{\footnotesize ).}}, $MTP_{2}$ implies that for each $i$ and each
non--decreasing function $\varphi $ of the remaining coordinates, the map $%
u\mapsto E[\varphi (X_{-i}^{H})\mid X_{i}^{H}=u]$ is non--decreasing;
similarly, $MRR_{2}$ implies the reverse monotonicity. In other words, under
these regularity assumptions, $MTP_{2}$ (resp.\ $MRR_{2}$) yields a positive
(resp.\ negative) regression--type monotonicity property (akin to $PRD/NRD$%
\textit{\footnote{{\footnotesize The quantification \textquotedblleft for
every coordinatewise non--decreasing }$\varphi ${\footnotesize %
\textquotedblright\ is the hallmark of \textit{positive regression dependence%
} (PRD), also called \textit{stochastic monotonicity} or \textit{stochastic
increasingness}; the reverse inequality corresponds to its negative analogue
(NRD). See, e.g., the discussion of \textit{stochastic monotonicity} and its
links to regression dependence in \cite{[60]} and \cite{[61]} and more
recent treatments such as \cite{[62]}.}}}). In particular, $MTP_{2}$ (resp.\ 
$MRR_{2}$) implies the coordinatewise $TP_{2}$ (resp.\ $RR_{2}$) property in
the sense of \textit{Definition~\ref{d-tp2-rr2}} when we take $\varphi
(x_{-i})=\prod_{j\neq i}x_{j}$. We have all the notions in a strict sense
when the respective inequalities or monotonicity are strict, and we put an
"S" in front the abbreviation to denote that.

For the remainder we aim results under the weaker, but more directly usable,
hypothesis of coordinatewise $TP_{2}$/$SRR_{2}$. Whenever, it is necessary
to work with the classical $MTP_{2}$/$SMRR_{2}$ or $PRD/SNRD$ notions, it
suffices to note that they imply \textit{Definition~\ref{d-tp2-rr2}}.Yet,
similarly to the bivariate case, we can demonstrate the decorrelation
properties of our \textit{Lorenz map} if we start with a strong dependence
concept. \textit{Theorem~\ref{Dd-main-long} }from the main text is defined
in such a way.

\begin{flushleft}
\textbf{Transferability of }$MTP_{2}$\textbf{/}$SMRR_{2}$
\end{flushleft}

In this subsection we prove that $MTP_{2}$ and $SMRR_{2}$ are preserved by
the iteration (\ref{941b}). This is the first genuinely $d$--dimensional
ingredient. We have the following closure lemmas:

\begin{lemma}
\label{Dd-closure-map} Let $l$ be a strictly positive $MTP_{2}$ (resp.\ $%
MRR_{2}$) density on $[0,1]^{d}$. Let $\psi :[0,1]^{d}\rightarrow \lbrack
0,1]^{d}$ be coordinatewise non-decreasing. Then $l\circ \psi $ is $MTP_{2}$
(resp.\ $MRR_{2}$).
\end{lemma}

\begin{proof}
\medskip \noindent Let $x,y\in \lbrack 0,1]^{d}$. Since $\psi $ is
coordinatewise non-decreasing, $\psi (x\wedge y)=\psi (x)\wedge \psi (y)$
and $\psi (x\vee y)=\psi (x)\vee \psi (y)$. Therefore, 
\begin{equation}
(l\circ \psi )(x\wedge y)\,(l\circ \psi )(x\vee y)=l(\psi (x)\wedge \psi
(y))\,l(\psi (x)\vee \psi (y))\geq (\leq )\ l(\psi (x))\,l(\psi (y))=(l\circ
\psi )(x)\,(l\circ \psi )(y),  \label{2006}
\end{equation}%
which is exactly (\ref{2000a}) (resp.\ (\ref{2000b})) for $l\circ \psi $.
\medskip
\end{proof}

\begin{lemma}
\label{Dd-closure-prod} If $g$ and $h$ are strictly positive $MTP_{2}$
(resp.\ $MRR_{2}$) functions on $[0,1]^{d}$, then $gh$ is $MTP_{2}$ (resp.\ $%
MRR_{2}$). Moreover, any strictly positive function of the form 
\begin{equation}
a(x)=\prod_{i=1}^{d}a_{i}(x_{i})  \label{2007}
\end{equation}%
(with $a_{i}:(0,1)\rightarrow (0,\infty )$) is both $MTP_{2}$ and $MRR_{2}$
(with equality in (\ref{2000a})/(\ref{2000b})).
\end{lemma}

\begin{proof}
\medskip For the product: multiply the lattice inequalities for $g$ and $h$
and use positivity. For the coordinatewise factor $a$, note that 
\begin{equation}
a(x\wedge y)\,a(x\vee y)=\prod_{i=1}^{d}a_{i}(\min
\{x_{i},y_{i}\})\,a_{i}(\max
\{x_{i},y_{i}\})=\prod_{i=1}^{d}a_{i}(x_{i})\,a_{i}(y_{i})=a(x)\,a(y),
\label{2008}
\end{equation}%
so the lattice inequality holds with equality. \medskip
\end{proof}

\begin{lemma}
\label{Dd-closure-strict} Let $g$ be $SMRR_{2}$ on $(0,1)^{d}$ and let $h$
be strictly positive and $MRR_{2}$ on $(0,1)^{d}$. Then $gh$ is $SMRR_{2}$.
Moreover, if $\psi $ is coordinatewise increasing and continuous, then $%
g\circ \psi $ is $SMRR_{2}$ on compact subsets of $(0,1)^{d}$ on which $\psi 
$ stays in a compact subset of $(0,1)^{d}$.
\end{lemma}

\begin{proof}
\medskip The (non-strict) $MRR_{2}$ inequality is preserved under products
and monotone compositions by \textit{Lemmas~\ref{Dd-closure-prod} and~\ref%
{Dd-closure-map}}. For strictness, fix a compact $K\subset (0,1)^{d}$ with $%
g $ bounded away from $0$ on $K$, and take incomparable $x,y\in K$. By $%
SMRR_{2}$ of $g$ we have $g(x\wedge y)g(x\vee y)<g(x)g(y)$. Since $h$ is $%
MRR_{2}$, $h(x\wedge y)h(x\vee y)\leq h(x)h(y)$. Multiplying yields strict
inequality for $gh$. The composition statement follows because incomparable $%
x,y$ imply incomparable $\psi (x),\psi (y)$ for coordinatewise increasing $%
\psi $, and $\psi (x)\wedge \psi (y)=\psi (x\wedge y)$, $\psi (x)\vee \psi
(y)=\psi (x\vee y)$. \medskip
\end{proof}

The above lemmas allow to prove the transferability result.

\begin{claim}
\label{Dd-transfer} Assume $L_{n}$ has strictly positive density $l_{n}$ on $%
(0,1)^{d}$. If $l_{n}$ is $MTP_{2}$, then $L_{n+1}$ has strictly positive
density $l_{n+1}$ which is $MTP_{2}$. If $l_{n}$ is $SMRR_{2}$, then $%
L_{n+1} $ has strictly positive density $l_{n+1}$ which is $SMRR_{2}$.
\end{claim}

\begin{proof}
We use the density recursion (\ref{941c}) from the main text. Writing it
explicitly (and suppressing the superscript $n$ on $T_{i}^{n}$ for
readability), there exists a strictly positive normalizing constant $%
c_{n}=1/I_{n}$ such that 
\begin{equation}
l_{n+1}(x)=c_{n}\;(\prod_{i=1}^{d}T_{i}(x_{i}))\;l_{n}(T(x))\;%
\prod_{i=1}^{d}(T_{i})^{\prime }(x_{i}),\qquad x\in (0,1)^{d},  \label{2009}
\end{equation}%
where $T(x)=(T_{1}(x_{1}),\dots ,T_{d}(x_{d}))$.

Set $a(x)=\prod_{i=1}^{d}T_{i}(x_{i})$ and $b(x)=\prod_{i=1}^{d}(T_{i})^{%
\prime }(x_{i})$ so that $l_{n+1}=c_{n}\cdot a\cdot b\cdot (l_{n}\circ T)$.
By \textit{Lemma~\ref{Dd-closure-prod}}, $a$ and $b$ are coordinatewise
factors and hence both $MTP_{2}$ and $MRR_{2}$ (with equality in the lattice
inequality). By \textit{Lemma~\ref{Dd-closure-map}}, $(l_{n}\circ T)$ is $%
MTP_{2}$ (resp.\ $MRR_{2}$) whenever $l_{n}$ is. Hence, by \textit{Lemma~\ref%
{Dd-closure-prod}}, $l_{n+1}$ is $MTP_{2}$ if $l_{n}$ is $MTP_{2}$.

For the strict reverse case, assume $l_{n}$ is $SMRR_{2}$. Then $(l_{n}\circ
T)$ is $SMRR_{2}$ on compact sets by \textit{Lemma~\ref{Dd-closure-strict}}
(composition part), and multiplying by $a$ and $b$ preserves strictness by 
\textit{Lemma~\ref{Dd-closure-strict}} (product part), since $a,b$ are $%
MRR_{2}$. Therefore $l_{n+1}$ is $SMRR_{2}$. \medskip
\end{proof}

\begin{flushleft}
\textbf{Crossing patterns in dimension }$d$\textbf{\ (}$TP$\textbf{-route
and }$RR$\textbf{-route)}
\end{flushleft}

This subsection provides the second genuinely $d$--dimensional ingredient:
the multivariate crossing patterns. In the $TP$-route we need the same
subdiagonality input as in \textit{Appendix~D.1} (formulated on
two-dimensional coordinates). In the $RR$-route,
overdiagonality/single-crossing is a property of the inverse-marginal maps $%
T_{i}^{n}$ (and hence of the compounds $\Phi _{n}$), exactly as in \textit{%
Appendix~D.2.}

Fix indices $1\leq i<j\leq d$ and fix $z\in (0,1)^{d-2}$ for the remaining
coordinates. Define the two-dimensional slice density 
\begin{equation}
l_{z}^{(i,j)}(u_{i},u_{j})=l(u_{1},\dots ,u_{d})\mid _{u_{k}=z_{k},\ k\notin
\{i,j\}},\qquad (u_{i},u_{j})\in (0,1)^{2}.  \label{2010}
\end{equation}

\begin{lemma}
\label{Dd-slice-TP} If $l$ is $MTP_{2}$ on $(0,1)^{d}$, then for each $i<j$
and each fixed $z$, the bivariate slice $l_{z}^{(i,j)}$ is $TP_{2}$ on $%
(0,1)^{2}$.
\end{lemma}

\begin{proof}
\medskip Take two points $(u_{i},u_{j})$ and $(v_{i},v_{j})$ in $(0,1)^{2}$
and embed them in $R^{d}$ as $x$ and $y$ by setting $x_{k}=y_{k}=z_{k}$ for $%
k\notin \{i,j\}$ and $(x_{i},x_{j})=(u_{i},u_{j})$, $%
(y_{i},y_{j})=(v_{i},v_{j})$. Then $x\wedge y$ and $x\vee y$ coincide with
taking coordinatewise min/max in the $(i,j)$ coordinates and leaving the
others fixed at $z$. Applying (\ref{2000a}) to $(x,y)$ yields precisely the $%
TP_{2}$ inequality for $l_{z}^{(i,j)}$. \medskip
\end{proof}

\begin{lemma}
\label{Dd-slice-SRR} If $l$ is $SMRR_{2}$ on $(0,1)^{d}$, then for each $i<j$
and each fixed $z$, the bivariate slice $l_{z}^{(i,j)}$ is strictly $RR_{2}$
($SRR_{2}$) on compact subsets of $(0,1)^{2}$ where it is bounded away from $%
0$.
\end{lemma}

\begin{proof}
\medskip Let $K\subset (0,1)^{2}$ be compact and assume $l_{z}^{(i,j)}$ is
bounded away from $0$ on $K$. Embed $(u_{i},u_{j}),(v_{i},v_{j})\in K$ into $%
x,y\in (0,1)^{d}$ as in \textit{Lemma~\ref{Dd-slice-TP}}. If the two points
in $K$ are incomparable in $R^{2}$, then the embedded points $x,y$ are
incomparable in $R^{d}$ (only coordinates $i,j$ vary). Hence, by $SMRR_{2}$
of $l$, 
\begin{equation}
l(x\wedge y)\,l(x\vee y)\ <\ l(x)\,l(y).  \label{2011}
\end{equation}%
Restricting the fixed coordinates to $z$ shows that this is exactly the
strict $RR_{2}$ inequality for the slice density $l_{z}^{(i,j)}$ on $K$.
\medskip
\end{proof}

The above lemmas allow us to provide consecutively results for the crossing
pattern of $L_{n}$ frozen to two coordinates.

\begin{claim}
\label{Dd-subdiag} Assume that for some $n\geq 0$, the density $l_{n}$ of $%
L_{n}$ is $MTP_{2}$ on $(0,1)^{d}$. Then, on every two-dimensional
coordinate slice (obtained by freezing $d-2$ coordinates), the restriction
of $L_{n}$ satisfies the same subdiagonality (one-sided crossing) property
as in Appendix~D.1. Consequently, the multivariate subdiagonality hypothesis
used in the multivariate section is satisfied.
\end{claim}

\begin{proof}
Fix a pair $i<j$ and freeze the remaining coordinates at $z\in (0,1)^{d-2}$.
By \textit{Lemma~\ref{Dd-slice-TP}}, the induced bivariate slice density is $%
TP_{2}$. Therefore the bivariate subdiagonality lemma from \textit{%
Appendix~D.1} applies on that slice. Since the multivariate subdiagonality
notion is defined via these coordinate slices/paths, the conclusion follows.
\medskip
\end{proof}

For each $n\geq 0$ and $i\in \{1,\dots ,d\}$ let $T_{i}^{n}=(L_{i}^{n})^{-1}$
be the inverse marginal. Define the compound maps $\Phi _{n}^{\,i}$ exactly
as in \textit{Appendix~D.2} (with the same conventions).

\begin{claim}
\label{Dd-overdiag-T} Assume that for some $n\geq 0$, the density $l_{n}$ of 
$L_{n}$ is $SMRR_{2}$ on $(0,1)^{d}$. Then, on every two-dimensional
coordinate slice (obtained by freezing $d-2$ coordinates), the corresponding
bivariate inverse-marginal maps satisfy the same overdiagonality and
(single) crossing properties as in Appendix~D.2 (in particular, the strict
version needed in the return--shift argument). Consequently, the bivariate
D.2 crossing input holds for the maps $T_{i}^{n}$ on each slice.
\end{claim}

\begin{proof}
Fix $i<j$ and freeze the remaining coordinates at $z\in (0,1)^{d-2}$. By 
\textit{Lemma~\ref{Dd-slice-SRR}}, the induced bivariate slice density is $%
SRR_{2}$ on compact subsets where it is bounded away from $0$. Therefore the
bivariate overdiagonality and single-crossing lemmas of \textit{Appendix~D.2}
apply on that slice, and those lemmas are stated in terms of the
inverse-marginal maps used to build $\Phi _{n}$. This yields the desired
slice-level properties for $T_{i}^{n}$. \medskip
\end{proof}

\begin{corollary}
\label{Dd-phi-crossing} Assume the hypotheses of Claim~\ref{Dd-overdiag-T}
hold along the iteration (i.e.\ for all $n$). Then for each $i$, the
sequence of compound maps $\Phi _{n}^{\,i}$ inherits the same
single-crossing restrictions as in Appendix~D.2 (with the strict form on
compact subsets).
\end{corollary}

\begin{proof}
\medskip \textit{Appendix~D.2} transfers crossings to $\Phi _{n}$ by
monotone composition.
\end{proof}

\begin{remark}
\label{alt-route} Besides the slice reduction used above, there is a second
(conceptually more \textquotedblleft one--dimensional\textquotedblright )
way to obtain the same shape (crossing) properties needed in \textit{%
Appendices~D.1--D.2}. It is based on the marginal representation 
\begin{equation}
L_{i}^{n+1}(x)=\frac{1}{I_{n}}\int_{0}^{T_{i}^{n}(x)}u\,m_{i}^{L_{n}}(u)%
\,l_{i}^{n}(u)\,du,\qquad x\in \lbrack 0,1],  \label{3000}
\end{equation}%
where $T_{i}^{n}=(L_{i}^{n})^{-1}$ and 
\begin{equation}
m_{i}^{L_{n}}(u)=E\![\prod_{k\neq i}X_{k}^{L_{n}}\,\mid
\,X_{i}^{L_{n}}=u],\qquad u\in (0,1).  \label{3001}
\end{equation}

Formula (\ref{3000}) follows by integrating out $X_{-i}$ in the defining
identity for $L_{n+1}$ (equivalently, by applying the tower property to the
product tilt by $\prod_{k}X_{k}$). The key point is that $MTP_{2}$/$MRR_{2}$
assumptions allow we to control the monotonicity of the conditional
expectation $m_{i}^{L_{n}}$ in (\ref{3001}), which in turn yields the same
one--dimensional crossing restrictions for $T_{i}^{n}$ (and hence for the
compound maps $\Phi _{n}$) as those obtained above.
\end{remark}

\begin{remark}
\label{mtp2-prd} Under the standing absolutely continuous setting with
strictly positive density on $(0,1)^{d}$, $MTP_{2}$ implies a positive
regression dependence property ($PRD$): for each $i$ and each coordinatewise
non--decreasing test function $\varphi $ of the remaining coordinates, 
\begin{equation}
u\longmapsto E[\varphi (X_{-i}^{H})\mid X_{i}^{H}=u]\quad \text{is
non--decreasing,}  \label{3001a}
\end{equation}%
and similarly $MRR_{2}$ implies the reverse monotonicity ($NRD$). In
particular, choosing $\varphi (x_{-i})=\prod_{k\neq i}x_{k}$ yields that,
for each $n$, 
\begin{equation}
u\longmapsto m_{i}^{L_{n}}(u)=E\![\prod_{k\neq i}X_{k}^{L_{n}}\,\mid
\,X_{i}^{L_{n}}=u]\quad \text{is non--decreasing under }MTP\text{$_{2}$ and
non--increasing under }MRR\text{$_{2}$,}  \label{3002}
\end{equation}%
with the strict form on compact subintervals in the $SMRR_{2}$ case (as in
the bivariate $SRR_{2}$ route). This provides exactly the monotonicity input
needed to run the one--dimensional arguments in the next remark. (For
background on the implication \textquotedblleft $MTP_{2}$ $\Rightarrow $ $%
PRD $\textquotedblright\ and its reverse form, see the references cited in
the main text.)
\end{remark}

\begin{remark}
\label{alt-proof-outline} We briefly outline how (\ref{3000}) together with (%
\ref{3002}) yields the same one--dimensional crossing properties as those
used in \textit{Appendices~D.1--D.2}.

\smallskip \noindent \textbf{Step 1 (reduction to a one--dimensional kernel).%
} For fixed $i$ and $n$, define the (unnormalized) kernel 
\begin{equation}
w_{i}^{n}(u)=u\,m_{i}^{L_{n}}(u)\,l_{i}^{n}(u),\qquad u\in (0,1),
\label{3003}
\end{equation}%
so that (\ref{3000}) reads 
\begin{equation}
L_{i}^{n+1}(x)=\frac{1}{I_{n}}\int_{0}^{T_{i}^{n}(x)}w_{i}^{n}(u)\,du.
\label{3004}
\end{equation}%
Since $T_{i}^{n}$ is increasing, the shape of $L_{i}^{n+1}$ and the relative
position of the inverse map $T_{i}^{n+1}$ are governed by the relative shape
of $w_{i}^{n}$ and the normalization $I_{n}$.

\smallskip \noindent \textbf{Step 2 (monotone likelihood ratio type
comparison).} Under $MTP_{2}$, $m_{i}^{L_{n}}$ is non--decreasing, hence $%
u\mapsto u\,m_{i}^{L_{n}}(u)$ is non--decreasing as well. Under $SMRR_{2}$, $%
m_{i}^{L_{n}}$ is strictly non--increasing on compact subintervals, hence $%
u\,m_{i}^{L_{n}}(u)$ is unimodal/one--crossing in the precise sense used in 
\textit{Appendix~D.2}. Combining this with the positivity of $l_{i}^{n}$
yields that ratios of the form 
\begin{equation}
\frac{\int_{0}^{t}w_{i}^{n}(u)\,du}{\int_{0}^{1}w_{i}^{n}(u)\,du}\quad \text{%
as functions of $t$}  \label{3005}
\end{equation}%
inherit the same subdiagonal/one--crossing behavior as in the bivariate case.

\smallskip \noindent \textbf{Step 3 (translation to $T_{i}^{n}$ and to $\Phi
_{n}$).} The arguments in \textit{Appendix~D.1} and \textit{Appendix~D.2}
that turn the above one--dimensional shape information into (i)
subdiagonality in the $TP$-route and (ii) the \textit{D.2} dichotomy (either
subdiagonal or exactly one-crossing) in the $RR$-route use only:

\begin{itemize}
\item monotonicity of inverse marginals $T_{i}^{n}$;

\item stability of crossing counts under composition with increasing maps;

\item and (in the $RR$-route) strictness on compact subintervals to exclude
flat segments.
\end{itemize}

Therefore, replacing the bivariate $TP_{2}$/$SRR_{2}$ input by the
regression monotonicity (\ref{3002}) yields the same conclusions for the $%
T_{i}^{n}$ and hence for the compound maps $\Phi _{n}$. In other words, the
slice reduction and the conditional-expectation route lead to the same
one--dimensional crossing restrictions needed for the compactness and
return--shift/invariance arguments of \textit{Appendices~D.1--D.2}.
\end{remark}

\begin{flushleft}
\textbf{Denominator non-collapse}
\end{flushleft}

We start with core technical results. Although \textit{Claim \ref{com-claim1}%
}, as noted therein, is also valid in the $d$-dimensional case, we need to
formally prove similarly to the bivariate case that the normalizing
constants 
\begin{equation}
I_{n}=\int_{[0,1]^{d}}\left( \prod_{j=1}^{d}u_{j}\right)
\,dL_{n}(u_{1},\ldots ,u_{d})  \label{4000}
\end{equation}%
satisfy 
\begin{equation}
\inf_{n\geq 0}I_{n}\geq c_{d}>0.  \label{4001}
\end{equation}

Consider a "bad" subsequence $\{n_{k}\}_{k\geq 0}$ such that 
\begin{equation}
I_{n_{k}}\longrightarrow 0.  \label{4001a}
\end{equation}%
If, along a further subsequence, 
\begin{equation}
L_{n_{k}}\Rightarrow L^{\ast },  \label{4001b}
\end{equation}%
then 
\begin{equation}
\int_{\lbrack 0,1]^{d}}x_{1}\cdots x_{d}\,dL^{\ast }(x)=0.  \label{4001c}
\end{equation}%
Hence 
\begin{equation}
L^{\ast }\{x_{1}\cdots x_{d}=0\}=1.  \label{4001d}
\end{equation}%
Thus a bad limit is supported on the union of the coordinate hyperplanes. In
particular, the \textit{Lorenz map} cannot be applied to $L^{\ast }$,
because the product denominator at $L^{\ast }$ is zero. The problem is
therefore to prevent a valid finite orbit from approaching such a
zero-denominator boundary law.

This boundary condition has a useful marginal consequence. Let $%
Y=(Y_{1},\ldots ,Y_{d})\sim L^{\ast }$, and put 
\begin{equation}
A_{i}=\{Y_{i}=0\},\qquad a_{i}=L_{i}^{\ast }(0)=P(Y_{i}=0),\qquad 1\leq
i\leq d.  \label{4001e}
\end{equation}%
Since $Y_{i}\in \lbrack 0,1]$, the condition (\ref{4001d}) is equivalent to 
\begin{equation}
P\left( \bigcup_{i=1}^{d}A_{i}\right) =1.  \label{4001f}
\end{equation}%
Consequently, 
\begin{equation}
1=P\left( \bigcup_{i=1}^{d}A_{i}\right) \leq
\sum_{i=1}^{d}P(A_{i})=\sum_{i=1}^{d}a_{i}.  \label{4001g}
\end{equation}%
In particular, at least one marginal of $L^{\ast }$ has a strictly positive
atom at zero. Thus every bad limit is not merely supported on the union of
the coordinate hyperplanes; it also forces a genuine zero-atom in at least
one marginal.

Equivalently, if 
\begin{equation}
\mathcal{A}^{\ast }=\{i:\,a_{i}>0\},  \label{4001h}
\end{equation}%
then $\mathcal{A}^{\ast }\neq \varnothing $, and the bad limit places all
its mass on 
\begin{equation}
\bigcup_{i\in \mathcal{A}^{\ast }}\{x_{i}=0\}.  \label{4001j}
\end{equation}%
The numbers $a_{i}$ record the marginal amount of mass carried by the
corresponding coordinate hyperplanes. More precisely, by
inclusion--exclusion, 
\begin{equation}
1=\sum_{i}P(A_{i})-\sum_{i<j}P(A_{i}\cap A_{j})+\cdots
+(-1)^{d+1}P(A_{1}\cap \cdots \cap A_{d}).  \label{4001k}
\end{equation}%
Thus the case $\sum_{i}a_{i}=1$ corresponds to zero-coordinate events which
are essentially disjoint, while $\sum_{i}a_{i}>1$ means that the bad limit
has positive mass on intersections of coordinate hyperplanes.

This observation also links denominator collapse to the marginal crossing
structure. Suppose $a_{i}>0$. Then, for every $x\in (0,a_{i})$ which is a
continuity point of the limiting marginal $L_{i}^{\ast }$, 
\begin{equation}
L_{i}^{\ast }(x)\geq L_{i}^{\ast }(0)=a_{i}>x.  \label{4001l}
\end{equation}%
Hence, along the subsequence, 
\begin{equation}
L_{i}^{n_{k}}(x)\longrightarrow L_{i}^{\ast }(x)>x.  \label{4001m}
\end{equation}%
Therefore, whenever the usual one-crossing property for the marginal $%
L_{i}^{n}$ is available, the corresponding crossing point satisfies 
\begin{equation*}
c_{i,n_{k}}\leq x
\end{equation*}%
for all sufficiently large $k$. Since $x>0$ can be chosen arbitrarily small,
we obtain 
\begin{equation}
c_{i,n_{k}}\longrightarrow 0\qquad \text{for every }i\text{ with }a_{i}>0.
\label{4001n}
\end{equation}%
Thus a bad subsequential limit necessarily manifests itself through the
collapse of at least one marginal crossing point to the left endpoint.

Assume that, for the current marginals, the \textit{Fr\'{e}chet--Hoeffding
lower-bound }is attainable. Denote the corresponding lower extremal
distribution by $L_{n}^{-}$, and denote its copula by $C_{n}^{-}$. If $C_{n}$
is the copula of $L_{n}$, then from \textit{Appendix~B} we have 
\begin{equation}
C_{n}^{-}\leq _{c}C_{n}.  \label{4001aa}
\end{equation}%
Since 
\begin{equation}
(x_{1},\ldots ,x_{d})\mapsto x_{1}\cdots x_{d}  \label{4001bb}
\end{equation}%
is increasing and supermodular on $[0,1]^{d}$, the comparison methodology of 
\textit{Claim~\ref{com-claim1}} yields 
\begin{equation}
\int x_{1}\cdots x_{d}\,dL_{n}^{-}(x)\leq \int x_{1}\cdots
x_{d}\,dL_{n}(x)=I_{n}.  \label{4001cc}
\end{equation}%
Consequently, whenever a genuine \textit{Fr\'{e}chet--Hoeffding lower-bound} 
$L_{n}^{-}$ exists and its denominator is bounded away from zero, the
denominator $I_{n}$ of the actual iterate is also bounded away from zero.
Thus the only potentially dangerous branches are those exceptional cases in
which the $d$-dimensional \textit{lower-bound} is attainable and the
corresponding \textit{Lorenz} \textit{trajectory} can approach the
zero-product boundary.

For $d=2$, the \textit{Fr\'{e}chet--Hoeffding lower-bound} is always
attainable: it is the distribution function of the countermonotone coupling
as we have already discussed in \textit{Appendix A.}

For $d\geq 3$, the situation is fundamentally different. The formal \textit{%
lower-bound} expression 
\begin{equation}
F^{\wedge }(x_{1},\ldots ,x_{d})=\left(
\sum_{i=1}^{d}F_{i}(x_{i})-d+1\right) _{+}  \label{4002a}
\end{equation}%
is not, in general, a $d$-dimensional distribution function. Dall'Aglio's
analysis of \textit{Fr\'{e}chet classes} and compatibility of distribution
functions shows that, apart from lower-dimensional degeneracies,
attainability of (\ref{4002a}) is possible only in exceptional large-jump
situations; see \cite{[76]} and the discussion in \cite[Chapter 6]{[5]}.

We shall only need the following consequence of that classification. Let 
\begin{equation}
\ell _{i}=\func{ess}X_{i},\qquad r_{i}=\func{ess}X_{i},\qquad i=1,\ldots ,d.
\label{4002b}
\end{equation}%
Then the attainable \textit{lower-bound} cases relevant here fall into the
following three alternatives.

\noindent \textbf{Case 0: effectively bivariate.} At most two marginals are
nondegenerate. Then the \textit{lower-bound} coupling is the ordinary
bivariate countermonotone coupling, embedded into $d$ dimensions by
deterministic coordinates. This case can occur even with continuous
nondegenerate marginals in the two genuinely random coordinates. It is
therefore not excluded by continuity alone. However, $\inf_{n\geq
0}I_{-,n}=0 $ can be ruled out from the analysis in \textit{Appendix D.2}.

\noindent \textbf{Case 1: lower-endpoint jump case.} Set 
\begin{equation}
q_{i}=P(X_{i}>\ell _{i}).  \label{4002c}
\end{equation}%
The \textit{lower-bound} case is characterized by 
\begin{equation}
\sum_{i=1}^{d}q_{i}\leq 1.  \label{4002d}
\end{equation}%
Equivalently, the events $\{X_{i}>\ell _{i}\}$ are disjoint under the 
\textit{lower-bound} law. Thus almost surely at most one coordinate is
strictly above its lower endpoint. This forces large atoms at the lower
endpoints. Indeed, if at least two marginals are nondegenerate and
continuous, then for each such marginal 
\begin{equation}
P(X_{i}>\ell _{i})=1,  \label{4002e}
\end{equation}%
so the sum in (\ref{4002e}) is already at least $2$, and (\ref{4002e})
cannot hold.

\noindent \textbf{Case 2: upper-endpoint jump case.} Set 
\begin{equation}
p_{i}=P(X_{i}<r_{i}).  \label{4002f}
\end{equation}%
The upper-endpoint case is characterized by 
\begin{equation}
\sum_{i=1}^{d}p_{i}\leq 1.  \label{4002g}
\end{equation}%
Equivalently, the events $\{X_{i}<r_{i}\}$ are disjoint under the \textit{%
lower-bound} law. Thus almost surely at most one coordinate is strictly
below its upper endpoint. This forces large atoms at the upper endpoints.
Indeed, if at least two marginals are nondegenerate and continuous, then for
each such marginal 
\begin{equation}
P(X_{i}<r_{i})=1,  \label{4002h}
\end{equation}%
so the sum in (\ref{4002g}) is already at least $2$, and (\ref{4002g})
cannot hold.

\smallskip

Consequently, if $F_{i}$ is continuous and nondegenerate for every $%
i=1,\ldots ,d,d\geq 3$, then none of the three exceptional attainability
cases above can occur which was exactly our assumption from the very
beginning regarding the extramal cases. Cases 1 and 2 are excluded because
continuous nondegenerate marginals have no endpoint atoms. Case 0 is
excluded because all $d\geq 3$ coordinates are genuinely random as well as
the $\inf_{n\geq 0}I_{-,n}>0$ argument.

\begin{flushleft}
\textbf{Lipschitz bounds and compactness in }$d$\textbf{\ dimensions}
\end{flushleft}

Following the non-collapse result established in the previous section, the
next lemma provides the $d$-dimensional analogue of \textit{Lemma~\ref{L-Lip}%
}. This result is central to establishing the compactness of the family.

\begin{lemma}
\noindent \label{L-Lip-d} For each $n\geq 1$, the map $L_{n}:[0,1]^{d}%
\rightarrow \lbrack 0,1]$ is uniformly Lipschitz with respect to the $\ell
^{1}$--norm. More precisely, 
\begin{equation}
|L_{n}(x)-L_{n}(y)|\leq \frac{1}{c_{d}}\Vert x-y\Vert _{1},\qquad x,y\in
\lbrack 0,1]^{d}.  \label{4002}
\end{equation}

Consequently, the family $(L_{n})_{n\geq 1}$ is $\sqrt{d}/c_{d}$-Lipschitz
with respect to the Euclidean norm. Moreover, since under the standing
density assumptions $L_{0}\in C([0,1]^{d})$, the full family $(L_{n})_{n\geq
0}$ is relatively compact in $C([0,1]^{d})$ with the uniform topology.

\noindent In particular, for any sequence $n_{k}\rightarrow \infty $ there
exists a subsequence $n_{k_{\ell }}$ and a continuous distribution function $%
L^{\ast }$ on $[0,1]^{d}$ such that 
\begin{equation}
\sup_{x\in \lbrack 0,1]^{d}}|L_{n_{k_{\ell }}}(x)-L^{\ast }(x)|\rightarrow
0\qquad \text{as }\ell \rightarrow \infty .  \label{4003}
\end{equation}%
The marginals $L_{i}^{n_{k_{\ell }}}$ converge uniformly on $[0,1]$ to the
marginals $L_{i}^{\ast }$ of $L^{\ast }$. Each $L_{i}^{\ast }$ is a
continuous distribution function on $[0,1]$.
\end{lemma}

\begin{proof}
We first prove \textit{uniform Lipschitz bounds} for the one-dimensional
marginals. Recall that 
\begin{equation}
L_{i}^{n}(x)=L_{n}(1,\ldots ,1,x,1,\ldots ,1),  \label{4004}
\end{equation}%
where $x$ appears in the $i$-th coordinate. Fix $i\in \{1,\ldots ,d\}$, $%
n\geq 0$, and $0\leq a\leq b\leq 1$. Put 
\begin{equation}
\alpha =L_{i}^{n,-1}(a),\qquad \beta =L_{i}^{n,-1}(b).  \label{4005}
\end{equation}%
Taking all coordinates except the $i$-th one equal to $1$ in the iteration
formula (\ref{941b}), we obtain 
\begin{equation}
L_{i}^{n+1}(x)=\frac{\int_{[0,1]^{i-1}}\int_{0}^{L_{i}^{n,-1}(x)}%
\int_{[0,1]^{d-i}}\left( \prod_{j=1}^{d}u_{j}\right) \,dL_{n}(u_{1},\ldots
,u_{d})}{I_{n}}.  \label{4006}
\end{equation}%
Therefore 
\begin{equation}
L_{i}^{n+1}(b)-L_{i}^{n+1}(a)=\frac{\int_{[0,1]^{i-1}}\int_{\alpha }^{\beta
}\int_{[0,1]^{d-i}}\left( \prod_{j=1}^{d}u_{j}\right) \,dL_{n}(u_{1},\ldots
,u_{d})}{I_{n}}.  \label{4007}
\end{equation}%
Since $0\leq \prod_{j=1}^{d}u_{j}\leq 1$ on $[0,1]^{d}$, it follows that 
\begin{equation}
L_{i}^{n+1}(b)-L_{i}^{n+1}(a)\leq \frac{\int_{\lbrack
0,1]^{i-1}}\int_{\alpha }^{\beta }\int_{[0,1]^{d-i}}\,dL_{n}(u_{1},\ldots
,u_{d})}{I_{n}}.  \label{4008}
\end{equation}%
The numerator in (\ref{4008}) is exactly the $i$-th marginal mass assigned
by $L_{i}^{n}$ to the interval $[\alpha ,\beta ]$. Under the standing
regularity assumptions used in the paper, the marginals are continuous and
strictly increasing, and hence 
\begin{equation}
L_{i}^{n}(\alpha )=a,\qquad L_{i}^{n}(\beta )=b.  \label{4009}
\end{equation}%
Thus 
\begin{equation}
\int_{\lbrack 0,1]^{i-1}}\int_{\alpha }^{\beta
}\int_{[0,1]^{d-i}}\,dL_{n}(u_{1},\ldots ,u_{d})=L_{i}^{n}(\beta
)-L_{i}^{n}(\alpha )=b-a.  \label{4010}
\end{equation}%
Combining (\ref{4008}), (\ref{4010}), and (\ref{4001}), we get 
\begin{equation}
L_{i}^{n+1}(b)-L_{i}^{n+1}(a)\leq \frac{b-a}{I_{n}}\leq \frac{1}{c_{d}}(b-a).
\label{4011}
\end{equation}%
Since $i$ was arbitrary, the marginals $\{L_{i}^{n}\}_{n\geq 1}$ are \textit{%
uniformly Lipschitz} on $[0,1]$, with common \textit{Lipschitz constant} $%
1/c_{d}$.

\noindent We now pass from the marginal estimates to a \textit{Lipschitz
estimate} for the joint distribution functions. Let $x,y\in \lbrack 0,1]^{d}$%
. Define the intermediate points 
\begin{equation}
z^{(0)}=x,\qquad z^{(i)}=(y_{1},\ldots ,y_{i},x_{i+1},\ldots ,x_{d}),\quad
i=1,\ldots ,d.  \label{4012}
\end{equation}%
Then $z^{(d)}=y$, and by the triangle inequality, 
\begin{equation}
|L_{n}(x)-L_{n}(y)|\leq \sum_{i=1}^{d}|L_{n}(z^{(i)})-L_{n}(z^{(i-1)})|.
\label{4013}
\end{equation}%
For each $i$, the two points $z^{(i-1)}$ and $z^{(i)}$ differ only in their $%
i$-th coordinate. Suppose first that $x_{i}\leq y_{i}$. Then, if $%
U^{(n)}=(U_{1}^{(n)},\ldots ,U_{d}^{(n)})$ has distribution function $L_{n}$%
, the corresponding difference is the probability of an event contained in 
\begin{equation}
\{x_{i}<U_{i}^{(n)}\leq y_{i}\}.  \label{4014}
\end{equation}%
Consequently, 
\begin{equation}
0\leq L_{n}(z^{(i)})-L_{n}(z^{(i-1)})\leq L_{i}^{n}(y_{i})-L_{i}^{n}(x_{i}).
\label{4015}
\end{equation}%
If $y_{i}\leq x_{i}$, the same argument with $x_{i}$ and $y_{i}$
interchanged gives 
\begin{equation}
|L_{n}(z^{(i)})-L_{n}(z^{(i-1)})|\leq |L_{i}^{n}(y_{i})-L_{i}^{n}(x_{i})|.
\label{4016}
\end{equation}%
Thus, for all $i=1,\ldots ,d$, 
\begin{equation}
|L_{n}(z^{(i)})-L_{n}(z^{(i-1)})|\leq |L_{i}^{n}(y_{i})-L_{i}^{n}(x_{i})|.
\label{4017}
\end{equation}%
Using the marginal Lipschitz estimate (\ref{4011}), we obtain, for every $%
n\geq 1$, 
\begin{equation}
|L_{n}(x)-L_{n}(y)|\leq
\sum_{i=1}^{d}|L_{i}^{n}(y_{i})-L_{i}^{n}(x_{i})|\leq \frac{1}{c_{d}}%
\sum_{i=1}^{d}|y_{i}-x_{i}|=\frac{1}{c_{d}}\Vert y-x\Vert _{1}.  \label{4018}
\end{equation}%
This proves (\ref{4002}). Since 
\begin{equation}
\Vert y-x\Vert _{1}\leq \sqrt{d}\,\Vert y-x\Vert _{2},  \label{4019}
\end{equation}%
the same estimate gives a \textit{Euclidean} \textit{Lipschitz constant} $%
\sqrt{d}/c_{d}$.

\noindent The family $(L_{n})_{n\geq 1}$ is therefore equicontinuous. It is
also uniformly bounded, since 
\begin{equation}
0\leq L_{n}(x)\leq 1,\qquad x\in \lbrack 0,1]^{d}.  \label{4020}
\end{equation}%
By the \textit{Arzel\`{a}--Ascoli's theorem}, the family $(L_{n})_{n\geq 1}$
is relatively compact in $C([0,1]^{d})$ with the uniform topology.

\noindent It remains to include $L_{0}$. Under the standing assumption that
the starting distribution admits a density and that $0<E(%
\prod_{j=1}^{d}X_{j})<\infty $, the \textit{tilted measure} 
\begin{equation}
dQ_{F}(x_{1},\ldots ,x_{d})=\frac{\prod_{j=1}^{d}x_{j}}{E(%
\prod_{j=1}^{d}X_{j})}\,dF(x_{1},\ldots ,x_{d})  \label{4021}
\end{equation}%
is absolutely continuous with respect to \textit{Lebesgue measure}. Hence it
has no atoms, and the initial \textit{Lorenz curve} $L_{0}=L_{F}$ is
continuous on $[0,1]^{d}$. Adding the single continuous function $L_{0}$ to
a relatively compact subset of $C([0,1]^{d})$ does not affect relative
compactness. Therefore the full family $(L_{n})_{n\geq 0}$ is relatively
compact in $C([0,1]^{d})$.

\noindent Now let $n_{k}\rightarrow \infty $. Relative compactness gives a
subsequence $n_{k_{\ell }}$ and a continuous function $L^{\ast }\in
C([0,1]^{d})$ such that (\ref{4003}) holds. The uniform limit $L^{\ast }$ is
again a distribution function: coordinatewise monotonicity, the boundary
conditions, and the $d$-dimensional rectangle inequalities are all preserved
under uniform limits. Thus $L^{\ast }$ is a continuous distribution function
on $[0,1]^{d}$.
\end{proof}

We have also a result for the marginals.

\begin{corollary}
\textbf{\label{AA-d} }Let $n_{k}\rightarrow \infty $ be a subsequence such
that 
\begin{equation}
\sup_{x\in \lbrack 0,1]^{d}}|L_{n_{k}}(x)-L^{\ast }(x)|\rightarrow 0.
\label{4025}
\end{equation}%
Then, for each $i\in \{1,\ldots ,d\}$, 
\begin{equation}
\sup_{x\in \lbrack 0,1]}|L_{i}^{n_{k}}(x)-L_{i}^{\ast }(x)|\rightarrow 0,
\label{4026}
\end{equation}%
where 
\begin{equation}
L_{i}^{\ast }(x)=L^{\ast }(1,\ldots ,1,x,1,\ldots ,1).  \label{4027}
\end{equation}%
If, in addition, $L_{i}^{\ast }$ is strictly increasing, then 
\begin{equation}
T_{i}^{n_{k}}=L_{i}^{n_{k},-1}\rightarrow T_{i}^{\ast }=L_{i}^{\ast ,-1}
\label{4028}
\end{equation}%
uniformly on $[0,1]$.
\end{corollary}

\begin{proof}
\noindent Since each $L_{n}$ is a distribution function with strictly
increasing continuous marginals, standard arguments (monotonicity and
continuity are preserved under uniform convergence) show that $L^{\ast }$ is
also a distribution function on $[0,1]^{d}$ with strictly increasing
continuous marginals $L_{i}^{\ast }$, and $L_{i}^{n_{k_{\ell }}}\rightarrow
L_{i}^{\ast }$ uniformly on $[0,1]$. As $L_{i}^{\ast }$ is strictly
increasing and continuous, the uniform convergence of the generalized
inverses follows from \textit{Lemma~\ref{inverse}}.
\end{proof}

Combining \textit{Lemma}~\ref{L-Lip-d} and \textit{Corollary~\ref{AA-d}} and
passing to subsequences as needed, we obtain:

\begin{lemma}
\label{Lstar-d} Let $\{L_{n}\}_{n\geq 0}$ be the sequence generated by (\ref%
{941b}). Then for every sequence of indices $n_{k}\rightarrow \infty $ there
exists a subsequence (still denoted $n_{k}$) and a distribution function $%
L^{\ast }$ on $[0,1]^{d}$ such that 
\begin{equation}
\sup_{x\in \lbrack 0,1]^{d}}\left\vert L_{n_{k}}(x)-L^{\ast }(x)\right\vert
\rightarrow 0\quad \text{as }k\rightarrow \infty ,  \label{954}
\end{equation}%
and the marginals $L_{i}^{n_{k}}$ converge uniformly to the marginals $%
L_{i}^{\ast }$ of $L^{\ast }$, with inverses $T_{i}^{n_{k}}\rightarrow
T_{i}^{\ast }=(L_{i}^{\ast })^{-1}$ uniformly on $[0,1]$.\hfill\ 
\end{lemma}

\begin{flushleft}
\textbf{Convergence of }$\Phi _{n}$\textbf{: reduction to Appendices D.1 and
D.2}
\end{flushleft}

Assume $f$ is $MTP_{2}$. By \textit{Claim~\ref{Dd-transfer}}, each $l_{n}$
is $MTP_{2}$. By \textit{Claim~\ref{Dd-subdiag}}, the subdiagonality
hypothesis needed for the $TP$-route holds (on the coordinate slices/paths
used in the multivariate section). Together with the
equicontinuity/compactness results established earlier in the multivariate
section (e.g.\ \textit{Lemma~\ref{L-Lip-d}} and its \textit{Arzel\`{a}%
--Ascoli} consequence), this is exactly the input used in \textit{%
Appendix~D.1}. Hence:

\begin{conclusion}
\label{Dd-TP-cite} Under $MTP_{2}$, the proof of Appendix~D.1 applies
verbatim (with $d$ replacing $2$), yielding the same structural convergence
input required by the final cancellation step.
\end{conclusion}

Assume $f$ is $SMRR_{2}$. By \textit{Claim~\ref{Dd-transfer}}, each $l_{n}$
is $SMRR_{2}$. By \textit{Claim~\ref{Dd-overdiag-T}} and \textit{Corollary~%
\ref{Dd-phi-crossing}}, the single-crossing hypothesis for the $\Phi _{n}$
maps is exactly as in \textit{Appendix~D.2} (including the strictness on
compact sets). Compactness is again provided by \textit{Lemma~\ref{L-Lip-d}}
and its consequences. Hence \textit{Appendix~D.2} applies verbatim and
yields:

\begin{theorem}
\label{Dd-phi-constant} Assume $f$ is $SMRR_{2}$. Then every subsequential
limit of the corresponding $\Phi _{n}$--maps (on compact subintervals of $%
(0,1)$) is a constant.
\end{theorem}

\begin{proof}
\medskip \noindent This is precisely the conclusion of \textit{Appendix~D.2}
once strict single-crossing and compactness hold. \medskip
\end{proof}

\newpage

\begin{flushleft}
{\Large Appendix E}
\end{flushleft}

The main body of the paper works in the regular, atom-free setting. This
appendix explains why this assumption is essential for the \textit{Fr\'{e}%
chet--Hoeffding lower-bound} comparison argument used in \textit{Lemma~\ref%
{lbound}}, and then studies the genuinely new dynamics that appear when
atoms are allowed. In particular, endpoint atoms may produce valid finite
iterates whose normalizing denominators converge to zero.

Throughout this appendix, atoms are treated according to the same
generalized-inverse convention as in the rest of the paper. Thus an atom is
not split over its rank interval\footnote{{\footnotesize If }$K(a-)=p$%
{\footnotesize \ and }$K(a)=p+m${\footnotesize , then the atom at }$a$%
{\footnotesize \ corresponds to the quantile plateau }$(p,p+m]$%
{\footnotesize , on which }$K^{-1}(u)=a${\footnotesize . Randomizing inside
the atom would spread its rank over this interval. With the
generalized-inverse convention used here, no such randomization is made: the
whole atom is assigned to one edge of the jump. }}; it is assigned to the
left edge of the corresponding marginal jump. This convention is exactly
where the continuous proofs and the atomic dynamics diverge.

It is interesting to note that one-dimensional \textit{Lorenz curve} studied
in \cite{[3]} is defined by the quantile integral 
\begin{equation*}
L_{F}(x)=\frac{\int_{0}^{x}F^{-1}(u)\,du}{\int_{0}^{1}F^{-1}(u)\,du},
\end{equation*}%
which automatically spreads an atom over its whole quantile interval. Hence
no separate atomic case is needed in the one-dimensional theorem. The
alternative threshold formula sometimes used in the economic literature 
\begin{equation*}
\frac{\int_{0}^{F^{-1}(x)}t\,dF(t)}{\int_{0}^{1}t\,dF(t)}
\end{equation*}%
coincides with the quantile formula when $F$ is continuous, but not in
general. Thus the continuous main body is consistent with \cite{[3]}, while
the present appendix studies the additional atomic dynamics created by the
generalized-inverse convention.

\begin{flushleft}
\textbf{Failure of Lemma~\ref{lbound} in the atomic case}
\end{flushleft}

We first identify the exact point at which the proof of \textit{Lemma~\ref%
{lbound}} uses atom-freeness. Introduce for convenience the notation 
\begin{equation*}
K_{n}=L_{1-}^{\,n},S_{n}=K_{n}^{-1}=T_{1-,n},J_{n}=I_{-,n}.
\end{equation*}%
In the \textit{Fr\'{e}chet--Hoeffding lower-bound} trajectory, the
normalizing denominator is 
\begin{equation*}
I_{-,n}=\int_{0}^{1}T_{1-,n}(u)T_{2-,n}(1-u)\,du.
\end{equation*}%
Under the complementary identity 
\begin{equation*}
T_{2-,n}(1-u)=1-S_{n}(u),
\end{equation*}%
this becomes 
\begin{equation*}
I_{-,n}=\int_{0}^{1}S_{n}(u)(1-S_{n}(u))\,du.
\end{equation*}

We now derive the formula for the next first marginal directly from the 
\textit{Lorenz update}. The bivariate \textit{lower-bound} iterate satisfies 
\begin{equation*}
L_{-}^{\,n+1}(x_{1},x_{2})=\frac{\int_{0}^{T_{1-,n}(x_{1})}%
\int_{0}^{T_{2-,n}(x_{2})}u_{1}u_{2}\,dL_{-}^{\,n}(u_{1},u_{2})}{I_{-,n}}.
\end{equation*}%
Setting $x_{2}=1$ gives the first marginal 
\begin{equation*}
K_{n+1}(x)=L_{-}^{\,n+1}(x,1)=\frac{\int_{0}^{T_{1-,n}(x)}%
\int_{0}^{1}u_{1}u_{2}\,dL_{-}^{\,n}(u_{1},u_{2})}{I_{-,n}}.
\end{equation*}%
Equivalently, if $(Y_{1},Y_{2})\sim L_{-}^{\,n}$, then 
\begin{equation}
K_{n+1}(x)=\frac{E\!\left[ Y_{1}Y_{2}1_{\{Y_{1}\leq T_{1-,n}(x)\}}\right] }{%
E[Y_{1}Y_{2}]}.  \label{eq:Knplus1-expectation}
\end{equation}%
Since $T_{1-,n}=S_{n}$, the event in the numerator is 
\begin{equation*}
\{Y_{1}\leq S_{n}(x)\}.
\end{equation*}%
For the \textit{Fr\'{e}chet--Hoeffding lower-bound} \textit{coupling} we may
write 
\begin{equation*}
Y_{1}=S_{n}(U),\qquad Y_{2}=T_{2-,n}(1-U),\qquad U\sim U(0,1).
\end{equation*}%
This gives 
\begin{equation*}
Y_{1}Y_{2}=S_{n}(U)(1-S_{n}(U)),
\end{equation*}%
and the event $\{Y_{1}\leq S_{n}(x)\}$ becomes 
\begin{equation*}
\{S_{n}(U)\leq S_{n}(x)\}.
\end{equation*}%
Substituting this into (\ref{eq:Knplus1-expectation}), we obtain the exact
generalized-inverse formula 
\begin{equation}
K_{n+1}(x)=\frac{\int_{\{u:\,S_{n}(u)\leq S_{n}(x)\}}S_{n}(u)(1-S_{n}(u))\,du%
}{\int_{0}^{1}S_{n}(u)(1-S_{n}(u))\,du}.  \label{eq:atomic-exact-K-update}
\end{equation}

In the proof of \textit{Lemma~\ref{lbound}}, formula (\ref%
{eq:atomic-exact-K-update}) was replaced by 
\begin{equation*}
K_{n+1}(x)=\frac{\int_{0}^{x}S_{n}(u)(1-S_{n}(u))\,du}{J_{n}}.
\end{equation*}%
This replacement is valid only if 
\begin{equation}
\{u:\,S_{n}(u)\leq S_{n}(x)\}=[0,x]\quad \text{up to Lebesgue-null sets}.
\label{eq:atomic-no-plateau-condition}
\end{equation}%
Equivalently, 
\begin{equation}
K_{n}(S_{n}(x))=x.  \label{eq:atomic-KS-condition}
\end{equation}%
Condition~(\ref{eq:atomic-KS-condition}) holds when $K_{n}$ is continuous.
It may fail when $K_{n}$ has atoms, because then $S_{n}$ has flat pieces.

In general, since $S_{n}$ is nondecreasing, we have 
\begin{equation*}
\{u:\,S_{n}(u)\leq S_{n}(x)\}=[0,\alpha _{n}(x)]\quad \text{up to null sets},
\end{equation*}%
where 
\begin{equation}
\alpha _{n}(x)=K_{n}(S_{n}(x)).  \label{eq:alpha-n-def}
\end{equation}%
Thus the correct atomic formula is 
\begin{equation}
K_{n+1}(x)=\frac{\int_{0}^{\alpha _{n}(x)}S_{n}(u)(1-S_{n}(u))\,du}{J_{n}}%
,\qquad \alpha _{n}(x)=K_{n}(S_{n}(x)).  \label{eq:atomic-corrected-K-update}
\end{equation}%
Formula~(\ref{345a30f}) is precisely the special case of (\ref%
{eq:atomic-corrected-K-update}) in which $\alpha _{n}(x)=x$.

This is not a harmless distinction. Suppose that $K_{n}$ has an atom at its
upper endpoint $1$, so that 
\begin{equation*}
K_{n}(1-)=p<1,\qquad K_{n}(1)=1.
\end{equation*}%
Then 
\begin{equation*}
S_{n}(x)=1,\qquad p<x\leq 1.
\end{equation*}%
Hence, for $p<x\leq 1$, 
\begin{equation*}
\alpha _{n}(x)=K_{n}(S_{n}(x))=K_{n}(1)=1.
\end{equation*}%
The correct formula~(\ref{eq:atomic-corrected-K-update}) gives 
\begin{equation*}
K_{n+1}(x)=\frac{\int_{0}^{1}S_{n}(u)(1-S_{n}(u))\,du}{J_{n}}=1,\qquad
p<x\leq 1.
\end{equation*}%
Thus the whole upper-endpoint atom is assigned to the left edge of the jump.
By contrast, formula~(\ref{345a30f}) would give 
\begin{equation*}
K_{n+1}(x)=\frac{\int_{0}^{x}S_{n}(u)(1-S_{n}(u))\,du}{J_{n}},
\end{equation*}%
which is generally strictly smaller than $1$ for $x<1$.

Therefore the first failure of the proof of \textit{Lemma~\ref{lbound}} in
the atomic case is the identity (\ref{345a30f}). Consequently the next step, 
\begin{equation*}
q_{n}(u)=K_{n+1}^{\prime }(u)=\frac{S_{n}(u)(1-S_{n}(u))}{J_{n}},
\end{equation*}%
is not justified either. If $K_{n}$ has an atom, then $K_{n+1}$ may have a
jump, so $K_{n+1}$ need not be absolutely continuous. The layer-cake
argument for unimodal densities leading to (\ref{345a40}) and the sharper
estimate (\ref{345a42}) therefore does not apply. Hence \textit{Lemma~\ref%
{lbound}}, \textit{Lemma~\ref{denominator-noncollapse-2d}}, and \textit{%
Corollary~\ref{srr2-denominator-noncollapse}} must be read under the
atom-free hypotheses used in the main body.

\begin{flushleft}
\bigskip \textbf{The atom-free condition expressed in terms of the initial
law}
\end{flushleft}

We now translate the previous condition into the initial law $F$. Let $%
F_{1},F_{2}$ be the one-dimensional marginals of $F$, and write 
\begin{equation*}
Q_{i}=F_{i}^{-1},\qquad i=1,2.
\end{equation*}%
The \textit{Fr\'{e}chet--Hoeffding lower-bound} coupling associated with $%
F_{1},F_{2}$ is 
\begin{equation*}
X_{1}=Q_{1}(U),\qquad X_{2}=Q_{2}(1-U),\qquad U\sim U(0,1).
\end{equation*}%
Its initial \textit{lower-bound} normalizer is 
\begin{equation}
J^{Q}=\int_{0}^{1}Q_{1}(u)Q_{2}(1-u)\,du.  \label{eq:atomic-J0}
\end{equation}%
Assume $J^{Q}>0$. The first marginal of the initial lower extremal \textit{%
Lorenz curve} is 
\begin{equation}
L_{1-}^{0}(x)=\frac{\int_{\{u:\,Q_{1}(u)\leq
Q_{1}(x)\}}Q_{1}(u)Q_{2}(1-u)\,du}{J^{Q}}.  \label{eq:atomic-L10-F}
\end{equation}%
Similarly, 
\begin{equation}
L_{2-}^{0}(x)=\frac{\int_{\{u:\,Q_{2}(1-u)\leq
Q_{2}(x)\}}Q_{1}(u)Q_{2}(1-u)\,du}{J^{Q}}.  \label{eq:atomic-L20-F}
\end{equation}

Let $y$ be an atom of $F_{1}$. Put 
\begin{equation*}
a_{y}=F_{1}(y-),\qquad b_{y}=F_{1}(y).
\end{equation*}%
Then $Q_{1}(u)=y$ for $u\in (a_{y},b_{y}]$. The jump of $L_{1-}^{0}$ created
by this atom has size 
\begin{equation}
\Delta L_{1-}^{0}(y)=\frac{y\int_{a_{y}}^{b_{y}}Q_{2}(1-u)\,du}{J^{Q}}.
\label{eq:atomic-L10-jump}
\end{equation}%
Therefore 
\begin{equation}
L_{1-}^{0}\ \text{is continuous}\quad \Longleftrightarrow \quad
y\int_{F_{1}(y-)}^{F_{1}(y)}Q_{2}(1-u)\,du=0\quad \text{for every atom }y%
\text{ of }F_{1}.  \label{eq:atomic-L10-continuity-equivalence}
\end{equation}

Likewise, if $z$ is an atom of $F_{2}$, then the jump of $L_{2-}^{0}$
created by this atom has size 
\begin{equation}
\Delta L_{2-}^{0}(z)=\frac{z\int_{F_{2}(z-)}^{F_{2}(z)}Q_{1}(1-v)\,dv}{J^{Q}}%
.  \label{eq:atomic-L20-jump}
\end{equation}%
Consequently 
\begin{equation}
L_{2-}^{0}\ \text{is continuous}\quad \Longleftrightarrow \quad
z\int_{F_{2}(z-)}^{F_{2}(z)}Q_{1}(1-v)\,dv=0\quad \text{for every atom }z%
\text{ of }F_{2}.  \label{eq:atomic-L20-continuity-equivalence}
\end{equation}

Thus the exact minimal condition on the initial law $F$, for the \textit{%
lower-bound} marginals to be atom-free at the initial step, is (\ref%
{eq:atomic-L10-continuity-equivalence}) and (\ref%
{eq:atomic-L20-continuity-equivalence}), together with $J^{Q}>0$.

A simpler sufficient assumption is 
\begin{equation}
F_{1}\ \text{and}\ F_{2}\ \text{are continuous and nondegenerate},\qquad
J^{Q}>0.  \label{ass:atomic-simple-continuity}
\end{equation}%
Here nondegenerate means that the marginal is not a \textit{Dirac mass}.
Under (\ref{ass:atomic-simple-continuity}), the atom conditions above hold
automatically, $L_{1-}^{0}$ and $L_{2-}^{0}$ are continuous, and the proof
of \textit{Lemma~\ref{lbound}} applies as written.

\begin{flushleft}
\textbf{Connection with the Dall'Aglio endpoint cases}
\end{flushleft}

We do not repeat here the full Dall'Aglio classification discussed earlier
in \textit{Appendix D.4}. We only record the point that is needed for the
present appendix.

For $d\geq 3$, the formal \textit{Fr\'{e}chet--Hoeffding lower-bound}
expression 
\begin{equation*}
F^{\wedge }(x_{1},\ldots ,x_{d})=\left(
\sum_{i=1}^{d}F_{i}(x_{i})-d+1\right) _{+}
\end{equation*}%
is a genuine $d$-dimensional distribution only in exceptional cases. As
recalled in \textit{Appendix D.4}, apart from the effectively bivariate
case, these are endpoint-jump cases. In the notation used there, 
\begin{equation*}
q_{i}=P(X_{i}>\ell _{i}),\qquad p_{i}=P(X_{i}<r_{i}),
\end{equation*}%
the lower-endpoint case is characterized by 
\begin{equation*}
\sum_{i=1}^{d}q_{i}\leq 1,
\end{equation*}%
and the upper-endpoint case is characterized by 
\begin{equation*}
\sum_{i=1}^{d}p_{i}\leq 1.
\end{equation*}%
Thus the endpoint cases are precisely the atomic cases in which the
generalized-inverse convention can create rank jumps. The examples below
show that these endpoint atoms can lead to denominator collapse.

\begin{flushleft}
\textbf{A bivariate endpoint-atomic lower-bound family}
\end{flushleft}

We now give a bivariate example for which the iteration is closed and
explicit. Let $\mu _{n}$ denote the probability measure whose distribution
function is the $n$-th \textit{Lorenz iterate}. Fix $a_{0}\in (0,1)$, set 
\begin{equation*}
s_{0}=1-a_{0},
\end{equation*}%
and let $W_{0}$ be a random variable on $[0,1]$. Define 
\begin{equation}
\mu _{0}=a_{0}\delta _{(1,1)}+\frac{s_{0}}{2}\mathrm{Law}(W_{0},1)+\frac{%
s_{0}}{2}\mathrm{Law}(1,W_{0}).  \label{eq:atomic-biv-star-initial}
\end{equation}%
This is a bivariate \textit{Fr\'{e}chet--Hoeffding lower-bound} law with
endpoint atoms. Indeed, if $p_{0}=\frac{s_{0}}{2}$ and $H_{0}$ is the
distribution function of $W_{0}$, then the common marginal quantile is 
\begin{equation*}
Q_{0}(u)=%
\begin{cases}
H_{0}^{-1}(u/p_{0}), & 0<u\leq p_{0}, \\ 
1, & p_{0}<u\leq 1.%
\end{cases}%
\end{equation*}%
The countermonotone coupling 
\begin{equation*}
(Q_{0}(U),Q_{0}(1-U)),\qquad U\sim U(0,1),
\end{equation*}%
has precisely the law (\ref{eq:atomic-biv-star-initial}).

Suppose that at step $n$ the law has the form 
\begin{equation}
\mu _{n}=a_{n}\delta _{(r_{n},r_{n})}+\frac{s_{n}}{2}\mathrm{Law}%
(r_{n}W_{n},r_{n})+\frac{s_{n}}{2}\mathrm{Law}(r_{n},r_{n}W_{n}),\qquad
s_{n}=1-a_{n}.  \label{eq:atomic-biv-star-n}
\end{equation}%
Let 
\begin{equation*}
\theta _{n}=E[W_{n}],\qquad D_{n}=a_{n}+s_{n}\theta _{n}.
\end{equation*}%
Then the denominator of the \textit{Lorenz iteration} is 
\begin{equation}
I_{n}=\int x_{1}x_{2}\,d\mu _{n}(x_{1},x_{2})=r_{n}^{2}(a_{n}+s_{n}\theta
_{n})=r_{n}^{2}D_{n}.  \label{eq:atomic-biv-In}
\end{equation}

The product tilt multiplies the top atom by $r_n^2$ and each arm by $%
r_n^2W_n $. Therefore the new top mass and total arm mass are 
\begin{equation}
a_{n+1} = \frac{a_n}{a_n+s_n\theta_n}, \qquad s_{n+1} = \frac{s_n\theta_n}{%
a_n+s_n\theta_n}.  \label{eq:atomic-biv-mass-recursion}
\end{equation}
For each coordinate, the mass below the old upper endpoint $r_n$ is $s_n/2$.
Since the generalized-inverse convention sends the upper endpoint atom to
the left edge of its jump, the next upper endpoint is 
\begin{equation}
r_{n+1}=\frac{s_n}{2}.  \label{eq:atomic-biv-r-recursion}
\end{equation}

It remains to describe the law of the new arm variable. Let $H_{n}$ be the
distribution function of $W_{n}$. Let $\widehat{W}_{n}$ be the size-biased
version of $W_{n}$: 
\begin{equation}
P(\widehat{W}_{n}\in dw)=\frac{w}{\theta _{n}}P(W_{n}\in dw).
\label{eq:atomic-size-biased}
\end{equation}%
On an arm, after product tilting, the varying coordinate is governed by $%
\widehat{W}_{n}$. Its old marginal rank inside the arm is $H_{n}(\widehat{W}%
_{n})$. After division by the new endpoint $r_{n+1}=s_{n}/2$, the new arm
variable is therefore 
\begin{equation}
W_{n+1}=H_{n}(\widehat{W}_{n}).  \label{eq:atomic-biv-W-recursion}
\end{equation}%
Thus the class (\ref{eq:atomic-biv-star-n}) is invariant.

The scalar mass ratio has the particularly simple recursion 
\begin{equation}
\frac{s_{n+1}}{a_{n+1}}=\theta _{n}\frac{s_{n}}{a_{n}}.
\label{eq:atomic-biv-ratio}
\end{equation}%
Writing 
\begin{equation*}
R_{n}=\frac{s_{n}}{a_{n}},
\end{equation*}%
we obtain 
\begin{equation}
R_{n}=R_{0}\prod_{k=0}^{n-1}\theta _{k},\qquad a_{n}=\frac{1}{1+R_{n}}%
,\qquad s_{n}=\frac{R_{n}}{1+R_{n}}.  \label{eq:atomic-biv-closed-masses}
\end{equation}%
Moreover, 
\begin{equation}
r_{n+1}=\frac{s_{n}}{2}=\frac{R_{n}}{2(1+R_{n})}.
\label{eq:atomic-biv-closed-r}
\end{equation}

\begin{flushleft}
\textbf{The beta-arm subclass and explicit collapse}
\end{flushleft}

Assume now that 
\begin{equation*}
W_{0}\sim \func{Beta}(\kappa _{0},1),\qquad \kappa _{0}>0.
\end{equation*}%
If 
\begin{equation*}
W_{n}\sim \func{Beta}(\kappa _{n},1),
\end{equation*}%
then 
\begin{equation*}
H_{n}(w)=w^{\kappa _{n}},\qquad \theta _{n}=E[W_{n}]=\frac{\kappa _{n}}{%
\kappa _{n}+1}.
\end{equation*}%
The size-biased version satisfies 
\begin{equation*}
\widehat{W}_{n}\sim \func{Beta}(\kappa _{n}+1,1).
\end{equation*}%
Hence 
\begin{equation*}
W_{n+1}=H_{n}(\widehat{W}_{n})=\widehat{W}_{n}^{\kappa _{n}}\sim \func{Beta}%
\left( 1+\frac{1}{\kappa _{n}},1\right) .
\end{equation*}%
Therefore the beta family is invariant and 
\begin{equation}
\kappa _{n+1}=1+\frac{1}{\kappa _{n}}.  \label{eq:atomic-kappa-recursion}
\end{equation}%
The sequence $(\kappa _{n})$ converges to 
\begin{equation*}
\kappa _{\ast }=\frac{1+\sqrt{5}}{2}.
\end{equation*}%
Consequently 
\begin{equation*}
\theta _{n}=\frac{\kappa _{n}}{\kappa _{n}+1}\longrightarrow \frac{\kappa
_{\ast }}{\kappa _{\ast }+1}=\frac{1}{\kappa _{\ast }}<1.
\end{equation*}%
Thus the product in (\ref{eq:atomic-biv-closed-masses}) converges to zero
exponentially, and 
\begin{equation*}
R_{n}\rightarrow 0,\qquad s_{n}\rightarrow 0,\qquad a_{n}\rightarrow
1,\qquad r_{n}\rightarrow 0.
\end{equation*}%
Using (\ref{eq:atomic-biv-In}), we conclude 
\begin{equation}
I_{n}=r_{n}^{2}(a_{n}+s_{n}\theta _{n})\longrightarrow 0.
\label{eq:atomic-biv-In-collapse}
\end{equation}

In the special case $W_{0}\sim U(0,1)$, one has $\kappa _{0}=1$, and 
\begin{equation*}
\kappa _{n}=\frac{F_{n+2}}{F_{n+1}},
\end{equation*}%
where $F_{1}=F_{2}=1$ are the \textit{Fibonacci numbers}. Hence 
\begin{equation*}
\theta _{n}=\frac{\kappa _{n}}{\kappa _{n}+1}=\frac{F_{n+2}}{F_{n+3}},
\end{equation*}%
and therefore 
\begin{equation*}
R_{n}=R_{0}\prod_{k=0}^{n-1}\frac{F_{k+2}}{F_{k+3}}=\frac{R_{0}}{F_{n+2}}.
\end{equation*}%
Thus, in the uniform-arm case, 
\begin{equation}
a_{n}=\frac{F_{n+2}}{F_{n+2}+R_{0}},\qquad s_{n}=\frac{R_{0}}{F_{n+2}+R_{0}}%
,\qquad r_{n+1}=\frac{R_{0}}{2(F_{n+2}+R_{0})}.
\label{eq:atomic-uniform-closed-form}
\end{equation}%
This gives a fully explicit collapsing orbit.

This example shows that the statement 
\begin{equation*}
\inf_{n}I_{-,n}>0
\end{equation*}%
is false for arbitrary endpoint-atomic lower-bound trajectories under the
literal generalized-inverse convention. The correct statement is that 
\textit{Lemma~\ref{lbound}} proves noncollapse for the atom-free lower-bound
trajectory.

\begin{flushleft}
\textbf{The }$d$\textbf{-dimensional upper-endpoint atomic family}
\end{flushleft}

The same mechanism works in all dimensions. Let $d\geq 2$. Fix $a_{0}\in
(0,1)$, set $s_{0}=1-a_{0}$, and let $W_{0}$ be a random variable on $[0,1]$%
. Define 
\begin{equation}
\mu _{0}=a_{0}\delta _{\mathbf{1}}+\frac{s_{0}}{d}\sum_{i=1}^{d}\mathrm{Law}%
(1,\ldots ,1,W_{0},1,\ldots ,1),  \label{eq:atomic-d-star-initial}
\end{equation}%
where $W_{0}$ appears in the $i$-th coordinate and 
\begin{equation*}
\mathbf{1}=(1,\ldots ,1).
\end{equation*}%
For each coordinate, 
\begin{equation*}
P_{\mu _{0}}(X_{i}<1)=\frac{s_{0}}{d},
\end{equation*}%
and so 
\begin{equation*}
\sum_{i=1}^{d}P_{\mu _{0}}(X_{i}<1)=s_{0}\leq 1.
\end{equation*}%
Thus (\ref{eq:atomic-d-star-initial}) belongs to the upper-endpoint
Dall'Aglio case.

Assume that at step $n$ 
\begin{equation}
\mu_n = a_n\delta_{r_n\mathbf{1}} + \frac{s_n}{d}\sum_{i=1}^d \mathrm{Law}%
(r_n,\ldots,r_n,r_nW_n,r_n,\ldots,r_n), \qquad s_n=1-a_n.
\label{eq:atomic-d-star-n}
\end{equation}
Let 
\begin{equation*}
\theta_n=E[W_n], \qquad D_n=a_n+s_n\theta_n.
\end{equation*}
Then 
\begin{equation}
I_n = \int \prod_{i=1}^d x_i\,d\mu_n(x) = r_n^d(a_n+s_n\theta_n) = r_n^dD_n.
\label{eq:atomic-d-star-In}
\end{equation}
Exactly as in the bivariate case, product tilting gives 
\begin{equation}
a_{n+1} = \frac{a_n}{a_n+s_n\theta_n}, \qquad s_{n+1} = \frac{s_n\theta_n}{%
a_n+s_n\theta_n}.  \label{eq:atomic-d-star-mass-recursion}
\end{equation}
For each coordinate, the mass below the old upper endpoint $r_n$ is $s_n/d$.
Hence 
\begin{equation}
r_{n+1}=\frac{s_n}{d}.  \label{eq:atomic-d-star-r-recursion}
\end{equation}
The arm variable evolves by the same size-biased rank recursion: 
\begin{equation}
W_{n+1}=H_n(\widehat W_n), \qquad P(\widehat W_n\in dw) = \frac{w}{\theta_n}%
P(W_n\in dw).  \label{eq:atomic-d-star-W-recursion}
\end{equation}
Therefore 
\begin{equation}
\frac{s_{n+1}}{a_{n+1}} = \theta_n\frac{s_n}{a_n}.
\label{eq:atomic-d-star-ratio}
\end{equation}

If $W_{0}\sim \func{Beta}(\kappa _{0},1)$, then again 
\begin{equation*}
\kappa _{n+1}=1+\frac{1}{\kappa _{n}}.
\end{equation*}%
Consequently $s_{n}/a_{n}\rightarrow 0$, so 
\begin{equation*}
s_{n}\rightarrow 0,\qquad a_{n}\rightarrow 1,\qquad r_{n}\rightarrow 0.
\end{equation*}%
By (\ref{eq:atomic-d-star-In}), 
\begin{equation*}
I_{n}=r_{n}^{d}(a_{n}+s_{n}\theta _{n})\longrightarrow 0.
\end{equation*}%
Thus the $d$-dimensional upper-endpoint star family gives a closed-form
atomic bad orbit for every $d\geq 2$.

\begin{flushleft}
\textbf{A nonsymmetric atomic family}
\end{flushleft}

The symmetry in the preceding example is not essential. Let 
\begin{equation*}
\mu _{n}=a_{n}\delta _{(r_{1,n},\ldots ,r_{d,n})}+\sum_{i=1}^{d}p_{i,n}%
\mathrm{Law}(r_{1,n},\ldots ,r_{i-1,n},r_{i,n}W_{i,n},r_{i+1,n},\ldots
,r_{d,n}),
\end{equation*}%
where 
\begin{equation*}
a_{n}+\sum_{i=1}^{d}p_{i,n}=1.
\end{equation*}%
Set 
\begin{equation*}
\theta _{i,n}=E[W_{i,n}],\qquad D_{n}=a_{n}+\sum_{i=1}^{d}p_{i,n}\theta
_{i,n}.
\end{equation*}%
Then 
\begin{equation*}
I_{n}=\left( \prod_{j=1}^{d}r_{j,n}\right) D_{n}.
\end{equation*}%
The mass recursion is 
\begin{equation}
a_{n+1}=\frac{a_{n}}{D_{n}},\qquad p_{i,n+1}=\frac{p_{i,n}\theta _{i,n}}{%
D_{n}}.  \label{eq:atomic-nonsymmetric-mass}
\end{equation}%
The new upper endpoint in coordinate $i$ is the old mass below the old upper
endpoint in that coordinate: 
\begin{equation}
r_{i,n+1}=p_{i,n}.  \label{eq:atomic-nonsymmetric-r}
\end{equation}%
Hence 
\begin{equation}
\frac{p_{i,n+1}}{a_{n+1}}=\theta _{i,n}\frac{p_{i,n}}{a_{n}}.
\label{eq:atomic-nonsymmetric-ratio}
\end{equation}%
If each $\theta _{i,n}$ is eventually bounded above by a constant strictly
smaller than $1$, then 
\begin{equation*}
p_{i,n}\rightarrow 0,\qquad r_{i,n}\rightarrow 0,\qquad I_{n}\rightarrow 0.
\end{equation*}%
Thus denominator collapse is not an artefact of exchangeability.

\begin{flushleft}
\textbf{Atomic paired countermonotonicity}
\end{flushleft}

We next record the corresponding correction for \textit{paired
countermonotonicity}. Let $d=2m$, and pair the coordinates as 
\begin{equation*}
(1,2),(3,4),\ldots ,(2m-1,2m).
\end{equation*}%
The continuous PCM formulas in the main text remain valid when the block
marginals are continuous. With atoms, the block factorization still holds,
but the one-dimensional limits in the block integrals must be corrected.

Consider one block $(i,j)$. At step $n$, let the block marginals be $%
L_{i}^{n},L_{j}^{n}$, with generalized inverses 
\begin{equation*}
T_{i}^{n}=(L_{i}^{n})^{-1},\qquad T_{j}^{n}=(L_{j}^{n})^{-1}.
\end{equation*}%
The countermonotone representation in this block is 
\begin{equation*}
(X_{i},X_{j})=(T_{i}^{n}(U),T_{j}^{n}(1-U)),\qquad U\sim U(0,1).
\end{equation*}%
Set 
\begin{equation*}
\mu _{ij,n}=\int_{0}^{1}T_{i}^{n}(u)T_{j}^{n}(1-u)\,du.
\end{equation*}%
Define the plateau endpoints 
\begin{equation*}
\alpha _{i}^{n}(x)=L_{i}^{n}(T_{i}^{n}(x)),\qquad \alpha
_{j}^{n}(x)=L_{j}^{n}(T_{j}^{n}(x)).
\end{equation*}%
Then the corrected atomic block marginal updates are 
\begin{equation}
L_{i}^{n+1}(x)=\frac{\int_{0}^{\alpha
_{i}^{n}(x)}T_{i}^{n}(u)T_{j}^{n}(1-u)\,du}{\mu _{ij,n}},
\label{eq:atomic-PCM-i}
\end{equation}%
and 
\begin{equation}
L_{j}^{n+1}(x)=\frac{\int_{1-\alpha
_{j}^{n}(x)}^{1}T_{i}^{n}(u)T_{j}^{n}(1-u)\,du}{\mu _{ij,n}}.
\label{eq:atomic-PCM-j}
\end{equation}%
Similarly, the corrected bivariate block \textit{Lorenz factor} is 
\begin{equation}
L_{ij}^{n+1}(x,y)=\frac{\int_{1-\alpha _{j}^{n}(y)}^{\alpha
_{i}^{n}(x)}T_{i}^{n}(u)T_{j}^{n}(1-u)\,du}{\mu _{ij,n}},
\label{eq:atomic-PCM-block}
\end{equation}%
with the convention that the numerator is zero if 
\begin{equation*}
\alpha _{i}^{n}(x)<1-\alpha _{j}^{n}(y).
\end{equation*}%
If the block marginals are continuous, then 
\begin{equation*}
\alpha _{i}^{n}(x)=x,\qquad \alpha _{j}^{n}(y)=y,
\end{equation*}%
and (\ref{eq:atomic-PCM-i})--(\ref{eq:atomic-PCM-block}) reduce to the
continuous PCM formulas.

\begin{flushleft}
\textbf{\bigskip Finite-support laws and the atomic Lorenz map}
\end{flushleft}

We now record the finite-support form of the atomic \textit{Lorenz map}. Let 
\begin{equation*}
\mu =\sum_{k=1}^{N}m_{k}\delta _{x^{(k)}},\qquad m_{k}>0,\qquad
\sum_{k=1}^{N}m_{k}=1,
\end{equation*}%
where 
\begin{equation*}
x^{(k)}=(x_{1}^{(k)},\ldots ,x_{d}^{(k)})\in \lbrack 0,1]^{d}.
\end{equation*}%
Assume 
\begin{equation*}
I(\mu )=\sum_{k=1}^{N}m_{k}\prod_{i=1}^{d}x_{i}^{(k)}>0.
\end{equation*}%
Define 
\begin{equation*}
w_{k}=\prod_{i=1}^{d}x_{i}^{(k)}.
\end{equation*}%
The product-tilted weights are 
\begin{equation*}
\widetilde{m}_{k}=\frac{m_{k}w_{k}}{I(\mu )}.
\end{equation*}%
Under the generalized-inverse convention, the rank image of the support
point $x^{(k)}$ is 
\begin{equation*}
\rho ^{(k)}=(\rho _{1}^{(k)},\ldots ,\rho _{d}^{(k)}),
\end{equation*}%
where 
\begin{equation}
\rho _{i}^{(k)}=\mu \{x:\,x_{i}<x_{i}^{(k)}\}.  \label{eq:atomic-left-rank}
\end{equation}%
Thus the atomic \textit{Lorenz image} is 
\begin{equation}
\mathcal{L}(\mu )=\sum_{k=1}^{N}\widetilde{m}_{k}\delta _{\rho ^{(k)}},
\label{eq:atomic-finite-map}
\end{equation}%
with atoms having the same rank vector merged. Formula (\ref%
{eq:atomic-finite-map}) is the finite-dimensional version of (\ref%
{eq:atomic-corrected-K-update}).

\begin{flushleft}
\textbf{Finite-support fixed points and boundary attractors}
\end{flushleft}

\bigskip Suppose that 
\begin{equation*}
\mu =\sum_{k=1}^{N}m_{k}\delta _{x^{(k)}}
\end{equation*}%
is an atomic fixed point and that no merging of rank images occurs. Then
there exists a permutation $\pi $ of $\{1,\ldots ,N\}$ such that 
\begin{equation*}
x^{(\pi (k))}=\rho ^{(k)}
\end{equation*}%
and 
\begin{equation}
m_{\pi (k)}=\frac{m_{k}w_{k}}{I(\mu )},\qquad
w_{k}=\prod_{i=1}^{d}x_{i}^{(k)}.  \label{eq:atomic-fixed-balance}
\end{equation}%
Following a cycle 
\begin{equation*}
k,\pi (k),\ldots ,\pi ^{\ell -1}(k),
\end{equation*}%
we obtain 
\begin{equation}
I(\mu )^{\ell }=\prod_{r=0}^{\ell -1}w_{\pi ^{r}(k)}.
\label{eq:atomic-cycle-condition}
\end{equation}%
Thus every cycle of the rank permutation must have product-weight geometric
mean equal to the same normalizer $I(\mu )$.

In the special case where the rank image fixes each atom individually, 
\begin{equation*}
\pi (k)=k,
\end{equation*}%
condition~(\ref{eq:atomic-fixed-balance}) gives 
\begin{equation*}
w_{k}=I(\mu )\qquad \text{for every support point }x^{(k)}.
\end{equation*}%
Therefore every atom must lie on the same product-level surface 
\begin{equation*}
\prod_{i=1}^{d}x_{i}=I(\mu ).
\end{equation*}%
This is a strong restriction.

The endpoint-star examples above do not converge to genuine interior fixed
points. For the bivariate family, 
\begin{equation*}
\mu _{n}=a_{n}\delta _{(r_{n},r_{n})}+\frac{s_{n}}{2}\mathrm{Law}%
(r_{n}W_{n},r_{n})+\frac{s_{n}}{2}\mathrm{Law}(r_{n},r_{n}W_{n}),
\end{equation*}%
the beta-arm computation gives 
\begin{equation*}
a_{n}\rightarrow 1,\qquad s_{n}\rightarrow 0,\qquad r_{n}\rightarrow 0.
\end{equation*}%
Hence 
\begin{equation*}
\mu _{n}\Rightarrow \delta _{(0,0)}.
\end{equation*}%
But 
\begin{equation*}
\int x_{1}x_{2}\,d\delta _{(0,0)}=0.
\end{equation*}%
Thus the limit is not a fixed point of the \textit{Lorenz map}; it is a
boundary attractor where the next normalizing denominator is zero. The same
conclusion holds in the $d$-dimensional star family: 
\begin{equation*}
\mu _{n}\Rightarrow \delta _{\mathbf{0}},\qquad I_{n}\rightarrow 0.
\end{equation*}

\begin{flushleft}
\textbf{Discrete complete mixability}
\end{flushleft}

Atomic examples also arise from discrete complete mixability. Suppose that
at step $n$ the support is indexed by $r=1,\ldots ,N$, with masses $%
m_{r}^{(n)}$, and that the coordinates are arranged by permutations $\sigma
_{i}\in S_{N}$: 
\begin{equation*}
x^{(r,n)}=(x_{1,\sigma _{1}(r)}^{(n)},\ldots ,x_{d,\sigma _{d}(r)}^{(n)}).
\end{equation*}%
A discrete complete mix satisfies 
\begin{equation}
\sum_{i=1}^{d}x_{i,\sigma _{i}(r)}^{(n)}=k_{n},\qquad r=1,\ldots ,N.
\label{eq:atomic-discrete-CM}
\end{equation}%
The \textit{Lorenz denominator} is 
\begin{equation*}
I_{n}=\sum_{r=1}^{N}m_{r}^{(n)}\prod_{i=1}^{d}x_{i,\sigma _{i}(r)}^{(n)}.
\end{equation*}%
The tilted masses are 
\begin{equation}
\widetilde{m}_{r}^{(n)}=\frac{m_{r}^{(n)}\prod_{i=1}^{d}x_{i,\sigma
_{i}(r)}^{(n)}}{I_{n}}.  \label{eq:atomic-discrete-tilt}
\end{equation}%
The next support is obtained by the left-rank map: 
\begin{equation}
x_{i,\sigma _{i}(r)}^{(n+1)}=\mu _{n}\{x:\,x_{i}<x_{i,\sigma
_{i}(r)}^{(n)}\},\qquad i=1,\ldots ,d.  \label{eq:atomic-discrete-rank}
\end{equation}%
Atoms with identical rank vectors are merged.

The complete-mixability constraint (\ref{eq:atomic-discrete-CM}) is not
automatically invariant under the coordinatewise rank map (\ref%
{eq:atomic-discrete-rank}). Therefore a closed all-$n$ discrete
complete-mixability theory requires an invariant subclass: after the tilt
and rank update, the new support must again be representable by permutations
satisfying a complete-mixability identity. In such subclasses, the iteration
is finite-dimensional and explicitly computable from (\ref%
{eq:atomic-discrete-tilt})--(\ref{eq:atomic-discrete-rank}).

\begin{flushleft}
\textbf{Atomic one-factor }$\Sigma $\textbf{-countermonotonicity}
\end{flushleft}

Finally, consider an atomic or one-factor model 
\begin{equation*}
X_{i}=q_{i}(U),\qquad i=1,\ldots ,d,\qquad U\sim U(0,1),
\end{equation*}%
and put 
\begin{equation*}
K(u)=\prod_{i=1}^{d}q_{i}(u),\qquad I=\int_{0}^{1}K(u)\,du.
\end{equation*}%
The one-step \textit{Lorenz marginal }is exactly 
\begin{equation}
L_{i}(x)=\frac{\int_{0}^{1}K(u)1_{\{q_{i}(u)\leq T_{i}(x)\}}\,du}{I},
\label{eq:atomic-one-factor-exact}
\end{equation}%
where $T_{i}$ is the marginal generalized inverse.

If $q_{i}=T_{i}\circ \pi _{i}$ for a measure-preserving map $\pi _{i}$, then
in the atom-free case one may simplify 
\begin{equation*}
\{q_{i}(u)\leq T_{i}(x)\}=\{\pi _{i}(u)\leq x\}\quad \text{up to null sets}.
\end{equation*}%
In the atomic case this simplification is not valid without a jump-splitting
convention. The correct formula is the event formula (\ref%
{eq:atomic-one-factor-exact}), or, in finite support, the left-rank map (\ref%
{eq:atomic-left-rank}).

Suppose the one-factor model is $\Sigma $\textit{-countermonotone} in the
sense that for every nonempty proper subset $A\subset \{1,\ldots ,d\}$ there
is a decreasing function $\varphi _{A}$ such that 
\begin{equation*}
\sum_{j\notin A}q_{j}(u)=\varphi _{A}\left( \sum_{i\in A}q_{i}(u)\right) .
\end{equation*}%
Product tilting preserves the support, but the subsequent coordinatewise
rank transformation need not preserve these functional relations. Hence, as
for discrete complete mixability, a closed all-$n$ atomic $\Sigma $\textit{%
-countermonotone} iteration requires an invariant subclass. Outside such
subclasses, $\Sigma $\textit{-countermonotonicity} gives correct one-step
formulas but does not by itself close the \textit{Lorenz dynamics}.

\newpage

\begin{flushleft}
{\Large Appendix F}
\end{flushleft}

\begin{figure}[tbph]
\centering\includegraphics[width=0.8\linewidth]{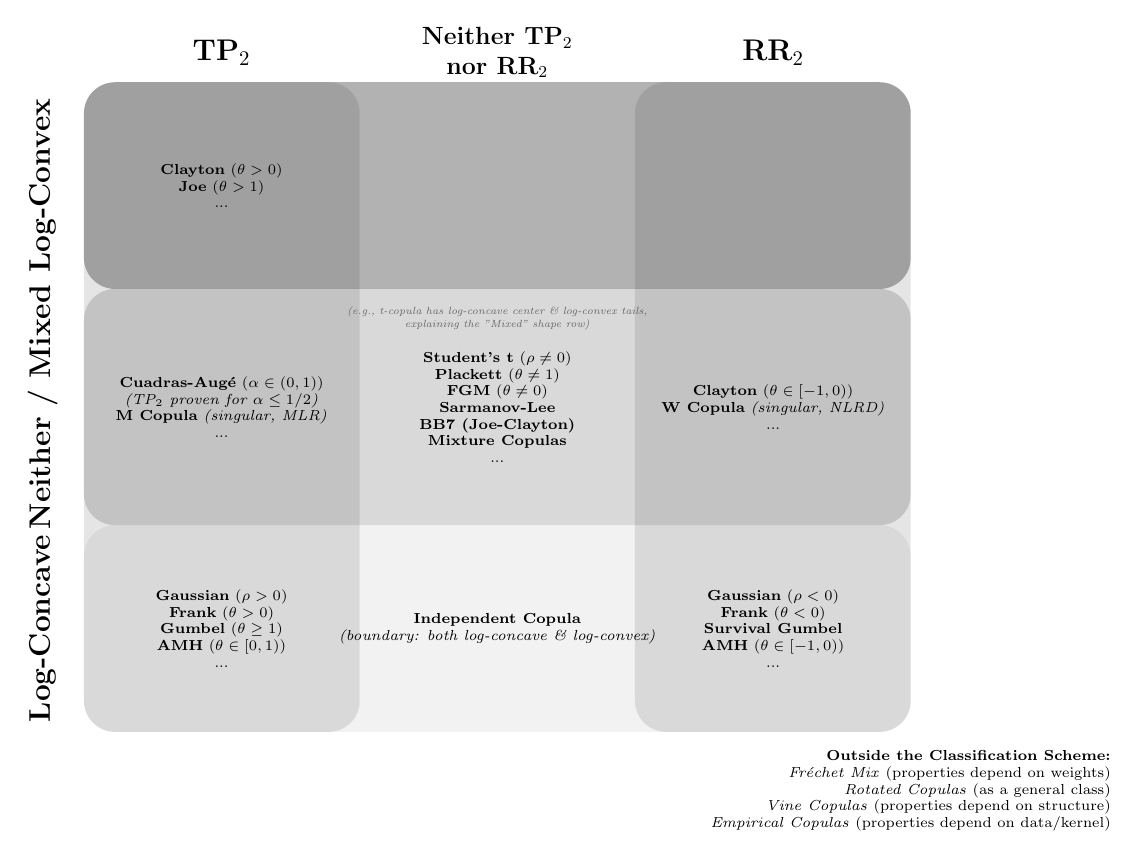}
\caption{Figure F1: Copulas properties}
\label{fig:figure15}
\end{figure}

\bigskip

\bigskip

\end{document}